\documentclass[12pt]{article}
\usepackage{amsmath}
\usepackage{amsthm}
\usepackage{amssymb}
\usepackage{amsfonts}
\usepackage{graphicx}
\usepackage{enumerate}
\usepackage{epsfig}
\usepackage{pxfonts}
\usepackage{amscd}
\usepackage[cmtip,all]{xy}
\usepackage{mathrsfs}
\usepackage{oldgerm}
\usepackage{fancybox}
\usepackage{draftcopy}

\normalsize
\newdir{ >}{!/-10pt/@{>}}

\newtheoremstyle{theoremstyle}
  {10pt}      
  {5pt}       
  {\itshape}  
  {}          
  {\bfseries} 
  {:}         
  {.5em}      
  {}          

\newtheoremstyle{examplestyle}
  {10pt}      
  {5pt}       
  {}          
  {}          
  {\bfseries} 
  {:}         
  {.5em}      
  {}          

\theoremstyle{theoremstyle}
\newtheorem{theorem}{Theorem}[section]
\newtheorem*{theorem*}{Theorem}
\newtheorem{lemma}[theorem]{Lemma}
\newtheorem{proposition}[theorem]{Proposition}
\newtheorem*{proposition*}{Proposition}

\newtheorem*{conjecture*}{Conjecture}
\newtheorem{corollary}[theorem]{Corollary}
\newtheorem*{corollary*}{Corollary}

\theoremstyle{examplestyle}
\newtheorem{example}[theorem]{Example}
\newtheorem{definition}[theorem]{Definition}
\newtheorem{definition*}{Definition}
\newtheorem{remark}[theorem]{Remark}
\newtheorem{remark*}{Remark}

\newcommand{\Hom}{\operatorname{Hom}}
\newcommand{\ev}{{\operatorname{ev}}}
\newcommand{\CC}{{\operatorname{C}}}
\newcommand{\Ev}{{\operatorname{Ev}}}
\newcommand{\Fr}{{\operatorname{Fr}}}
\newcommand{\Group}{{\operatorname{G}}}
\newcommand{\id}{{\operatorname{id}}}

\newcommand{\op}{{\operatorname{op}}}
\newcommand{\fin}{{\operatorname{fin}}}
\newcommand{\sgn}{{\operatorname{sgn}}}
\newcommand{\C}{\mathbb{C}}
\newcommand{\FF}{\mathbb{F}}
\newcommand{\AAA}{\mathbb{A}}
\newcommand{\PPP}{\mathbb{P}}
\newcommand{\Q}{\mathbb{Q}}
\newcommand{\R}{\mathbb{R}}
\newcommand{\N}{\mathbb{N}}
\newcommand{\ZZ}{\mathbb{Z}}

\newcommand{\D}{\mathcal{D}}
\newcommand{\A}{\mathcal{A}}
\newcommand{\Hilbert}{\mathcal{H}}
\newcommand{\K}{\mathcal{K}}
\newcommand{\T}{\mathcal{T}}
\newcommand{\GG}{\mathcal{G}}
\newcommand{\M}{\mathcal{M}}
\newcommand{\NN}{\mathcal{N}}
\newcommand{\QQ}{\mathcal{Q}}
\newcommand{\CCC}{\mathcal{C}}
\newcommand{\DDD}{\mathcal{D}}
\newcommand{\RRR}{\mathcal{R}}
\newcommand{\SSSS}{\mathcal{S}}
\newcommand{\TTTT}{\mathcal{T}}
 
\newcommand{\E}{\mathcal{E}}
\newcommand{\F}{\mathcal{F}}
\newcommand{\G}{\mathcal{G}}
\newcommand{\HHH}{\mathcal{H}}
\newcommand{\Hspt}{\mathcal{H}}
\newcommand{\X}{\mathcal{X}}
\newcommand{\Y}{\mathcal{Y}}
\newcommand{\Z}{\mathcal{Z}}
\newcommand{\W}{\mathcal{W}}
\newcommand{\U}{\mathcal{U}}
\newcommand{\V}{\mathcal{V}}
\newcommand{\Ch}{{\bf Ch}}
\newcommand{\bb}{{\bf b}}
\newcommand{\HH}{{\bf H}}
\newcommand{\HL}{{\bf HL}}
\newcommand{\KK}{{\bf KK}}
\newcommand{\MM}{{\bf M}}
\newcommand{\RR}{{\bf R}}
\newcommand{\SH}{{\bf SH}}
\newcommand{\Ho}{{\bf Ho}}
\newcommand{\htp}{{\bf ht}}
\newcommand{\Alg}{{\bf Alg}}
\newcommand{\Rep}{{\bf Rep}}
\newcommand{\Spc}{{\bf Spc}}
\newcommand{\Set}{{\bf Set}}

\newcommand{\Spt}{{\bf Spt}}
\newcommand{\Sym}{{\bf Sym}}
\newcommand{\Alt}{{\bf Alt}}
\newcommand{\Mod}{{\bf Mod}}
\newcommand{\Ab}{{\bf Ab}}
\newcommand{\SSS}{{\bf S}}

\newcommand{\ddet}{{\bf det}}
\newcommand{\Ex}{\operatorname{Ex}}
\newcommand{\cyl}{\operatorname{cyl}}

\newcommand{\ttop}{\operatorname{top}}
\newcommand{\Sing}{\operatorname{Sing}}
\newcommand{\cone}{\operatorname{cone}}
\newcommand{\colim}{\operatorname{colim}}
\newcommand{\holim}{\operatorname{holim}}
\newcommand{\hofib}{\operatorname{hofib}}

\newcommand{\hocolim}{\operatorname{hocolim}}
\newcommand{\cof}{\rightarrowtail}
\newcommand{\fib}{\twoheadrightarrow}
\newcommand{\ccc}{\mathsf{c}}
\newcommand{\tr}{\mathsf{tr}}
\newcommand{\fd}{\mathsf{fd}}
\newcommand{\eff}{\mathsf{eff}}
\newcommand{\f}{\mathsf{f}}
\newcommand{\ppt}{\mathsf{pt}}
\newcommand{\hht}{\mathsf{ht}}
\newcommand{\sst}{\mathsf{st}}

\title{{Homotopy theory of $\CC^{\ast}$-algebras}}
\author{Paul Arne {\O}stv{\ae}r
\footnote{Department of Mathematics, University of Oslo, Norway. 
\newline MSC Primary: 46L99, 55P99}}

\begin{document}
\maketitle

\begin{abstract}
In this work we construct from ground up a homotopy theory of $\CC^{\ast}$-algebras.
This is achieved in parallel with the development of classical homotopy theory by first introducing an unstable model structure 
and second a stable model structure.
The theory makes use of a full fledged import of homotopy theoretic techniques into the subject of $\CC^{\ast}$-algebras.
\vspace{0.1in}

The spaces in $\CC^{\ast}$-homotopy theory are certain hybrids of functors represented by $\CC^{\ast}$-algebras and spaces studied 
in classical homotopy theory.  
In particular, 
we employ both the topological circle and the $\CC^{\ast}$-algebra circle of complex-valued continuous functions on the real numbers 
which vanish at infinity.  
By using the inner workings of the theory, 
we may stabilize the spaces by forming spectra and bispectra with respect to either one of these circles or their tensor product.
These stabilized spaces or spectra are the objects of study in stable $\CC^{\ast}$-homotopy theory.
\vspace{0.1in}

The stable homotopy category of $\CC^{\ast}$-algebras gives rise to invariants such as stable homotopy groups and bigraded cohomology 
and homology theories.  
We work out examples related to the emerging subject of noncommutative motives and zeta functions of $\CC^{\ast}$-algebras.
In addition, 
we employ homotopy theory to define a new type of $K$-theory of $\CC^{\ast}$-algebras.
\end{abstract}

\newpage 
\tableofcontents

\newpage 
\section{Introduction}
In this work we present some techniques and results which lead to new invariants of $\CC^{\ast}$-algebras.
The fundamental organizational principle of $\CC^{\ast}$-homotopy theory infers there exists a 
homotopy theory of $\CC^{\ast}$-algebras determined by short exact sequences, 
matrix invariance and by complex-valued functions on the topological unit interval. 
We shall make this precise by constructing model structures on certain spaces which are build up of 
$\CC^{\ast}$-algebras in much the same way as every natural number acquires a prime factorization.
Our approach combines a new take on $\CC^{\ast}$-algebras dictated by category theory and the recently 
perfected homotopy theory of cubical sets.
The idea of combining $\CC^{\ast}$-algebras and cubical sets into a category of 
``cubical $\CC^{\ast}$-spaces'' may perhaps be perceived as quite abstract on a first encounter.
However, these spaces arise naturally from a homotopy theoretic viewpoint.
We observe next the failure of a more straightforward topological approach.
\vspace{0.1in}

By employing the classical homotopy lifting property formulated in \cite{Schochet:III} for maps between 
$\CC^{\ast}$-algebras one naturally arrives at the notion of a $\CC^{\ast}$-algebra cofibration.
The definition is rigged such that under the Gelfand-Naimark correspondence a map between locally
compact Hausdorff spaces $X\rightarrow Y$ is a topological cofibration if and only if the induced map 
$C_{0}(Y)\rightarrow C_{0}(X)$ is a $\CC^{\ast}$-algebra cofibration.
Now a standard argument shows every ${\ast}$-homomorphism factors as the composition of an injective 
homotopy equivalence and a $\CC^{\ast}$-algebra cofibration.
This might suggest to willing minds that there exists a bona fide model structure on $\CC^{\ast}$-algebras  
with fibrations the $\CC^{\ast}$-algebra cofibrations and weak equivalences the homotopy equivalences.
In this aspiring model structure every $\CC^{\ast}$-algebra is fibrant and for a suitable tensor product 
the suspension functor $\Sigma=C_0(\R)\otimes-$ acquires a left adjoint.
Thus for every diagram of $\CC^{\ast}$-algebras indexed by some ordinal $\lambda$ the suspension functor 
$\Sigma$ induces a homotopy equivalence
\begin{equation*}
\xymatrix{
\Sigma\prod_{i\in\lambda} A_i\ar[r] & \prod_{i\in\lambda}\Sigma A_i. }
\end{equation*}
But this map is clearly not a homotopy equivalence;
for example, applying $K_1$ to the countable constant diagram with value the complex numbers yields an 
injective map with image the subgroup of bounded sequences in $\prod_{\N}\mathbb{Z}$.
The categories of cubical and simplicial $\CC^{\ast}$-spaces introduced in this work offer alternate approaches to a homotopy theory 
employing constructions which are out of reach in the traditional confines of $\CC^{\ast}$-algebra theory, 
as in e.g.~\cite{BlackadarKbook} and \cite{Rosenberg:noncommutativetopology}.
\vspace{0.1in}

One of the main goals of $\CC^{\ast}$-homotopy theory is to provide a modern framework for 
cohomology and homology theories of $\CC^{\ast}$-algebras.
In effect, 
we introduce the stable homotopy category $\SH^{\ast}$ of $\CC^{\ast}$-algebras by 
stabilizing the model structure on cubical $\CC^{\ast}$-spaces with respect to the tensor product 
$C\equiv S^1\otimes C_0(\R)$ of the cubical circle $S^1$ and the nonunital $\CC^{\ast}$-algebra $C_0(\R)$.
Combining these two circles allow us to define a bigraded cohomology theory 
\begin{equation}
\label{cohomology}
\E^{p,q}(A)
\equiv
\SH^{\ast}\bigl(\Sigma_C^{\infty} A,S^{p-q}\otimes C_0(\R^q)\otimes\E\bigr)
\end{equation} 
and a bigraded homology theory
\begin{equation}
\label{homology}
\E_{p,q}(A)
\equiv
\SH^{\ast}\bigl(\Sigma_C^{\infty} S^{p-q}\otimes C_0(\R^q),A\otimes \E\bigr).
\end{equation} 
Here, 
$A$ is a $\CC^{\ast}$-algebra considered as a discrete cubical $\CC^{\ast}$-space and $\E$ is a 
$\CC^{\ast}$-spectrum in the stable homotopy category $\SH^{\ast}$.
Kasparov's $KK$-theory and Connes-Higson's $E$-theory, 
suitably extended to cubical $\CC^{\ast}$-spaces,
give rise to examples of $C$-spectra.
The precise definitions of the tensor product $\otimes$ and the suspension functor $\Sigma_C^{\infty}$ 
will be explained in the main body of the text.
An allied theory of noncommutative motives rooted in algebraic geometry is partially responsible for the choice of bigrading, 
see \cite{Ostvar:noncommutativemotives}. 
Inserting the sphere spectrum $\Sigma_C^{\infty}\C$ into the formula (\ref{homology}) yields a theory of 
bigraded stable homotopy groups which receives a canonical map from the classical homotopy groups of spheres.
An intriguing problem, 
which will not be attempted in this paper, 
is to compute the commutative endomorphism ring of $\Sigma_C^{\infty}\C$.
\vspace{0.1in}

The spaces in $\CC^{\ast}$-homotopy theory are convenient generalizations of $\CC^{\ast}$-algebras.
A $\CC^{\ast}$-space is build out of $\CC^{\ast}$-algebras considered as representable set-valued functors.
In unstable $\CC^{\ast}$-homotopy theory we work with cubical $\CC^{\ast}$-spaces $\X$. 
By definition, 
for every $\CC^{\ast}$-algebra $A$ we now get homotopically meaningful objects in form of cubical 
sets $\X(A)$.
Using the homotopy theory of cubical sets, 
which models the classical homotopy theory of topological spaces,
declare a map $\X\rightarrow\Y$ of cubical $\CC^{\ast}$-spaces to be a pointwise weak equivalence if 
$\X(A)\rightarrow\Y(A)$ is a weak equivalence of cubical sets. 
This is a useful but at the same time an extremely coarse notion of weak equivalence in our setting.
In order to introduce a much finer notion of weak equivalence reflecting the data of short exact sequences, 
matrix invariance and homotopy equivalence of $\CC^{\ast}$-algebras, 
we shall localize the pointwise model structure with respect to a set of maps in the category $\Box\Spc$ 
of cubical $\CC^{\ast}$-spaces.
We define the unstable homotopy category $\HH$ of $\CC^{\ast}$-algebras as the homotopy category of the 
localized model structure.

The homotopy category is universal in the sense that the localized model structure gives the initial example 
of a left Quillen functor $\Box\Spc\rightarrow\MM$ to some model category $\MM$ with the property that every 
member of the localizing set of maps derives to an isomorphism in the homotopy category of $\MM$. 
\vspace{0.1in}

Section \ref{section:Cstarhomotopytheory} details the constructions and some basic properties 
of the unstable model structures in $\CC^{\ast}$-homotopy theory.
Moreover, 
for the noble purpose of stabilizing, 
we note there exist entirely analogous model structures for pointed cubical $\CC^{\ast}$-spaces 
and a pointed unstable homotopy category $\HH^{\ast}$ of $\CC^{\ast}$-algebras.
As in topology, 
every pointed cubical $\CC^{\ast}$-space gives rise to homotopy groups indexed by the 
non-negative integers. 
These invariants determine precisely when a map is an isomorphism in $\HH^{\ast}$.
We use the theory of representations of $\CC^{\ast}$-algebras to interpret Kasparov's $KK$-groups
as maps in $\HH^{\ast}$.
Due to the current setup of $K$-theory,
in this setting it is convenient to work with simplicial rather than cubical $\CC^{\ast}$-spaces.
However, 
this distinction makes no difference since the corresponding homotopy categories are equivalent.
\vspace{0.1in}

Section \ref{section:stableCstarhomotopytheory} introduces the ``spaces'' employed in stable $\CC^{\ast}$-homotopy theory, 
namely spectra in the sense of algebraic topology with respect to the suspension coordinate $C$.
These are sequences $\X_{0},\X_{1},\cdots$ of pointed cubical $\CC^{\ast}$-spaces equipped with 
structure maps $C\otimes\X_{n}\rightarrow\X_{n+1}$ for every $n\geq 0$.
We show there exists a stable model structure on spectra and define $\SH^{\ast}$ as the associated 
homotopy category.
There is a technically superior category of $\CC^{\ast}$-symmetric spectra with a closed symmetric
monoidal product.
The importance of this category is not emphasized in full in this paper, 
but one would expect that it will play a central role in further developments of the subject.
Its stable homotopy category is equivalent to $\SH^{\ast}$.  
\vspace{0.1in}

Using the pointed model structure we define a new type of $K$-theory of $\CC^{\ast}$-algebras.
We note that it contains the $K$-theory of the cubical sphere spectrum or Waldhausen's $A$-theory of a point as a retract.
This observation brings our $\CC^{\ast}$-homotopy theory in contact with manifold theory.
More generally, 
working relative to some $\CC^{\ast}$-algebra $A$, 
we construct a $K$-theory spectrum $K(A)$ whose homotopy groups cannot be extracted from the ordinary $K_{0}$- and 
$K_{1}$-groups of $A$.
It would of course be of considerable interest to explicate some of the $K$-groups arising from this construction, 
even for the trivial $\CC^{\ast}$-algebra.
We show that the pointed (unstable) model structure and the stable model structure on $S^{1}$-spectra of spaces in unstable 
$\CC^{\ast}$-homotopy theory give rise to equivalent $K$-theories.
This is closely related to the triangulated structure on $\SH^{\ast}$.   
\vspace{0.1in}

In \cite{Ostvar:noncommutativemotives} we construct a closely related theory of noncommutative motives. 
On the level of $\CC^{\ast}$-symmetric spectra this corresponds to twisting Eilenberg-MacLane spectra in ordinary stable 
homotopy theory by the $KK$-theory of tensor products of $C_{0}(\R)$.  
This example relates to $K$-homology and $K$-theory and is discussed in Section \ref{subsection:KK-theory}.
The parallel theory of Eilenberg-MacLane spectra twisted by local cyclic homology is sketched in Section \ref{subsection:localcyclichomology}.
As it turns out there is a Chern-Connes character, 
with highly structured multiplicative properties, 
between the $\CC^{\ast}$-symmetric spectra build from Eilenberg-MacLane spectra twisted by $KK$-theory and local cyclic homology.  
This material is covered in Section \ref{subsection:theChernConnescharacter}.
An alternate take on motives has been worked out earlier by Connes-Consani-Marcolli in \cite{CCM}.
\vspace{0.1in}

Related to the setup of motives we introduce for a $\CC^{\ast}$-space $\E$ its zeta function $\zeta_{\E}(t)$ taking values in the formal power 
series over a certain Grothendieck ring.
The precise definition of zeta functions in this setting is given in Section \ref{subsection:zetafunctions}.
In the same section we provide some motivation by noting an analogy with zeta functions defined for algebraic varieties.
As in the algebro-geometric situation a key construction is that of symmetric powers.
It turns out that $\zeta_{\E}(t)$ satisfies a functional equation involving Euler characteristics $\chi(\E)$ and $\chi_{+}(\E)$ and the determinant 
$\ddet(\E)$ provided $\E$ is ``finite dimensional'' in some sense.
If $\D\E$ denotes the dual of $\E$ then the functional equation reads as follows:
\begin{equation*}
\zeta_{\D\E}(t^{-1})=
(-1)^{\chi_{+}(\E)}\ddet(\E) t^{\chi(\E)}\zeta_{\E}(t)
\end{equation*}
\vspace{0.1in}

In the last section we show there exists a filtration of the stable homotopy category $\SH^{\ast}$ by full triangulated subcategories:
\begin{equation*}
\cdots\subseteq 
\Sigma_{C}^{1}\SH^{\ast,\eff}\subseteq
\SH^{\ast,\eff}\subseteq
\Sigma_{C}^{-1}\SH^{\ast,\eff}\subseteq
\cdots
\end{equation*}
Here, 
placed in degree zero is the so-called effective stable $\CC^{\ast}$-homotopy category comprising all suspension spectra.
The above is a filtration of the stable $\CC^{\ast}$-homotopy category in the sense that the smallest triangulated subcategory containing 
$\Sigma_{C}^{q}\SH^{\ast,\eff}$ for every integer $q$ coincides with $\SH^{\ast}$.
In order to make this construction work we use the fact that $\SH^{\ast}$ is a compactly generated triangulated category.
The filtration points toward a whole host open problems reminiscent of contemporary research in motivic homotopy theory \cite{VVopenproblems}.
A first important problem in this direction is to identify the zero slice of the sphere spectrum.
\vspace{0.1in}

The results described in the above extend to $\CC^{\ast}$-algebras equipped with a strongly continuous representation of a locally compact group 
by $\CC^{\ast}$-algebra automorphisms.
This fact along with the potential applications have been a constant inspiration for us.
Much work remains to develop the full strength of the equivariant setup.
\vspace{0.1in}

Our overall aim in this paper is to formulate,
by analogy with classical homotopy theory, 
a first thorough conceptual introduction to $\CC^{\ast}$-homotopy theory.
A next step is to indulge in the oodles of open computational questions this paper leaves behind.
Some of these emerging questions should be resolved by making difficult things easy as a consequence of the setup, 
while others will require considerable hands-on efforts.
\vspace{0.1in}

{\bf Acknowledgments.}
Thanks go to the the members of the operator algebra and geometry/topology groups in Oslo for interest in this work. 
We are grateful to Clark Barwick, George Elliott, Nigel Higson, Rick Jardine, Andr\'e Joyal, 
Jack Morava, Sergey Neshveyev, Oliver R{\"o}ndigs and Claude Schochet for inspiring correspondence and discussions.
We extend our gratitude to Michael Joachim for explaining his joint work with Mark Johnson on a 
model category structure for sequentially complete locally multiplicatively convex $\CC^{\ast}$-algebras 
with respect to some infinite ordinal number \cite{JJ:modelstructure}.
The two viewpoints turned out to be wildly different.
Aside from the fact that we are not working with the same underlying categories,
one of the main differences is that the model structure in \cite{JJ:modelstructure} is right proper, 
since every object is fibrant, 
but it is not known to be left proper. 
Hence it is not suitable fodder for stabilization in terms of today's (left) Bousfield localization machinery.
One of the main points in our work is that fibrancy is a special property; 
in fact, it governs the whole theory, 
while left properness is required for defining the stable $\CC^{\ast}$-homotopy category.

\newpage 
\section{Preliminaries}
\label{section:preliminaries}

\subsection{$\CC^{\ast}$-spaces}
\label{subsection:CC-spaces}
Let $\CC^{\ast}-\Alg$ denote the category of separable $\CC^{\ast}$-algebras and $\ast$-homomorphisms.
It is an essentially small category with small skeleton the set of $\CC^{\ast}$-algebras which are 
operators on a fixed separable Hilbert space of countably infinite dimension.
In what follows, 
all $\CC^{\ast}$-algebras are objects of $\CC^{\ast}-\Alg$ so that commutative $\CC^{\ast}$-algebras 
can be identified with pointed compact metric spaces via Gelfand-Naimark duality. 
Let $\K$ denote the $\CC^{\ast}$-algebra of compact operators on a separable, 
infinite dimensional Hilbert space, 
e.g.~the space $\ell^2$ of square summable sequences. 
\vspace{0.1in}

The object of main interest in this section is obtained from $\CC^{\ast}-\Alg$ via embeddings 
\begin{equation*}
\xymatrix{
\CC^{\ast}-\Alg \ar[r] &
\CC^{\ast}-\Spc \ar[r] &
\Box\CC^{\ast}-\Spc. }
\end{equation*}
A $\CC^{\ast}$-space is a set-valued functor on $\CC^{\ast}-\Alg$. 
Let $\CC^{\ast}-\Spc$ denote the category of $\CC^{\ast}$-spaces and natural transformations.
By the Yoneda lemma there exists a full and faithful contravariant embedding of $\CC^{\ast}-\Alg$ into 
$\CC^{\ast}-\Spc$ which preserves limits.
This entails in particular natural bijections $\CC^{\ast}-\Alg(A,B)=\CC^{\ast}-\Spc(B,A)$ for all 
$\CC^{\ast}$-algebras $A$, $B$.
Since, 
as above, 
the context will always clearly indicate the meaning we shall throughout identify every 
$\CC^{\ast}$-algebra with its corresponding representable $\CC^{\ast}$-space.
Note that every set determines a constant $\CC^{\ast}$-space.
A pointed $\CC^{\ast}$-space consists of a $\CC^{\ast}$-space $\X$ together with a map of 
$\CC^{\ast}$-spaces from the trivial $\CC^{\ast}$-algebra to $\X$.
We let $\CC^{\ast}-\Spc_{0}$ denote the category of pointed $\CC^{\ast}$-spaces.
There exists a functor $\CC^{\ast}-\Spc\rightarrow\CC^{\ast}-\Spc_{0}$ obtained by taking pushouts of 
diagrams of the form $\X\leftarrow\emptyset\rightarrow 0$; it is left adjoint to the forgetful functor.
Observe that every $\CC^{\ast}$-algebra is canonically pointed. 
The category $\Box\CC^{\ast}-\Spc$ of cubical $\CC^{\ast}$-spaces consists of possibly void collections of 
$\CC^{\ast}$-spaces $\X_n$ for all $n\geq 0$ together with face maps 
$d_i^{\alpha}\colon\X_n\rightarrow\X_{n-1},1\leq i\leq n,\alpha=0,1$
(corresponding to the $2n$ faces of dimension $n-1$ in a standard geometrical $n$-cube),   
and degeneracy maps $s_i\colon\X_{n-1}\rightarrow\X_n$ where $1\leq i\leq n$ subject to the cubical 
identities $d_i^{\alpha} d_j^{\beta}=d_{j-1}^{\beta} d_i^{\alpha}$ for $i<j$, $s_i s_j=s_{j+1}s_i$ for 
$i\leq j$ and  
\begin{align*}
d_i^{\alpha}s_j=
\begin{cases}
s_{j-1} d_i^{\alpha} & i<j\\
\id & i=j\\
s_j d_{i-1}^{\alpha} & i>j.
\end{cases}
\end{align*}
A map of cubical $\CC^{\ast}$-spaces is a collection of maps of $\CC^{\ast}$-spaces $\X_n\rightarrow\Y_n$ 
for all $n\geq 0$ which commute with the face and degeneracy maps.
An alternate description uses the box category $\Box$ of abstract hypercubes representing the 
combinatorics of power sets of finite ordered sets \cite[\S3]{Jardine:Categoricalhomotopytheory}.
The box category $\Box$ is the subcategory of the category of poset maps $1^{n}\rightarrow 1^{m}$ 
which is generated by the face and degeneracy maps.
Here, 
$1^{n}=1^{\times n}=\{(\epsilon_{1},\cdots,\epsilon_{n})\vert \epsilon_{i}=0,1\}$ is the $n$-fold hypercube.
As a poset $1^{n}$ is isomorphic to the power set of $\{0,1,\cdots,n\}$. 
The category $\Box\Set$ of cubical sets consists of functors $\Box^{\op}\rightarrow\Set$ and natural 
transformations.
With these definitions we may identify $\Box\CC^{\ast}-\Spc$ with the functor category 
$[\CC^{\ast}-\Alg,\Box\Set]$ of cubical set-valued functors on $\CC^{\ast}-\Alg$.
Note that every cubical set defines a constant cubical $\CC^{\ast}$-space by extending degreewise the 
correspondence between sets and $\CC^{\ast}$-spaces.  
A particularly important example is the standard $n$-cell defined by $\Box^{n}\equiv\Box(-,1^{n})$.
Moreover, every $\CC^{\ast}$-algebra defines a representable $\CC^{\ast}$-space which can be viewed 
as a discrete cubical $\CC^{\ast}$-space.
The category $\Box\CC^{\ast}-\Spc_{0}$ of pointed cubical $\CC^{\ast}$-spaces is defined using the exact same 
script as above. 
Hence it can be identified with the functor category of pointed cubical set-valued functors on 
$\CC^{\ast}-\Alg$. 
\vspace{0.1in}

We shall also have occasion to work with the simplicial category $\Delta$ of finite ordinals 
$[n]=\{0<1<\cdots<n\}$ for $n\geq 0$ and order-preserving maps.
The category $\Delta\CC^{\ast}-\Spc$ of simplicial $\CC^{\ast}$-spaces consists of $\CC^{\ast}$-spaces
$\X_{n}$ for all $n\geq 0$ together with face maps $d_i\colon\X_n\rightarrow\X_{n-1}$, $1\leq i\leq n$, 
and degeneracy maps $s_i\colon\X_{n-1}\rightarrow\X_n$, $1\leq i\leq n$, subject to the simplicial 
identities $d_{i} d_{j}=d_{j-1} d_{i}$ for $i<j$, $s_i s_j=s_{j+1}s_i$ for $i\leq j$ and  
\begin{align*}
d_{i} s_{j}=
\begin{cases}
s_{j-1} d_i & i<j\\
\id & i=j,j+1\\
s_{j} d_{i-1} & i>j+1.
\end{cases}
\end{align*}

Let $\otimes_{\CC^{\ast}-\Alg}$ denote a suitable monoidal product on $\CC^{\ast}-\Alg$ with unit 
the complex numbers.
Later we shall specialize to the symmetric monoidal maximal and minimal tensor products, 
but for now it is not important to choose a specific monoidal product. 
In \S\ref{modelcategories} we recall the monoidal product $\otimes_{\Box\Set}$ in Jardine's closed symmetric 
monoidal structure on cubical sets \cite[\S3]{Jardine:Categoricalhomotopytheory}.
We shall outline an extension of these data to a closed monoidal structure on $\Box\CC^{\ast}-\Spc$ following 
the work of Day \cite{Day:closedfunctorsI}.
The external monoidal product of two cubical $\CC^{\ast}$-spaces 
$\X,\Y\colon\CC^{\ast}-\Alg\rightarrow\Box\Set$ is defined by setting 
\begin{equation*}
\X\widetilde{\otimes}\Y\equiv\otimes_{\Box\Set}\circ (\X\times\Y). 
\end{equation*}
Next we introduce the monoidal product $\X\otimes\Y$ of $\X$ and $\Y$ by taking the left Kan extension 
of $\otimes_{\CC^{\ast}-\Alg}$ along $\X\widetilde{\otimes}\Y$ or universal filler in the diagram:
\begin{equation*}
\xymatrix{ 
\CC^{\ast}-\Alg\times\CC^{\ast}-\Alg \ar[r]^-{\X\widetilde{\otimes}\Y} \ar[d]_-{\otimes_{\CC^{\ast}-\Alg}} &
\Box\Set\\
\CC^{\ast}-\Alg  \ar@{ -->}[ur] }
\end{equation*}
Thus the $\Box\Set$-values of the monoidal product are given by the formulas
\begin{equation*}
\X\otimes\Y(A)\equiv
\underset{A_1\otimes_{\CC^{\ast}-\Alg}A_2\rightarrow A}{\colim}
\X(A_1)\otimes_{\Box\Set}\Y(A_2).
\end{equation*}
The colimit is indexed on the category with objects 
$\alpha\colon A_1\otimes_{\CC^{\ast}-\Alg}A_2\rightarrow A$ and maps pairs of maps 
$(\phi,\psi)\colon (A_1,A_2)\rightarrow (A'_1,A'_2)$ such that $\alpha'(\psi\otimes\phi)=\alpha$.
By functoriality of colimits it follows that $\X\otimes\Y$ is a cubical $\CC^{\ast}$-space.
When couched as a coend, 
the tensor product is a weighted average of all of the handicrafted external tensor products 
$\X\widetilde{\otimes}\Y\equiv\otimes_{\Box\Set}\circ (\X\times\Y)$ in the sense that 
\begin{equation*}
\X\otimes\Y(A)=
\int^{A_1,A_2\in\CC^{\ast}-\Alg}
\bigl(\X(A_1)\otimes_{\Box\Set}\Y(A_2)\bigr)
\otimes_{\Box\Set}
\CC^{\ast}-\Alg(A_1\otimes_{\CC^{\ast}-\Alg}A_2,A).
\end{equation*}
Since the tensor product is defined by a left Kan extension, 
it is characterized by the universal property
\begin{equation*}
\Box\CC^{\ast}-\Spc(\X\otimes\Y,\Z)=
[\CC^{\ast}-\Alg,\Box\CC^{\ast}-\Spc] 
(\X\widetilde{\otimes}\Y,\Z\circ\otimes_{{\CC^{\ast}-\Alg}}).
\end{equation*}
The bijection shows that maps between cubical $\CC^{\ast}$-spaces $\X\otimes\Y\rightarrow\Z$ are uniquely 
determined by maps of cubical sets $\X(A)\otimes_{\Box\Set}\Y(B)\rightarrow\Z(A\otimes_{\CC^{\ast}-\Alg}B)$ 
which are natural in $A$ and $B$.
Note also that the tensor product of representable $\CC^{\ast}$-spaces $A\otimes B$ is represented by the 
monoidal product $A\otimes_{\CC^{\ast}-\Alg}B$ and for cubical sets $K$, $L$, 
$K\otimes L=K\otimes_{\Box\Set} L$, 
i.e.~$(\CC^{\ast}-\Alg,\otimes_{\CC^{\ast}-\Alg})\rightarrow (\Box\CC^{\ast}-\Spc,\otimes)$ and 
$(\Box,\otimes_{\Box\Set})\rightarrow (\Box\CC^{\ast}-\Spc,\otimes)$ are monoidal functors in the strong 
sense that both of the monoidal structures are preserved to within coherent isomorphisms. 
According to our standing hypothesis, 
the $\CC^{\ast}$-algebra $\C$ (the complex numbers) represents the unit for the monoidal product $\otimes$.

If $\widetilde{\Z}$ is a cubical set-valued functor on $\CC^{\ast}-\Alg\times\CC^{\ast}-\Alg$ 
and $\Y$ is a cubical $\CC^{\ast}$-space,
define the external function object $\underline{\widetilde{\Hom}}(\Y,\widetilde{\Z})$ by 
\begin{equation*}
\underline{\widetilde{\Hom}}(\Y,\widetilde{\Z})(A)
\equiv
\Box\CC^{\ast}-\Spc\bigl(\Y,\widetilde{\Z}(A,-)\bigr).
\end{equation*}
Then for every cubical $\CC^{\ast}$-space $\X$ there is a bijection
\begin{equation*}
\Box\CC^{\ast}-\Spc\bigl(\X,\underline{\widetilde{\Hom}}(\Y,\widetilde{\Z})\bigr)=
[\CC^{\ast}-\Alg\times\CC^{\ast}-\Alg,\Box\Set]
(\X\widetilde{\otimes}\Y,\widetilde{\Z}).
\end{equation*}
A pair of cubical $\CC^{\ast}$-spaces $\Y$ and $\Z$ acquires an internal hom object
\begin{equation*}
\underline{\Hom}(\Y,\Z)
\equiv
\underline{\widetilde{\Hom}}(\Y,\Z\circ\otimes_{\CC^{\ast}-\Alg}).
\end{equation*}
Using the characterization of the monoidal product it follows that 
\begin{equation*}
\xymatrix{
\Z\ar@{|->}[r] & \underline{\Hom}(\Y,\Z) }
\end{equation*}
determines a right adjoint of the functor
\begin{equation*}
\xymatrix{
\X\ar@{|->}[r] & \X\otimes\Y.} 
\end{equation*}
Observe that $\Box\CC^{\ast}-\Spc$ equipped with $\otimes$ and $\underline{\Hom}$ becomes a closed symmetric 
monoidal category provided the monoidal product $\otimes_{\CC^{\ast}-\Alg}$ is symmetric, 
which we may assume.

According to the adjunction the natural evaluation map $\underline{\Hom}(\Y,\Z)\otimes\Y\rightarrow\Z$ 
determines an exponential law
\begin{equation*}
\Box\CC^{\ast}-\Spc(\X\otimes\Y,\Z)=\Box\CC^{\ast}-\Spc\bigl(\X,\underline{\Hom}(\Y,\Z)\bigr).
\end{equation*}
Using these data, 
standard arguments imply there exist natural isomorphisms
\begin{equation*}
\underline{\Hom}(\X\otimes\Y,\Z)=\underline{\Hom}\bigl(\X,\underline{\Hom}(\Y,\Z)\bigr)
\end{equation*}
\begin{equation*}
\Box\CC^{\ast}-\Spc(\Y,\Z)=
\Box\CC^{\ast}-\Spc(\C\otimes\Y,\Z)=
\Box\CC^{\ast}-\Spc\bigl(\C,\underline{\Hom}(\Y,\Z)\bigr) 
\end{equation*}
and 
\begin{equation*}
\underline{\Hom}(\C,\Z)=\Z.
\end{equation*}

In what follows we introduce a cubical set tensor and cotensor structure on $\Box\CC^{\ast}-\Spc$.
This structure will greatly simplify the setup of the left localization theory of model structures 
on cubical $\CC^{\ast}$-spaces.
If $\X$ and $\Y$ are cubical $\CC^{\ast}$-spaces and $K$ is a cubical set, 
define the tensor $\X\otimes K$ by 
\begin{equation}
\label{cubicalCstarspacetensor}
\X\otimes K(A)\equiv
\X(A)\otimes_{\Box\Set} K
\end{equation}
and the cotensor $\Y^K$ in terms of the ordinary cubical function complex
\begin{equation}
\label{cubicalCstarspacecotensor}
\Y^K(A)\equiv{\hom}_{\Box\Set}\bigl(K,\Y(A)\bigr).
\end{equation}
The cubical function complex ${\hom}_{\Box\CC^{\ast}-\Spc}(\X,\Y)$ of $\X$ and $\Y$ is defined by setting
\begin{equation*}
{\hom}_{\Box\CC^{\ast}-\Spc}(\X,\Y)_n\equiv\Box\CC^{\ast}-\Spc(\X\otimes\Box^{n},\Y).
\end{equation*}
By the Yoneda lemma there exists a natural isomorphism of cubical sets
\begin{equation}
\label{functioncomplexofrepresentable}
{\hom}_{\Box\CC^{\ast}-\Spc}(A,\Y)=\Y(A).
\end{equation}
Using these definitions one verifies easily that $\Box\CC^{\ast}-\Spc$ is enriched in cubical sets $\Box\Set$. 
Moreover, 
there are natural isomorphisms of cubical sets
\begin{equation*}
{\hom}_{\Box\CC^{\ast}-\Spc}(\X\otimes K,\Y)
=
{\hom}_{\Box\Set}\bigl(K,{\hom}_{\Box\CC^{\ast}-\Spc}(\X,\Y)\bigr)
 =
{\hom}_{\Box\CC^{\ast}-\Spc}(\X,\Y^K).
\end{equation*}
In particular, taking $0$-cells we obtain the natural isomorphisms
\begin{equation}
\label{0-celladjunction}
\Box\CC^{\ast}-\Spc(\X\otimes K,\Y)=
\Box\Set\bigl(K,{\hom}_{\Box\CC^{\ast}-\Spc}(\X,\Y)\bigr)=
\Box\CC^{\ast}-\Spc(\X,\Y^K).
\end{equation}

It is useful to note that the cubical function complex is the global sections of the internal 
hom object,
and more generally that 
\begin{equation*}
\underline{\Hom}(\X,\Y)(A)={\hom}_{\Box\CC^{\ast}-\Spc}\bigl(\X,\Y(-\otimes A)\bigr).
\end{equation*}
In effect, 
note that according to the Yoneda lemma and the exponential law for cubical $\CC^{\ast}$-spaces, 
we have
\begin{align*}
\underline{\Hom}(\X,\Y)(A) 
& ={\hom}_{\Box\CC^{\ast}-\Spc}\bigl(A,\underline{\Hom}(\X,\Y)\bigr)\\
& ={\hom}_{\Box\CC^{\ast}-\Spc}\bigl((\X\otimes A),\Y\bigr).
\end{align*}
Hence,
since the Yoneda embedding of $(\CC^{\ast}-\Alg)^{\op}$ into $\Box\CC^{\ast}-\Spc$ is monoidal, 
we have 
\begin{equation*}
\underline{\Hom}(B,\Y)=\Y(-\otimes B).
\end{equation*}
The above allows to conclude there are natural isomorphisms
\begin{align*}
\underline{\Hom}(\X,\Y)(A) 
& ={\hom}_{\Box\CC^{\ast}-\Spc}\bigl((\X\otimes A),\Y\bigr)\\
& ={\hom}_{\Box\CC^{\ast}-\Spc}\bigl(\X,\underline{\Hom}(A,\Y)\bigr)\\
& ={\hom}_{\Box\CC^{\ast}-\Spc}\bigl(\X,\Y(-\otimes A)\bigr).
\end{align*}
In particular, 
the above entails natural isomorphisms
\begin{equation}
\label{representableinternalhomevalution}
\underline{\Hom}(B,\Y)(A)=\Y(A\otimes B).
\end{equation}

There exist entirely analogous constructs for pointed cubical $\CC^{\ast}$-spaces and pointed cubical sets.
In short, 
there exists a closed monoidal category $(\Box\CC^{\ast}-\Spc_{0},\otimes,\underline{\Hom})$ and all the 
identifications above hold in the pointed context.
Similarly, 
there are closed monoidal categories $(\Delta\CC^{\ast}-\Spc,\otimes,\underline{\Hom})$ and
$(\Delta\CC^{\ast}-\Spc_{0},\otimes,\underline{\Hom})$ of simplicial and pointed simplicial 
$\CC^{\ast}$-spaces constructed by the same method.
Here we consider the categories of simplicial sets $\Delta\Set$ and pointed simplicial sets 
$\Delta\Set_{\ast}$ with their standard monoidal products.
\vspace{0.1in}

Next we recall some size-related concepts which are also formulated in \cite[\S2.1.1]{Hovey:Modelcategories}.
One of the lessons of the next sections is that these issues matter when dealing with 
model structures on cubical $\CC^{\ast}$-spaces. 
Although the following results are stated for cubical $\CC^{\ast}$-spaces,
all results hold in the pointed category $\Box\CC^{\ast}-\Spc_{0}$ as well.

Let $\lambda$ be an ordinal, 
i.e.~the partially ordered set of all ordinals $<\lambda$.
A $\lambda$-sequence or transfinite sequence indexed by $\lambda$ in $\Box\CC^{\ast}-\Spc$ is a 
functor $F\colon\lambda\rightarrow\Box\CC^{\ast}-\Spc$ which is continuous at every limit ordinal 
$\beta<\lambda$ in the sense that there is a naturally induced isomorphism
$\colim_{\alpha<\beta}F_{\alpha}\rightarrow F_{\beta}$.
If $\lambda$ is a regular cardinal, 
then no $\lambda$-sequence has a cofinal subsequence of shorter length. 

Let $\kappa$ be a cardinal.
A cubical $\CC^{\ast}$-space $\X$ is $\kappa$-small relative to a class of maps $I$ if for 
every regular cardinal $\lambda\geq\kappa$ and $\lambda$-sequence $F$ in $\Box\CC^{\ast}-\Spc$ 
for which each map $F_{\alpha}\rightarrow F_{\alpha+1}$ belongs to $I$, 
there is a naturally induced isomorphism
\begin{equation*}
\xymatrix{
\colim_{\alpha}\Box\CC^{\ast}-\Spc(\X,F_{\alpha})\ar[r] & \Box\CC^{\ast}-\Spc(\X,\colim_{\alpha}F_{\alpha}).}
\end{equation*}
The idea is that every map from $\X$ into the colimit factors through $F_{\alpha}$ 
for some $\alpha<\lambda$ and the factoring is unique up to refinement.
Moreover, 
$\X$ is small relative to $I$ if it is $\kappa$-small relative to $I$ for some cardinal $\kappa$, 
and small if it is small relative to $\Box\CC^{\ast}-\Spc$.
Finitely presentable cubical $\CC^{\ast}$-spaces are $\omega$-small cubical $\CC^{\ast}$-spaces, 
where as usual $\omega$ denotes the smallest infinite cardinal number.
\begin{example}
\label{example:finitelypresentable}
Every $\CC^{\ast}$-algebra $A$ is $\kappa$-small for every cardinal $\kappa$, 
and $\Box^{n}$ is a finitely presentable cubical set for all $n\geq 0$ since every representable 
cubical set has only a finite number of non-degenerate cells.
Thus $A\otimes\Box^{n}$ is a finitely presentable cubical $\CC^{\ast}$-space.
\end{example}

Since $\CC^{\ast}-\Alg\times\Box^{\op}$ is a small category $\Box\CC^{\ast}-\Spc$ is locally presentable 
according to \cite[5.2.2b]{Borceux:Handbook2}, 
i.e.~$\Box\CC^{\ast}-\Spc$ is cocomplete and there is a regular cardinal $\lambda$ and a set $\A$ of 
$\lambda$-small cubical $\CC^{\ast}$-spaces such that every cubical $\CC^{\ast}$-space is a 
$\lambda$-filtered colimit of objects from $\A$.
\begin{lemma}
\label{lemma:finitelypresentable}
The category of cubical $\CC^{\ast}$-spaces is locally presentable.
\end{lemma}
This observation implies the set of all representable cubical $\CC^{\ast}$-spaces is a strong
generator for $\Box\CC^{\ast}-\Spc$ \cite[pg.~18]{AR:book}.
We shall refer repeatedly to Lemma \ref{lemma:finitelypresentable} when localizing model structures 
on cubical $\CC^{\ast}$-spaces.
\vspace{0.1in}
 
The next straightforward lemmas are bootstrapped for finitely presentable objects.
\begin{lemma}
\label{lemma:filteredandfinitecolimit}
Every cubical $\CC^{\ast}$-space is a filtered colimit of finite colimits of cubical $\CC^{\ast}$-spaces 
of the form $A\otimes\Box^{n}$ where $A$ is a $\CC^{\ast}$-algebra.
\end{lemma}

We let ${\bf fp}\Box\CC^{\ast}-\Spc$ denote the essentially small category of finitely presentable 
cubical $\CC^{\ast}$-spaces.
It is closed under retracts, 
finite colimits and tensors in $\Box\CC^{\ast}-\Spc$.
\begin{lemma}
\label{lemma:filteredandfinitecolimitfinitelypresentable}
The subcategory ${\bf fp}\Box\CC^{\ast}-\Spc$ exhausts $\Box\CC^{\ast}-\Spc$ in the sense that every cubical 
$\CC^{\ast}$-space is a filtered colimit of finitely presentable cubical $\CC^{\ast}$-spaces.
\end{lemma}
\begin{remark}
In this paper we shall employ the pointed analog of ${\bf fp}\Box\CC^{\ast}-\Spc$ when defining $K$-theory and 
also as the source category for a highly structured model for the stable $\CC^{\ast}$-homotopy category.
The results above hold in the pointed context.
\end{remark}
\begin{corollary}
\label{corollary:internalhomfinitelypresentable}
A cubical $\CC^{\ast}$-space $\X$ is finitely presentable if and only if the internal hom functor 
$\underline{\Hom}(\X,-)$ is finitely presentable. 
\end{corollary}
\begin{example}
\label{example:internalhomfinitelypresentable}
The internal hom object $\underline{\Hom}\bigl(S^{1}\otimes C_{0}(\R),-\bigr)$ is finitely presentable. 
\end{example}
\vspace{0.1in}

Next we introduce the geometric realization functor for cubical $\CC^{\ast}$-spaces.

Denote by $\Box_{\ttop}^{\bullet}$ the topological standard cocubical set equipped with the coface maps 
\begin{equation*}
\xymatrix{
\delta_{0}^{\alpha}\colon \ast=I^{0}\ar[r] & I^{1};\, 
\ast\ar@{|->}[r] &\alpha,
\\
\delta_i^{\alpha}\colon I^{n-1}\ar[r] & I^{n};\,
(t_{1},\dots,t_{n-1})\ar@{|->}[r] & (t_{1},\dots,t_{i-1},\alpha,t_{i},\dots,t_{n-1}), }
\end{equation*}
where $I$ is the topological unit interval, 
$1\leq i\leq n,\alpha=0,1$,  
and the codegeneracy maps 
\begin{equation*}
\xymatrix{
\epsilon_{0}\colon I^{1}\ar[r] &\ast=I^{0};\, 
t\ar@{|->}[r] &\ast
\\
\epsilon_{i}\colon I^{n}\ar[r] & I^{n-1};\,
(t_{1},\dots,t_{n})\ar@{|->}[r] & (t_1,\cdots,\widehat{t_i},\cdots,t_n). }
\end{equation*}

These maps satisfies the cocubical identities $\delta_j^{\beta} \delta_i^{\alpha}=\delta_{i+1}^{\alpha} \delta_j^{\beta}$ and 
$\epsilon_i \epsilon_j=\epsilon_{j}\epsilon_{i+1}$ for $j\leq i$, 
and  
\begin{align*}
\epsilon_j\delta_i^{\alpha}=
\begin{cases}
\delta_{i-1}^{\alpha}
\epsilon_{j} & j<i\\
\id & j=i\\
\delta_{i}^{\alpha}\epsilon_{j-1} & j>i.
\end{cases}
\end{align*}

Denote by $C(\Box_{\ttop}^{\bullet})$ the standard cubical $\CC^{\ast}$-algebra 
\begin{equation}
\label{equation:standardcubicalC*algebra}
\xymatrix{
n\ar@{|->}[r] & C(\Box_{\ttop}^{n})\colon 
\C &
C(I^{1}) \ar@<3pt>[l]\ar@<-3pt>[l] & 
C(I^{2}) \ar@<9pt>[l]\ar@<3pt>[l]\ar@<-3pt>[l]\ar@<-9pt>[l] &
\ar@<15pt>[l]\ar@<9pt>[l]\ar@<3pt>[l]\ar@<-3pt>[l]\ar@<-9pt>[l]\ar@<-15pt>[l]
\cdots }
\end{equation}
comprising continuous complex-valued functions on the topological standard $n$-cube.
Its cubical structure is induced in the evident way by the coface and codegeneracy maps of $\Box_{\ttop}^{\bullet}$ given above.
\vspace{0.1in}

For legibility we shall use the same notation $C(\Box_{\ttop}^{\bullet})$ for the naturally induced cocubical $\CC^{\ast}$-space 
\begin{equation*}
\xymatrix{
\CC^{\ast}-\Alg\bigl(C(\Box_{\ttop}^{\bullet}),-\bigr)\colon\CC^{\ast}-\Alg
\ar[r] & (\Box\Set)^{\op}.} 
\end{equation*}

The singular functor
\begin{equation*}
\xymatrix{
\Sing_{\Box}^{\bullet}\colon 
\Box\CC^{\ast}-\Spc\ar[r] & \Box\CC^{\ast}-\Spc}
\end{equation*}
is an endofunctor of cubical $\CC^{\ast}$-spaces.
Its value at a $\CC^{\ast}$-space $\X$ is by definition given as the internal hom object
\begin{equation*}
\Sing_{\Box}^{\bullet}(\X)\equiv\underline{\Hom}\bigl(C(\Box_{\ttop}^{\bullet}),\X\bigr).
\end{equation*}
The cubical structure of $\Sing_{\Box}^{\bullet}(\X)$ is obtained from the cocubical structure of
$C(\Box_{\ttop}^{\bullet})$.

Plainly this functor extends to an endofunctor of $\Box\CC^{\ast}-\Spc$ by taking the diagonal of 
the bicubical $\CC^{\ast}$-space
\begin{equation*}
\xymatrix{
(m,n)\ar@{|->}[r] & 
\underline{\Hom}\bigl(C(\Box_{\ttop}^{m}),\X_n\bigr)=
\X_n\bigl(C(\Box_{\ttop}^{m})\otimes -\bigr).}
\end{equation*}
In particular, 
the singular functor specializes to a functor from $\CC^{\ast}$-spaces 
\begin{equation*}
\xymatrix{
\Sing_{\Box}^{\bullet}\colon 
\CC^{\ast}-\Spc\ar[r] & \Box\CC^{\ast}-\Spc.}
\end{equation*}
Its left adjoint is the geometric realization functor
\begin{equation*}
\xymatrix{
\vert \cdot \vert \colon\Box\CC^{\ast}-\Spc\ar[r] & \CC^{\ast}-\Spc.}
\end{equation*}
If $\X$ is a cubical $\CC^{\ast}$-space,  
then its geometric realization $\vert\X\vert$ is the coend of the functor 
$\Box\times\Box^{\op}\rightarrow\CC^{\ast}-\Spc$ given by $(1^m,1^{n})\mapsto C(\Box_{\ttop}^{m})\otimes\X_n$.
Hence there is a coequalizer in $\CC^{\ast}-\Spc$ 
\begin{equation*}
\xymatrix{
\coprod_{\Theta\colon 1^{m}\rightarrow 1^{n}}C(\Box_{\ttop}^{m})\otimes\X_n
\ar@<3pt>[r]\ar@<-3pt>[r] & 
\coprod_{1^{n}}C(\Box_{\ttop}^{n})\otimes\X_n \ar[r] & 
\vert \X \vert\equiv\int^{\Box^{\op}}C(\Box_{\ttop}^{n})\otimes\X_n.} 
\end{equation*}
The two parallel maps in the coequalizer associated with the maps $\Theta\colon 1^{m}\rightarrow 1^{n}$ in the 
box category $\Box$ are gotten from the natural maps 
\begin{equation*}
\xymatrix{
C(\Box_{\ttop}^{m})\otimes\X_n\ar[r] & 
C(\Box_{\ttop}^{n})\otimes\X_n\ar[r] & 
\coprod_{n}C(\Box_{\ttop}^{n})\otimes\X_n }
\end{equation*}
and 
\begin{equation*}
\xymatrix{
C(\Box_{\ttop}^{m})\otimes\X_n\ar[r] & 
C(\Box_{\ttop}^{m})\otimes\X_m\ar[r] & 
\coprod_{n}C(\Box_{\ttop}^{n})\otimes\X_n.}
\end{equation*}
\begin{example}
\label{example:monomorphism}
For every cubical $\CC^{\ast}$-space $\X$ there is a monomorphism $\X\rightarrow\Sing_{\Box}^{\bullet}(\X)$.
In $n$-cells it is given by the canonical map 
\begin{equation*}
\xymatrix{
\X_n\ar[r] & 
\underline{\Hom}\bigl(C(\Box_{\ttop}^{n}),\X_n\bigr).}
\end{equation*}
\end{example}
\begin{example}
For $n\geq 0$ there are natural isomorphisms
\begin{equation*}
\xymatrix{
\Sing_{\Box}^{n}(\X)(A)=\X\bigl(A\otimes C(\Box_{\ttop}^{n})\bigr)}
\end{equation*}
and 
\begin{equation*}
\xymatrix{
\vert \X\otimes C(\Box_{\ttop}^{n}) \vert=\vert \X \vert\otimes C(\Box_{\ttop}^{n}).}
\end{equation*}
\end{example}
\begin{remark}
The cognoscenti of homotopy theory will notice the formal similarities between $\vert \cdot \vert$ 
and the geometric realization functors of Milnor from semi-simplicial complexes to CW-complexes 
\cite{Milnor:realization} and of Morel-Voevodsky from simplicial sheaves to sheaves on some site 
\cite{MV:IHES}. 
Note that using the same script we obtain a geometric realization functor for every cocubical 
$\CC^{\ast}$-algebra.
The standard cocubical $\CC^{\ast}$-space meshes well with the monoidal products we shall consider in the 
sense that $C(\Box_{\ttop}^{n})$ and $C(\Box_{\ttop}^{1})\otimes\cdots\otimes C(\Box_{\ttop}^{1})$ 
are isomorphic as $\CC^{\ast}$-algebras, and hence as $\CC^{\ast}$-spaces.
\end{remark}
\begin{remark}
Note that $n\mapsto C(\Box_{\ttop}^{n})$ defines a functor 
$\Box\rightarrow\CC^{\ast}-\Spc\subset\Box\CC^{\ast}-\Spc$.
Since the category $\Box\CC^{\ast}-\Spc$ is cocomplete this functor has an enriched symmetric monoidal 
left Kan extension $\Box\Set\rightarrow\Box\CC^{\ast}-\Spc$ which commutes with colimits and sends 
$\Box^{\bullet}$ to $C(\Box_{\ttop}^{\bullet})$.
\end{remark}
The next result is reminiscent of \cite[Lemma 3.10]{MV:IHES} and \cite[Lemma B.1.3]{Jardine:MSS}.
\begin{lemma}
\label{lemma:geometricrealizationpreservesmonomorphisms}
The geometric realization functor $\vert\cdot\vert\colon\Box\CC^{\ast}-\Spc\rightarrow\CC^{\ast}-\Spc$
preserves monomorphisms.
\end{lemma}
\begin{proof}
For $i<j$ and $n\geq 2$ the cosimplicial identities imply there are pullback diagrams:
\begin{equation*}
\xymatrix{
C(\Box_{\ttop}^{n-2})\ar[r]^{d^{j-1}}\ar[d]_{d^{i}} & C(\Box_{\ttop}^{n-1})\ar[d]^{d^{i}}\\
C(\Box_{\ttop}^{n-1})\ar[r]^{d^{j}} & C(\Box_{\ttop}^{n}) }
\end{equation*}
Hence $\vert\partial\Box^{n}\vert$ is isomorphic to the union $\partial C(\Box_{\ttop}^{n})$ 
of the images $d^{i}\colon C(\Box_{\ttop}^{n-1})\rightarrow C(\Box_{\ttop}^{n})$,
and $\vert\partial\Box^{n}\vert\rightarrow\vert\Box^{n}\vert$ is a monomorphism for $n\geq 2$.
And therefore the lemma is equivalent to the fact that $C(\Box_{\ttop}^{\bullet})$ is augmented, 
i.e.~the two maps $\partial\Box^{1}\subseteq\Box^{1}$ induce an injection 
$C(\Box_{\ttop}^{0})\coprod C(\Box_{\ttop}^{0})\rightarrow C(\Box_{\ttop}^{1})$.
\end{proof}
\begin{remark}
Denote by $\Delta_{\ttop}^{\bullet}$ the topological standard cosimplicial set and by 
$C(\Delta_{\ttop}^{\bullet})$ the simplicial $\CC^{\ast}$-algebra $n\mapsto C(\Delta_{\ttop}^{n})$ 
of continuous complex-valued functions on $\Delta_{\ttop}^{n}$ which vanish at infinity.
As in the cubical setting, the corresponding cosimplicial $\CC^{\ast}$-space 
$C(\Delta_{\ttop}^{\bullet})$ defines a singular functor $\Sing_{\Delta}^{\bullet}$ and a geometric
realization functor $\vert\cdot\vert\colon\Delta\CC^{\ast}-\Spc\rightarrow\CC^{\ast}-\Spc$.
We note that $C(\Delta_{\ttop}^{\bullet})$ does not mesh well with monoidal products in the sense 
that $C(\Delta_{\ttop}^{n})\neq C(\Delta_{\ttop}^{1})\otimes\cdots\otimes C(\Delta_{\ttop}^{1})$. 
The other properties of the cubical singular and geometric realization functors in the above hold 
simplicially.
\end{remark}

\subsection{$\Group-\CC^{\ast}$-spaces}
\label{subsection:GG-CC-spaces}
Let $\Group$ be a locally compact group. 
In this section we indicate the steps required to extend the results in the previous section to $\Group-\CC^{\ast}$-algebras. 
Recall that a $\Group-\CC^{\ast}$-algebra is a $\CC^{\ast}$-algebra equipped with  a strongly continuous representation of $\Group$ 
by $\CC^{\ast}$-algebra automorphisms.
There is a corresponding category $\Group-\CC^{\ast}-\Alg$ comprised of $\Group-\CC^{\ast}$-algebras and $\Group$-equivariant $\ast$-homomorphisms.
Since every $\CC^{\ast}$-algebra acquires a trivial $\Group$-action, 
there is an evident functor $\CC^{\ast}-\Alg\rightarrow \Group-\CC^{\ast}-\Alg$.
It gives a unique $\Group-\CC^{\ast}$-algebra structure to $\C$ because the identity is its only automorphism.
Denote by $\otimes_{\Group-\CC^{\ast}-\Alg}$ a symmetric monoidal tensor product on $\Group-\CC^{\ast}-\Alg$ with unit $\C$.
To provide examples, 
note that if $\otimes_{\CC^{\ast}-\Alg}$ denotes the maximal or minimal tensor product on $\CC^{\ast}-\Alg$,
then $A\otimes_{\CC^{\ast}-\Alg} B$ inherits two strongly continuous $\Group$-actions and hence the structure of a 
$(\Group\times \Group)-\CC^{\ast}$-algebra for all objects $A,B\in \Group-\CC^{\ast}-\Alg$.
Thus $A\otimes_{\CC^{\ast}-\Alg} B$ becomes a $\Group-\CC^{\ast}$-algebra by restricting the $(\Group\times \Group)$-action to the diagonal.
For both choices of a tensor product on $\CC^{\ast}-\Alg$ this construction furnishes symmetric monoidal structures on $\Group-\CC^{\ast}-\Alg$ 
with unit the complex numbers turning $\CC^{\ast}-\Alg\rightarrow \Group-\CC^{\ast}-\Alg$ into a symmetric monoidal functor.
\vspace{0.1in}

With the above as background we obtain embeddings 
\begin{equation*}
\xymatrix{
\Group-\CC^{\ast}-\Alg \ar[r] &
\Group-\CC^{\ast}-\Spc \ar[r] &
\Box \Group-\CC^{\ast}-\Spc }
\end{equation*}
by running the same tape as for $\CC^{\ast}-\Alg$.
The following properties can be established using the same arguments as in the previous section.  
\begin{itemize}
\item 
$\Box \Group-\CC^{\ast}-\Spc$ is a closed symmetric monoidal category with symmetric monoidal product $\X\otimes^{\Group}\Y$,  
internal hom object $\underline{\Hom}^{\Group}(\X,\Y)$ and cubical function complex ${\hom}_{\Box \Group-\CC^{\ast}-\Spc}(\X,\Y)$
for cubical $\Group-\CC^{\ast}$-spaces $\X$ and $\Y$. 
The unit is representable by the complex numbers.
\item 
$\Box \Group-\CC^{\ast}-\Spc$ is enriched in cubical sets.
\item
$\Box \Group-\CC^{\ast}-\Spc$ is locally presentable.
\item
$\underline{\Hom}^{\Group}\bigl(S^{1}\otimes C_{0}(\R),-\bigr)$ is finitely presentable. 
\item
There exists a $G$-equivariant singular functor 
\begin{equation*}
\xymatrix{
\Sing^{\Group,\bullet}_{\Box}\colon\Box \Group-\CC^{\ast}-\Spc \ar[r] &
\Box \Group-\CC^{\ast}-\Spc. }
\end{equation*}
\end{itemize}
The categories of pointed cubical $\Group-\CC^{\ast}$-spaces, simplicial $\Group-\CC^{\ast}$-spaces and pointed simplicial $\Group-\CC^{\ast}$-spaces
acquire the same formal properties as $\Box \Group-\CC^{\ast}-\Spc$.

\subsection{Model categories}
\label{modelcategories}
In order to introduce $\CC^{\ast}$-homotopy theory properly we follow Quillen's ideas 
for axiomatizing categories in which we can ``do homotopy theory.''
A striking beauty of the axioms for a model structure is that algebraic categories such as 
chain complexes also admit natural model structures, 
as well as the suggestive geometric examples of topological spaces and simplicial sets. 
The axioms for a stable homotopy category, or even for a triangulated category, 
are often so cumbersome to check that the best way to construct such structures is as the 
homotopy category of some model structure.
The standard references for this material include \cite{DS:Modelcategories}, \cite{GJ:Modelcategories}, 
\cite{Hirschhorn:Modelcategories}, \cite{Hovey:Modelcategories} and \cite{Quillen:Homotopicalalgebra}.

\begin{definition}
\label{definition:modelcategory}
A model category is a category $\M$ equipped with three classes of maps called weak equivalences, 
cofibrations and fibrations which are denoted by $\overset{\sim}{\rightarrow}$, $\cof$ and $\fib$ 
respectively.
Maps which are both cofibrations and weak equivalences are called acyclic cofibrations and denoted by 
$\overset{\sim}{\cof}$; acyclic fibrations are defined similarly and denoted by $\overset{\sim}{\fib}$.
The following axioms are required \cite[Definition 1.1.4]{Hovey:Modelcategories}:
\begin{enumerate}[{\bf CM} 1:]
\item $\M$ is bicomplete.
\item (Saturation or two-out-of-three axiom) 
If $f\colon\X\rightarrow\Y$ and $g\colon\Y\rightarrow\W$ are maps in $\M$ and any two of $f$, $g$, 
and $gf$ are weak equivalences, then so is the third.
\item (Retract axiom) Every retract of a weak equivalence (respectively cofibration, fibration) 
is a weak equivalence (respectively cofibration, fibration).
\item (Lifting axiom) Suppose there is a commutative square in $\M$:
\begin{equation*}
\xymatrix{ 
\X \ar[r]\ar@{ >->}[d]_{p} & \Z \ar@{->>}[d]^{q} \\
\Y \ar@{ -->}[ur] \ar[r] & \W }
\end{equation*}
Then the indicated lifting $\Y\rightarrow\Z$ exists if either $p$ or $q$ is a weak equivalence.
\item\label{f} (Factorization axiom)
Every map $\X\rightarrow\W$ may be functorially factored in two ways, 
as $\X\overset{\sim}{\cof}\Y\fib\W$ and as $\X\cof\Z\overset{\sim}{\twoheadrightarrow}\W$.  
\end{enumerate}
\end{definition}

If every square as in {\bf CM} 4 has a lifting $\Y\rightarrow\Z$, 
then $\X\rightarrow\Y$ is said to have the left lifting property with respect to $\Z\rightarrow\W$.
The right lifting property is defined similarly.
When $\M$ is a model category, 
one may formally invert the weak equivalences to obtain the homotopy category $\Ho(\M)$ of $\M$ 
\cite[I.1]{Quillen:Homotopicalalgebra}.  
A model category is called pointed if the initial object and terminal object are the same. 
The homotopy category of any pointed model category acquires a suspension functor denoted by $\Sigma$.  
It turns out that $\Ho(\M)$ is a pre-triangulated category in a natural way 
\cite[\S7.1]{Hovey:Modelcategories}.
When the suspension is an equivalence, 
$\M$ is called a stable model category, 
and in this case $\Ho(\M)$ becomes a triangulated category \cite[\S7.1]{Hovey:Modelcategories}.  
We will give examples of such model structures later in this text.
\vspace{0.1in}

A Quillen map of model categories $\M\rightarrow\NN$ consists of a pair of adjoint functors 
\begin{equation*}
\xymatrix{
L\colon \M \ar@<3pt>[r] & \NN \colon R \ar@<3pt>[l] }
\end{equation*}
where the left adjoint $L$ preserves cofibrations and trivial cofibrations, 
or equivalently that $R$ preserves fibrations and trivial fibrations.  
Every Quillen map induces adjoint total derived functors between the homotopy categories 
\cite[I.4]{Quillen:Homotopicalalgebra}.  
The map is a Quillen equivalence if and only if the total derived functors are adjoint 
equivalences of the homotopy categories.  
\vspace{0.1in}

For the definition of a cofibrantly generated model category $\M$ with generating cofibrations $I$ and
generating acyclic cofibrations $J$ and related terminology we refer to \cite[\S2.1]{Hovey:Modelcategories}.
The definition entails that in order to check whether a map in $\M$ is an acyclic fibration or fibration it 
suffices to test the right lifting property with respect to $I$ respectively $J$.
In addition, 
the domains of $I$ are small relative to $I$-cell and likewise for $J$ and $J$-cell.
It turns out the (co)domains of $I$ and $J$ often have additional properties.
Next we first recall \cite[Definition 4.1]{Hovey:spectra}.
\begin{definition}
\label{definition:finitelygenerated}
A cofibrantly generated model category is called finitely generated if the domains and codomains of 
$I$ and $J$ are finitely presentable,
and almost finitely generated if the domains and codomains of $I$ are finitely presentable and there 
exists a set of trivial cofibrations $J'$ with finitely presentable domains and codomains such that a map 
with fibrant codomain is a fibration if and only if it has the right lifting property with respect to $J'$, 
i.e.~the map is contained in $J'$-inj.
\end{definition}

In what follows we will use the notion of a weakly finitely generated model structure introduced in \cite[Definition 3.4]{DRO:general}.   
\begin{definition}
\label{definition:weaklyfinitelygenerated}
A cofibrantly generated model category is called weakly finitely generated if the domains and 
the codomains of $I$ are finitely presentable,
the domains of the maps in $J$ are small, 
and if there exists a subset $J'$ of $J$ of maps with finitely presentable domains and codomains 
such that a map with fibrant codomain is a fibration if and only if it is contained in $J'$-inj.
\end{definition}

\begin{lemma}
\label{lemma:afgmodelcategories} (\cite[Lemma 3.5]{DRO:general}).
In weakly finitely generated model categories,   
the classes of acyclic fibrations, fibrations with fibrant codomains, fibrant objects, 
and weak equivalences are closed under filtered colimits.
\end{lemma}
\begin{remark}
Lemma \ref{lemma:afgmodelcategories} implies that in weakly finitely generated model categories,
the homotopy colimit of a filtered diagram maps by a weak equivalence to the colimit of the diagram.
This follows because the homotopy colimit is the total left derived functor of the colimit and 
filtered colimits preserves weak equivalences.
\end{remark}

Two fundamental examples of model structures are the standard model structures on the functor categories 
of simplicial sets $\Delta\Set\equiv[\Delta^{\op},\Set]$ constructed by 
Quillen \cite{Quillen:Homotopicalalgebra} and of cubical sets $\Box\Set\equiv[\Box^{\op},\Set]$ constructed 
independently by Cisinski \cite{Cisinski:presheavesasmodelsforhomotopytypes} and 
Jardine \cite{Jardine:Categoricalhomotopytheory}.
The box category $\Box$ has objects $1^{0}=\{0\}$ and $1^{n}=\{0,1\}^{n}$ for every $n\geq 1$.
The maps in $\Box$ are generated by two distinct types of maps which are subject to the dual of the 
cubical relations,  
and defined as follows.
For $n\geq 1$, $1\leq i\leq n$ and $\alpha=0,1$ define the coface map 
$\delta_{n}^{i,\alpha}\colon 1^{n-1}\rightarrow 1^{n}$ by 
$(\epsilon_{1},\cdots,\epsilon_{n-1})\mapsto
(\epsilon_{1},\cdots,\epsilon_{i-1},\alpha,\epsilon_{i},\cdots,\epsilon_{n-1})$.
And for $n\geq 0$ and $1\leq i\leq n+1$ the codegeneracy map 
$\sigma_{n}^{i}\colon 1^{n+1}\rightarrow 1^{n}$ is defined by  
$(\epsilon_{1},\cdots,\epsilon_{n+1})\mapsto
(\epsilon_{1},\cdots,\epsilon_{i-1},\epsilon_{i+1},\cdots,\epsilon_{n+1})$.
Recall that a map $f$ in $\Box\Set$ is a weak equivalence if applying the triangulation functor yields 
a weak equivalence $\vert f\vert$ of simplicial sets.
A cofibration of cubical sets is a monomorphism.
The Kan fibrations are forced by the right lifting property with respect to all acyclic monomorphisms.

\begin{theorem}
\label{theorem:cubicalsetmodelstructure}
(\cite{Cisinski:presheavesasmodelsforhomotopytypes},\cite{Jardine:Categoricalhomotopytheory})
The weak equivalences, cofibrations and Kan fibrations define a cofibrantly generated and proper 
model structure on $\Box\Set$ for which the triangulation functor is a Quillen equivalence.
\end{theorem}
\begin{example}
The cubical set $\partial\Box^{n}$ is the subobject of the standard $n$-cell $\Box^{n}$ 
generated by all faces $d_{i}^{\alpha}:\Box^{n-1}\rightarrow\Box^{n}$. 
It follows there is a coequalizer diagram of cubical sets
\begin{equation*}
\xymatrix{
\coprod_{0\leq i<j\leq n,(\alpha_{1},\alpha_{2})}
\Box^{n-2}\ar@<3pt>[r]\ar@<-3pt>[r] & 
\coprod_{(i,\alpha)}\Box^{n-1}\ar[r] &
\partial\Box^{n}. }
\end{equation*}
The cubical set $\sqcap^{n}_{(\alpha,i)}$ is the subobject of $\Box^{n}$ generated by all faces 
$d_{j}^{\gamma}\colon\Box^{n-1}\rightarrow\Box^{n}$ for $(j,\gamma)\neq(i,\alpha)$.
There is a coequalizer diagram of cubical sets where the first disjoint union is indexed over  
pairs for which $0\leq j_{1}<j_{2}\leq n$ and $(j_{k},\gamma_{k})\neq (i,\alpha)$ for $k=1,2$ 
\begin{equation*}
\xymatrix{
\coprod_{(j_{1},\gamma_{1}),(j_{2},\gamma_{2})}
\Box^{n-2}\ar@<3pt>[r]\ar@<-3pt>[r] & 
\coprod_{(j,\gamma)\neq(i,\alpha)}\Box^{n-1}\ar[r] &
\sqcap^{n}_{(\alpha,i)}. }
\end{equation*}
\end{example}

The sets of all monomorphisms $\partial\Box^{n}\subset\Box^{n}$ and $\sqcap^{n}_{(\alpha,i)}\subset\Box^{n}$
furnish generators for the cofibrations respectively the acyclic cofibrations of cubical sets.
By using these generators one can show that the model structure in Theorem 
\ref{theorem:cubicalsetmodelstructure} is weakly finitely generated and monoidal with respect to 
the closed symmetric monoidal product $\otimes_{\Box\Set}$ on $\Box\Set$ introduced by Jardine in 
\cite[\S3]{Jardine:Categoricalhomotopytheory}.
The monoidal product is determined by $\Box^{m}\otimes_{\Box\Set}\Box^{n}=\Box^{m+n}$ and the internal 
homs or cubical function complexes are defined by 
${\hom}_{\Box\Set}(K,L)_{n}\equiv\Box\Set(K\otimes_{\Box\Set}\Box^{n},L)$
as in Day's work \cite{Day:closedfunctorsI}.
This structure allows to define a notion of cubical model categories in direct analogy with Quillen's 
{\bf SM}7 axiom for simplicial model categories.
We include a sketch proof of the next result.
\begin{lemma}
\label{lemma:cubicalmappingcylinderfactorization}
Suppose $\M$ is a cubical model category and $f\colon\X\rightarrow\Y$ is a map between cofibrant objects.
Then the cubical mapping cylinder $\cyl(f)$ is cofibrant, 
$\X\rightarrow\cyl(f)$ is a cofibration and $\cyl(f)\rightarrow\Y$ is a cubical homotopy equivalence.
\end{lemma}
\begin{proof}
The cubical mapping cylinder $\cyl(f)$ is defined as the pushout of the diagram    
\begin{equation*}
\xymatrix{
\X\otimes\Box^{1} & \X \ar[l] \ar[r]^{f} & \Y,} 
\end{equation*}
induced by the embedding $1^{0}\rightarrow 1^{1}$, $0\mapsto 0$, 
via the Yoneda lemma. 
This construction uses the isomorphism $\X=\X\otimes\Box^{0}$. 
The second embedding $1^{0}\rightarrow 1^{1}$, $0\mapsto 1$, yields the map $\X\rightarrow\cyl(f)$,
while the diagram
\begin{equation*}
\xymatrix{
\X \ar[r]^-{f}\ar[d] & \Y \ar@{=}[d] \\
\X\otimes\Box^{1}  \ar[r] & \Y } 
\end{equation*}
implies there is a map $\cyl(f)\rightarrow\Y$, 
where the lower horizontal map 
\begin{equation*}
\xymatrix{
\X\otimes\Box^{1}\ar[r] &
\X\otimes\Box^{0}=\X\ar[r] & \Y }
\end{equation*}
is induced by the unique map $1^{1}\rightarrow 1^{0}$.
This produces the desired factorization.
\vspace{0.1in}

The cofibrancy assumption on $\X$ implies $\X\otimes (\partial\Box^{1}\subset\Box^{1})$ is a cofibration
since the model structure is cubical. 
It follows that $\X\rightarrow\X\coprod\Y\rightarrow\cyl(f)$ is a cofibration on account of the pushout diagram: 
\begin{equation*}
\xymatrix{
\X\otimes\partial\Box^{1}=\X\coprod\X \ar[r]^-{\id_{\X}\coprod f}\ar[d] & \X\coprod\Y \ar[d] \\
\X\otimes\Box^{1}  \ar[r] & \cyl(f) } 
\end{equation*}
Clearly this shows $\cyl(f)$ is cofibrant. 
Finally,
using that $\X\otimes \Box^{1}$ is a cylinder object for any cofibrant $\X$, cf.~\cite[II Lemma 3.5]{GJ:Modelcategories},  
one verifies routinely that $\cyl(f)\rightarrow\Y$ is a cubical homotopy equivalence.
\end{proof}
\begin{remark}
\label{remark:pointedcubicalsets}
The above remains valid for pointed cubical sets $\Box\Set_{\ast}\equiv[\Box^{\op},\Set_{\ast}]$.
Note that the cofibrations are generated by the monomorphisms $(\partial\Box^{n}\subset\Box^{n})_{+}$ 
and the acyclic cofibrations by the monomorphisms $(\sqcap^{n}_{(\alpha,i)}\subset\Box^{n})_{+}$.
\end{remark}
\begin{proposition}
\label{proposition:homotopyclassesofmaps}
Suppose $\X$ is cofibrant and $\Y$ is fibrant in some cubical model category $\M$ with cubical function complex ${\hom}_{\M}(\X,\Y)$.
Then there is an isomorphism
\begin{equation*}
\Ho(\M)(\X,\Y)=\pi_{0}{\hom}_{\M}(\X,\Y).
\end{equation*}
\end{proposition}
\begin{proof}
By {\bf SM}7, 
which ensures that ${\hom}_{\M}(\X,\Y)$ is fibrant, 
the right hand side is the set of homotopies $\Box^{1}\rightarrow{\hom}_{\M}(\X,\Y)$, 
or equivalently $\X\otimes\Box^{1}\rightarrow\Y$,
i.e.~homotopies between $\X$ and $\Y$ because $\X\otimes\Box^{1}$ is a cylinder object for $\X$.
\end{proof}
\begin{corollary}
\label{corollary:weakequivalencesfibrantobjects}
A map $\X\rightarrow\Y$ in a cubical model category $\M$ with a cofibrant replacement functor $\QQ\rightarrow\id_{\M}$ is a weak equivalence 
if and only if for every fibrant object $\Z$ of $\M$ the induced map ${\hom}_{\M}(\QQ\Y,\Z)\rightarrow{\hom}_{\M}(\QQ\X,\Z)$ is a weak equivalence 
of cubical sets.
\end{corollary}
\begin{proof}
The map $\X\rightarrow\Y$ is a weak equivalence if and only if $\QQ\X\rightarrow\QQ\Y$ is so.
For the if implication it suffices to show there is an induced isomorphism 
\begin{equation*}
\xymatrix{
\Ho(\M)(\QQ\Y,\Z)\ar[r] & \Ho(\M)(\QQ\X,\Z)  }
\end{equation*}
for every fibrant $\Z$.
This follows from Proposition \ref{proposition:homotopyclassesofmaps} since, 
by the assumption, 
${\hom}_{\M}(\QQ\Y,\Z)\rightarrow {\hom}_{\M}(\QQ\X,\Z)$ is a weak equivalence of cubical sets, 
and hence there is an isomorphism $\pi_{0}{\hom}_{\M}(\QQ\Y,\Z)\rightarrow\pi_{0}{\hom}_{\M}(\QQ\X,\Z)$.
\vspace{0.1in}

Conversely,
Ken Brown's lemma \cite[Lemma 1.1.12]{Hovey:Modelcategories} shows that we may assume $\QQ\X\rightarrow\QQ\Y$ is an acyclic cofibration since 
weak equivalences have the two-out-of-three property.
By {\bf SM}7, 
the map ${\hom}_{\M}(\QQ\Y,\Z)\rightarrow {\hom}_{\M}(\QQ\X,\Z)$ is then an acyclic fibration.
\end{proof}
\begin{remark}
The dual of Corollary \ref{corollary:weakequivalencesfibrantobjects} shows that $\X\rightarrow\Y$ is a weak equivalence 
if and only if for every cofibrant object $\W$ of $\M$ and fibrant replacement functor $\RRR$ the induced map
${\hom}_{\M}(\W,\RRR\X)\rightarrow{\hom}_{\M}(\W,\RRR\Y)$ is a weak equivalence 
of cubical sets.
\end{remark}
\vspace{0.1in}

The homotopy colimit of a small diagram of cubical sets ${\bf X}\colon I\rightarrow\Box\Set$ is the 
cubical set defined by 
\begin{equation}
\label{homotopycolimit}
\underset{I}{\hocolim}\,\,{\bf X}\equiv 
B(-\downarrow I)^{\op}\otimes_{[I,\Box\Set]}{\bf X}.
\end{equation}
Here $B(i\downarrow I)$ is the cubical nerve of the undercategory $i\downarrow I$ so that there is 
a natural map from (\ref{homotopycolimit}) to the colimit of ${\bf X}$.
In model categorical terms the homotopy colimit of ${\bf X}$ is a left derived functor of the colimit
\begin{equation}
\label{leftderivedcolimit}
{\bf L}\,\,\underset{I}{\colim}\,\,{\bf X}\equiv 
\underset{I}{\colim}\,\,\QQ\,\,{\bf X}.
\end{equation}
Here $\QQ$ is a cofibrant replacement functor.
The homotopy limit is defined dually.
Items (\ref{homotopycolimit}) and (\ref{leftderivedcolimit}) define homotopy functors and a naturally 
induced weak equivalence 
\begin{equation*}
\xymatrix{
{\bf L}\,\,\underset{I}{\colim}\,\,{\bf X}\ar@<3pt>[r] & 
\underset{I}{\hocolim}\,\,{\bf X}.}
\end{equation*}
Homotopy colimits and homotopy limits of small diagrams of simplicial sets are defined by the same script.
\vspace{0.1in}

Combining the notions of cofibrantly generated and locally presentable one arrives at the following definition.
\begin{definition}
A model category is called combinatorial if it is cofibrantly generated and the underlying category is locally presentable.
\end{definition}

It is useful to note there exists an accessible fibrant replacement functor in every combinatorial model category
\cite[Proposition 3.2]{Rosicky:brownrepresentability}.
Recall that a functor between $\lambda$-accessible categories is $\lambda$-accessible if it preserves $\lambda$-filtered colimits.
We shall use this when setting up model structures on $\CC^{\ast}$-functors in Section \ref{subsection:Cstarfunctors}.
\vspace{0.1in}

Next we review the process of localizing model structures as in \cite{Hirschhorn:Modelcategories}.
Suppose $\mathcal{L}$ is a set of maps in a model category $\M$.
Then the Bousfield localization $\M_{\mathcal{L}}$ of $\M$ with respect to $\mathcal{L}$ is a new 
model structure on $\M$ having the same class of cofibrations, 
but in which the maps of $\mathcal{L}$ are weak equivalences.  
Furthermore, 
$\M_{\mathcal{L}}$ is the initial such model structure in the sense that if $\M\rightarrow\NN$ is 
some Quillen functor that sends the maps of $\mathcal{L}$ to weak equivalences, 
then $\M_{\mathcal{L}}\rightarrow\NN$ is also a Quillen functor.  
The total right derived functor of the identity $\M_{\mathcal{L}}\rightarrow\M$ is fully faithful.
The Bousfield localization $\M_{\mathcal{L}}$ exists when $\M$ is left proper and combinatorial by (unpublished) work 
of Jeff Smith.
Reference \cite{Barwick} gives a streamlined presentation of this material.
\vspace{0.1in}

In \cite{Hirschhorn:Modelcategories},
Bousfield localizations are shown to exist for left proper cellular model categories, 
which are special kinds of cofibrantly generated model categories.
Three additional hypotheses are required on the sets of generating cofibrations $I$ and acyclic cofibrations $J$:
\begin{itemize}
\item
The domains and codomains of $I$ are compact relative to $I$.
\item
The domains of $J$ are small relative to the cofibrations.
\item
Cofibrations are effective monomorphisms.
\end{itemize}

When presented with the definition of ``compact relative to $I$'' for the first time it is helpful to keep the example of CW-complexes in mind.
We recall:
\begin{definition}
\label{definition:presentationofcellcomplexes}
A presentation of a relative $I$-cell complex $f\colon\X\rightarrow\Y$, 
i.e.~a transfinite composition of pushouts of coproducts of maps in $I$, 
consists of a presentation ordinal $\lambda$, 
a $\lambda$-sequence $F$ in $\M$, 
a collection $\{(T^{\beta},e^{\beta},h^{\beta})_{\beta<\lambda}\}$ where $T^{\beta}$ is a set and $e^{\beta}:T^{\beta}\rightarrow I$ a map for which the 
following properties hold.
\begin{itemize}
\item The composition of $F$ is isomorphic to $f$.
\item If $i\in T^{\beta}$ and $e^{\beta}_{i}\colon C_{i}\rightarrow D_{i}$ is the image of $i$, 
      then $h^{\beta}_{i}:C_{i}\rightarrow F_{\beta}$ is a map.
\item For every $\beta<\lambda$ there is a pushout:
\begin{equation*}
\xymatrix{
\coprod_{i\in T^{\beta}} C_{i} \ar[r]^-{\coprod e^{\beta}_{i}} \ar[d]_-{\coprod h^{\beta}_{i}} & \coprod_{i\in T^{\beta}}  D_{i} \ar[d] \\
F_{\beta} \ar[r] & F_{\beta +1} }
\end{equation*}
\end{itemize}
The map $f$ has said to have size the cardinality of its set of cells $\coprod _{\beta <\lambda}T^{\beta}$ and the presentation ordinal of a cell 
$e$ of $f$ is the ordinal $\beta$ such that $e\in T^{\beta}$.
\end{definition}

Next we need the definition of a subcomplex;
the motivational example is that of a CW-subcomplex.
\begin{definition}
\label{definition:subcomplexes}
A subcomplex of a presentation $F\colon\lambda\rightarrow\M,\{(T^{\beta},e^{\beta},h^{\beta})_{\beta<\lambda}\}$ of an $I$-cell complex 
$f\colon\X\rightarrow\Y$ is a collection $\{(\widetilde{T}^{\beta},\widetilde{e}^{\beta},\widetilde{h}^{\beta})_{\beta<\lambda}\}$ such that
\begin{itemize}
\item For every $\beta <\lambda$, $\widetilde{T}^{\beta}\subseteq T^{\beta}$ and $\widetilde{e}^{\beta}$ is the restriction of 
$e^{\beta}$ to $\widetilde{T}^{\beta}$.
\item There is a $\lambda$-sequence $\widetilde{F}$ such that $\widetilde{F}_{0}=F_{0}$ and a natural transformation 
$\widetilde{F}\rightarrow F$ such that for every $\beta<\lambda$ and every $i\in\widetilde{T}^{\beta}$,
the map $\widetilde{h}^{\beta}_{i}:C_{i}\rightarrow \widetilde{F}_{\beta}$ is a factorization of $h^{\beta}_{i}:C_{i}\rightarrow F_{\beta}$ 
through $\widetilde{F}_{\beta}\rightarrow F_{\beta}$.
\item For every $\beta<\lambda$ there is a pushout:
\begin{equation*}
\xymatrix{
\coprod_{i\in\widetilde{T}^{\beta}} C_{i} \ar[r]^-{\coprod\widetilde{e}^{\beta}_{i}}
\ar[d]_-{\coprod\widetilde{h}^{\beta}_{i}} & \coprod_{i\in\widetilde{T}^{\beta}}  D_{i} \ar[d] \\
\widetilde{F}_{\beta} \ar[r] & \widetilde{F}_{\beta +1} }
\end{equation*}
\end{itemize}
\end{definition}

We are ready to make precise the condition ``compact relative to $I$.''
\begin{definition}
\label{definition:compactrelativeto}
\begin{itemize}
\item Let $\kappa$ be a cardinal. 
An object $\Z$ of $\M$ is $\kappa$-compact relative to $I$ if for every presented relative $I$-cell complex $f\colon\X\rightarrow\Y$,
every map $\Z\rightarrow\Y$ factors through a subcomplex of $f$ of size at most $\kappa$.
\item An object $\Z$ of $\M$ is compact relative to $I$ if it is $\kappa$-compact relative to $I$ for some cardinal $\kappa$.
\end{itemize}
\end{definition}
\vspace{0.1in}

Recall that a map $\X\rightarrow\Y$ in $\M$ is an effective monomorphism if it is the equalizer of the two naturally induced maps
\begin{equation*}
\xymatrix{
\Y 
\ar@<3pt>[r]\ar@<-3pt>[r] & 
\Y\coprod_{\X}\Y. }
\end{equation*}
In the category of sets, 
the effective monomorphisms are precisely the injective maps. 
It is important to note that the Bousfield localization of a left proper cellular model category is also a left proper and cellular model category.
\vspace{0.1in}

In the following sections we shall detail the localization process for various model structures 
on cubical $\CC^{\ast}$-spaces.  
A common theme for all of these model structures is that we know precisely what the fibrant objects 
should be in the localized model structure, 
and this forces the new weak equivalences defined by cubical function complexes.
Regardless of the shape of $J$ in $\M$, 
it is often problematic to explicate a new set of generating trivial cofibrations in the localized 
model structure.

\begin{proposition}
\label{proposition:almostfinitelygeneratedpreservedunderlocalization}
(\cite[Proposition 4.2]{Hovey:spectra})
If $\M$ is an almost finitely generated, combinatorial, cubical and left proper model category,
$\mathcal{L}$ a set of cofibrations of $\M$ such that for every domain and codomain $\X$ of $\mathcal{L}$ 
and every finitely presentable cubical set $K$, $\X\otimes K$ is finitely presentable,
then the Bousfield localization of $\M$ with respect to $\mathcal{L}$ is almost finitely generated.
\end{proposition}
\begin{proof}
If $\X\rightarrow\Y$ is a map in $\mathcal{L}$ the set of maps 
\begin{equation}
\label{extramaps}
\xymatrix{
(\X\otimes\Box^{n})\coprod_{\X\otimes\sqcap^{n}_{(\alpha,i)}}
(\Y\otimes\sqcap^{n}_{(\alpha,i)})
\ar[r] & \Y\otimes \Box^{n} }
\end{equation}
detect $\mathcal{L}$-local fibrant objects. 
An $\mathcal{L}$-local fibration between $\mathcal{L}$-local fibrant objects is an ordinary fibration.  
Now let $J'$ consist of the maps in (\ref{extramaps}) together with the old set of generating trivial cofibrations.
\end{proof}

If $\M$ is a symmetric monoidal category, 
then in order for the homotopy category $\Ho(\M)$ to acquire an induced closed symmetric 
monoidal structure the monoidal structure is not required to preserve equivalences on the nose.
By \cite[Theorem 4.3.2]{Hovey:Modelcategories} it suffices that the unit is cofibrant and the 
monoidal product $\otimes$ satisfies the pushout product axiom, 
i.e.~for cofibrations $\X\rightarrow\Y$ and $\Z\rightarrow\W$ the pushout product map
\begin{equation*}
\xymatrix{
(\X\otimes\W)\coprod_{\X\otimes\Z} (\Y\otimes\Z)\ar[r] & \Y\otimes\W}
\end{equation*}
is a cofibration, 
and if either of the two maps is an acyclic cofibration, 
then so is their pushout product map.
If this property holds, 
then $\M$ is called a monoidal model category \cite[\S4.2]{Hovey:Modelcategories}.
\newpage

\section{Unstable $\CC^{\ast}$-homotopy theory}
\label{section:Cstarhomotopytheory}
In this section we shall introduce four types of unstable model structures on cubical $\CC^{\ast}$-spaces.
The two pointwise model structures are lifted from the model structure on cubical sets in canonical ways,
while the other three pairs of model structures are determined by short exact sequences of 
$\CC^{\ast}$-algebras, 
matrix invariance and homotopy invariance. 
Throughout we fix a small skeleton for $\CC^{\ast}-\Alg$.

\subsection{Pointwise model structures}
\begin{definition}
A map $\X\rightarrow\Y$ of cubical $\CC^{\ast}$-spaces is a projective fibration if for every 
$\CC^{\ast}$-algebra $A$ there is a Kan fibration of cubical sets $\X(A)\rightarrow\Y(A)$.
Pointwise weak equivalences are defined similarly.
Projective cofibrations of cubical $\CC^{\ast}$-spaces are maps having the 
left lifting property with respect to all pointwise acyclic projective fibrations.
\end{definition}

We refer to the following as the pointwise projective model structure.
\begin{theorem}
\label{theorem:projectiveobjectwise}
The classes of projective cofibrations, projective fibrations and pointwise weak equivalences 
determine a combinatorial model structure on $\Box\CC^{\ast}-\Spc$.
\end{theorem}

The proof of Theorem \ref{theorem:projectiveobjectwise} is straightforward and will not be considered 
in much detail. 
Rather we give an outline and refer to \cite{Hirschhorn:Modelcategories} for a more general result 
concerning diagram categories. 
Evaluating cubical $\CC^{\ast}$-spaces at $A$ yields the $A$-sections functor
\begin{equation*}
\xymatrix{ \Ev_A\colon\Box\CC^{\ast}-\Spc\ar[r] & \Box\Set.}
\end{equation*}
Lemma \ref{lemma:finitelypresentable} shows $\Box\CC^{\ast}-\Spc$ is locally finitely presentable 
and the same holds for $\Box\Set$.
Now since limits and colimits are formed pointwise in these categories, 
it follows from \cite[Theorem 1.66]{AR:book} that $\Ev_A$ acquires a left adjoint functor 
which the Yoneda lemma shows is given by $A\otimes -$.
By using the model structure on $\Box\Set$ we find that the class of projective cofibrations is generated by the set
\begin{equation*}
I_{\Box\CC^{\ast}-\Spc}\equiv\{A\otimes(\partial\Box^{n}\subset\Box^{n})\}^{n\geq 0} 
\end{equation*}
indexed by representatives $A$ of the isomorphism classes in $\CC^{\ast}-\Alg$.
Likewise, 
the class of pointwise acyclic projective cofibrations is generated by 
\begin{equation*}
J_{\Box\CC^{\ast}-\Spc}\equiv\{A\otimes(\sqcap^{n}_{(\alpha,i)}\subset\Box^{n})\}^{n>0}_{0\leq i\leq n} .
\end{equation*}
It follows that every map between cubical $\CC^{\ast}$-spaces acquires a factorization through some 
sequential $I_{\Box\CC^{\ast}-\Spc}$-cell (respectively $J_{\Box\CC^{\ast}-\Spc}$-cell) composed with a 
pointwise acyclic projective fibration (respectively projective fibration).

For example, 
to prove the claim for $J_{\Box\CC^{\ast}-\Spc}$, 
assume $\X\rightarrow\Y$ is a projective fibration and consider commutative diagrams 
of cubical $\CC^{\ast}$-spaces of the form:
\begin{equation}
\label{firstliftingdiagram}
\minCDarrowwidth20pt
\begin{CD}
A\otimes\sqcap^{n}_{(\alpha,i)} @>>> \X \\
@VVV @VVV\\
A\otimes\Box^{n} @>>> \Y
\end{CD}
\end{equation}
By the Yoneda lemma (\ref{functioncomplexofrepresentable}) and (\ref{0-celladjunction}), 
such diagrams are in one-to-one correspondence with commutative diagrams of cubical sets of the form:
\begin{equation}
\label{secondliftingdiagram}
\minCDarrowwidth20pt
\begin{CD}
\sqcap^{n}_{(\alpha,i)} @>>> {\hom}_{\Box\CC^{\ast}-\Spc}(A,\X) @>\cong>> \X(A)\\
@VVV @VVV @VVV\\
\Box^{n} @>>> {\hom}_{\Box\CC^{\ast}-\Spc}(A,\Y) @>\cong>> \Y(A)
\end{CD}
\end{equation}
The assumption implies there exists a lifting $\Box^{n}\rightarrow\X(A)$ in (\ref{secondliftingdiagram}), 
which means there is a lifting $A\otimes\Box^{n}\rightarrow\X$ in (\ref{firstliftingdiagram}).  
Now suppose that $\X\rightarrow\Y$ is a map in $\Box\CC^{\ast}-\Spc$.
Define the pushout $\Y^{0}$ by combining all commutative diagrams of the following form, 
where $n\geq 0$, $\alpha=0,1$ and $A$ runs through the isomorphism classes of $\CC^{\ast}$-algebras:
\begin{equation}
\label{smallobjectdiagram}
\minCDarrowwidth20pt
\begin{CD}
A\otimes\sqcap^{n}_{(\alpha,i)} @>>> \X \\
@VVV @VVV\\
A\otimes\Box^{n} @>>> \Y
\end{CD}
\;\;\;\;
\rightsquigarrow
\;\;\;\;
\minCDarrowwidth20pt
\begin{CD}
\coprod A\otimes\sqcap^{n}_{(\alpha,i)} @>>> \X \\
@VVV @VVV\\
\coprod A\otimes\Box^{n} @>>> \Y^{0}
\end{CD}
\end{equation}
Then $\X\rightarrow\Y^{0}$ is a pointwise acyclic projective cofibration because the induced map 
of coproducts in (\ref{smallobjectdiagram}) is so and the evaluation functor $\Ev_{A}$ preserves 
pushouts and acyclic cofibrations.
By iterating this construction countably many times we obtain a diagram where the horizontal maps 
are pointwise acyclic projective cofibrations:
\begin{equation}
\label{countablediagram}
\xymatrix{
\X\ar[r] & \Y^{0}\ar[r] & \Y^{1}\ar[r] & \cdots\ar[r] & \Y^{n}\ar[r] & \cdots }
\end{equation}
Letting $\Y^{\infty}$ denote the colimit of (\ref{countablediagram}) yields a factorization of the 
original map
\begin{equation}
\label{factorization}
\xymatrix{
\X\ar[r] & \Y^{\infty}\ar[r] & \Y.}
\end{equation}
It is straightforward to show that $\X\rightarrow\Y^{\infty}$ is a pointwise acyclic projective 
cofibration and the map $\Y^{\infty}\rightarrow\Y$ acquires the right lifting property with respect 
to the set $J_{\Box\CC^{\ast}-\Spc}$.
This small object argument is due to Quillen \cite{Quillen:Homotopicalalgebra} and shows that every map of 
cubical $\CC^{\ast}$-spaces factors into a pointwise acyclic cofibration composed with a projective fibration
(\ref{factorization}).
The remaining arguments constituting a proof of Theorem \ref{theorem:projectiveobjectwise} are of a similar 
flavor,  
and one shows easily that the pointwise projective model is weakly finitely generated.
We deduce the following result.
\begin{corollary}
Pointwise weak equivalences, projective fibrant objects, 
pointwise acyclic projective fibrations and projective fibrations are closed under filtered colimits.
\end{corollary}

Recall that $\X\rightarrow\Y$ is a projective cofibration if and only if it is a relative 
$I_{\Box\CC^{\ast}-\Spc}$-cell complex or a retract thereof.
Hence projective cofibrations are pointwise cofibrations, 
a.k.a.~monomorphisms or injective cofibrations of cubical $\CC^{\ast}$-spaces.
\begin{lemma}
\label{projectiveimpliesinjective}
Every projective cofibration is a monomorphism.
\end{lemma}
\begin{lemma}
\label{lemma:pprojectiveproperness}
The pointwise projective model is a proper model structure.
\end{lemma}
\begin{proof}
This is a routine check using properness of the model structure on cubical sets and Lemma \ref{projectiveimpliesinjective}.
\end{proof}

If $A$, $B$ are $\CC^{\ast}$-algebras and $K$, $L$ are cubical sets there is a natural isomorphism 
\begin{equation*}
(A\otimes K)\otimes (B\otimes L)=
(A\otimes_{\CC^{\ast}-\Alg}B)\otimes (K\otimes_{\Box\Set}L).
\end{equation*}
This follows since for every cubical $\CC^{\ast}$-space $\X$ there are natural isomorphisms 
\begin{align*}
\Box\CC^{\ast}-\Spc\bigl((A\otimes K)\otimes (B\otimes L),\X\bigr) & =
\Box\Set\bigl(K\otimes_{\Box\Set}L,\X(A\otimes_{\CC^{\ast}-\Alg}B)\bigr) \\ & =
\Box\CC^{\ast}-\Spc\bigl((A\otimes_{\CC^{\ast}-\Alg}B)\otimes (K\otimes_{\Box\Set}L),\X\bigr).
\end{align*}
\begin{lemma}
\label{lemma:projectivecofibrant}
If $A$ is a $\CC^{\ast}$-algebra and $K$ is a cubical sets, 
then $A\otimes K$ is a projective cofibrant cubical $\CC^{\ast}$-space.
In particular, 
every $\CC^{\ast}$-algebra is projective cofibrant.
\end{lemma}
\begin{proof}
By definition, 
$A\otimes-\colon\Box\Set\rightarrow\Box\CC^{\ast}-\Spc$ is a left Quillen functor for the pointwise 
projective model structure on $\Box\CC^{\ast}-\Spc$ and every cubical set is cofibrant.
We note the assertion for $\CC^{\ast}$-algebras follows by contemplating the set $I_{\Box\CC^{\ast}-\Spc}$ 
of generating projective cofibrations.
\end{proof}

Since every map between discrete cubical sets is a Kan fibration we get:
\begin{lemma}
\label{lemma:projectivefibrant}
Every $\CC^{\ast}$-algebra is projective fibrant.
\end{lemma}

In addition to the Quillen adjunction gotten by evaluating cubical $\CC^{\ast}$-spaces at a fixed 
$\CC^{\ast}$-algebra, 
the constant diagram functor from $\Box\Set$ to $\Box\CC^{\ast}-\Spc$ has as left adjoint the colimit 
functor, which can be derived.
Each evaluation functor passes directly to a functor between the corresponding homotopy categories.
\vspace{0.1in}

Next we observe that the projective model is compatible with the enrichment of cubical 
$\CC^{\ast}$-spaces in cubical sets introduced in \S\ref{subsection:CC-spaces}.

\begin{lemma}
\label{lemma:pprojectivecubical}
The pointwise projective model is a cubical model structure.
\end{lemma}
\begin{proof}
If $i\colon\X\cof\Y$ is a projective cofibration of cubical $\CC^{\ast}$-spaces and $p\colon K\cof L$ 
is a cofibration of cubical sets,  
we claim the naturally induced pushout map
\begin{equation*}
\xymatrix{
\X\otimes L\coprod_{\X\otimes K}\Y\otimes K\ar[r] & \Y\otimes L}
\end{equation*}
is a projective cofibration, 
and a pointwise acyclic projective cofibration if in addition either $i$ is a pointwise weak equivalence 
or $p$ is a weak equivalence. 
In effect, 
the model structure on $\Box\Set$ is cubical and, as was noted above,  
evaluating projective cofibrations produce cofibrations or monomorphisms of cubical sets.
Since colimits of cubical $\CC^{\ast}$-spaces, 
in particular all pushouts, 
are formed pointwise the assertions follow.
\end{proof}

One shows easily using the cotensor structure that Lemma \ref{lemma:pprojectivecubical} is equivalent to
the following result: 
\begin{corollary}
\label{corollary:pprojectivecubicalstructurecorollary}
The following statements hold and are equivalent.
\begin{itemize}
\item 
If $j\colon\U\cof\V$  is a projective cofibration and $k\colon\Z\fib\W$ is a projective fibration, 
then the pullback map
\begin{equation*}
\xymatrix{
{\hom}_{\Box\CC^{\ast}-\Spc}(\V,\Z)\ar[r] &
{\hom}_{\Box\CC^{\ast}-\Spc}(\V,\W)
\times_{{\hom}_{\Box\CC^{\ast}-\Spc}(\U,\W)}{\hom}_{\Box\CC^{\ast}-\Spc}(\U,\Z)}
\end{equation*}
is a Kan fibration of cubical sets which is a weak equivalence if in addition either $j$ or $k$ is 
pointwise acyclic.
\item
If $p\colon K\cof L$ is a cofibration of cubical sets and $k\colon\Z\fib\W$ is a projective fibration, 
then 
\begin{equation*}
\xymatrix{
\Z^L\ar[r] &
\Z^K\times_{\W^{K}}\W^L}
\end{equation*}
is a projective fibration which is pointwise acyclic if either $k$ is pointwise acyclic or $p$ is acyclic.
\end{itemize} 
\end{corollary}

Recall that every $\CC^{\ast}$-algebra,
in particular the complex numbers, 
determines a projective cofibrant representable cubical $\CC^{\ast}$-space.
Thus the next result verifies that the projective model is a monoidal model structure.

\begin{lemma}
\label{lemma:pprojectivemonoidal}
The following statements hold and are equivalent.
\begin{itemize}
\item
If $i\colon\X\cof\Y$ and $j\colon\U\cof\V$ are projective cofibrations, then 
\begin{equation*}
\xymatrix{
\X\otimes\V\coprod_{\X\otimes\U}\Y\otimes\U\ar[r] & \Y\otimes\V}
\end{equation*}
is a projective cofibration which is pointwise acyclic if either $i$ or $j$ is.
\vspace{0.1cm}
\item
If $j\colon\U\cof\V$ is a projective cofibration and $k\colon\Z\fib\W$ is a projective fibration, 
then the pullback map
\begin{equation*}
\xymatrix{
\underline{\Hom}(\V,\Z)
\ar[r] & 
\underline{\Hom}(\V,\W)\times_{\underline{\Hom}(\U,\W)}\underline{\Hom}(\U,\Z)}
\end{equation*}
is a projective fibration which is pointwise acyclic if either $j$ or $k$ is.
\end{itemize} 
\end{lemma}
\begin{remark}
Lemma \ref{lemma:pprojectivemonoidal} shows that the pointwise projective model structures on 
$\Box\CC^{\ast}-\Spc$ is monoidal \cite[Definition 3.1]{SS:monoidal}.
Now since the complex numbers $\C$ is a projective cofibrant cubical $\CC^{\ast}$-space, 
it follows that $\otimes$ induces a monoidal product on the associated homotopy category.
\end{remark}
\begin{remark}
We note that since limits are defined pointwise, 
the second part of Lemma \ref{lemma:pprojectivemonoidal} implies the first part of Corollary \ref{corollary:pprojectivecubicalstructurecorollary} 
by evaluating internal hom objects at the complex numbers.
\end{remark}
\begin{proof}
To prove the first statement we may assume 
$$
i=A\otimes(\partial\Box^m\subset\Box^m)\text{ and } j=B\otimes(\partial\Box^{n}\subset\Box^{n}).
$$
Projective cofibrations are retracts of relative $I_{\Box\CC^{\ast}-\Spc}$-cell complexes so if $j$ 
is a retract of a transfinite composition of cobase changes of maps in $I_{\Box\CC^{\ast}-\Spc}$ the 
pushout product map of $i$ and $j$ is a retract of a transfinite composition of cobase changes of 
pushout product maps between $i$ and members of $I_{\Box\CC^{\ast}-\Spc}$.
Thus, 
by analyzing the generating projective cofibrations, 
we may assume the pushout product map in question is of the form
\begin{equation*}
\xymatrix{
(A\otimes B)
\otimes(\Box^m\otimes\partial\Box^{n}\coprod_{\partial\Box^m
\otimes\partial\Box^{n}}\partial\Box^m\otimes\Box^{n}
\ar[r] & \Box^{m+n}).}
\end{equation*}
This is a projective cofibration since $(A\otimes B)\otimes -$ is a left Quillen functor for the pointwise 
projective model structure and the map of cubical sets in question is a cofibration.
\\ \indent
The remaining claim is proved analogously: 
We may assume 
$i=A\otimes(\partial\Box^m\subset\Box^m)$ and $j=B\otimes(\sqcap^{n}_{(\alpha,i)}\subset\Box^{n})$, 
so that the pushout product map is of the form
\begin{equation*}
\xymatrix{
(A\otimes B)\otimes(\Box^{m}\otimes\sqcap^{n}_{(\alpha,i)}
\coprod_{\partial\Box^m\otimes\sqcap^{n}_{(\alpha,i)}}
\partial\Box^{m}\otimes\Box^{n}\ar[r] & \Box^{m+n}).}
\end{equation*}
This map is a pointwise acyclic projective cofibration by the previous argument.
\vspace{0.1in}

In order to prove the second part,
note that by adjointness there is a one-to-one correspondence between the following types of commutative diagrams:
\begin{equation*}
\minCDarrowwidth20pt
\begin{CD}
\X\otimes\V\coprod_{\X\otimes\U}\Y\otimes\U @>>> \Z \\
@VVV @VVV\\
\Y\otimes\V @>>> \W
\end{CD}
\;\;\;\;
\leftrightarrow
\;\;\;\;
\minCDarrowwidth20pt
\begin{CD}
\X @>>> \underline{\Hom}(\V,\Z) \\
@VVV @VVV\\
\Y @>>> \underline{\Hom}(\V,\W)\times_{\underline{\Hom}(\U,\W)}\underline{\Hom}(\U,\Z)
\end{CD}
\end{equation*}
Hence the lifting properties relating projective cofibrations and projective fibrations combined with the first part finish the proof.
\end{proof}

\begin{lemma}
\label{lemma:pprojectiveboxing}
Suppose $\Z$ is a projective cofibrant cubical $\CC^{\ast}$-space.
Then 
\begin{equation*}
\xymatrix{
\Z\otimes-\colon\Box\CC^{\ast}-\Spc\ar[r] & \Box\CC^{\ast}-\Spc} 
\end{equation*}
preserves the classes of projective cofibrations, 
acyclic projective cofibrations and pointwise weak equivalences between projective cofibrant cubical 
$\CC^{\ast}$-spaces.
\end{lemma}
\begin{proof}
This follows because the pointwise projective model structure is monoidal.
\end{proof}

\begin{lemma}
\label{lemma:pprojectivemonoid}
The monoid axiom holds in the pointwise projective model structure.
\end{lemma}
\begin{proof}
We need to check that $(\Box\CC^{\ast}-\Spc\otimes J_{\Box\CC^{\ast}-\Spc})$-cell consists of 
weak equivalences \cite[Definition 3.3]{SS:monoidal}. 
Since the monoid axiom holds for cubical sets and colimits in $\Box\CC^{\ast}-\Spc$ are defined pointwise, 
it suffices to show that $(\Box\CC^{\ast}-\Spc\otimes J_{\Box\CC^{\ast}-\Spc})(B)$ is contained in 
$\Box\Set\otimes J_{\Box\Set}$ for every $\CC^{\ast}$-algebra $B$. 
This follows from the equalities
\begin{align*}
\Bigl(\X\otimes \bigl(A\otimes(\sqcap^{n}_{(\alpha,i)}\subset\Box^{n})\bigr)\Bigr)(B)
& = 
\bigl((\X\otimes A)\otimes (\sqcap^{n}_{(\alpha,i)}\subset\Box^{n})\bigr)(B) \\
& = 
(\X\otimes A)(B)\otimes (\sqcap^{n}_{(\alpha,i)}\subset\Box^{n}).
\end{align*}
\end{proof}
\begin{remark}
Lemma \ref{lemma:pprojectivemonoid} combined with the work of Schwede-Shipley \cite{SS:monoidal} implies that modules over a monoid in 
$\Box\CC^{\ast}-\Spc$ inherit a module structure, 
where the fibrations and weak equivalences of modules over the monoid are just the module maps that are fibrations and weak equivalences
in the underlying model structure.
In the refined model structures on $\Box\CC^{\ast}-\Spc$ the same result holds for cofibrant monoids by reference to 
\cite{Hovey:monoidsandalgebras}.
\end{remark}
\vspace{0.1in}

Next we turn to the injective model structure on $\Box\CC^{\ast}-\Spc$.
Let $\kappa$ be the first infinite cardinal number greater than the cardinality of the set of maps of $\CC^{\ast}-\Spc$.
If $\omega$, 
as usual, 
denotes the cardinal of the continuum, 
define the cardinal number $\gamma$ by $\gamma\equiv\kappa\omega^{\kappa\omega}$.
\begin{definition}
Let $I_{\Box\CC^{\ast}-\Spc}^{\kappa}$ denote the set of maps $\X\rightarrow\Y$ such that 
$\X(A)\rightarrow\Y(A)$ is a cofibration of cubical sets of cardinality less than $\kappa$ for all $A$.
Likewise,
let $J_{\Box\CC^{\ast}-\Spc}^{\gamma}$ denote the set of  maps $\X\rightarrow\Y$ such that 
$\X(A)\rightarrow\Y(A)$ is an acyclic cofibration of cubical sets of cardinality less than 
$\gamma$ for all $A$.
\end{definition}
\vspace{0.1in}

The next result follows by standard tricks of the model category trade.
\begin{proposition}
\label{proposition:injectiveliftings}
Let $q\colon\Z\rightarrow\W$ be a map of cubical $\CC^{\ast}$-spaces.
\begin{itemize}
\item
If $q$ has the right lifting property with respect to every map in $I_{\Box\CC^{\ast}-\Spc}^{\kappa}$,
then $q$ has the right lifting property with respect to all maps $\X\rightarrow\Y$ for which 
$\X(A)\rightarrow\Y(A)$ is a cofibration of cubical sets for all $A$.
\item
If $q$ has the right lifting property with respect to every map in $J_{\Box\CC^{\ast}-\Spc}^{\gamma}$,
then $q$ has the right lifting property with respect to all maps $\X\rightarrow\Y$ for which 
$\X(A)\rightarrow\Y(A)$ is an acyclic cofibration of cubical sets for all $A$.
\end{itemize} 
\end{proposition}
Note that $\X\rightarrow\Y$ is a monomorphism if and only if $\X(A)\rightarrow\Y(A)$ is a cofibration 
of cubical sets for all $A$.

From Proposition \ref{proposition:injectiveliftings} we obtain immediately the next result.
\begin{corollary}
\label{corollary:injectiveliftings}
Let $p\colon\X\rightarrow\Y$ be a map of cubical $\CC^{\ast}$-spaces.
\begin{itemize}
\item
The map $p$ is an $I_{\Box\CC^{\ast}-\Spc}^{\kappa}$-cofibration,
i.e.~has the left lifting property with respect to every map with the right lifting property with 
respect to $I_{\Box\CC^{\ast}-\Spc}^{\kappa}$, 
if and only if $\X(A)\rightarrow\Y(A)$ is a cofibration of cubical sets for all $A$.
\item
The map $p$ is a $J_{\Box\CC^{\ast}-\Spc}^{\gamma}$-cofibration,
i.e.~has the left lifting property with respect to every map with the right lifting property with 
respect to $J_{\Box\CC^{\ast}-\Spc}^{\gamma}$, 
if and only if $\X(A)\rightarrow\Y(A)$ is an acyclic cofibration of cubical sets for all $A$.
\end{itemize}
\end{corollary}
\begin{definition}
A map $\X\rightarrow\Y$ of cubical $\CC^{\ast}$-spaces is an injective fibration if it has the 
right lifting property with respect to all maps which are simultaneously a monomorphism and a
pointwise weak equivalence.
\end{definition}
\vspace{0.1in}

We refer to the following as the pointwise injective model structure.
\begin{theorem}
\label{theorem:injectiveobjectwise}
The classes of monomorphisms a.k.a.~injective cofibrations, injective fibrations and pointwise 
weak equivalences determine a combinatorial model structure on $\Box\CC^{\ast}-\Spc$.
\end{theorem}
\begin{lemma}
\label{lemma:injectiveobjectwise}
The following hold for the pointwise injective model structure.
\begin{itemize}
\item
It is a proper model structure.
\item
It is a cubical model structure.
\item
Every cubical $\CC^{\ast}$-space is cofibrant.
\item
The identity functor  on $\Box\CC^{\ast}-\Spc$ yields a Quillen equivalence with the pointwise 
projective model structure.
\end{itemize} 
\end{lemma}

By adopting the same definitions to simplicial $\CC^{\ast}$-spaces we get a combinatorial 
projective model structure on $\Delta\CC^{\ast}-\Spc$. 
Using standard notations for boundaries and horns in $\Delta\Set$, 
a set of generating cofibrations is given by 
\begin{equation*}
I_{\Delta\CC^{\ast}-\Spc}\equiv\{A\otimes(\partial\Delta^{n}\subset\Delta^{n})\}^{n\geq 0} 
\end{equation*}
and a set of generating acyclic cofibrations by
\begin{equation*}
J_{\Delta\CC^{\ast}-\Spc}\equiv\{A\otimes(\Lambda^{n}_{i}\subset\Delta^{n})\}^{n>0}_{0\leq i\leq n}.
\end{equation*}
The projective model structure on $\Delta\CC^{\ast}-\Spc$ acquires the same additional properties as 
the projective model structure on cubical $\CC^{\ast}$-spaces.
\vspace{0.1in}

From Theorem \ref{theorem:cubicalsetmodelstructure} we deduce that the corresponding homotopy categories
are equivalent:
\begin{lemma}
\label{lemma:projectiveBoxDelta}
There is a naturally induced Quillen equivalence between projective model structures:
\begin{equation*}
\xymatrix{
\Box\CC^{\ast}-\Spc \ar@<3pt>[r] & \Delta\CC^{\ast}-\Spc \ar@<3pt>[l] }
\end{equation*}
\end{lemma}
\vspace{0.1in}

There is also a pointwise injective model structure on $\Delta\CC^{\ast}-\Spc$ where again the 
cofibrations and the weak equivalences are defined pointwise. 
It defines a simplicial model structure, and acquires the same formal properties as the pointwise 
injective model structure on $\Box\CC^{\ast}-\Spc$.
In particular, 
the identity on $\Delta\CC^{\ast}-\Spc$ defines a Quillen equivalence between the pointwise injective 
and projective model structures.
\vspace{0.1in}

At last in this section we shall verify the rather technical cellular model category conditions for our
examples at hand.
\begin{proposition}
\label{proposition:cellular}
The pointwise injective and projective model structures on $\Box\CC^{\ast}-\Spc$ and $\Delta\CC^{\ast}-\Spc$
and their pointed versions are cellular model structures.
\end{proposition}
\begin{lemma}
\label{lemma:smallness}
Every cubical or simplicial $\CC^{\ast}$-space $\X$ is small.
\end{lemma}
\begin{proof}
Suppose $\X$ a cubical $\CC^{\ast}$-space (the following argument applies also to simplicial $\CC^{\ast}$-spaces).
Let $\kappa$ be the regular successor cardinal of the set 
\begin{equation*}
\coprod_{A\in\CC^{\ast}-\Alg, n\geq 0} \X_{n}(A).
\end{equation*}
For every regular cardinal $\lambda\geq\kappa$ and $\lambda$-sequence $F$ in $\Box\CC^{\ast}-\Spc$ we claim there is a 
naturally induced isomorphism
\begin{equation}
\label{equation:smallness}
\xymatrix{
\colim_{\alpha<\lambda}\Box\CC^{\ast}-\Spc(\X,F_{\alpha})\ar[r] & \Box\CC^{\ast}-\Spc(\X,\colim_{\alpha<\lambda}F_{\alpha}).}
\end{equation}
Injectivity of (\ref{equation:smallness}) follows by taking sections and using the fact that every cubical (simplicial) 
set is small, 
cf.~\cite[Lemma 3.1.1]{Hovey:Modelcategories}.
The restriction of $\X\rightarrow\colim_{\alpha}F_{\alpha}$ in $\Box\CC^{\ast}-\Spc$ to any cell of $\X$ factors through 
$F_{\alpha}$ for some $\alpha<\lambda$.
Regularity of $\lambda$ and the fact that there are less than $\kappa<\lambda$ cells in $\X$ implies the restriction to 
any cell of $\X$ factors through $F_{\alpha}$ for some $\alpha<\kappa$. 
Hence the map in (\ref{equation:smallness}) is surjective.
\end{proof}
\begin{lemma}
\label{lemma:compactness}
The domains and codomains of the generating cofibrations $I$ of the injective model structures on $\Box\CC^{\ast}-\Spc$ 
and $\Delta\CC^{\ast}-\Spc$ are compact relative to $I$.
\end{lemma}
\begin{proof}
Let $\kappa$ be the regular successor cardinal of the cardinal of the set
\begin{equation*}
\coprod_{\X\rightarrow\Y\in I}\,\,
\coprod_{A\in\CC^{\ast}-\Alg,n\geq 0}\X_{n}(A)\sqcup\Y_{n}(A).
\end{equation*}
If $\Z$ is a domain or codomain of $I$ we will show that $\Z$ is $\kappa$-compact relative to $I$.
Suppose $f\colon\X\rightarrow\Y$ is an $I$-cell complex and $\Z\rightarrow\Y$ a map.
Note that since all monomorphisms are $I$-cells for the injective model structure, 
$\emptyset\rightarrow\X$ is an $I$-cell complex which when combined with any presentation of $f$ yields a presentation of $\Y$
as an $I$-cell complex with $\X$ as a subcomplex.
Provided the claim holds for the induced $I$-cell complex $\emptyset\rightarrow\Y$ then $\Z$ factors through a subcomplex $\X'$ 
of size less than $\kappa$.
The union of $\X$ and $\X'$ is a subcomplex of the same size as $\X'$ and $\Z\rightarrow\Y$ factors through it.
This would imply that $\Z$ is $\kappa$-compact relative to $I$.
It remains to consider $I$-cell complexes of the form $f\colon\emptyset\rightarrow\Y$.
\\ \indent

Consider an $I$-cell presentation of $f\colon\emptyset\rightarrow\Y$ with presentation ordinal $\lambda$
\begin{equation*}
\xymatrix{
\emptyset=
F_{0}\ar[r] &
F_{1}\ar[r] &
\cdots\ar[r] &
F_{\beta<\lambda}\ar[r] &
\cdots.}
\end{equation*}
We use transfinite induction to show that every cell of $\Y$ is contained in a subcomplex of size less than $\kappa$.  
The induction starts for the presentation ordinal $0$ which produces a subcomplex of size $1$.
Suppose $e_{\beta}$ is a cell with presentation ordinal $\beta<\lambda$.
Using induction on $\beta$ and regularity of $\lambda$ one shows that the attaching map $h_{e_{\beta}}$ has image contained 
in a subcomplex $\Y'$ of size less than $\kappa$.
The subcomplex obtained from $\Y'$ by attaching $e_{\beta}$ via $h_{e_{\beta}}$ yields the desired subcomplex.
Since the cardinality of $\Z$ is bounded by $\kappa$, 
the image of $\Z\rightarrow\Y$ is contained in less than $\kappa$ cells of $\Y$.
Every such cell is in turn contained in a subcomplex of $\Y$ of size less than $\kappa$ by the argument above.
Taking the union of these subcomplexes yields a subcomplex $\Y'$ of $\Y$ of size less than the regular cardinal $\kappa$.
Clearly $\Z\rightarrow\Y$ factors through $\Y'$.
\end{proof}
\begin{lemma}
\label{lemma:effectivenessness}
The cofibrations in the injective model structures on $\Box\CC^{\ast}-\Spc$ and $\Delta\CC^{\ast}-\Spc$ are effective monomorphisms.
\end{lemma}
\begin{proof}
Note that $\X\rightarrow\Y$ is an effective monomorphism if and only if $\X_{n}(A)\rightarrow\Y_{n}(A)$ is an effective monomorphism
of sets for all $A\in\CC^{\ast}-\Alg$ and $n\geq 0$. 
This follows since all limits and colimits are formed pointwise. 
To conclude, 
use that for sets the class of effective monomorphisms coincides with the class of injective maps. 
\end{proof}

According to Proposition \ref{proposition:cellular} the pointwise model structures on $\Box\CC^{\ast}-\Spc$ can be localized using 
the theory of cellular model structures \cite{Hirschhorn:Modelcategories}.
\newpage

\subsection{Exact model structures}
Clearly the pointwise injective and projective model structures involves too many homotopy types of 
$\CC^{\ast}$-algebras.
We shall remedy the situation slightly by introducing the exact model structures.
This involves a strengthening of the fibrancy condition by now requiring that every fibrant cubical 
$\CC^{\ast}$-space turns certain types of short exact sequences in $\CC^{\ast}-\Alg$ into homotopy 
fiber sequences of cubical sets.
First we fix some standard conventions concerning exact sequences and monoidal structures on 
$\CC^{\ast}-\Alg$ \cite{Cuntz:bivariant}.
Corresponding to the maximal tensor product we consider all short exact sequences of $\CC^{\ast}$-algebras
\begin{equation}
\label{ses}
\xymatrix{
0\ar[r] & A\ar[r] & E\ar[r] & B\ar[r] & 0.}
\end{equation}
That is,
the image of the injection $A\rightarrow E$ is a closed $2$-sided ideal, 
the composition $A\rightarrow B$ is trivial and the induced map $E/A\rightarrow B$ is an isomorphism of 
$\CC^{\ast}$-algebras.
Corresponding to the minimal tensor product we restrict to completely positive split
short exact sequences.
Recall that $f\colon A\rightarrow B$ is positive if $f(a)\in B_{+}$ for every positive element $a\in A_{+}$,  
i.e.~$a=a^{\ast}$ and $\sigma_{A}(a)\subseteq [0,\infty)$, 
and $f$ is completely 
positive if $M_{n}(A)\rightarrow M_{n}(B)$, $(a_{ij})\mapsto \bigl(f(a_{ij})\bigr)$, is positive for all $n$.
With these conventions, 
the monoidal product is flat in the sense that (\ref{ses}) remains exact when tensored with any 
$\CC^{\ast}$-algebra.
\begin{example}
The short exact sequence 
$0\rightarrow C_0(\R)\rightarrow C(I)\rightarrow \C\oplus\C\rightarrow 0$ is not split since
$[0,1]$ does not map continuously onto $\{0,1\}$, 
while $0\rightarrow \K\rightarrow\T\rightarrow C(S^1)\rightarrow 0$ where the shift generator of the 
Toeplitz algebra $\T$ is send to the unitary generator of $C(S^1)$ acquires a completely positive 
splitting by sending $f\in C(S^1)$ to the operator $T_{f}$ on the Hardy space $H^{2}\subset L^{2}(S^1)$;
here $T_{f}(g)\equiv\pi(fg)$, 
where $\pi\colon L^{2}(S^1)\rightarrow H^{2}$ denotes the orthogonal projection.  
Note that the splitting is not a map of algebras.
\end{example}
\begin{example}
A deformation of $B$ into $A$ is a continuous field of $\CC^{\ast}$-algebras over a half-open interval 
$[0,\varepsilon)$, 
locally trivial over $(0,\varepsilon)$, 
whose fibers are all isomorphic to $A$ except for the fiber over $0$ which is $B$. 
Every such deformation gives rise to a short exact sequence of $\CC^{\ast}$-algebras
\begin{equation}
\label{extension:deformation}
\xymatrix{
0\ar[r] & C_0\bigl((0,\varepsilon),A\bigr)\ar[r] & E\ar[r] & B\ar[r] & 0.}
\end{equation}
When $A$ is nuclear, 
so that the maximal and minimal tensor products with $A$ coincide, 
then (\ref{extension:deformation}) has a completely positive splitting.
In particular, 
this holds whenever $A$ is finite dimensional, 
commutative or of type I \cite{Blackadar}.  
\end{example}
\begin{remark}
Kasparov's $KK$-theory \cite{Kasparov:KK1}, \cite{Kasparov:KK2} and the $E$-theory of 
Connes-Higson \cite{CH:Etheory} are the universal bivariant theories corresponding to 
the minimal and maximal tensor products respectively, cf.~\cite{Cuntz:bivariant}.
\end{remark}

The set of exact squares consists of all diagrams
\begin{equation}
\label{exactsquare}
\E\equiv
\minCDarrowwidth20pt
\begin{CD}
A @>>> E \\
@VVV @VVV\\
0 @>>> B
\end{CD}
\end{equation}
obtained from a short exact sequence of $\CC^{\ast}$-algebras as in (\ref{ses}), 
and the degenerate square with only one entry $A=0$ in the upper left hand corner.
Note that exactness of the sequence (\ref{ses}) implies the exact square in (\ref{exactsquare}) 
is both a pullback and a pushout diagram in $\CC^{\ast}-\Alg$ \cite{Pedersen:pullbackpushout}.
The exact model structures will be rigged such that exact squares turn into homotopy cartesian 
squares when viewed in $\Box\CC^{\ast}-\Spc$.

A cubical $\CC^{\ast}$-space $\Z$ is called flasque if it takes every exact square to a 
homotopy pullback square.
In detail, 
we require that $\Z(0)$ is contractible and applying $\Z$ to every exact square $\E$ obtained 
from some short exact sequence of $\CC^{\ast}$-algebras yields a homotopy cartesian diagram of 
cubical sets:
\begin{equation}
\label{homotopypullback}
\E\equiv
\minCDarrowwidth20pt
\begin{CD}
A @>>> E \\
@VVV @VVV\\
0 @>>> B
\end{CD}
\;\;\;\;
\rightsquigarrow
\;\;\;\;
\Z(\E)\equiv
\minCDarrowwidth20pt
\begin{CD}
\Z(A) @>>> \Z(E) \\
@VVV @VVV\\
\ast @>>> \Z(B)
\end{CD}
\end{equation}
The definition translates easily into the statement that $\Z$ is flasque if and only if applying $\Z$ 
to any short exact sequence of $\CC^{\ast}$-algebras (\ref{ses}) yields a homotopy fiber sequence of 
cubical sets
\begin{equation*}
\xymatrix{
\Z(A)\ar[r] & \Z(E)\ar[r] & \Z(B).}
\end{equation*}
Recall that $\Z(\E)$ is a homotopy cartesian diagram if the canonical map from $\Z(A)$ to the 
homotopy pullback of the diagram $\ast\rightarrow\Z(B)\leftarrow\Z(E)$ is a weak equivalence.
Right properness of the model structure on $\Box\Set$ implies that if $\Z(B)$ and $\Z(E)$ are 
fibrant the homotopy pullback and homotopy limit of $\ast\rightarrow\Z(B)\leftarrow\Z(E)$ are 
naturally weakly equivalent.
If $\Z(B)$ is contractible, 
then $\Z(\E)$ is a homotopy cartesian diagram if and only if $\Z(A)\rightarrow\Z(E)$ is a 
weak equivalence.
We denote homotopy fibers of maps in model categories by $\hofib$.

If ${\bf X}\colon I\rightarrow\Box\CC^{\ast}-\Spc$ is a small diagram, 
the homotopy limit of ${\bf X}$ is the cubical $\CC^{\ast}$-space defined by 
\begin{equation*}
\underset{I}{\holim}\,\,{\bf X}(A)\equiv
\underset{I}{\holim}\,\,\Ev_{A}\circ {\bf X}.
\end{equation*}
Here $\Ev_{A}\circ {\bf X}\colon I\rightarrow\Box\Set$ and the homotopy limit is formed in cubical sets.
Likewise, 
the homotopy colimit of ${\bf X}$ is the cubical $\CC^{\ast}$-space defined using 
(\ref{homotopycolimit}) by setting
\begin{equation*}
\underset{I}{\hocolim}\,\,{\bf X}(A)\equiv
\underset{I}{\hocolim}\,\,\Ev_{A}\circ {\bf X}.
\end{equation*}

\begin{definition}
A cubical $\CC^{\ast}$-space $\Z$ is exact projective fibrant if it is projective fibrant and flasque.
A map $f\colon\X\rightarrow\Y$ is an exact projective weak equivalence if for every exact projective 
fibrant cubical $\CC^{\ast}$-space $\Z$ there is a naturally induced weak equivalence of cubical sets
\begin{equation*}
\xymatrix{
{\hom}_{\Box\CC^{\ast}-\Spc}(\QQ f,\Z)
\colon
{\hom}_{\Box\CC^{\ast}-\Spc}(\QQ\Y,\Z)\ar[r] & {\hom}_{\Box\CC^{\ast}-\Spc}(\QQ\X,\Z).}
\end{equation*}
Here $\QQ\rightarrow\id_{\Box\CC^{\ast}-\Spc}$ denotes a cofibrant replacement functor in the pointwise 
projective model structure.
Exact projective fibrations of cubical $\CC^{\ast}$-spaces are maps having the right lifting property 
with respect to exact projective acyclic cofibrations.
\end{definition}
\begin{remark}
\label{remark:cubicallocalization}
Since the pointwise projective model structure on $\Box\CC^{\ast}-\Spc$ is a cubical model structure according 
to Lemma \ref{lemma:pprojectivecubical}, 
we may recast the localization machinery in \cite{Hirschhorn:Modelcategories} based on homotopy 
function complexes in terms of cubical function complexes.
\end{remark}

We are now ready to introduce the exact projective model structure as the first out of three types of 
localizations of the pointwise model structures on cubical $\CC^{\ast}$-spaces.

\begin{theorem}
\label{theorem:exactprojective}
The classes of projective cofibrations, exact projective fibrations, 
and exact projective weak equivalences determine a combinatorial cubical model structure on 
$\Box\CC^{\ast}-\Spc$.
\end{theorem}
\begin{proof}
We show the exact projective model structure arise as the localization of the pointwise projective 
model with respect to the maps $\hocolim (\E)\rightarrow A$,
i.e.~the set of maps $\hocolim (0\leftarrow B\rightarrow E)\rightarrow A$ and $\emptyset\rightarrow 0$
indexed by exact squares.
The localized model structure exists because the projective model structure is combinatorial and left 
proper according to Theorem \ref{theorem:projectiveobjectwise} and Lemma \ref{lemma:pprojectiveproperness}.
Since the cofibrations and the fibrant objects determine the weak equivalences in any model structure,
it suffices to identify the fibrant objects in the localized model structure with the exact projective 
fibrant ones defined in terms of short exact sequences.

In effect, 
note that $\Z$ is fibrant in the localized model structure if and only if it is projective fibrant, 
the cubical set $\Z(0)$ is contractible and for all short exact sequences (\ref{ses}) the cubical set 
maps 
\begin{equation*}
\xymatrix{
{\hom}_{\Box\CC^{\ast}-\Spc}(A,\Z)
\ar[r] &
{\hom}_{\Box\CC^{\ast}-\Spc}\bigl(\hocolim (0\leftarrow B\rightarrow E),\Z\bigr)}
\end{equation*}
are weak equivalences.
By (\ref{functioncomplexofrepresentable}),
the latter holds if and only if there exist naturally induced weak equivalences of cubical sets
\begin{align*}
\xymatrix{
\Z(A) \ar[r] & \holim{\hom}_{\Box\CC^{\ast}-\Spc}\bigl((0\leftarrow B\rightarrow E),\Z\bigr) 
= \holim\bigl(\Z(0)\rightarrow\Z(B)\leftarrow\Z(E)\bigr).}
\end{align*}
Put differently, 
the cubical set $\Z(0)$ is contractible and there exist naturally induced weak equivalences
\begin{equation*}
\xymatrix{
\Z(A)\ar[r] & \hofib\bigl(\Z(E)\rightarrow\Z(B)\bigr).} 
\end{equation*}
This shows that $\Z$ is fibrant in the localized model structure if and only if it is exact 
projective fibrant.
It follows that the classes of maps in question form part of a model structure with the stated 
properties.
\end{proof}
\begin{remark}
By construction,
every exact square gives rise to a homotopy cartesian square of cubical $\CC^{\ast}$-spaces in the 
exact projective model structure. 
Taking pushouts along the exact projective weak equivalence $\emptyset\rightarrow 0$ shows that every 
cubical $\CC^{\ast}$-space is exact projective weakly equivalent to its image in $\Box\CC^{\ast}-\Spc_{0}$.
\end{remark}
\begin{lemma}
\label{lemma:projectiveexactleftproper}
The exact projective model structure is left proper.
\end{lemma}
\begin{proof}
Left properness is preserved under localizations of left proper model structures.
Lemma \ref{lemma:pprojectiveproperness} shows the pointwise projective model is left proper.
\end{proof}

We are interested in explicating sets of generating acyclic cofibrations for the exact projective model.
In the following we shall use the cubical mapping cylinder construction to produce a convenient set of generators.
In effect, 
apply the cubical mapping cylinder construction $\cyl$ to exact squares and form the pushouts: 
\begin{equation*}
\E\equiv
\minCDarrowwidth20pt
\begin{CD}
A @>>> E \\
@VVV @VVV\\
0 @>>> B
\end{CD}
\;\;\;\;
\rightsquigarrow
\;\;\;\;
\minCDarrowwidth20pt
\begin{CD}
B @>>> \cyl(B\rightarrow E) @>>> E\\
@VVV @VVV @VVV\\
0 @>>> \cyl(B\rightarrow E)\coprod_{B}0 @>>> A
\end{CD}
\end{equation*}
By Theorem \ref{theorem:exactprojective} the exact projective model structure is cubical. 
Lemmas \ref{lemma:cubicalmappingcylinderfactorization} and \ref{lemma:projectivecofibrant}
imply $B\rightarrow\cyl(B\rightarrow E)$ is a projective cofibration between projective cofibrant 
cubical $\CC^{\ast}$-spaces.
Thus $s(\E)\equiv\cyl(B\rightarrow E)\coprod_{B}0$ is projective cofibrant 
\cite[Corollary 1.11.1]{Hovey:Modelcategories}.
For the same reasons, 
applying the cubical mapping cylinder to $s(\E)\rightarrow A$ and setting 
$t(\E)\equiv\cyl\bigl(s(\E)\rightarrow A\bigr)$ we get a projective cofibration
\begin{equation}
\label{exactprojectivecofibration}
\xymatrix{
\cyl(\E)\colon s(\E)\ar[r] & t(\E).}
\end{equation}
We claim the map $\cyl(\E)$ is an exact projective weak equivalence.
To wit, 
since cubical homotopy equivalences are pointwise weak equivalences, 
it suffices by Lemma \ref{lemma:cubicalmappingcylinderfactorization} to prove that $s(\E)\rightarrow A$ 
is an exact projective weak equivalence. 
The canonical map $E\coprod_{B}0 \rightarrow A$ is an exact projective weak equivalence and there is a 
factoring $E\coprod_{B}0\rightarrow s(\E)\rightarrow A$.
Moreover, 
since $E\rightarrow\cyl(B\rightarrow E)$ is an exact acyclic projective cofibration, 
so is $E\coprod_{B}0\rightarrow s(\E)$. 

Let $J^{\E}_{\Box\CC^{\ast}-\Spc}$ denote the set of maps $J_{\Box\CC^{\ast}-\Spc}\cup J^{\cyl(\E)}_{\Box\CC^{\ast}-\Spc}$
where $J^{\cyl(\E)}_{\Box\CC^{\ast}-\Spc}$ consists of all pushout product maps 
\begin{equation*} 
\xymatrix{
s(\E)\otimes\Box^{n}\coprod_{s(\E)\otimes\partial\Box^{n}} 
t(\E)\otimes\partial\Box^{n}\ar[r] & t(\E)\otimes\Box^{n}.}
\end{equation*}

\begin{proposition}
\label{proposition:projectiveexactJ}
A cubical $\CC^{\ast}$-space is exact projective fibrant if and only if 
it has the right lifting property with respect to the set $J^{\E}_{\Box\CC^{\ast}-\Spc}$.
\end{proposition}
\begin{remark}
Theorem \ref{theorem:exactprojective} shows the members of $J^{\cyl(\E)}_{\Box\CC^{\ast}-\Spc}$ are 
exact acyclic projective cofibrations because the exact projective model structure is cubical and 
the map $\cyl(\E)$ in (\ref{exactprojectivecofibration}) is an exact acyclic projective cofibration.
\end{remark}
\begin{proof}
Note that a projective fibration $\X\rightarrow\Y$ has the right lifting property with respect to 
$J^{\E}_{\Box\CC^{\ast}-\Spc}$ if and only if it has the right lifting property with respect to 
$J^{\cyl(\E)}_{\Box\CC^{\ast}-\Spc}$.
By adjointness, 
the latter holds if and only if $\X(0)\rightarrow\Y(0)$ is a weak equivalence of cubical sets and for 
every exact square $\E$ obtained from a short exact sequence of $\CC^{\ast}$-algebras as in (\ref{ses}) 
there exist liftings in all diagrams of the following form:
\begin{equation*}
\minCDarrowwidth20pt
\begin{CD}
\partial\Box^{n} @>>> 
\hom_{\Box\CC^{\ast}-\Spc}\bigl(t(\E),\X\bigr)\\
@VVV @VVV\\
\Box^{n}         @>>> 
\hom_{\Box\CC^{\ast}-\Spc}\bigl(s(\E),\X\bigr)
\times_{\hom_{\Box\CC^{\ast}-\Spc}\bigl(s(\E),\Y\bigr)}
\hom_{\Box\CC^{\ast}-\Spc}\bigl(t(\E),\Y\bigr)
\end{CD}
\end{equation*}
In other words there are homotopy cartesian diagrams of cubical sets:
\begin{equation*}
\minCDarrowwidth20pt
\begin{CD}
\hom_{\Box\CC^{\ast}-\Spc}\bigl(t(\E),\X\bigr) @>>> \hom_{\Box\CC^{\ast}-\Spc}\bigl(t(\E),\Y\bigr) \\
@VVV @VVV\\
\hom_{\Box\CC^{\ast}-\Spc}\bigl(s(\E),\X\bigr) @>>> \hom_{\Box\CC^{\ast}-\Spc}\bigl(s(\E),\Y\bigr)
\end{CD}
\end{equation*}
An equivalent statement obtained from Yoneda's lemma and the construction of $\cyl(\E)$ is to require 
that there are naturally induced homotopy cartesian diagrams:
\begin{equation*}
\minCDarrowwidth20pt
\begin{CD}
\X(A) @>>> \Y(A)\\
@VVV @VVV\\
\hom_{\Box\CC^{\ast}-\Spc}\bigl(\cyl(B\rightarrow E),\X\bigr)\times_{\X(B)}\X(0)
@>>> 
\hom_{\Box\CC^{\ast}-\Spc}\bigl(\cyl(B\rightarrow E),\Y\bigr)\times_{\Y(B)}\Y(0)
\end{CD}
\end{equation*}
In particular, 
a projective fibrant cubical $\CC^{\ast}$-space $\Z$ has the right lifting property with respect 
to $J^{\cyl(\E)}_{\Box\CC^{\ast}-\Spc}$ if and only if $\Z(0)$ is contractible and for every exact square $\E$ 
obtained from a short exact sequence of $\CC^{\ast}$-algebras there is a homotopy cartesian diagram:
\begin{equation*}
\minCDarrowwidth20pt
\Z(\E)\equiv
\begin{CD}
\Z(A) @>>> \Z(E)\\
@VVV @VVV\\
\ast @>>> \Z(B)
\end{CD}
\end{equation*}
This holds if and only if $\Z$ is flasque since the latter diagram coincides with (\ref{homotopypullback}).
\end{proof}
\begin{corollary}
The exact projective model is weakly finitely generated.
\end{corollary}
\begin{proof}
Members of $J^{\E}_{\Box\CC^{\ast}-\Spc}$ have finitely presentable domains and codomains. 
\end{proof}
\begin{corollary}
The classes of exact projective weak equivalences, 
exact acyclic projective fibrations, 
exact projective fibrations with exact projective fibrant codomains, 
and all exact projective fibrant objects are closed under filtered colimits.
\end{corollary}

Recall the contravariant Yoneda embedding yields a full and faithful embedding of $\CC^{\ast}-\Alg$ 
into cubical $\CC^{\ast}$-spaces.
In the next result we note that no non-isomorphic $\CC^{\ast}$-algebras become isomorphic in the 
homotopy category associated with the exact projective model structure.
This observation motivates to some extent the matrix invariant and homotopy invariant model structures 
introduced in the next sections. 
\begin{proposition}
\label{proposition:fullandfaithfulembedding}
The contravariant Yoneda embedding of $\CC^{\ast}-\Alg$ into $\Box\CC^{\ast}-\Spc$ yields a full and 
faithful embedding of the category of $\CC^{\ast}$-algebras into the homotopy category of 
the exact projective model structure.
\end{proposition}
\begin{proof}
Every $\CC^{\ast}$-algebra is projective cofibrant by Lemma \ref{lemma:projectivecofibrant},  
projective fibrant by Lemma \ref{lemma:projectivefibrant} and also flasque:
note that $\Z(\E)$ is a pullback of discrete cubical sets for every $\CC^{\ast}$-algebra $\Z$ so the assertion 
follows from \cite[II Remark 8.17]{GJ:Modelcategories} since $\ast\rightarrow \Z(B)$ is a fibration of cubical sets.
This shows that every $\CC^{\ast}$-algebra is exact projective fibrant.
Thus \cite[Theorem 1.2.10]{Hovey:Modelcategories} implies there is a bijection between maps in the exact 
projective homotopy category, say $\Ho(\Box\CC^{\ast}-\Spc)(B,A)$, and homotopy classes of maps $[B,A]$. 
Since the exact projective model structure is cubical, 
maps $B\rightarrow A$ are homotopic if and only if there exists a cubical homotopy 
$B\otimes\Box^{1}\rightarrow A$ by an argument analogous to the proof of \cite[II Lemma 3.5]{GJ:Modelcategories} 
which shows that $B\otimes\Box^{1}$ is a cylinder object for $B$ in $\Box\CC^{\ast}-\Spc$.
Using the Yoneda embedding and the fact that $\CC^{\ast}$-algebras are discrete cubical $\CC^{\ast}$-spaces, 
so that all homotopies are constant, 
we get bijections $\Ho(\Box\CC^{\ast}-\Spc)(B,A)=\CC^{\ast}-\Spc(B,A)=\CC^{\ast}-\Alg(A,B)$.
\end{proof}

In the next result we note there exists an explicitly constructed fibrant replacement functor for the exact projective model structure.
\begin{proposition}
\label{proposition:projectiveexactreplacement}
There exists a natural transformation
$\id\rightarrow\Ex^{J^{\E}_{\Box\CC^{\ast}-\Spc}}$
such that for every cubical $\CC^{\ast}$-space $\X$ the map 
$\X\rightarrow\Ex^{J^{\E}_{\Box\CC^{\ast}-\Spc}}\X$
is an exact projective weak equivalence with an exact projective fibrant codomain.
\end{proposition}
\begin{proof}
Use Quillen's small object argument with respect to the maps $A\otimes(\sqcap^{n}_{(\alpha,i)}\subset\Box^{n})$ and
$s(\E)\otimes\Box^{n}\coprod_{s(\E)\otimes\partial\Box^{n}} t(\E)\otimes\partial\Box^{n}\rightarrow t(\E)\otimes\Box^{n}$.
See \cite[\S7]{DS:Modelcategories} and \S3.1 for details.
\end{proof}
The fibrant replacement functor gives a way of testing whether certain maps are exact projective fibrations: 
\begin{corollary}
\label{characterization:exactprojectivefibration}
Suppose $f\colon\X\rightarrow\Y$ is a pointwise projective fibration and $\Y$ exact projective fibrant.
Then $f$ is an exact projective fibration if and only if the diagram 
\begin{equation*}
\minCDarrowwidth20pt
\begin{CD}
\X @>>> \Ex^{J^{\E}_{\Box\CC^{\ast}-\Spc}}\X\\
@V{f}VV @VV{\Ex^{J^{\E}_{\Box\CC^{\ast}-\Spc}}f}V\\
\Y @>>> \Ex^{J^{\E}_{\Box\CC^{\ast}-\Spc}}\Y
\end{CD}
\end{equation*}
is homotopy cartesian in the pointwise projective model structure.
\end{corollary}
\begin{proof}
Follows from \cite[Proposition 2.32]{Barwick} and Proposition \ref{proposition:projectiveexactreplacement}.
\end{proof}

In the following we show that the exact projective model structure is monoidal.
This is a highly desirable property from a model categorical viewpoint.
It turns out our standard conventions concerning short exact sequences of $\CC^{\ast}$-algebras is exactly 
the input we need in order to prove monoidalness. 

\begin{lemma}
\label{lemma:internalhomexactfibrant}
If $\X$ is projective cofibrant and $\Z$ is exact projective fibrant, 
then $\underline{\Hom}(\X,\Z)$ is exact projective fibrant.
\end{lemma}
\begin{proof}
Lemma \ref{lemma:pprojectivemonoidal}(ii) shows it suffices to check 
$\underline{\Hom}\bigl(A\otimes(\partial\Box^{n}\subset\Box^{n}),\Z\bigr)$ 
has the right lifting property with respect to $J^{\cyl(\E)}_{\Box\CC^{\ast}-\Spc}$.
By adjointness, 
if suffices to check that for every exact square $\E$ the pushout product map of 
\begin{equation*}
\xymatrix{
j_{\E}\equiv
s(\E)\otimes\Box^{m}
\coprod_{s(\E)\otimes\partial\Box^{m}} 
t(\E)\otimes\partial\Box^{m}
\ar[r] & 
t(\E)\otimes\Box^{m}}
\end{equation*}
and $A\otimes(\partial\Box^{n}\subset\Box^{n})$ is a composition of pushouts of maps in 
$J^{\cyl(\E)}_{\Box\CC^{\ast}-\Spc}$.
This follows because there is an isomorphism  $j_{\E}\otimes A=j_{\E\otimes A}$ where $\E\otimes A$ 
denotes the exact square obtained by tensoring with $A$, and the pushout product map of 
$\partial\Box^{m}\subset\Box^{m}$ and $\partial\Box^{n}\subset\Box^{n}$ is a monomorphism of cubical 
sets formed by attaching cells.
\end{proof}
\begin{proposition}
\label{proposition:eprojectivemonoidal}
The exact projective model structure is monoidal.
\end{proposition}
\begin{proof}
Suppose $\X\rightarrow\Y$ is an exact acyclic projective cofibration and moreover that $\Z$ is exact projective fibrant.
There is an induced commutative diagram of cubical function complexes:
\begin{equation}
\label{exactmonoidalhomotopypullback}
\minCDarrowwidth20pt
\begin{CD}
{\hom}_{\Box\CC^{\ast}-\Spc}\bigl(\Y,\underline{\Hom}(A\otimes\Box^{n},\Z)\bigr) 
@>>> {\hom}_{\Box\CC^{\ast}-\Spc}\bigl(\X,\underline{\Hom}(A\otimes\Box^{n},\Z)\bigr) \\
@VVV @VVV \\ 
{\hom}_{\Box\CC^{\ast}-\Spc}\bigl(\Y,\underline{\Hom}(A\otimes\partial\Box^{n},\Z)\bigr) 
@>>> {\hom}_{\Box\CC^{\ast}-\Spc}\bigl(\X,\underline{\Hom}(A\otimes\partial\Box^{n},\Z)\bigr)
\end{CD}
\end{equation}
Lemma \ref{lemma:internalhomexactfibrant} implies the horizontal maps in 
(\ref{exactmonoidalhomotopypullback}) are weak equivalences. 
Thus (\ref{exactmonoidalhomotopypullback}) is a homotopy cartesian diagram.
\end{proof}

Next we record the analog of Lemma \ref{lemma:pprojectiveboxing} in the exact projective model structure.
\begin{lemma}
\label{lemma:exactprojectiveboxing}
Suppose $\Z$ is a projective cofibrant cubical $\CC^{\ast}$-space.
Then 
\begin{equation*}
\xymatrix{
\Z\otimes-\colon\Box\CC^{\ast}-\Spc\ar[r] & \Box\CC^{\ast}-\Spc } 
\end{equation*}
preserves the classes of acyclic projective cofibrations and 
exact weak equivalences between projective cofibrant cubical $\CC^{\ast}$-spaces.
\end{lemma}

For reference we include the next result which captures equivalent formulations of the statement that the 
exact projective model structure is monoidal.
\begin{lemma}
\label{lemma:exactprojectivemonoidal}
The following  statements hold and are equivalent.
\begin{itemize}
\item
If $i\colon\X\cof\Y$ and $j\colon\U\cof\V$ are projective cofibrations and either $i$ or $j$ is 
an exact projective weak equivalence, then so is
\begin{equation*}
\xymatrix{
\X\otimes\V\coprod_{\X\otimes\U}\Y\otimes\U\ar[r] & \Y\otimes\V.}
\end{equation*}
\item
If $j\colon\U\cof\V$ is a projective cofibration and $k\colon\Z\fib\W$ is an exact projective fibration, 
then the pullback map
\begin{equation*}
\xymatrix{
\underline{\Hom}(\V,\Z)
\ar[r] & 
\underline{\Hom}(\V,\W)\times_{\underline{\Hom}(\U,\W)}\underline{\Hom}(\U,\Z)}
\end{equation*}
is an exact projective fibration which is exact acyclic if either $j$ or $k$ is.
\item
With the same assumptions as in the previous item, the induced map
\begin{equation*}
\xymatrix{
{\hom}_{\Box\CC^{\ast}-\Spc}(\V,\Z)\ar[r] & 
{\hom}_{\Box\CC^{\ast}-\Spc}(\V,\W)
\times_{{\hom}_{\Box\CC^{\ast}-\Spc}(\U,\W)}{\hom}_{\Box\CC^{\ast}-\Spc}(\U,\Z)}
\end{equation*}
is a Kan fibration which is a weak equivalence of cubical sets if in addition either $j$ or $k$ is 
exact acyclic.
\end{itemize} 
\end{lemma}

Next we construct the exact injective model structure on cubical $\CC^{\ast}$-spaces. 
\begin{definition}
A cubical $\CC^{\ast}$-space $\Z$ is exact injective fibrant if it is injective fibrant and flasque.
A map $f\colon\X\rightarrow\Y$ is an exact injective weak equivalence if for every exact injective 
fibrant cubical $\CC^{\ast}$-space $\Z$ there is a naturally induced weak equivalence of cubical sets
\begin{equation*}
\xymatrix{
{\hom}_{\Box\CC^{\ast}-\Spc}(f,\Z)
\colon
{\hom}_{\Box\CC^{\ast}-\Spc}(\Y,\Z)\ar[r] & {\hom}_{\Box\CC^{\ast}-\Spc}(\X,\Z).}
\end{equation*}
The exact injective fibrations of cubical $\CC^{\ast}$-spaces are maps having the right lifting property 
with respect to exact injective acyclic cofibrations.
\end{definition}
\begin{remark}
Note there is no cofibrant replacement functor involved in the definition of exact injective fibrant 
objects due to the fact that every cubical $\CC^{\ast}$-space is cofibrant in the injective model structure.
\end{remark}
The proof of the next result proceeds as the proof of Theorem \ref{theorem:exactprojective} by localizing 
the pointwise injective model structure with respect to the maps $\hocolim (\E)\rightarrow A$.
\begin{theorem}
\label{theorem:exactinjective}
The classes of monomorphisms, exact injective fibrations and exact injective weak equivalences determine a 
combinatorial, cubical and left proper model structure on $\Box\CC^{\ast}-\Spc$.
\end{theorem}
\begin{proposition}
\label{proposition:exactprojectiveinjective}
The classes of exact injective and projective weak equivalences coincide.
Hence the identity functor on $\Box\CC^{\ast}-\Spc$ is a Quillen equivalence between the exact injective 
and projective model structures.
\end{proposition}
\begin{proof}
If $\Z$ is exact injective fibrant then clearly $\Z$ is exact projective fibrant. 
Thus if $f\colon\X\rightarrow\Y$ is an exact projective weak equivalence,
then map ${\hom}_{\Box\CC^{\ast}-\Spc}(\QQ f,\Z)$ is a weak equivalence a cubical sets. 
Now $\QQ f$ maps to $f$ via pointwise weak equivalences, 
so ${\hom}_{\Box\CC^{\ast}-\Spc}(f,\Z)$ is also a weak equivalence.
\\ \indent
If $\Z$ is exact projective fibrant there exists a pointwise weak equivalence $Z\rightarrow\W$
where $\W$ is injective fibrant.
It follows that $\W$ is flasque.
Now if $f\colon\X\rightarrow\Y$ is an exact injective weak equivalence, 
using that the exact projective model structure is cubical we get the following diagram with 
vertical weak equivalences:
\begin{equation*}
\xymatrix{
{\hom}_{\Box\CC^{\ast}-\Spc}(\QQ\Y,\Z) \ar[d]_{\sim} \ar[r] & 
{\hom}_{\Box\CC^{\ast}-\Spc}(\QQ\X,\Z)\ar[d]^{\sim}\\
{\hom}_{\Box\CC^{\ast}-\Spc}(\QQ\Y,\W) \ar[r] & 
{\hom}_{\Box\CC^{\ast}-\Spc}(\QQ\X,\W) }
\end{equation*}
It remains to note that ${\hom}_{\Box\CC^{\ast}-\Spc}(\QQ f,\W)$ is a weak equivalence because $\QQ f$ is an 
exact injective weak equivalence.
\end{proof}
\begin{remark}
The proof of Proposition \ref{proposition:exactprojectiveinjective} applies with small variations to both the 
matrix invariant and the homotopy invariant model structures on $\Box\CC^{\ast}-\Spc$ which will be constructed  
in Sections \ref{subsection:matrixinvariantmodelstructures} and \ref{subsection:homotopyinvariantmodelstructures}
respectively;
details in these proofs will be left implicit in the next sections. 
By Proposition \ref{proposition:cellular} we may appeal to localizations of left proper cellular model structures
for the existence of the exact model structures.
The same applies to the matrix invariant and homotopy invariant model structures.
\end{remark}

By localizing the pointwise model structures on $\Delta\CC^{\ast}-\Spc$ with respect to the maps 
$\hocolim(\E)\rightarrow A$ as above,
we obtain exact model structures on simplicial $\CC^{\ast}$-spaces.
They acquire the same additional properties as the corresponding exact model structures on cubical $\CC^{\ast}$-spaces.
As a special case of \cite[Theorem 3.3.20]{Hirschhorn:Modelcategories} and Lemma \ref{lemma:projectiveBoxDelta} 
we deduce that the corresponding homotopy categories are equivalent:
\begin{lemma}
\label{lemma:exactprojectiveBoxDelta}
There are naturally induced Quillen equivalences between exact injective and projective 
model structures:
\begin{equation*}
\xymatrix{
\Box\CC^{\ast}-\Spc \ar@<3pt>[r] & \Delta\CC^{\ast}-\Spc \ar@<3pt>[l] }
\end{equation*}
\end{lemma}
\begin{remark}
In the following we shall introduce the matrix invariant and homotopy invariant model structures.
As above, 
these model structures furnish two Quillen equivalences between $\Box\CC^{\ast}-\Spc$ and 
$\Delta\CC^{\ast}-\Spc$.
This observation will be employed implicitly in later sections in the proof of representability 
of Kasparov's $KK$-groups in the pointed unstable homotopy category and when dealing with the
triangulated structure of the stable homotopy category of $\CC^{\ast}$-algebras.
\end{remark}
\newpage

\subsection{Matrix invariant model structures}
\label{subsection:matrixinvariantmodelstructures}
In this section we refine the exact projective model structures by imposing a natural fibrancy condition 
determined by the highly noncommutative data of Morita-Rieffel equivalence or matrix invariance.
This amounts to the choice of a rank-one projection $p\in\K$ such that the corner embedding 
$A\rightarrow A\otimes\K=\colim M_{n}(A)$ given by $a\mapsto a\otimes p$ becomes a ``matrix exact''
weak equivalence.
To achieve this we shall localize the exact model structures with respect to such a rank-one projection.
With this approach the results and techniques in the previous section carry over in gross outline.
However, 
there are a couple of technical differences and the exposition tends to emphasize these.
As a motivation for what follows,
recall that matrix invariance is a natural and basic property in the theory of $K$-theory of 
$\CC^{\ast}$-algebras \cite{Connes:noncommutative}.
\vspace{0.1in}

A cubical $\CC^{\ast}$-space $\Z$ is matrix exact projective fibrant if $\Z$ is exact projective fibrant 
and for every $\CC^{\ast}$-algebra $A$ the induced map of cubical $\CC^{\ast}$-spaces 
\begin{equation}
\label{matrixinvariantlocalizationset}
\xymatrix{
A\otimes\K\ar[r] & A } 
\end{equation}
given by a rank one projection induces a weak equivalence of cubical sets
\begin{equation}
\label{matrixinvariantset}
\xymatrix{
\Z(A)={\hom}_{\Box\CC^{\ast}-\Spc}(A,\Z)
\ar[r] &
{\hom}_{\Box\CC^{\ast}-\Spc}(A\otimes\K,\Z)=\Z(A\otimes\K).}
\end{equation}
The definition of $\Z$ being matrix exact projective fibrant is independent of the choice of a rank one projection.

Recall $\QQ$ is a pointwise projective cofibrant replacement functor.
A map between cubical $\CC^{\ast}$-spaces $\X\rightarrow\Y$ is a matrix exact projective weak equivalence 
if for every matrix exact projective fibrant $\Z$ there is a naturally induced weak equivalence of cubical sets
\begin{equation}
\label{matrixexactweakequivalence}
\xymatrix{
{\hom}_{\Box\CC^{\ast}-\Spc}(\QQ\Y,\Z)
\ar[r] &
{\hom}_{\Box\CC^{\ast}-\Spc}(\QQ\X,\Z).}
\end{equation}
\begin{example}
\label{example:matrixexactweakequivalence}
The map $A\otimes\K\rightarrow A$ is a matrix exact projective weak equivalence for every 
$\CC^{\ast}$-algebra $A$ because representable cubical $\CC^{\ast}$-spaces are projective cofibrant. 
If $A\otimes\K\rightarrow M_{n}(A)$ is a matrix exact projective weak equivalence for some $n\geq 1$, 
then so is $M_{n}(A)\rightarrow A$.
\end{example}
The matrix invariant projective model structure is defined by taking the Bousfield localization 
of the exact projective model structure on $\Box\CC^{\ast}-\Spc$ with respect to the set of maps 
obtained by letting $A$ run through all isomorphism classes of $\CC^{\ast}$-algebras in 
(\ref{matrixinvariantlocalizationset}).
Thus the next result is a consequence of 
Theorem \ref{theorem:exactprojective} and Lemma \ref{lemma:projectiveexactleftproper}.
\begin{theorem}
\label{theorem:matrixexactprojective}
The classes of matrix exact projective weak equivalences defined by (\ref{matrixexactweakequivalence}), 
matrix exact projective fibrations and projective cofibrations form a combinatorial, cubical and left 
proper model structure on $\Box\CC^{\ast}-\Spc$.
\end{theorem}

Applying the cubical mapping cylinder to the map in (\ref{matrixinvariantlocalizationset}) yields a 
factoring
\begin{equation}
\label{matrixexactfactoring}
\xymatrix{
A\otimes\K\ar[r] & \cyl^{A}_{\K}\ar[r] & A.} 
\end{equation}
Recall the map $A\otimes\K\rightarrow \cyl^{A}_{\K}$ is a projective cofibration and 
$\cyl^{A}_{\K}\rightarrow A$ is a cubical homotopy equivalence.
In particular, $\cyl^{A}_{\K}$ is projective cofibrant.
Example \ref{example:matrixexactweakequivalence} and saturation imply  
$A\otimes\K\rightarrow \cyl^{A}_{\K}$ is a matrix exact projective weak equivalence.
Since the matrix invariant model structure is cubical, 
the factoring (\ref{matrixexactfactoring}) and the generating cofibrations 
$\partial\Box^{n}\subset \Box^{n}$ for $\Box\Set$ induce matrix exact acyclic projective cofibrations.
\vspace{0.1in}

Let $J^{\cyl(\K)}_{\Box\CC^{\ast}-\Spc}$ be the set consisting of the matrix exact acyclic projective 
cofibrations
\begin{equation}
\xymatrix{
(A\otimes\K)\otimes\Box^{n}\coprod_{(A\otimes\K)\otimes\partial\Box^{n}}
\cyl^{A}_{\K}\otimes\partial\Box^{n}\ar[r] & \cyl^{A}_{\K}\otimes\Box^{n} }
\end{equation}
where $A\in\CC^{\ast}-\Alg$ and $n\geq 0$. 

\begin{proposition}
\label{proposition:projectivematrixexactJ}
Define
\begin{equation*}
J^{\K}_{\Box\CC^{\ast}-\Spc}\equiv 
J_{\Box\CC^{\ast}-\Spc}
\cup
J^{\cyl(\E)}_{\Box\CC^{\ast}-\Spc}
\cup
J^{\cyl(\K)}_{\Box\CC^{\ast}-\Spc}.
\end{equation*}
Then a map of cubical $\CC^{\ast}$-spaces 
with a matrix exact projective fibrant codomain has the right lifting property with respect 
to $J^{\K}_{\Box\CC^{\ast}-\Spc}$ if and only if it is a matrix exact projective fibration.
\end{proposition}
\begin{proof}
Follows from Proposition \ref{proposition:projectiveexactJ} and \cite[3.3.16]{Hirschhorn:Modelcategories}.
\end{proof}

We shall use the set $J^{\cyl(\K)}_{\Box\CC^{\ast}-\Spc}$ to prove the following crucial result.
\begin{proposition}
\label{proposition:miprojectivemonoidal}
The matrix invariant projective model structure is monoidal.
\end{proposition}
\begin{proof}
If $\Z$ is projective cofibrant, 
Lemmas \ref{lemma:pprojectiveboxing} and \ref{lemma:matrixinvariantboxing} imply the functor 
\begin{equation*}
\xymatrix{
\Z\otimes-\colon\Box\CC^{\ast}-\Spc\ar[r] & \Box\CC^{\ast}-\Spc}
\end{equation*}
preserves matrix exact acyclic projective cofibrations $f\colon\X\rightarrow\Y$.
In particular, 
this result applies to the domains and codomains of the generating projective cofibrations 
$I_{\Box\CC^{\ast}-\Spc}$.
Hence for every $\CC^{\ast}$-algebra $A$ and $n\geq0$ there is a commutative diagram where the horizontal 
maps are matrix exact acyclic projective cofibrations:
\begin{equation*}
\minCDarrowwidth20pt
\begin{CD}
(A\otimes\partial\Box^{n})\otimes\X @>>> (A\otimes\partial\Box^{n})\otimes\Y \\
@VVV @VVV \\
(A\otimes\Box^{n})\otimes\X @>>> (A\otimes\Box^{n})\otimes\Y
\end{CD}
\end{equation*}
Thus,
by \cite[Corollary 1.1.11]{Hovey:Modelcategories}, 
the pushout map
\begin{equation*}
\xymatrix{
(A\otimes\Box^{n})\otimes\X
\ar[r] & 
(A\otimes\Box^{n})\otimes\X
\coprod_{(A\otimes\partial\Box^{n})\otimes\X}
(A\otimes\partial\Box^{n})\otimes\Y}
\end{equation*}
of $(A\otimes\partial\Box^{n})\otimes f$ along 
$\bigl(A\otimes(\partial\Box^{n}\subset\Box^{n})\bigr)\otimes\X$ 
is a matrix exact acyclic projective cofibration.
By saturation if follows that the pushout product map of $A\otimes(\partial\Box^{n}\subset\Box^{n})$ 
and $f$ is a matrix exact projective weak equivalence. 
\end{proof}

To complete the proof of Proposition \ref{proposition:miprojectivemonoidal} it remains to prove the next result.
\begin{lemma}
\label{lemma:matrixinvariantboxing}
If $\X\rightarrow\Y$ is a matrix exact weak equivalence and $\W$ is projective cofibrant,
then the induced map $\X\otimes\W\rightarrow\Y\otimes\W$ is a matrix exact projective weak equivalence.
\end{lemma}
\begin{proof}
Suppose that $\Z$ is matrix exact projective fibrant.
We need to show there is an induced weak equivalence of cubical sets
\begin{equation*}
\xymatrix{
\hom_{\Box\CC^{\ast}-\Spc}(\Y\otimes\W,\Z)\ar[r] & \hom_{\Box\CC^{\ast}-\Spc}(\X\otimes\W,\Z).}
\end{equation*}
By adjointness the latter identifies with the map  
\begin{equation*}
\xymatrix{
\hom_{\Box\CC^{\ast}-\Spc}\bigl(\Y,\underline{\Hom}(\W,\Z)\bigr)\ar[r] & 
\hom_{\Box\CC^{\ast}-\Spc}\bigl(\X,\underline{\Hom}(\W,\Z)\bigr).}
\end{equation*}
Thus it suffices to show $\underline{\Hom}(\W,\Z)$ is matrix exact projective fibrant, 
see the next lemma.
\end{proof}

\begin{lemma}
\label{lemma:internalhommatrixexactfibrant}
If $\X$ is projective cofibrant and $\Z$ is matrix exact projective fibrant, 
then the internal hom $\underline{\Hom}(\X,\Z)$ is matrix exact projective fibrant.
\end{lemma}
\begin{proof}
Lemma \ref{lemma:exactprojectivemonoidal}(ii) shows it suffices to check that for every 
$\CC^{\ast}$-algebra $B$ the map 
$\underline{\Hom}\bigl(B\otimes(\partial\Box^{n}\subset\Box^{n}),\Z\bigr)$ 
has the right lifting property with respect to $J^{\cyl(\K)}_{\Box\CC^{\ast}-\Spc}$.

Using adjointness, 
if suffices to check that for all $\CC^{\ast}$-algebras $A$ and $B$ the pushout product map of 
\begin{equation}
\xymatrix{
(A\otimes\K)\otimes\Box^{m}\coprod_{(A\otimes\K)\otimes\partial\Box{m}} 
\cyl^{A}_{\K}\otimes\partial\Box^{m}\ar[r] & \cyl^{A}_{\K}\otimes\Box^{m}}
\end{equation}
and $B\otimes(\partial\Box^{n}\subset\Box^{n})$ is a composition of pushouts of maps in 
$J^{\K}_{\Box\CC^{\ast}-\Spc}$.
This follows using the isomorphism $\cyl^{A}_{\K}\otimes B=\cyl^{A\otimes B}_{\K}$,
cp.~the proof of Lemma \ref{lemma:internalhomexactfibrant}.
\end{proof}

The next result summarizes the monoidal property of the matrix invariant model structure.
\begin{lemma}
\label{lemma:matrixexactprojectivemonoidal}
The following statements hold and are equivalent.
\begin{itemize}
\item
If $i\colon\X\cof\Y$ and $j\colon\U\cof\V$ are projective cofibrations and either 
$i$ or $j$ is a matrix exact projective weak equivalence, 
then so is
\begin{equation*}
\xymatrix{
\X\otimes\V\coprod_{\X\otimes\U}\Y\otimes\U\ar[r] & \Y\otimes\V.}
\end{equation*}
\item
If $j\colon\U\cof\V$ is a projective cofibration and $k\colon\Z\fib\W$ is a 
matrix exact projective fibration, 
then the pullback map
\begin{equation*}
\xymatrix{
\underline{\Hom}(\V,\Z)\ar[r] &
\underline{\Hom}(\V,\W)\times_{\underline{\Hom}(\U,\W)}\underline{\Hom}(\U,\Z)}
\end{equation*}
is a matrix exact projective fibration which is matrix exact acyclic if either $j$ or $k$ is.
\item
With the same assumptions as in the previous item, the induced map
\begin{equation*}
\xymatrix{
{\hom}_{\Box\CC^{\ast}-\Spc}(\V,\Z)\ar[r] &
{\hom}_{\Box\CC^{\ast}-\Spc}(\V,\W)
\times_{{\hom}_{\Box\CC^{\ast}-\Spc}(\U,\W)}{\hom}_{\Box\CC^{\ast}-\Spc}(\U,\Z)}
\end{equation*}
is a Kan fibration which is a weak equivalence of cubical sets if in addition either 
$j$ or $k$ is matrix exact acyclic.
\end{itemize} 
\end{lemma}

The matrix invariant injective model structure on cubical $\CC^{\ast}$-spaces arises in an analogous
way by declaring that $\Z$ is matrix exact injective fibrant if it is exact injective fibrant and 
$\Z(A)\rightarrow\Z(A\otimes\K)$ is a weak equivalence for all $A$.
A map $\X\rightarrow\Y$ is a matrix exact weak equivalence if for every matrix exact injective fibrant 
$\Z$ there is a naturally induced weak equivalence of cubical sets
\begin{equation*}
\xymatrix{
{\hom}_{\Box\CC^{\ast}-\Spc}(f,\Z)\colon
{\hom}_{\Box\CC^{\ast}-\Spc}(\Y,\Z)
\ar[r] &
{\hom}_{\Box\CC^{\ast}-\Spc}(\X,\Z).}
\end{equation*}

We are ready to formulate the main results concerning the class of matrix exact weak equivalences.
\begin{theorem}
\label{theorem:matrixexactinjective}
The classes of monomorphisms, matrix invariant injective fibrations and matrix exact injective weak 
equivalences determine a combinatorial, cubical and left proper model structure on $\Box\CC^{\ast}-\Spc$.
\end{theorem}
\begin{proposition}
\label{proposition:matrixexactprojectiveinjective}
The classes of matrix invariant injective and projective weak equivalences coincide.
Hence the identity functor on $\Box\CC^{\ast}-\Spc$ is a Quillen equivalence between the matrix invariant 
injective and projective model structures.
\end{proposition}
\begin{proof}
See the proof of Proposition \ref{proposition:exactprojectiveinjective}.
\end{proof}

The category $\Delta\CC^{\ast}-\Spc$ acquires matrix invariant model structures.
We have:
\begin{lemma}
\label{lemma:matrixexactprojectiveBoxDelta}
There are naturally induced Quillen equivalences between matrix invariant injective and projective 
model structures:
\begin{equation*}
\xymatrix{
\Box\CC^{\ast}-\Spc \ar@<3pt>[r] & \Delta\CC^{\ast}-\Spc \ar@<3pt>[l] }
\end{equation*}
\end{lemma}
\begin{remark}
For completeness we note that the Quillen equivalent matrix invariant model structures on $\Box\CC^{\ast}-\Spc$
and $\Delta\CC^{\ast}-\Spc$ are examples of cellular model structures.
\end{remark}
\vspace{0.1in}

With the matrix invariant model structures in hand we are now ready to construct the last in the series of model structure appearing in 
unstable $\CC^{\ast}$-homotopy theory.
\newpage

\subsection{Homotopy invariant model structures}
\label{subsection:homotopyinvariantmodelstructures}
Let $A$ be a $\CC^{\ast}$-algebra, let $I=[0,1]$ denote the topological unit interval and $C(I,A)$ the 
$\CC^{\ast}$-algebra of continuous functions from $I$ to $A$ with pointwise operations and the supremum norm.
At time $t$, 
$0\leq t\leq 1$, 
there is an evaluation map 
\begin{equation*}
\xymatrix{
\ev^A_t\colon C(I,A)\ar[r] & A. }
\end{equation*}
Recall that $\ast$-homomorphisms $h_t\colon A\rightarrow B$ for $t=0,1$ are homotopic if there 
exists a map $H\colon A\rightarrow C(I,B)$ such that $\ev^B_t\circ H=h_t$.
The notions of homotopies between $\ast$-homomorphisms and contractible $\CC^{\ast}$-algebras 
are defined in terms of $C(I)\equiv C(I,\C)$ and the trivial $\CC^{\ast}$-algebra exactly as for 
topological spaces. 
There is an isomorphism of $\CC^{\ast}$-algebras $C(I,B)\cong C(I)\otimes B$ for the tensor products we consider. 
These definitions correspond under Gelfand-Naimark duality to the usual topological 
definitions in the event that $A$ and $B$ are commutative.
It turns out the cone of every $\CC^{\ast}$-algebra is contractible and every contractible $\CC^{\ast}$-algebra is nonunital. 
If $A$ is contractible, then its unitalization is homotopy equivalent to $\C$.
There is a canonical $\ast$-homomorphism, 
the constant function map, 
from $A$ to $C(I,A)$ sending elements of $A$ to constant functions. 
Composing this map with $\ev^A_t$ gives the identity map on $A$ for all $t$.
\vspace{0.1in}

Motivated by the notion of homotopies between maps of $\CC^{\ast}$-algebras we shall now introduce the 
homotopy invariant model structures on cubical $\CC^{\ast}$-spaces. 
The main idea behind these models is to enlarge the class of weak equivalences by formally adding all 
homotopy equivalences based on employing $C(I)$ as the unit interval.
Indeed, these model structures give rise to the correct unstable homotopy category in the sense that 
homotopic $\CC^{\ast}$-algebras become isomorphic upon inverting the weak equivalences in the homotopy 
invariant models.
Existence of the homotopy invariant model structures is shown using localization techniques, 
as one would expect.
We show there is an abstract characterization of the weak equivalences in the homotopy invariant 
model structure and introduce homotopy groups.
These invariants give a way of testing whether a map between $\CC^{\ast}$-spaces is a weak equivalence.
\vspace{0.1in}

A cubical $\CC^{\ast}$-space $\Z$ is called $\CC^{\ast}$-projective fibrant if $\Z$ is matrix exact 
projective fibrant and for every $\CC^{\ast}$-algebra $A$ the canonically induced map of cubical 
$\CC^{\ast}$-spaces $C(I,A)\rightarrow A$ induces a weak equivalence of cubical sets
\begin{equation}
\label{homotopysetone}
\xymatrix{
{\hom}_{\Box\CC^{\ast}-\Spc}(A,\Z)
\ar[r] & 
{\hom}_{\Box\CC^{\ast}-\Spc}\bigl(C(I,A),\Z\bigr).}
\end{equation}
It follows immediately that a matrix exact projective fibrant $\Z$ is $\CC^{\ast}$-projective fibrant 
if and only if for every $A$ and for some $0\leq t\leq 1$ the evaluation map $\ev^A_t$ yields a weak 
equivalence
\begin{equation}
\label{homotopysettwo}
\xymatrix{
{\hom}_{\Box\CC^{\ast}-\Spc}\bigl(C(I,A),\Z\bigr)
\ar[r] & 
{\hom}_{\Box\CC^{\ast}-\Spc}(A,\Z).}
\end{equation}
Moreover, 
note that the map in (\ref{homotopysetone}) is a weak equivalence if and only if the induced map 
$\pi_0\Z(A)\rightarrow\pi_0\Z\bigl(C(I,A)\bigr)$ is a surjection  and for every $0$-cell 
$x$ of $\Z(A)$ and $n\geq 1$ there is a similarly induced surjective map of higher homotopy groups
\begin{equation*}
\xymatrix{
\pi_{n}\bigl(\Z(A),x\bigr)
\ar[r] & 
\pi_{n}\biggl(\Z\bigl(C(I,A)\bigr),x\biggr).}
\end{equation*}
Likewise, the map (\ref{homotopysettwo}) induces surjections on all higher homotopy groups.
An alternate formulation of $\Z$ being $\CC^{\ast}$-projective fibrant is to require that for 
$0\leq t\leq 1$ there are naturally induced pointwise weak equivalences 
\begin{equation*}
\minCDarrowwidth20pt
\begin{CD}
\Z(-)@>>>
\Z\bigl(C(I)\otimes -\bigr)
@>\Z(\ev^{\C}_{t}\otimes -) >> 
\Z(-).
\end{CD}
\end{equation*}
In terms of internal hom objects, 
yet another equivalent formulation obtained from (\ref{representableinternalhomevalution})
is that for every $\CC^{\ast}$-algebra $A$ evaluating the naturally induced maps 
\begin{equation*}
\xymatrix{
\underline{\Hom}(A,\Z)\ar[r] & 
\underline{\Hom}\bigl(C(I,A),\Z\bigr)\ar[r] & 
\underline{\Hom}(A,\Z)}
\end{equation*}
at the complex numbers yield weak equivalences of cubical sets.

\begin{remark}
The notion of an $\CC^{\ast}$-projective fibrant cubical $\CC^{\ast}$-space depends only on the unit 
interval $C(I)$ and the matrix exact projective fibrancy condition in the sense that 
it may be checked using any of the evaluation maps or the constant function map.
\end{remark}

A map $\X\rightarrow\Y$ is a projective $\CC^{\ast}$-weak equivalence if for every 
$\CC^{\ast}$-projective fibrant $\Z$ there is an induced weak equivalence of cubical sets
\begin{equation*}
\xymatrix{
{\hom}_{\Box\CC^{\ast}-\Spc}(\QQ\Y,\Z)\ar[r] & {\hom}_{\Box\CC^{\ast}-\Spc}(\QQ\X,\Z).}
\end{equation*}
Recall that $\QQ$ is our notation for a cofibrant replacement in the pointwise projective model structure.
Every $\CC^{\ast}$-algebra $A$ is projective $\CC^{\ast}$-weakly equivalent to $C(I,A)$.
All matrix exact projective weak equivalences are examples of projective $\CC^{\ast}$-weak equivalences 
since the matrix exact projective model structure is cubical.

A map $\X\rightarrow\Y$ is a projective $\CC^{\ast}$-fibration if it has the right lifting property with 
respect to every $\CC^{\ast}$-acyclic projective cofibration.
The class of projective $\CC^{\ast}$-fibrations coincides with the fibrations in the 
$\CC^{\ast}$-projective model structure which we define by localizing the matrix invariant projective 
model at the set of maps $C(I,A)\rightarrow A$. 
\begin{theorem}
The projective cofibrations and projective $\CC^{\ast}$-weak equivalences determine a combinatorial, 
cubical and left proper model structure on $\Box\CC^{\ast}-\Spc$.
\end{theorem}
The unstable $\CC^{\ast}$-homotopy category,
denoted by $\HH$,  
is defined by inverting the class of projective $\CC^{\ast}$-weak equivalences between cubical 
$\CC^{\ast}$-spaces.
\vspace{0.1in}

We trust that the notions of $\CC^{\ast}$-injective fibrant cubical $\CC^{\ast}$-spaces, 
injective $\CC^{\ast}$-weak equivalences and injective $\CC^{\ast}$-fibrations are clear from the 
above and the definitions of the injective model structures constructed in the previous sections. 
Next we state two basic results concerning the injective homotopy invariant model structure.
\begin{theorem}
\label{theorem:homotopyinvariantinjective}
The classes of monomorphisms, injective $\CC^{\ast}$-fibrations and injective $\CC^{\ast}$-weak equivalences 
determine a combinatorial, cubical and left proper model structure on $\Box\CC^{\ast}-\Spc$.
\end{theorem}
\begin{proposition}
\label{proposition:homotopyinvariantprojectiveinjective}
The classes of injective and projective $\CC^{\ast}$-weak equivalences coincide.
Hence the identity functor on $\Box\CC^{\ast}-\Spc$ is a Quillen equivalence between the homotopy invariant 
injective and projective model structures.
\end{proposition}

In the following we write $\CC^{\ast}$-weak equivalence rather than injective or projective $\CC^{\ast}$-weak equivalence.
We note there exist corresponding homotopy invariant model structures for simplicial $\CC^{\ast}$-spaces, 
and include the following observation.
\begin{lemma}
\label{lemma:homotopyinvariantBoxDelta}
There are naturally induced Quillen equivalences between homotopy invariant injective and projective 
model structures:
\begin{equation*}
\xymatrix{
\Box\CC^{\ast}-\Spc \ar@<3pt>[r] & \Delta\CC^{\ast}-\Spc \ar@<3pt>[l] }
\end{equation*}
\end{lemma}
\vspace{0.1in}

An elementary homotopy between maps $h_t\colon\X\rightarrow\Y$ of cubical $\CC^{\ast}$-spaces is a map
$H\colon\X\otimes C(I)\rightarrow\Y$ such that $H\circ(\id_{\X}\otimes\ev^{\C}_t)=h_t$ for $t=0,1$.
Two maps $f$ and $g$ are homotopic if there exists a sequence of maps $f=f_{0},f_{1},\cdots,f_{n}=g$ 
such that $f_{i-1}$ is elementary homotopic to $f_{i}$ for $1\leq i\leq n$. 
And $f\colon\X\rightarrow\Y$ is a homotopy equivalence if there exists a map $g\colon\Y\rightarrow\X$ 
such that $f\circ g$ and $g\circ f$ are homotopic to the respective identity maps.
\begin{remark}
Note that maps between representable cubical $\CC^{\ast}$-spaces are homotopic if and only if the maps 
between the corresponding $\CC^{\ast}$-algebras are so.
\end{remark}
\begin{lemma}
\label{homotopyimpliesweakequivalence}
Homotopy equivalences are $\CC^{\ast}$-weak equivalences.
\end{lemma}
\begin{proof}
The proof reduces to showing that elementary homotopic maps $h_t\colon\X\rightarrow\Y$ become isomorphic 
in the unstable $\CC^{\ast}$-homotopy category:
If $f\colon\X\rightarrow\Y$ is a homotopy equivalence with homotopy inverse $g$ we need to show that 
$f\circ g$ and $g\circ f$ are equal to the corresponding identity maps in the homotopy category, 
but the composite maps are homotopic to the corresponding identity maps.
Now for the projective cofibrant replacement $\QQ\X\rightarrow\X$ the assertion holds for the maps 
$\QQ\X\rightarrow\QQ\X\otimes C(I)$ induced by evaluating at $t=0$ and $t=1$.
And hence the same holds for the two composite maps 
$\QQ\X\rightarrow\QQ\X\otimes C(I)\rightarrow\X\otimes C(I)$.
Composing these maps with the homotopy yields maps naturally isomorphic to $h_0$ and $h_1$ in 
the unstable $\CC^{\ast}$-homotopy category.
\end{proof}
\begin{remark}
Lemma \ref{homotopyimpliesweakequivalence} shows that cubical $\CC^{\ast}$-spaces represented 
by homotopy equivalent $\CC^{\ast}$-algebras are $\CC^{\ast}$-weakly equivalent.
\end{remark}
\begin{lemma}
\label{preservinghomotopies}
Suppose $f,g\colon\X\rightarrow\Y$ are homotopy equivalent maps and $\Z$ is a cubical $\CC^{\ast}$-space.
Then $\underline{\Hom}(f,\Z)$ and $\underline{\Hom}(g,\Z)$ respectively $f\otimes\Z$ and $g\otimes\Z$ 
are homotopy equivalent maps. 
Thus the internal hom functor $\underline{\Hom}(-,\Z)$ and the tensor functor $-\otimes\Z$ preserves 
homotopy equivalences.
\end{lemma}
\begin{proof}
An elementary homotopy from $f$ to $g$ determines a map of cubical $\CC^{\ast}$-spaces 
\begin{equation*}
\xymatrix{
\underline{\Hom}(\Y,\Z)\ar[r] & 
\underline{\Hom}\bigl(\X\otimes C(I),\Z\bigr).}
\end{equation*}
According to the closed symmetric monoidal structure of $\Box\CC^{\ast}-\Spc$ detailed in 
\S\ref{subsection:CC-spaces} there exists by adjointness a map 
\begin{equation*}
\xymatrix{
\underline{\Hom}(\Y,\Z)\ar[r] & 
\underline{\Hom}\bigl(C(I),\underline{\Hom}(\X,\Z)\bigr).}
\end{equation*}
By adjointness of the latter map we get the desired elementary homotopy
\begin{equation*}
\xymatrix{
\underline{\Hom}(\Y,\Z)\otimes C(I)
\ar[r] & 
\underline{\Hom}(\X,\Z).}
\end{equation*}
The claims concerning $f\otimes\Z$ and $g\otimes\Z$ are clear.
\end{proof}
\begin{corollary}
\label{elementaryweakequivalences}
For every cubical $\CC^{\ast}$-space $\X$ and $n\geq 0$ the canonical map 
\begin{equation*}
\xymatrix{
\X\ar[r] & \X\bigl(C(\Box_{\ttop}^{n})\otimes -\bigr)} 
\end{equation*}
is a $\CC^{\ast}$-weak equivalence.
\end{corollary}
\begin{corollary}
\label{corollary:singCstarweakequivalence}
The canonical map $\X\rightarrow\Sing_{\Box}^{\bullet}(\X)$ is a $\CC^{\ast}$-weak equivalence.
\end{corollary}
\begin{proof}
Applying the homotopy colimit functor yields a commutative diagram with naturally induced vertical 
pointwise weak equivalences \cite[Corollary 18.7.5]{Hirschhorn:Modelcategories}:
\begin{equation*}
\minCDarrowwidth20pt
\begin{CD}
\underset{\Box^{\op}}{\hocolim}\,\,\X_{n} @>>>
\underset{\Box^{\op}}{\hocolim}\,\,\underline{\Hom}\bigl(C(\Box_{\ttop}^{n}),\X_{n}\bigr)\\
@VVV @VVV \\
\X @>>> \Sing_{\Box}^{\bullet}(\X)
\end{CD}
\end{equation*}
In $n$-cells there is a homotopy equivalence 
$\X_{n}\rightarrow\underline{\Hom}\bigl(C(\Box_{\ttop}^{n}),\X_{n}\bigr)$. 
The same map is a $\CC^{\ast}$-weak equivalence of discrete cubical $\CC^{\ast}$-spaces, 
so the upper horizontal map is a $\CC^{\ast}$-weak equivalence on account of 
Corollary \ref{levelprojectiveweakequivalenceimpliesprojectiveweakequivalence}.
\end{proof}
\begin{lemma}
Suppose ${\bf X}\rightarrow {\bf Y}$ is a natural transformation of small diagrams 
$I\rightarrow\Box\CC^{\ast}-\Spc$ such that for every $i\in I$ the induced map 
${\bf X}(i)\rightarrow{\bf Y}(i)$ is a $\CC^{\ast}$-weak equivalence. 
Then the induced map
\begin{equation*}
\xymatrix{
\underset{I}{\hocolim}\,\,{\bf X}\ar[r] & \underset{I}{\hocolim}\,\,{\bf Y} }
\end{equation*}
is a $\CC^{\ast}$-weak equivalence.
\end{lemma}
\begin{proof}
If $\Z$ is a cubical $\CC^{\ast}$-space there is a canonical isomorphism of cubical sets
\begin{equation*}
{\hom}_{\Box\CC^{\ast}-\Spc}(\underset{I}{\hocolim}\,\,{\bf X},\Z)=
\underset{I^{\op}}{\holim}\,\,{\hom}_{\Box\CC^{\ast}-\Spc}({\bf X},\Z).
\end{equation*}
The lemma can also be proven using general properties of homotopy colimits.
\end{proof}
\begin{corollary}
\label{levelprojectiveweakequivalenceimpliesprojectiveweakequivalence}
A map 
of cubical $\CC^{\ast}$-spaces which induces
$\CC^{\ast}$-weak equivalences of discrete cubical $\CC^{\ast}$-spaces 
in all cells is a $\CC^{\ast}$-weak equivalence.
\end{corollary}
\begin{lemma}
\label{lemma:projectiveCstarweakequivalencebetweenfibrantobjects}
Suppose $f\colon\X\rightarrow\Y$ is a map between projective $\CC^{\ast}$-fibrant cubical 
$\CC^{\ast}$-spaces.
The following statements are equivalent where in the third item we assume $\X$ is projective cofibrant.
\begin{itemize}
\item
$f$ is a $\CC^{\ast}$-weak equivalence.
\vspace{0.1cm}
\item
$f$ is a pointwise weak equivalence.
\vspace{0.1cm}
\item
$f$ is a cubical homotopy equivalence.
\end{itemize}
\end{lemma}
\begin{proof}
The equivalence between the first two items follows because the homotopy invariant model structure is a 
localization of the pointwise projective model structure.
Now if $\X$ is projective cofibrant, 
then $\X\otimes\Box^{1}$ is a cylinder object for $\X$ since we are dealing with a cubical model 
structure, 
cf.~\cite[II Lemma 3.5]{GJ:Modelcategories}.
Hence the first item is equivalent to the third by \cite[Theorem 1.2.10]{Hovey:Modelcategories}.
\end{proof}
\begin{remark}
We leave the formulation of Lemma \ref{lemma:projectiveCstarweakequivalencebetweenfibrantobjects} 
for maps between injective $\CC^{\ast}$-fibrant cubical $\CC^{\ast}$-spaces to the reader.
Note that no cofibrancy condition is then required in the third item since every cubical $\CC^{\ast}$-space 
is cofibrant in the injective homotopy invariant model structure. 
\end{remark}

A map between $\CC^{\ast}$-spaces is called a $\CC^{\ast}$-weak equivalence if the associated map 
of constant cubical $\CC^{\ast}$-spaces is a $\CC^{\ast}$-weak equivalence, 
and a cofibration if it is a monomorphism.  
Fibration of $\CC^{\ast}$-spaces are defined by the right lifting property.
\vspace{0.1in}

We are ready to state the analog in $\CC^{\ast}$-homotopy theory of \cite[Theorem B.4]{Jardine:MSS} which
can be verified by similar arguments using Lemma \ref{lemma:geometricrealizationpreservesmonomorphisms}
and Corollary \ref{corollary:singCstarweakequivalence}.
Further details are left to the interested reader.
In this setting, 
\begin{equation*}
\hom_{\CC^{\ast}-\Spc}(\X,\Y)_{n}\equiv\CC^{\ast}-\Spc(\X\otimes\Box^{n},\Y)
\end{equation*}
and
\begin{equation*}
\X^{K}\equiv\lim_{\Box^{n}\rightarrow K}\,\,\underline{\Hom}\bigl(C_{0}(\Box^{n}_{\ttop}),\X\bigr),
\end{equation*}
where $\X$, $\Y$ are $\CC^{\ast}$-spaces, 
and the limit is taken over the cell category of the cubical set $K$.
\begin{theorem}
\label{theorem:spacemodelstructure}
The classes of monomorphisms, fibrations and $\CC^{\ast}$-weak equivalence form a combinatorial, 
cubical and left proper model category on $\CC^{\ast}-\Spc$. 

The singular and geometric realization functors yield a Quillen equivalence:
\begin{equation*}
\xymatrix{
\vert\cdot\vert\colon\Box\CC^{\ast}-\Spc \ar@<3pt>[r] & 
\CC^{\ast}-\Spc\colon\Sing_{\Box}^{\bullet} \ar@<3pt>[l] }
\end{equation*}
\end{theorem}

Next we introduce unstable $\CC^{\ast}$-homotopy group (functors) $\pi^{\ast}_{n}$ for integers $n\geq 0$, 
and show that $f\colon\X\rightarrow\Y$ is a $\CC^{\ast}$-weak equivalence if and only if 
$\pi^{\ast}_0\X\rightarrow\pi^{\ast}_0\Y$ is a bijection and for every $0$-cell $x$ of $\X$ and $n\geq 1$
there is a group object isomorphism
\begin{equation}
\label{homtopygroupmap}
\xymatrix{
\pi^{\ast}_{n}(\X,x)\ar[r] & \pi^{\ast}_{n}\bigl(\Y,f(x)\bigr). }
\end{equation}
If $n\geq 2$, then $\pi^{\ast}_{n}(\X,x)$ takes values in abelian groups.
To achieve this we require the construction of a fibrant replacement functor in the 
$\CC^{\ast}$-projective model structure.
\vspace{0.1in}

Using the cubical mapping cylinder we may factor the constant function map of cubical 
$\CC^{\ast}$-spaces
\begin{equation*}
\xymatrix{
C(I,A)\ar[r] & A}
\end{equation*}
into a projective cofibration composed with a cubical homotopy equivalence 
\begin{equation}
\label{homotopycubicalmappingcylinderfactorization}
\xymatrix{
C(I,A)\ar[r] &
\cyl\bigl(C(I,A)\rightarrow A\bigr)\ar[r] & A.}
\end{equation}
Observe that $\cyl\bigl(C(I,A)\rightarrow A\bigr)$ is finitely presentable 
projective cofibrant and the maps in (\ref{homotopycubicalmappingcylinderfactorization}) are 
$\CC^{\ast}$-weak equivalences.
\begin{definition}
Let $J^{\cyl(I)}_{\Box\CC^{\ast}-\Spc}$ denote the set of pushout product maps from
\begin{equation*}
C(I,A)\otimes\Box^{n}
\coprod_{C(I,A)\otimes\partial\Box^{n}} 
\cyl\bigl(C(I,A)\rightarrow A\bigr)\otimes\partial\Box^{n}
\end{equation*}
to 
$\cyl\bigl(C(I,A)\rightarrow A\bigr)\otimes\Box^{n}$ indexed by $n\geq 0$ and $A\in\CC^{\ast}-\Alg$.
\end{definition}
\begin{lemma}
A matrix exact projective fibration whose codomain is $\CC^{\ast}$-projective fibrant is a 
$\CC^{\ast}$-projective fibration if and only if it has the right lifting property with respect to 
$J^{\cyl(I)}_{\Box\CC^{\ast}-\Spc}$.
\end{lemma}
\begin{proof}
The cubical function complex ${\hom}_{\Box\CC^{\ast}-\Spc}(\Z,-)$ preserves cubical homotopies, 
which are examples of pointwise weak equivalences.
Proposition \ref{proposition:miprojectivemonoidal} and a check using only the definitions reveal 
that a matrix exact projective fibrant cubical $\CC^{\ast}$-space $\Z$ is $\CC^{\ast}$-projective 
fibrant if and only if the map $\Z\rightarrow\ast$ has the right lifting property with respect to 
$J^{\ast}_{\Box\CC^{\ast}-\Spc}$.
This completes the proof by \cite[Proposition 3.3.6]{Hirschhorn:Modelcategories}.
\end{proof}
\begin{corollary}
\label{homotopyweaklyfinitelygenerated}
The $\CC^{\ast}$-projective model is weakly finitely generated by the set
\begin{equation*}
J^{\ast}_{\Box\CC^{\ast}-\Spc}\equiv 
J_{\Box\CC^{\ast}-\Spc}
\cup
J^{\cyl(\E)}_{\Box\CC^{\ast}-\Spc}
\cup
J^{\cyl(\K)}_{\Box\CC^{\ast}-\Spc}
\cup
J^{\cyl(I)}_{\Box\CC^{\ast}-\Spc}.
\end{equation*}
\end{corollary}

Next we employ $J^{\ast}_{\Box\CC^{\ast}-\Spc}$ in order to explicate a fibrant replacement functor 
in the $\CC^{\ast}$-projective model structure by means of a routine small object argument as in the 
proof of Proposition \ref{proposition:projectiveexactreplacement}.
\begin{proposition}
\label{projective*replacement}
There exists a natural transformation 
\begin{equation*}
\xymatrix{
\id\ar[r] &\Ex^{J^{\ast}_{\Box\CC^{\ast}-\Spc}} }
\end{equation*}
of endofunctors of cubical $\CC^{\ast}$-spaces such that for every $\X$ the map 
\begin{equation*}
\xymatrix{
\X\ar[r] & \Ex^{J^{\ast}_{\Box\CC^{\ast}-\Spc}}\X }
\end{equation*}
is a $\CC^{\ast}$-weak equivalence with $\CC^{\ast}$-fibrant codomain. 
\end{proposition}
\begin{definition}
Let $(\X,x)$ be a pointed cubical $\CC^{\ast}$-space.
Define the $n$th $\CC^{\ast}$-homotopy group
\begin{equation*}
\xymatrix{
\pi^{\ast}_{n}(\X,x)\colon\CC^{\ast}-\Alg\ar[r] & \Set}
\end{equation*}
by
\begin{equation*}
\xymatrix{
A\ar@{|->}[r] & 
\pi^{\ast}_{n}(\X,x)(A)}
\equiv 
\begin{cases}
\pi_{0}\bigl(\Ex^{J^{\ast}_{\Box\CC^{\ast}-\Spc}}\X(A)\bigr)& n=0 \\
\pi_{n}\bigl(\Ex^{J^{\ast}_{\Box\CC^{\ast}-\Spc}}\X(A),x\bigr) & n>0.
\end{cases}
\end{equation*}
A cubical $\CC^{\ast}$-space $\X$ is $n$-connected if it is nonempty and $\pi^{\ast}_{i}(\X,x)$ 
is trivial for all $0\leq i\leq n$ and $x$.
Note that $\pi^{\ast}_{n}(\X,x)$ is a contravariant functor taking values in sets if $n=0$, groups if $n=1$ 
and abelian groups if $n\geq 2$.
\end{definition}
\begin{lemma}
A map $(\X,x)\rightarrow (\Y,y)$ is a $\CC^{\ast}$-weak equivalence if and only if for every 
integer $n\geq 0$ there are naturally induced isomorphisms between $\CC^{\ast}$-homotopy groups
\begin{equation*}
\xymatrix{
\pi^{\ast}_{n}(\X,x)\ar[r] & \pi^{\ast}_{n}(\Y,y).}
\end{equation*}
\end{lemma}
\begin{proof}
Using the properties of the natural transformation $\id\rightarrow\Ex^{J^{\ast}_{\Box\CC^{\ast}-\Spc}}$ appearing 
in Proposition \ref{projective*replacement} and the Whitehead theorem for localizations of model categories,
it follows that $\X\rightarrow\Y$ is a $\CC^{\ast}$-weak equivalence if and only if
\begin{equation*}
\xymatrix{
\Ex^{J^{\ast}_{\Box\CC^{\ast}-\Spc}}\X\ar[r] & 
\Ex^{J^{\ast}_{\Box\CC^{\ast}-\Spc}}\Y}
\end{equation*}
is a pointwise weak equivalence.
\end{proof}

Next we show a characterization of the class of $\CC^{\ast}$-weak equivalences.
\begin{proposition}
\label{proposition:charactericationprojective*weakequivalences}
The class of $\CC^{\ast}$-weak equivalences is the smallest class of maps 
${\bf \ast-weq}$ of cubical $\CC^{\ast}$-spaces which satisfies the following properties. 
\begin{itemize}
\item
${\bf \ast-weq}$ is saturated.
\item
${\bf \ast-weq}$ contains the class of exact projective weak equivalences.
\item
${\bf \ast-weq}$ contains the elementary matrix invariant weak equivalences 
\begin{equation*}
\xymatrix{
A\otimes\K\ar[r] & A.}
\end{equation*}
and the elementary $\CC^{\ast}$-weak equivalences
\begin{equation*}
\xymatrix{
C(I,A)\ar[r] & A.}
\end{equation*}
\item
Suppose there is a pushout square of cubical $\CC^{\ast}$-spaces where $f$ is in ${\bf \ast-weq}$:
\begin{equation*}
\minCDarrowwidth20pt
\begin{CD}
\X @>g>> \Z \\
@VfVV @VVhV \\
\Y @>>> \W
\end{CD}
\end{equation*}
If $g$ is a projective cofibration, 
then $h$ is in ${\bf \ast-weq}$.
If $f$ is a projective cofibration, 
then $h$ is a projective cofibration contained in ${\bf \ast-weq}$.
\item
Suppose ${\bf X}\colon I\rightarrow\Box\CC^{\ast}-\Spc$ is a small filtered diagram such that for every 
$i\rightarrow j$, 
the induced map ${\bf X}(i)\rightarrow{\bf X}(j)$ is a map in ${\bf \ast-weq}$.
Then the induced map
\begin{equation*}
\xymatrix{
{\bf X}(i)\ar[r] & \underset{j\in i\downarrow I}{\colim}\,\,{\bf X}(j)}
\end{equation*}
is in ${\bf \ast-weq}$.
\end{itemize}  
\end{proposition}
\vspace{0.1in}

It is clear that the class of $\CC^{\ast}$-weak equivalences satisfies the first, second and third conditions in 
Proposition \ref{proposition:charactericationprojective*weakequivalences}.
\vspace{0.1in}

The first part of the fourth item holds because the homotopy invariant projective model structure is left proper, 
while the second part holds because acyclic cofibrations are closed under pushouts in any model structure.
\vspace{0.1in}

If ${\bf X}\rightarrow{\bf Y}$ is a map of diagrams as in the fifth part, 
\cite[Proposition 17.9.1]{Hirschhorn:Modelcategories} implies there is an induced $\CC^{\ast}$-weak equivalence
\begin{equation*}
\xymatrix{
\underset{i\in I}{\colim}\,\,{\bf X}(i)\ar[r] & 
\underset{i\in I}{\colim}\,\,{\bf Y}(i).}
\end{equation*}

Now consider the small filtered undercategory $i\downarrow I$ with objects the maps $i\rightarrow j$ in $I$, 
and with maps the evident commutative triangles of objects.
Applying the above to ${\bf X}(i)\rightarrow{\bf X}$ implies the last item.
Note that in the formulation of the last item we may replace colimits by homotopy colimits.
\vspace{0.1in}
 
The proof of Proposition \ref{proposition:charactericationprojective*weakequivalences} makes use 
of a functorial fibrant replacement functor in the homotopy invariant projective model structure.
Denote by
\begin{equation*}
\xymatrix{
\id_{\Box\CC^{\ast}-\Spc}\ar[r] & (-)^{f}_{C(I)}}
\end{equation*}
the fibrant replacement functor obtained by applying the small object argument to the set of 
(isomorphism classes of) elementary $\CC^{\ast}$-weak equivalences. 
\vspace{0.1in}

With these preliminaries taken care of we are ready to begin the proof.
\vspace{0.1in}

\begin{proof}(of Proposition \ref{proposition:charactericationprojective*weakequivalences}.)
Every elementary weak equivalence $A\rightarrow C(I,A)$ is contained in ${\bf \ast-weq}$.
We shall prove that any $\CC^{\ast}$-weak equivalence $\X\rightarrow\Y$ can be constructed 
from elementary weak equivalences using constructions as in the statement of the proposition.
\vspace{0.1in}

There is a commutative diagram:
\begin{equation*}
\minCDarrowwidth20pt
\begin{CD}
\X @>>> \Y \\
@VVV @VVV \\
\X^{f}_{C(I)} @>>> \Y^{f}_{C(I)}
\end{CD}
\end{equation*}
The vertical maps are constructed out of direct colimits of pushouts of elementary weak equivalences.
The third and fourth items imply that the vertical maps are contained in ${\bf \ast-weq}$.
On the other hand, 
saturation for $\CC^{\ast}$-weak equivalences and the defining property of the fibrant replacement functor 
imply the lower horizontal map is a projective $\CC^{\ast}$-weak equivalence. 
This implies that it is a matrix exact projective weak equivalence.
Since ${\bf \ast-weq}$ contains all matrix exact projective weak equivalences according to the second and third conditions, 
the lower horizontal map is contained in ${\bf \ast-weq}$.
Thus saturation, 
or the two-out-of-three property, 
for the class ${\bf \ast-weq}$ which holds by the first item implies the $\CC^{\ast}$-weak equivalence $\X\rightarrow\Y$ is 
contained in ${\bf \ast-weq}$.
\end{proof}
\vspace{0.1in}

In order to construct a fibrant replacement functor for the injective homotopy invariant model structure we shall 
proceed a bit differently.
A flasque cubical $\CC^{\ast}$-space $\Z$ is called quasifibrant if the maps $\Z(A)\rightarrow\Z(A\otimes\K)$ 
- corresponding to matrix invariance - 
and $\Z(A)\rightarrow\Z\bigl(C(I,A)\bigr)$
- corresponding to homotopy invariance - 
are weak equivalences for every $\CC^{\ast}$-algebra $A$.
We note that every $\CC^{\ast}$-projective fibrant cubical $\CC^{\ast}$-space is quasifibrant.
\vspace{0.1in}

For a $\CC^{\ast}$-space $\X$, 
we set 
\begin{equation*}
(\Ex^{\cyl(\E,\K)}_{\Sing_{\Box}^{\bullet}}\X)_{0}\equiv\Sing_{\Box}^{\bullet}\X 
\end{equation*}
and form inductively pushout diagrams
\begin{equation*}
\xymatrix{
\coprod_{\alpha_{n}} s\alpha \ar[r]\ar[d] & \Sing_{\Box}^{\bullet}(\Ex^{\cyl(\E,\K)}_{\Sing_{\Box}^{\bullet}}\X)_{n}\ar[d] \\
\coprod_{\alpha_{n}} t\alpha \ar[r]       & (\Ex^{\cyl(\E,\K)}_{\Sing_{\Box}^{\bullet}}\X)_{n+1}  }
\end{equation*}
indexed by the set $\alpha_{n}$ of all commutative diagrams
\begin{equation*}
\xymatrix{
s\alpha \ar[r]\ar[d] & \Sing_{\Box}^{\bullet}(\Ex^{\cyl(\E,\K)}_{\Sing_{\Box}^{\bullet}}\X)_{n}\ar[d] \\
t\alpha \ar[r]       & \ast }
\end{equation*}
where $s\alpha\rightarrow t\alpha$ is member of $J^{\cyl(\E)}_{\Box\CC^{\ast}-\Spc}\cup J^{\cyl(\K)}_{\Box\CC^{\ast}-\Spc}$.
There is an induced map from $\X$ to the colimit $\Ex^{\cyl(\E,\K)}_{\Sing_{\Box}^{\bullet}}\X$ of the sequential diagram of
alternating injective acyclic $\CC^{\ast}$-weak equivalences according to Example \ref{example:monomorphism}, Corollary 
\ref{corollary:singCstarweakequivalence} and $J^{\cyl(\E)}_{\Box\CC^{\ast}-\Spc}\cup J^{\cyl(\K)}_{\Box\CC^{\ast}-\Spc}$-acyclic
cofibrations:
\begin{equation}
\label{equation:quasifibranttower}
\xymatrix{ 
\cdots\ar[r] & (\Ex^{\cyl(\E,\K)}_{\Sing_{\Box}^{\bullet}}\X)_{n}
\ar[r] & \Sing_{\Box}^{\bullet}(\Ex^{\cyl(\E,\K)}_{\Sing_{\Box}^{\bullet}}\X)_{n}\ar[r] &
(\Ex^{\cyl(\E,\K)}_{\Sing_{\Box}^{\bullet}}\X)_{n+1} \ar[r] & \cdots }
\end{equation}
\begin{lemma}
\label{lemma:quasifibrant}
There is an endofunctor $\Ex^{\cyl(\E,\K)}_{\Sing_{\Box}^{\bullet}}$ of $\Box\CC^{\ast}-\Spc$ and a natural transformation
\begin{equation*}
\xymatrix{
\id_{\Box\CC^{\ast}-\Spc}\ar[r] & \Ex^{\cyl(\E,\K)}_{\Sing_{\Box}^{\bullet}} }
\end{equation*}
such that $\Ex^{\cyl(\E,\K)}_{\Sing_{\Box}^{\bullet}}\X$ is quasifibrant for every cubical $\CC^{\ast}$-space $\X$ and the map 
\begin{equation*}
\xymatrix{
\X\ar[r] & \Ex^{\cyl(\E,\K)}_{\Sing_{\Box}^{\bullet}}\X }
\end{equation*}
is an injective acyclic $\CC^{\ast}$-weak equivalence.
\end{lemma}
\begin{proof}
The natural transformation exists by naturality of the map $\X\rightarrow\Ex^{\cyl(\E,\K)}_{\Sing_{\Box}^{\bullet}}\X$ from $\X$ 
to the colimit of (\ref{equation:quasifibranttower}), 
and by \cite[Proposition 17.9.1]{Hirschhorn:Modelcategories} it is an injective acyclic $\CC^{\ast}$-weak equivalence. 
To show quasifibrancy, 
note that homotopy invariance holds on account of the singular functor and that $\Ex^{\cyl(\E,\K)}_{\Sing_{\Box}^{\bullet}}\X$
has the right lifting property with respect to $J^{\cyl(\E)}_{\Box\CC^{\ast}-\Spc}\cup J^{\cyl(\K)}_{\Box\CC^{\ast}-\Spc}$ since
the domains and codomains of maps in the latter set preserve sequential colimits.
\end{proof}
\begin{corollary}
Let 
\begin{equation*}
\xymatrix{
\id_{\Box\CC^{\ast}-\Spc}\ar[r] & \RRR }
\end{equation*}
denote a fibrant replacement functor in the pointwise injective model structure.
Then 
\begin{equation*}
\xymatrix{
\id_{\Box\CC^{\ast}-\Spc}\ar[r] & \RRR\Ex^{\cyl(\E,\K)}_{\Sing_{\Box}^{\bullet}} }
\end{equation*}
is a fibrant replacement functor in the injective homotopy invariant model structure.
\end{corollary}
\begin{proof}
For every cubical $\CC^{\ast}$-space $\X$ the composite map 
\begin{equation*}
\xymatrix{
\X\ar[r] & 
\Ex^{\cyl(\E,\K)}_{\Sing_{\Box}^{\bullet}}\X\ar[r] & 
\RRR(\Ex^{\cyl(\E,\K)}_{\Sing_{\Box}^{\bullet}}\X) }
\end{equation*}
is an injective acyclic $\CC^{\ast}$-weak equivalence by Lemma \ref{lemma:quasifibrant} and the defining property of $\RRR$.
Moreover,
$\RRR(\Ex^{\cyl(\E,\K)}_{\Sing_{\Box}^{\bullet}}\X)$ is clearly injective fibrant, matrix invariant and homotopy invariant.
To show that it is flasque, note that for every exact square $\E$ the diagram
\begin{equation}
\label{equation:Exdiagram}
\Ex^{\cyl(\E,\K)}_{\Sing_{\Box}^{\bullet}}(\E)\equiv
\minCDarrowwidth20pt
\begin{CD}
\Ex^{\cyl(\E,\K)}_{\Sing_{\Box}^{\bullet}}(A) @>>> \Ex^{\cyl(\E,\K)}_{\Sing_{\Box}^{\bullet}}(E) \\
@VVV @VVV\\
\ast @>>> \Ex^{\cyl(\E,\K)}_{\Sing_{\Box}^{\bullet}}(B)
\end{CD}
\end{equation}
is homotopy cartesian due to Lemma \ref{lemma:quasifibrant}.
Applying the pointwise injective fibrant replacement functor $\RRR$ yields a pointwise weak equivalence between 
(\ref{equation:Exdiagram}) and:  
\begin{equation}
\label{equation:RRRExdiagram}
\RRR\Ex^{\cyl(\E,\K)}_{\Sing_{\Box}^{\bullet}}(\E)\equiv
\minCDarrowwidth20pt
\begin{CD}
\RRR\Ex^{\cyl(\E,\K)}_{\Sing_{\Box}^{\bullet}}(A) @>>> \RRR\Ex^{\cyl(\E,\K)}_{\Sing_{\Box}^{\bullet}}(E) \\
@VVV @VVV\\
\ast @>>> \RRR\Ex^{\cyl(\E,\K)}_{\Sing_{\Box}^{\bullet}}(B)
\end{CD}
\end{equation}
It follows that (\ref{equation:RRRExdiagram}) is homotopy cartesian \cite[Proposition 13.3.13]{Hirschhorn:Modelcategories}.
\end{proof}

Next we note the homotopy invariant projective model structure is compatible with the monoidal structure.
The proof is analogous to the proof of Proposition \ref{proposition:miprojectivemonoidal}, 
using that for $\CC^{\ast}$-algebras $A$ and $B$,
$\cyl\bigl(C(I,A)\rightarrow A\bigr)\otimes B=
\cyl\bigl(C(I,A\otimes B)\rightarrow A\otimes B\bigr)$.
\begin{proposition}
\label{proposition:hiprojectivemonoidal}
The homotopy invariant projective model structure on $\Box\CC^{\ast}-\Spc$ is monoidal.
\end{proposition}
\begin{lemma}
\label{lemma:projectiveinternalhomprojectiveCstarfibrant}
If $\X$ is projective cofibrant and $\Z$ is $\CC^{\ast}$-projective fibrant, 
then the internal hom object $\underline{\Hom}(\X,\Z)$ is $\CC^{\ast}$-projective fibrant.
\end{lemma}
\begin{lemma}
\label{lemma:hiprojectiveQuillenadjunction}
If $\X$ is a projective cofibrant cubical $\CC^{\ast}$-space, 
then 
\begin{equation*}
\bigl(-\otimes\X,\underline{\Hom}(\X,-)\bigr)
\end{equation*}
is a Quillen adjunction for the homotopy invariant projective model structure on $\Box\CC^{\ast}-\Spc$.
\end{lemma}
The above has the following consequence.
\begin{lemma}
\label{lemma:hiprojectivemonoidal2}
The following statements hold and are equivalent.
\begin{itemize}
\item
If $i\colon\X\cof\Y$ and $j\colon\U\cof\V$ are projective cofibrations and either $i$ or $j$ is a 
$\CC^{\ast}$-weak equivalence, 
then so is
\begin{equation*}
\xymatrix{
\X\otimes\V\coprod_{\X\otimes\U}\Y\otimes\U\ar[r] & \Y\otimes\V.}
\end{equation*}
\item
If $j\colon\U\cof\V$ is a projective cofibration and $k\colon\Z\fib\W$ is a projective $\CC^{\ast}$-fibration, 
then the pullback map
\begin{equation*}
\xymatrix{
\underline{\Hom}(\V,\Z)\ar[r] &
\underline{\Hom}(\V,\W)\times_{\underline{\Hom}(\U,\W)}\underline{\Hom}(\U,\Z)}
\end{equation*}
is a projective $\CC^{\ast}$-fibration which is $\CC^{\ast}$-acyclic if either $j$ or $k$ is.
\item
With the same assumptions as in the previous item, the induced map
\begin{equation*}
\xymatrix{
{\hom}_{\Box\CC^{\ast}-\Spc}(\V,\Z)\ar[r] &
{\hom}_{\Box\CC^{\ast}-\Spc}(\V,\W)
\times_{{\hom}_{\Box\CC^{\ast}-\Spc}(\U,\W)}{\hom}_{\Box\CC^{\ast}-\Spc}(\U,\Z)}
\end{equation*}
is a Kan fibration which is a weak equivalence of cubical sets if in addition either $j$ or $k$ is $\CC^{\ast}$-acyclic.
\end{itemize} 
\end{lemma}
\begin{lemma}
\label{lemma:projectiveCstarweakequivalenceinternalhom}
A map of cubical $\CC^{\ast}$-spaces $\X\rightarrow\Y$ is a $\CC^{\ast}$-weak equivalence 
if and only if for every projective $\CC^{\ast}$-fibrant $\Z$ the induced map of internal hom objects
\begin{equation}
\label{inducedinternalhommap}
\xymatrix{
\underline{\Hom}(\QQ\Y,\Z)\ar[r] & \underline{\Hom}(\QQ\X,\Z)}
\end{equation}
is a pointwise weak equivalence.
\end{lemma}
\begin{proof}
Lemma \ref{lemma:projectiveinternalhomprojectiveCstarfibrant} implies (\ref{inducedinternalhommap}) 
is a map between $\CC^{\ast}$-projective fibrant objects.
Hence, 
by Lemma \ref{lemma:projectiveCstarweakequivalencebetweenfibrantobjects},  
(\ref{inducedinternalhommap}) is a $\CC^{\ast}$-weak equivalence if and only if for every 
$\CC^{\ast}$-algebra $A$ the induced map 
\begin{equation*}
\xymatrix{
\underline{\Hom}(\QQ\Y,\Z)(A)\ar[r] & \underline{\Hom}(\QQ\X,\Z)(A), }
\end{equation*}
or equivalently 
\begin{equation*}
\xymatrix{
\hom_{\Box\CC^{\ast}-\Spc}\bigl(\QQ\Y,\underline{\Hom}(A,\Z)\bigr) \ar[r] & 
\hom_{\Box\CC^{\ast}-\Spc}\bigl(\QQ\X,\underline{\Hom}(A,\Z)\bigr) }
\end{equation*}
is a weak equivalence of cubical sets.
Since every $\CC^{\ast}$-algebra is projective cofibrant according to Lemma \ref{lemma:projectivecofibrant}, 
the internal hom object $\underline{\Hom}(A,\Z)$ is $\CC^{\ast}$-projective fibrant again by  
Lemma \ref{lemma:projectiveinternalhomprojectiveCstarfibrant}. 
\end{proof}

The next result follows now from \cite[Theorem 4.3.2]{Hovey:Modelcategories}.
\begin{corollary}
\label{corollary:Cstarhomotopycategory}
The total derived adjunction of $(\otimes,\underline{Hom})$ gives a closed symmetric monoidal structure 
on the unstable $\CC^{\ast}$-homotopy category $\HH$.
The associativity, commutativity and unit isomorphisms are derived from the corresponding isomorphisms 
in $\Box\CC^{\ast}-\Spc$.
\end{corollary}
\begin{remark}
Comparing with the corresponding derived adjunction obtained from the closed monoidal structure on 
$\Delta\CC^{\ast}-\Spc$ and the Quillen equivalent homotopy invariant model structure, 
we get compatible closed symmetric monoidal structures on the unstable $\CC^{\ast}$-homotopy category.
\end{remark}
\newpage

\subsection{Pointed model structures}
\label{subsection:pointedmodelstructures}
It is straightforward to show the results for the model structures on $\Box\CC^{\ast}-\Spc$ have 
analogs for the categories $\Box\CC^{\ast}-\Spc_{0}$ of pointed cubical $\CC^{\ast}$-spaces and 
$\Delta\CC^{\ast}-\Spc_{0}$ of pointed simplicial $\CC^{\ast}$-spaces.
Let $\HH^{\ast}$ denote the unstable pointed $\CC^{\ast}$-homotopy category.
In this section we identify a set of compact generators for $\HH^{\ast}$, 
formulate Brown representability for $\HH^{\ast}$ and compute Kasparov's $KK$-groups of 
$\CC^{\ast}$-algebras as maps in $\HH^{\ast}$.
To prove this result we use the simplicial category $\Delta\CC^{\ast}-\Spc_{0}$.

The next observation will be used in the context of cubical $\CC^{\ast}$-spectra.
\begin{lemma}
\label{lemma:pointedprojectiveinternalhomprojectiveCstarfibrant}
Suppose $\X$ is projective cofibrant and $\Z$ $\CC^{\ast}$-projective fibrant in $\Box\CC^{\ast}-\Spc_{0}$.
Then the pointed internal hom object $\underline{\Hom}_{0}(\X,\Z)$ is $\CC^{\ast}$-projective fibrant.
\end{lemma}
\begin{proof}
There is a pullback diagram of cubical $\CC^{\ast}$-spaces:
\begin{equation*}
\xymatrix{
\underline{\Hom}_{0}(\X,\Z) \ar[r]\ar[d] & \underline{\Hom}(\X,\Z)\ar[d] \\
\ast\ar[r] & \underline{\Hom}(\ast,\Z) }
\end{equation*}
By monoidalness in the form of Lemma \ref{lemma:hiprojectivemonoidal2} the right vertical map is a $\CC^{\ast}$-projective fibration.
Now use that fibrations pull back to fibrations in every model structure.
\end{proof}

Suppose $\M$ is a pointed model category.  
Recall that $\GG$ is a set of weak generators for $\Ho(\M)$ if for every nontrivial $\Y\in\Ho(\M)$
there is an $\X\in\GG$ such that $\Ho(\M)(\Sigma^{n}\X,\Y)$ is nontrivial.  
An object $\X\in\Ho(\M)$ is called small if, 
for every set $\{\X_{\alpha}\}_{\alpha \in\lambda}$ of objects of $\Ho(\M)$, 
there is a naturally induced isomorphism
\begin{equation*}
\xymatrix{
\underset{\lambda'\subseteq \lambda, \vert\lambda'\vert<\infty}{\colim}
\Ho(\M)(\X,\coprod _{\alpha\in\lambda'}\X_{\alpha}) 
\ar[r] & 
\Ho(\M)(\X,\coprod _{\alpha \in \lambda}X_{\alpha}). }
\end{equation*}
By \cite[Section~7.3]{Hovey:Modelcategories} the cofibers of the generating cofibrations in any 
cofibrantly generated model category $\M$ form a set of weak generators for $\Ho(\M)$.  
However,
it is more subtle to decide whether these weak generators are small in $\Ho(\M)$.
The argument given in~\cite[Section~7.4]{Hovey:Modelcategories} relies not only on smallness
properties of the domains and codomains of the generating cofibrations of $\M$, 
but also on detailed knowledge of the generating trivial cofibrations.  
For further details we refer to the proof of the analogous stable result Theorem \ref{theorem:stableweakgenerators}.
\begin{theorem}
\label{theorem:unstableweakgenerators}
The cofibers of the generating projective cofibrations 
\begin{equation*}
\{A\otimes(\partial\Box^{n}\subset\Box^{n})_{+}\}_{A}^{n\geq 0}
\end{equation*}
form a set of compact generators for the pre-triangulated homotopy category $\HH^{\ast}$ 
of the homotopy invariant model structure on pointed cubical $\CC^{\ast}$-spaces.
\end{theorem}

Next we formulate Brown representability for contravariant functors from the pointed homotopy 
category of $\CC^{\ast}$-spaces to pointed sets.  
\begin{theorem}
\label{theorem:unstablebrown}
Suppose the contravariant functor $\F$ from $\HH^{\ast}$ to $\Set_{\ast}$ satisfies the
following properties. 
\begin{itemize}
\item
$\F(0)$ is the one-point set.
\item
For every set $\{\X_{\alpha}\}$ of objects in $\Box\CC^{\ast}-\Spc_{0}$ there is a naturally induced 
bijective map
\begin{equation*}
\xymatrix{
\F(\bigvee\X_{\alpha})\ar[r] & \prod \F(\X_{\alpha}). }
\end{equation*}
\item
For every pointed projective cofibration $\X\rightarrow\Y$ and pushout diagram 
\begin{equation*}
\xymatrix{
\X\ar[r]\ar[d] & \Y \ar[d] \\
\Z\ar[r] & \Z\cup_{\X}\Y }
\end{equation*}
there is a naturally induced surjective map
\begin{equation*}
\xymatrix{
\F(\Z\cup_{\X}\Y)\ar[r] & \F(\Z)\times_{\F(\X)}\F(\Y). }
\end{equation*}
\end{itemize}
Then there exists a pointed cubical $\CC^{\ast}$-space $\W$ and a natural isomorphism
\begin{equation*}
\HH^{\ast}(-,\W)=\F(-).
\end{equation*} 
\end{theorem}
\begin{proof}
Theorem \ref{theorem:unstableweakgenerators} and left properness imply the $\CC^{\ast}$-projective 
model structure on $\Box\CC^{\ast}-\Spc_{0}$ satisfies the assumptions in Jardine's representability 
theorem for pointed model categories \cite{Jardine:representabilityformodelcategories}. 
\end{proof}
\begin{remark}
Left properness ensures the contravariant pointed set valued functor $\HH^{\ast}(-,\W)$ satisfies the 
conditions in the formulation of Theorem \ref{theorem:unstablebrown}.
\end{remark}
\begin{lemma}
\label{lemma:Cstarspacecofibrantadjunction}
If $\X$ is a projective cofibrant pointed $\CC^{\ast}$-space,
then there is a Quillen map 
\begin{equation*}
\xymatrix{
\X\otimes-\colon \Box\CC^{\ast}-\Spc_{0}\ar@<3pt>[r] & 
\Box\CC^{\ast}-\Spc_{0} \colon \underline{\Hom}(\X,-) \ar@<3pt>[l]}
\end{equation*}
of the homotopy invariant model structure. 
\end{lemma}

\begin{lemma}
\label{lemma:unstablerepresentability}
Suppose $\X$ is a projective $\CC^{\ast}$-fibrant pointed cubical $\CC^{\ast}$-space.
Then for every $\CC^{\ast}$-algebra $A$ and integer $n\geq 0$ there is a natural isomorphism
\begin{equation*}
\pi_{n}\X(A)=
\HH^{\ast}(A\otimes S^{n},\X).
\end{equation*}
\end{lemma}
\begin{proof}
Let $\simeq$ be the equivalence relation generated by cubical homotopy equivalence. 
Since $\Omega_{S^{1}}\X$ is projective $\CC^{\ast}$-fibrant by the assumption on $\X$,
the Yoneda lemma and Proposition \ref{proposition:homotopyclassesofmaps} imply there are isomorphisms 
\begin{eqnarray*}
\pi_{n}\X(A) &=& \pi_{n}\hom_{\Box\CC^{\ast}-\Spc_{0}}(A,\X)\\
&=& \Box\CC^{\ast}-\Spc_{0}(A,\Omega_{S^{1}}^{n}\X)/\simeq \\
&=& \HH^{\ast}(A\otimes S^{n},\X).
\end{eqnarray*}
\end{proof}

In the next theorem we use unstable $\CC^{\ast}$-homotopy theory to represent Kasparov's $KK$-groups.
The proof we give makes extensive use of $K$-theoretic techniques which are couched in simplicial sets;
it carries over to the cubical setting in the likely event that the cubical nerve furnishes an equivalent way of constructing $K$-theory. 
Section \ref{subsection:KtheoryofCstaralgebras} gives a fuller review of the $K$-theory machinery behind categories with cofibrations and 
weak equivalences.  
\begin{theorem}
\label{theorem:unstablerepresentabilityKK}
Let $F$ be a $\CC^{\ast}$-algebra.
The pointed simplicial $\CC^{\ast}$-space
\begin{equation*}
\xymatrix{
F^{\Rep}\colon\CC^{\ast}-\Alg\ar[r] & \Delta\Set_{0}}
\end{equation*}
defined by 
\begin{equation*}
F^{\Rep}(E)\equiv
K\bigl(\Rep(F,E)\bigr)
=\Omega\vert\,N\htp S_{\bullet}\Rep(F,E)\,\vert
\end{equation*}
is projective $\CC^{\ast}$-fibrant.  
For $n\geq 0$ there is a natural isomorphism
\begin{equation*}
KK_{n-1}(F,E)=\HH^{\ast}(E\otimes S^{n},F^{\Rep}).
\end{equation*}
\end{theorem}

Here $\Rep(A,B)$ is the idempotent complete additive category of representations between $\CC^{\ast}$-algebras $A$ and $B$.
It is a category with cofibrations the maps which are split monomorphisms and weak equivalences the isomorphisms.
Now passing to the $K$-theory of $\Rep(A,B)$ by using a fibrant geometric realization functor we get a pointed simplicial 
$\CC^{\ast}$-space $F^{\Rep}$ for every $\CC^{\ast}$-algebra $F$.
Next we briefly outline the part of the proof showing $F^{\Rep}$ is exact projective fibrant:
It is projective fibrant by construction (every simplicial abelian group is fibrant).
To show it is flasque we shall trade $\Rep(A,B)$ for the category $\Ch^{\bb}\bigl(\Rep(A,B)\bigr)$ of bounded chain complexes. 
The extra information gained by passing to chain complexes allows us to finish the proof.

The canonical inclusion of $\Rep(A,B)$ into $\Ch^{\bb}\bigl(\Rep(A,B)\bigr)$ as chain complexes 
of length one induces an equivalence in $K$-theory \cite[Theorem 1.11.7]{TT}. 
Thus we may assume $\Rep(A,B)$ acquires a cylinder functor and satisfies the cylinder, 
extension and saturation axioms.
Applying the fibration theorem \cite[Theorem 1.6.4]{Waldhausen:LNM1126} furnishes for every 
short exact sequence (\ref{ses}) of $\CC^{\ast}$-algebras with a completely positive splitting 
the desired homotopy fiber sequence 
\begin{equation*}
\xymatrix{
F^{\Rep}(A)\ar[r] & F^{\Rep}(E)\ar[r] & F^{\Rep}(B).}
\end{equation*}
The second part of Theorem \ref{theorem:unstablerepresentabilityKK} follows by combining the first 
part with Lemma \ref{lemma:unstablerepresentability} and work of Kandelaki \cite{Kandelaki:KKRep}.
\vspace{0.1in}

To prepare ground for the proof of Theorem \ref{theorem:unstablerepresentabilityKK} we shall recall 
some notions from \cite{Kandelaki:KKRep} and \cite{Kasparov:Hilbertmodules}.
In particular, 
we shall consider categories enriched in the symmetric monoidal category of $\CC^{\ast}$-algebras,
a.k.a.~$\CC^{\ast}$-categories.
The category of Hilbert spaces and bounded linear maps is an example.
Every unital $\CC^{\ast}$-algebra defines a $\CC^{\ast}$-category with one object and with the elements of the algebra as maps. 
\vspace{0.1in}

If $B$ is a $\CC^{\ast}$-algebra,  
then a Hilbert $B$-module $\Hilbert$ consists of a countably generated right Hilbert module over $B$ equipped with an 
inner product $\langle\,\,\mid\,\,\rangle\colon\Hilbert\times\Hilbert\rightarrow B$.
Denote by $\Hilbert(B)$ the additive $\CC^{\ast}$-category of Hilbert $B$-modules with respect to sums of Hilbert
modules and by $\K(B)$ its $\CC^{\ast}$-ideal of compact maps.
Next we consider pairs $(\Hilbert,\rho)$ where $\Hilbert\in\Hilbert(B)$ and 
$\rho\colon A\rightarrow\mathcal{L}(\Hilbert)$ is a $\ast$-homomorphism.
Here $\mathcal{L}(\Hilbert)$ is the algebra of linear operators on $\Hilbert$ which admit an adjoint with respect 
to the inner product. 
A map $(\Hilbert,\rho)\rightarrow (\Hilbert',\rho')$ consists of a map $f\colon\Hilbert\rightarrow\Hilbert'$ in 
$\Hilbert(B)$ such that $f\rho(a)-\rho'(a)f$ is in $\K(B)(\Hilbert,\Hilbert')$ for all $a\in A$.
This defines the structure of an additive $\CC^{\ast}$-category inherited from $\Hilbert(B)$.
Let $\Rep(A,B)$ denote its universal pseudoabelian $\CC^{\ast}$-category.
Its objects are triples $(\Hilbert,\rho,p)$ where $p\colon (\Hilbert,\rho)\rightarrow (\Hilbert,\rho)$
satisfies $p=p^{\ast}$ and $p^{2}=p$, and maps $(\Hilbert,\rho,p)\rightarrow (\Hilbert',\rho',p')$ 
consists of maps of pairs $f\colon (\Hilbert,\rho)\rightarrow (\Hilbert',\rho')$ as above, 
subject to the relation $fp=p'f=f$.
We note that triples are added according to the formula
$(\Hilbert,\rho,p)\oplus (\Hilbert',\rho',p')\equiv (\Hilbert\oplus\Hilbert',\rho\oplus\rho',p\oplus p')$.
\vspace{0.1in}

Let $\Ch^{\bb}\bigl(\Rep(A,B)\bigr)$ be the chain complex category of bounded chain complexes 
$E^{\bb}\colon 0\rightarrow E_{m}\rightarrow\cdots\rightarrow E_{n}\rightarrow 0$ in the additive 
category $\Rep(A,B)$.
It acquires the structure of a category with cofibrations and weak equivalences 
$\htp\Ch^{\bb}\bigl(\Rep(A,B)\bigr)$ in the sense of Waldhausen \cite{Waldhausen:LNM1126} 
with cofibrations the degreewise split monomorphisms and weak equivalences the maps whose 
mapping cones are homotopy equivalent to acyclic complexes in $\Ch^{\bb}\bigl(\Rep(A,B)\bigr)$.
\vspace{0.1in}

If $f\colon E^{\bb}\rightarrow F^{\bb}$ is a map in $\Ch^{\bb}\bigl(\Rep(A,B)\bigr)$, 
let $T(f)$ be the bounded chain complex given
\begin{equation*}
T(f)_{p}\equiv E_{p}\oplus E_{p-1}\oplus F_{p}.
\end{equation*}
The boundary maps of $T(f)$ are determined by the matrix: 
\[ 
\left( \begin{array}{ccc}
d_{E^{\bb}} & -\id & 0 \\
0 & -d_{E^{\bb}} & 0 \\
0 & f & d_{F^{\bb}} \end{array} 
\right)\] 
There exist natural inclusions of direct summands $i_{E^{\bb}}\colon E^{\bb}\subset T(f)$ and 
$i_{F^{\bb}}\colon F^{\bb}\subset T(f)$.
These maps fit into the commutative diagram:
\begin{equation*}
\xymatrix{
E^{\bb}\ar[r]^-{i_{E^{\bb}}}\ar[dr]_-{f} & T(F)\ar[d]^-{\pi} & F^{\bb}\ar[l]_-{i_{F^{\bb}}}\ar@{=}[dl] \\
& F^{\bb}  }
\end{equation*}
Here $\pi$ is defined degreewise by $\pi_{p}\equiv(f,0,\id)$.
Standard chain complex techniques imply $\htp\Ch^{\bb}\bigl(\Rep(A,B)\bigr)$ satisfies the cylinder axioms 
\cite[\S1.6]{Waldhausen:LNM1126}.
\begin{lemma}
\label{lemma:cylinderaxiom}
If $f\colon E^{\bb}\rightarrow F^{\bb}$ is a chain map, 
then $\pi$ is a chain homotopy equivalence, 
$i_{E^{\bb}}\oplus i_{F^{\bb}}$ is a degreewise split monomorphisms and
$T(0\rightarrow F^{\bb})=F^{\bb}$, $\pi=i_{F^{\bb}}=\id_{F^{\bb}}$.  
\end{lemma}
Moreover, 
the next lemma shows that $\htp\Ch^{\bb}\bigl(\Rep(A,B)\bigr)$ satisfies the extension axiom formulated in 
\cite[\S1.2]{Waldhausen:LNM1126}.
\begin{lemma}
\label{lemma:extensionaxiom}
Suppose 
\begin{equation*}
\xymatrix{
B^{\bb}\ar[r]\ar[d] & E^{\bb}\ar[r]\ar[d] & A^{\bb}\ar[d] \\
\widetilde{B}^{\bb}\ar[r] & \widetilde{E}^{\bb}\ar[r] & \widetilde{A}^{\bb} }
\end{equation*}
is a map of cofibration sequences in $\Ch^{\bb}\bigl(\Rep(A,B)\bigr)$.
If the left and right vertical maps are weak equivalences,
then so is the middle vertical map. 
\end{lemma}
Next we note that $\Ch^{\bb}\bigl(\Rep(A,B)\bigr)$ satisfies the saturation axiom 
\cite[\S1.2]{Waldhausen:LNM1126}.
\begin{lemma}
\label{lemma:saturationaxiom}
If $f$ and $g$ are composable maps in $\Ch^{\bb}\bigl(\Rep(A,B)\bigr)$ and two of the maps 
$f$, $g$ and $fg$ are weak equivalences, 
then the third map is a weak equivalence.
\end{lemma}

Suppose $F$ is a $\CC^{\ast}$-algebra. 
Applying the functor $\Ch^{\bb}\bigl(\Rep(F,-)\bigr)$ to a completely positive split short exact sequence
$0\rightarrow A\rightarrow E\rightarrow B\rightarrow 0$ yields functors
\begin{equation*}
\xymatrix{
\Ch^{\bb}\bigl(\Rep(F,A)\bigr)\ar[r] & 
\Ch^{\bb}\bigl(\Rep(F,E)\bigr)\ar[r] & 
\Ch^{\bb}\bigl(\Rep(F,B)\bigr).}
\end{equation*}
Denote by $\widetilde{\htp}\Ch^{\bb}\bigl(\Rep(F,E)\bigr)$ the category $\Ch^{\bb}\bigl(\Rep(F,E)\bigr)$ 
with cofibrations degreewise split monomorphisms and weak equivalences the chain maps with mapping cones
homotopy equivalent to acyclic complexes in $\Ch^{\bb}\bigl(\Rep(F,B)\bigr)$.
It inherits a cylinder functor from ${\htp}\Ch^{\bb}\bigl(\Rep(F,E)\bigr)$.
\begin{lemma}
\label{lemma:extensionsaturationcylinder}
The category $\widetilde{\htp}\Ch^{\bb}\bigl(\Rep(F,E)\bigr)$ satisfies the extension and saturation 
axioms and acquires a cylinder functor satisfying the cylinder axioms. 
\end{lemma}

Define ${\htp}\Ch^{\bb}\bigl(\Rep(F,E)\bigr)^{\widetilde{\htp}}$ to be the full subcategory of 
${\htp}\Ch^{\bb}\bigl(\Rep(F,E)\bigr)$ whose objects are $E^{\bb}$ such that $0\rightarrow E^{\bb}$
is a weak equivalence in $\widetilde{\htp}\Ch^{\bb}\bigl(\Rep(F,E)\bigr)$.
It acquires the structure of a category with cofibrations and weak equivalences inherited from 
$\htp\Ch^{\bb}\bigl(\Rep(F,E)\bigr)$.
With these definitions there are equivalences
\begin{equation*}
{\htp}\Ch^{\bb}\bigl(\Rep(F,E)\bigr)^{\widetilde{\htp}}\simeq\htp\Ch^{\bb}\bigl(\Rep(F,A)\bigr)
\end{equation*}
and 
\begin{equation*}
\widetilde{\htp}\Ch^{\bb}\bigl(\Rep(F,E)\bigr)\simeq\htp\Ch^{\bb}\bigl(\Rep(F,B)\bigr).
\end{equation*}
Clearly every weak equivalence in $\htp\Ch^{\bb}\bigl(\Rep(F,E)\bigr)$ is also a weak equivalence in 
$\widetilde{\htp}\Ch^{\bb}\bigl(\Rep(F,E)\bigr)$.
Thus by \cite[Theorem 1.6.4]{Waldhausen:LNM1126} there is a homotopy cartesian square:
\begin{equation*}
\label{homotopycartesiansquare}
\xymatrix{
\htp\Ch^{\bb}\bigl(\Rep(F,A)\bigr)\ar[r]\ar[d] & 
\widetilde{\htp}\Ch^{\bb}\bigl(\Rep(F,E)\bigr)^{\widetilde{\htp}}\simeq\ast\ar[d]\\
\htp\Ch^{\bb}\bigl(\Rep(F,E)\bigr)\ar[r] &
\htp\Ch^{\bb}\bigl(\Rep(F,B)\bigr) }
\end{equation*}
This implies $F^{\Rep}$ is flasque. 
Theorem \ref{theorem:unstablerepresentabilityKK} follows now simply by combining the isomorphism
\begin{equation*}
\pi_{n}F^{\Rep}(E)=KK_{n-1}(F,E)
\end{equation*}
for $n\geq 0$ \cite[Theorem 1.2]{Kandelaki:KKRep} and Lemma \ref{lemma:unstablerepresentability}.
\vspace{0.1in}

The results in \cite{Kandelaki:KKRep} employed in the above hold equivariantly.
Thus we may infer:  
\begin{theorem}
\label{theorem:unstablerepresentabilityequivariantKK}
Let $F$ be a $\Group-\CC^{\ast}$-algebra where $\Group$ is compact second countable.
Then the pointed simplicial $\Group-\CC^{\ast}$-space
\begin{equation*}
\xymatrix{
F^{\Rep}\colon\Group-\CC^{\ast}-\Alg\ar[r] & \Delta\Set_{0}}
\end{equation*}
defined by 
\begin{equation*}
F^{\Group-\Rep}(E)\equiv
K\bigl(\Group-\Rep(F,E)\bigr)=
\Omega\vert\,N\htp S_{\bullet}\Group-\Rep(F,E)\,\vert
\end{equation*}
is projective $\Group-\CC^{\ast}$-fibrant.  
For $n\geq 0$ there is a natural isomorphism 
\begin{equation*}
\Group-KK_{n-1}(F,E)=\Group-\HH^{\ast}(E\otimes S^{n},F^{\Group-\Rep}).
\end{equation*}
Here, 
the left hand side denotes the $\Group$-equivariant Kasparov $KK$-groups and the right hand side maps in the 
unstable pointed $G$-equivariant $\CC^{\ast}$-homotopy category.
\end{theorem}
\newpage

\subsection{Base change}
\label{subsection:basechange}
For every $\CC^{\ast}$-algebra $A$ the slice category $\Box\CC^{\ast}-\Spc\downarrow A$ consists of
cubical $\CC^{\ast}$-spaces together with a map to $A$.
Maps in $\Box\CC^{\ast}-\Spc\downarrow A$ are maps in $\Box\CC^{\ast}-\Spc$ which are compatible with 
the given maps to $A$.
We claim $\Box\CC^{\ast}-\Spc\downarrow A$ acquires the exact same four types of model structures as 
$\Box\CC^{\ast}-\Spc$ by defining the relevant homotopical data via the forgetful functor
\begin{equation*}
\xymatrix{
\Box\CC^{\ast}-\Spc\downarrow A\ar[r] & \Box\CC^{\ast}-\Spc.}
\end{equation*}
In the slice category setting the model structures on $\Box\CC^{\ast}-\Spc$ correspond to the trivial 
$\CC^{\ast}$-algebra.
More generally, we have the following result.
\begin{lemma}
\label{lemma:modelstructuresonundercategories}
For any of the pointwise, exact, matrix invariant and homotopy invariant model structures on 
$\Box\CC^{\ast}-\Spc$ the slice category $\Box\CC^{\ast}-\Spc\downarrow \X$ has a corresponding 
combinatorial and weakly finitely generated left proper model structure where a map $f$ is a weak equivalence 
(respectively cofibration, fibration) in $\Box\CC^{\ast}-\Spc\downarrow \X$ if and only if $f$ 
is a weak equivalence (respectively cofibration, fibration) in $\Box\CC^{\ast}-\Spc$.
\end{lemma}
\begin{proof}
The existence of the model structure follows from \cite[Theorem 7.6.5]{Hirschhorn:Modelcategories}.
Since pushouts are formed by taking pushouts of the underlying maps in $\Box\CC^{\ast}-\Spc$, 
it follows that $\Box\CC^{\ast}-\Spc\downarrow \X$ is left proper since $\Box\CC^{\ast}-\Spc$ is so.
With these definitions it is trivial to check that the (acyclic) cofibrations are generated by 
generating (acyclic) cofibrations over $\X$.
\end{proof}
\begin{remark}
\label{remark:pointedmodelstructuresonundercategories}
There is a straightforward analog of Lemma \ref{lemma:modelstructuresonundercategories} for 
pointed cubical and pointed simplicial $\CC^{\ast}$-spaces.
We leave the formulation of Brown representability in this setting to the reader. 
\end{remark}
If $f\colon\X\rightarrow\Y$ is a map between cubical $\CC^{\ast}$-spaces,
there is an induced Quillen pair between the corresponding slice categories:
\begin{equation}
\label{equation:Quillensliceadjunction}
\xymatrix{
f_{!}\colon\Box\CC^{\ast}-\Spc\downarrow\X
\ar@<3pt>[r] & 
\Box\CC^{\ast}-\Spc\downarrow\Y\colon f^{\ast} \ar@<3pt>[l].}
\end{equation}
The left adjoint is defined by $(\Z\rightarrow\X)\mapsto (\Z\rightarrow\X\rightarrow\Y)$ 
and the right adjoint by $(\Z\rightarrow\Y)\mapsto (\Z\times_{\Y}\X\rightarrow\X)$.
When $f$ is a weak equivalence between fibrant objects, 
then the adjunction (\ref{equation:Quillensliceadjunction}) is a Quillen equivalence, 
but without the fibrancy condition this may fail.
\vspace{0.1in}

For objects $\X$ and $\Y$ of $\Box\CC^{\ast}-\Spc$ let $\X\downarrow\Box\CC^{\ast}-\Spc\downarrow\Y$
be the category of objects of $\Box\CC^{\ast}-\Spc$ under $\X$ and over $\Y$ in which an object is a
diagram $\X\rightarrow\Z\rightarrow\Y$ of maps of cubical $\CC^{\ast}$-spaces.
A map from $\X\rightarrow\Z\rightarrow\Y$ to $\X\rightarrow\W\rightarrow\Y$ consists of a map 
$f\colon\Z\rightarrow\W$ such that the obvious diagram commutes.
The next result can be proved using similar arguments as in the proof of Lemma
\ref{lemma:modelstructuresonundercategories} and there are direct analogs for pointed cubical and 
pointed simplicial $\CC^{\ast}$-spaces which we leave implicit.
\begin{lemma}
\label{lemma:modelstructuresonoverundercategories}
For any of the pointwise, exact, matrix invariant and homotopy invariant model structures on 
$\Box\CC^{\ast}-\Spc$ and for every pair of cubical $\Box\CC^{\ast}$-spaces $\X$ and $\Y$ the 
category $\X\downarrow\Box\CC^{\ast}-\Spc\downarrow\Y$ has a corresponding combinatorial and weakly finitely generated 
left proper model structure where $f$ is a weak equivalence (respectively cofibration, fibration) in 
$\X\downarrow\Box\CC^{\ast}-\Spc\downarrow\Y$ if and only if $\Z\rightarrow\W$ is a weak equivalences 
(respectively cofibration, fibration) in $\Box\CC^{\ast}-\Spc$.
\end{lemma}
\vspace{0.1in}

The $K$-theory of a $\CC^{\ast}$-algebra or more generally of a cubical $\Box\CC^{\ast}$-space $\Z$ uses 
the homotopy theory of the retract category $(\Z,\Box\CC^{\ast}-\Spc,\Z)$ with objects triples 
$(\X,i\colon\Z\rightarrow\X,r\colon\X\rightarrow\Z)$ where $ri=\id$ and maps 
$f\colon(\X,i\colon\Z\rightarrow\X,r\colon\X\rightarrow\Z)\rightarrow
(\X,j\colon\Z\rightarrow\Y,s\colon\Y\rightarrow\Z)$ respecting the retractions and sections.
\vspace{0.1in}

We have the following variant of Lemma \ref{lemma:modelstructuresonoverundercategories}.

\begin{lemma}
\label{lemma:modelstructuresonretractivecategories}
For any of the pointwise, exact, matrix invariant and homotopy invariant model structures on 
$\Box\CC^{\ast}-\Spc$ the retract category $(\Z,\Box\CC^{\ast}-\Spc,\Z)$ of a cubical $\Box\CC^{\ast}$-space 
$\Z$ has a corresponding combinatorial and weakly finitely generated left proper cubical model structure where $f$ 
is declared a weak equivalence (respectively cofibration, fibration) in $(\Z,\Box\CC^{\ast}-\Spc,\Z)$ if and only 
if $\X\rightarrow\Y$ is a weak equivalences (respectively cofibration, fibration) in $\Box\CC^{\ast}-\Spc$.
\end{lemma}
\vspace{0.1in}

Since it is perhaps not completely obvious we define the cubical structure of $(\Z,\Box\CC^{\ast}-\Spc,\Z)$.
If $K$ is a cubical set the tensor  
\begin{equation*}
(\X,i\colon\Z\rightarrow\X,r\colon\X\rightarrow\Z)\otimes K
\end{equation*}
is defined as the pushout of the diagram
\begin{equation*}
\xymatrix{
\Z\cong
\Z\otimes\Box^{0} &
\Z\otimes K \ar[l] \ar[r] &
\X\otimes K, }
\end{equation*}
while the cotensor 
\begin{equation*}
(\X,i\colon\Z\rightarrow\X,r\colon\X\rightarrow\Z)^{K}
\end{equation*}
is defined as the pullback of the diagram
\begin{equation*}
\xymatrix{
\Z\cong
\Z^{\Box^{0}} \ar[r] &
\Z^{K}  &
\X^{K}. \ar[l] }
\end{equation*}
The cubical function complex 
\begin{equation*}
{\hom}_{(\Z,\Box\CC^{\ast}-\Spc,\Z)}
\bigl((\X,i\colon\Z\rightarrow\X,r\colon\X\rightarrow\Z),(\X',i'\colon\Z\rightarrow\X',r'\colon\X'\rightarrow\Z)\bigr)
\end{equation*}
of $\X$ and $\X'$ is the subcomplex of ${\hom}_{\Box\CC^{\ast}-\Spc}(\X,\X')$ comprising maps which respect the retraction and section 
\cite[II.~2 Proposition 6]{Quillen:Homotopicalalgebra}.
\begin{remark}
We leave implicit the formulations of the corresponding equivariant results in this section.
Several functoriality questions arise when the groups vary.
\end{remark}

\newpage
\section{Stable $\CC^{\ast}$-homotopy theory}
\label{section:stableCstarhomotopytheory}
Stable homotopy theory in the now baroque formulation of spectra is bootstrapped to represent all 
generalized homology and cohomology theories for topological spaces.
We are interested in an analogous theory for cubical $\CC^{\ast}$-spaces which captures suitably 
defined cohomology and homology theories in one snap maneuver.
The mixing of $\CC^{\ast}$-algebras and cubical sets in $\Box\CC^{\ast}-\Spc$ allows us to vary 
the suspension coordinate in a manner which is out of reach in the more confined settings of 
$\CC^{\ast}-\Alg$ and $\Box\Set$.
Indeed the ``circle'' $C$ we will be using is the tensor product $S^{1}\otimes C_0(\R)$ of the standard 
cubical set model $\Box^{1}/\partial\Box^{1}$ for the topological circle and the $\CC^{\ast}$-algebra 
of complex-valued continuous functions on the real numbers which vanish at infinity.
In the modern formulation of stable homotopy theory the use of symmetric spectra obviate ordinary 
spectra by solving the problem of finding a monoidal model structure which is Quillen equivalent to 
the stable model structure.
To set up the stable $\CC^{\ast}$-homotopy theory we consider symmetric spectra of pointed cubical 
$\CC^{\ast}$-spaces with respect to $C$.
A great deal of the results can be proved by referring to the works of Hovey \cite{Hovey:spectra} and 
Jardine \cite{Jardine:MSS},
a strategy we will follow to a large extend. 
Another valuable viewpoint which offers considerable flexibility and opens up some new subjects to 
explore is to consider model structures on enriched functors from the subcategory 
${\bf fp}\Box\CC^{\ast}-\Spc$ of finitely presentable cubical $\CC^{\ast}$-spaces into 
$\Box\CC^{\ast}-\Spc$. 
This falls into the realms of \cite{DRO:general}.

\subsection{$\CC^{\ast}$-spectra}
\label{subsection:spectra}
We start out by adapting the definition of spectra to our setting.
\begin{definition}
\label{definition:Cstarspectra}
The category $\Spt_{C}$ of cubical $\CC^{\ast}$-spectra consists of sequences 
$\E\equiv(\E_{n})_{n\geq 0}$ of pointed cubical $\CC^{\ast}$-spaces equipped with structure maps 
$\sigma_{n}^{\E}\colon\Sigma_{C}\E_{n}\rightarrow\E_{n+1}$ where $\Sigma_{C}\equiv C\otimes-$ is 
the suspension functor.
A map $f\colon\E\rightarrow\F$ of cubical $\CC^{\ast}$-spectra consists of compatible maps of pointed 
cubical $\CC^{\ast}$-spaces $f_{n}\colon\E_{n}\rightarrow\F_{n}$ in the sense that the diagrams
\begin{equation*}
\xymatrix{
\Sigma_{C}\E_{n}\ar[d]_-{\Sigma_{C}\otimes f_{n}}\ar[r]^-{\sigma_{n}^{\E}} & \E_{n+1}\ar[d]^-{f_{n+1}} \\
\Sigma_{C}\F_{n}\ar[r]^-{\sigma_{n}^{\F}} & \F_{n+1} }
\end{equation*}
commute for all $n\geq 0$.
\end{definition}
What follows is a list of examples of cubical $\CC^{\ast}$-spectra we will be working with.
\begin{example}
The suspension cubical $\CC^{\ast}$-spectrum of a $\CC^{\ast}$-space $\X$ is given by 
\begin{equation*}
\xymatrix{
\Sigma^{\infty}_{C}\X\equiv\{n\ar@{|->}[r] & C^{\otimes n}\otimes\X\}}
\end{equation*}
with structure maps the canonical isomorphisms 
$\Sigma_{C}C^{\otimes n}\otimes\X\rightarrow C^{\otimes (n+1)}\otimes\X$.
The sphere spectrum is the suspension cubical $\CC^{\ast}$-spectrum $\Sigma^{\infty}_{C}\C$ 
of the complex numbers. 
\end{example}
\begin{example}
If $\E$ is a cubical $\CC^{\ast}$-spectrum and $\X$ is a pointed cubical $\CC^{\ast}$-space,
there is a cubical $\CC^{\ast}$-spectrum $\E\wedge\X$ with $n$th level $\E_{n}\otimes\X$ and 
structure maps $\sigma_{n}^{\E}\otimes\X$.
The suspension $\E\wedge C$ of $\E$ is left adjoint to the $C$-loops cubical $\CC^{\ast}$-spectrum 
$\Omega_{C}\E$ of $\E$ defined by setting 
$(\Omega_{C}\E)_{n}\equiv\Omega_{C}(\E_{n})=\underline{\Hom}(C,\E_{n})$
and with structure maps 
$\sigma_{n}^{\Omega_{C}\E}\colon C\otimes\Omega_{C}(\E_{n})\rightarrow\Omega_{C}(\E_{n+1})$
adjoint to the composite 
$C\otimes\Omega_{C}(\E_{n})\otimes C\rightarrow C\otimes\E_{n}\rightarrow\E_{n+1}$.
\end{example}
\begin{example}
The fake suspension $\Sigma_{C}\E$ of $\E$ has $n$th level $C\otimes\E_{n}$ and structure maps 
$\sigma_{n}^{\Sigma_{C}\E}\equiv C\otimes\sigma_{n}^{\E}$.
We note that $\Sigma_{C}\colon\Spt_{C}\rightarrow\Spt_{C}$ is left adjoint to the fake $C$-loops functor 
$\Omega_{C}^{\ell}$ defined by $\Omega_{C}^{\ell}(\E)\equiv\Omega_{C}(\E_{n})$ and with structure maps 
adjoint to the maps 
$\Omega_{C}(\widetilde\sigma_{n}^{\E})\colon\Omega_{C}(\E_{n})\rightarrow\Omega_{C}^{2}(\E_{n+1})$. 
It is important to note that the adjoint of the structure map $\sigma_{n}^{\Omega_{C}\E}$ differs from 
$\Omega_{C}(\widetilde\sigma_{n}^{\E})$ by a twist of loop factors.
In particular, 
the fake $C$-loops functor is not isomorphic to the $C$-loops functor.
\end{example}
\begin{example}
If $\X$ is a pointed cubical $\CC^{\ast}$-space, 
denote by $\hom_{\Spt_{C}}(\X,\E)$ the cubical $\CC^{\ast}$-spectrum 
$\hom_{\Spt_{C}}(K,\E)_{n}\equiv\hom_{\Box\CC^{\ast}-\Spc_{0}}(K,\E_{n})$ 
with structure maps adjoint to the composite maps 
$C\otimes\hom_{\Box\CC^{\ast}-\Spc_{0}}(K,\E_{n})\otimes K\rightarrow C\otimes\E_{n}\rightarrow\E_{n+1}$.
With these definitions there is a natural bijection
\begin{equation*}
\Spt_{C}(\E\wedge K,\F)=
\Spt_{C}\bigl(\E,\hom_{\Spt_{C}}(K,\F)\bigr).
\end{equation*}
The function complex $\hom_{\Spt_{C}}(\E,\F)$ of cubical $\CC^{\ast}$-spectra $\E$ and $\F$ are defined 
in level $n$ as all maps $\E\wedge\Box_{+}^{n}\rightarrow\F$ of cubical $\CC^{\ast}$-spectra. 
\end{example}
\begin{example}
\label{example:Cstarshift}
The $m$th shift $\E[m]$ of a cubical $\CC^{\ast}$-spectrum $\E$ is defined by 
\begin{eqnarray*}
\E[m]_{n}\equiv 
\begin{cases} 
\E_{m+n} & m+n\geq 0 \\ 
\ast & m+n<0.
\end{cases}
\end{eqnarray*}
The structure maps are reindexed accordingly.
\end{example}
\begin{example}
\label{example:layerfiltration}
The layer filtration of $\E$ is obtained from the cubical $\CC^{\ast}$-spectra $L_{m}\E$
defined by 
\begin{eqnarray*}
(L_{m}\E)_{n}\equiv
\begin{cases}
\E_{n} & n\leq m \\
C^{\otimes n-m}\otimes\E_{m} & n>m.
\end{cases}
\end{eqnarray*}
There is a canonical map $\Sigma^{\infty}_{C}\E_{m}[-m]\rightarrow L_{m}\E$ 
and $L_{m+1}\E$ is the pushout of the diagram:
\begin{equation*}
\xymatrix{
\Sigma^{\infty}_{C}(\E_{m+1})[-m-1] &
\Sigma^{\infty}_{C}(C\otimes\E_{m})[-m-1]
\ar[l]\ar[r] &
L_{m}\E }
\end{equation*}
Observe that $\E$ and ${\colim}\,\,L_{m}\E$ are isomorphic.
\end{example}

A map $f\colon\E\rightarrow\F$ is a level weak equivalence (respectively fibration) if 
$f_{n}\colon\E_{n}\rightarrow\F_{n}$ is a $\CC^{\ast}$-weak equivalence 
(respectively projective $\CC^{\ast}$-fibration).
And $f$ is a projective cofibration if $f_{0}$ and the maps
\begin{equation*}
\xymatrix{
\E_{n+1}\coprod_{\Sigma_{C}\E_{n}}\Sigma_{C}\F_{n}\ar[r] & \F_{n+1} } 
\end{equation*}
are projective cofibrations for all $n\geq 0$.
By the results in \cite[\S1]{Hovey:spectra} we have:
\begin{proposition}
\label{proposition:strictmodelstructures}
The level weak equivalences, projective cofibrations and level fibrations furnish a combinatorial, 
cubical and left proper Quillen equivalent model structure on $\Spt_{C}$.
\end{proposition}

Our next objective is to define the stable model structure as a Bousfield localization of the 
level model structure. 
The fibrant objects in the localized model structure have been apprehended as $\Omega$-spectra 
since the days of yore.
In our setting this amounts to the following definition.
\begin{definition}
\label{definition:Cstarstablefibrant}
A cubical $\CC^{\ast}$-spectrum $\Z$ is stably fibrant if it is level fibrant and all the adjoints 
$\widetilde\sigma_{n}^{\Z}\colon\Z_{n}\rightarrow\underline{\Hom}(C,\Z_{n+1})$ of its structure maps 
are $\CC^{\ast}$-weak equivalences.
\end{definition}

The stably fibrant cubical $\CC^{\ast}$-spectra determine the stable weak equivalences of cubical 
$\CC^{\ast}$-spectra.
Stable fibrations are maps having the right lifting property with respect to all maps which are 
projective cofibrations and stable weak equivalences.
\begin{definition}
\label{definition:Cstarstableweakequivalence}
A map $f\colon\E\rightarrow\F$ of cubical $\CC^{\ast}$-spectra is a stable weak equivalence if for 
every stably fibrant $\Z$ taking a cofibrant replacement $\QQ f\colon\QQ\E\rightarrow\QQ\F$ of $f$
in the level model structure on $\Spt_{C}$ yields a weak equivalence of pointed cubical sets
\begin{equation*}
\xymatrix{
\hom_{\Spt_{C}}(\QQ f,\Z)\colon
\hom_{\Spt_{C}}(\QQ\F,\Z)\ar[r] & \hom_{\Spt_{C}}(\QQ\E,\Z).}
\end{equation*}
\end{definition}
\begin{example}
\label{example:stableweakequivalence}
The map $\Sigma^{\infty}_{C}\E_{m}[-m]\rightarrow L_{m}\E$ mentioned in 
Example \ref{example:layerfiltration} is a stable weak equivalence.
\end{example}
By specializing the collection of results in \cite[\S3]{Hovey:spectra} to our setting we have:
\begin{theorem}
\label{theorem:Cstarstablemodelstructure}
The classes of stable weak equivalences and projective cofibrations define a combinatorial, cubical 
and left proper model structure on $\Spt_{C}$.
Let $\SH^{\ast}$ denote the associated stable homotopy category of $\CC^{\ast}$-algebras. 
\end{theorem}

Denote by $\iota_{\E}\colon\E\rightarrow\Theta\E\equiv\Omega_{C}^{\ell}\E[1]$ the natural map where 
$(\Theta\E)_{n}\equiv\underline{\Hom}(C,\E_{n+1})$ and
$\sigma_{n}^{\Theta\E}\equiv\underline{\Hom}(C,\widetilde\sigma_{n}^{E})$.
The stabilization $\Theta^{\infty}\E$ of $\E$ is the colimit of the diagram:
\begin{equation*}
\xymatrix{
\E\ar[r]^-{\iota_{\E}} & 
\Theta\E \ar[r]^-{\Theta\iota_{\E}} &
\Theta^{2}\E\ar[r]^-{\Theta^{2}\iota_{\E}} & \cdots}
\end{equation*}
By \cite[Proposition 4.6]{Hovey:spectra} we get the following result because $\underline{\Hom}(C,-)$ 
preserves sequential colimits according to Example \ref{example:internalhomfinitelypresentable} 
and the projective homotopy invariant model structure on $\Box\CC^{\ast}-\Spc_{0}$ is almost 
finitely generated.
\begin{lemma}
\label{lemma:stabilizationstablefibrant}
The stabilization of every level fibrant cubical $\CC^{\ast}$-spectrum is stably fibrant.
\end{lemma}

Now let $f_{\E}\colon\E\rightarrow\RRR\E$ be a natural level fibrant replacement for $\E$ in $\Spt_{C}$, 
meaning that $f_{\E}$ a level weak equivalence and projective cofibration and $\RRR\E$ is level fibrant.  
Lemma \ref{lemma:stabilizationstablefibrant} motivates the next definition.
\begin{definition}
\label{definition:stablefibrantreplacement}
Let $\widetilde\iota_{\E}\colon\E\rightarrow\Theta^{\infty}\RRR\E$ be the composite of $f_{\E}$ and 
$\iota_{\RRR\E}^{\infty}\colon\RRR\E\rightarrow\Theta^{\infty}\RRR\E$.
\end{definition}

We have the following convenient characterization of stable weak equivalences given by 
\cite[Theorem 4.12]{Hovey:spectra} and a corollary which shows that certain stable maps 
can be approximated by unstable maps \cite[Corollary 4.13]{Hovey:spectra}.
\begin{theorem}
\label{theorem:stableweakequivalencesandstabilization}
A map $f\colon\E\rightarrow\F$ is a stable weak equivalence if and only if the induced map
$\widetilde\iota_{\E}(f)\colon\Theta^{\infty}\RRR\E\rightarrow\Theta^{\infty}\RRR\F$ is a 
level weak equivalence.
\end{theorem}
\begin{corollary}
\label{corollary:finitelypresentablecofibrantcolimit}
If $\X$ is a finitely presentable cofibrant cubical $\CC^{\ast}$-space and $\F$ is level fibrant, 
then there is a canonical isomorphism
\begin{equation}
\label{equation:SHcolimitH}
\SH^{\ast}(\Sigma^{\infty}_{C}\X,\F)=
\underset{n}{\colim}\,\,\HH^{\ast}(\X,\Omega_{C}^{n}\F_{n}).
\end{equation}
In addition, 
if $\F$ is stably fibrant so that all of the transition maps in the directed system in (\ref{equation:SHcolimitH}) 
are isomorphisms, 
then there is a canonical isomorphism
\begin{equation*}
\SH^{\ast}(\Sigma^{\infty}_{C}\X,\F)=
\HH^{\ast}(\X,\F_{0}).
\end{equation*}
\end{corollary}

Next we characterize an important class of stable fibrations.
\begin{lemma}
\label{lemma:Cstarstablecharacterization}
A level fibration $f\colon\E\rightarrow\F$ with a stably fibrant target is a stable fibration if and only if
the diagram  
\begin{equation*}
\xymatrix{
\E \ar[r]\ar[d]_-{f} & \Theta^{\infty}\RRR\E\ar[d]^-{\widetilde\iota_{\E}(f)}\\
\F \ar[r] & \Theta^{\infty}\RRR\F}
\end{equation*}
is a homotopy pullback in the levelwise model structure on $\CC^{\ast}-\Spc_{C}$.
\end{lemma}
\begin{proof}
See the proof of Corollary \ref{characterization:exactprojectivefibration}.
\end{proof}
\begin{lemma}
\label{lemma:loopspreservesstableweakequivalencesoffibrantobjects}
The loop functor $\E\mapsto\Omega_{C}\E$ preserves stable weak equivalences between 
level fibrant objects.
\end{lemma}
\begin{proof}
If $\E$ is level fibrant there is an isomorphism 
$\Theta^{\infty}(\Omega_{C}\E)_{n}=\Omega_{C}(\Theta^{\infty}\E)_{n}$.
\end{proof}

Next we seek an interpretation of stable weak equivalences in terms of 
bigraded stable homotopy groups $\pi_{p,q}$ for integers $p,q\in\ZZ$. 
Suppose $\E$ is level fibrant and consider the sequential diagram:
\begin{equation*}
\xymatrix{
\cdots\ar[r] &
\E_{n}\ar[r]^-{\widetilde\sigma_{C}} & \Omega_{C}\E_{n+1}\ar[r]^-{\Omega_{C}\widetilde\sigma_{C}} & 
\Omega_{C}^{2}\E_{n+2}\ar[r]^-{\Omega_{C}^{2}\widetilde\sigma_{C}} & \cdots}
\end{equation*}
In $A$-sections, 
the homotopy group $\pi_{p}\Theta^{\infty}\E_{n}(A)$ is isomorphic to the colimit of the sequential 
diagram:
\begin{equation*}
\xymatrix{
\cdots\ar[r] &
\pi_{p}(\E_{n})(A)\ar[r]^-{\pi_{p}(\widetilde\sigma_{C})(A)} & 
\pi_{p}(\Omega_{C}\E_{n+1})(A)\ar[r]^-{\pi_{p}(\Omega_{C}\widetilde\sigma_{C})(A)} & 
\pi_{p}(\Omega_{C}^{2}\E_{n+2})(A)\ar[r] \ar[r]^-{\pi_{p}(\Omega_{C}^{2}\widetilde\sigma_{C})(A)} & \cdots}
\end{equation*}
Passing to the homotopy category associated with the stable model structure on $\Spt_{C}\downarrow A$, 
we can recast the latter diagram as:
\begin{equation*}
\xymatrix{
\cdots\ar[r] &
[S^p,\E_{n}\vert A]\ar[r] & [S^p\otimes C,\E_{n+1}\vert A]\ar[r] & 
[S^p\otimes C^{\otimes 2},\E_{n+2}\vert A]\ar[r] &\cdots}
\end{equation*}
In the stable homotopy category $\SH^{\ast}_{A}$ of pointed cubical $\CC^{\ast}$-spaces over $A$, 
where there is no need to impose fibrancy, 
one obtains from the definition $C=S^{1}\otimes C_{0}(\R)$ an alternate description of the homotopy 
groups as the colimit of the sequential diagram:
\begin{equation*}
\xymatrix{
\cdots\ar[r] &
[S^{p},\E_{n}\vert A]\ar[r] & [S^{p+1}\otimes C_{0}(\R),\E_{n+1}\vert A]\ar[r] & 
[S^{p+2}\otimes C_{0}(\R^{2}),\E_{n+2}\vert A]\ar[r] &\cdots}
\end{equation*}
\begin{definition}
Let $\E$ be a cubical $\CC^{\ast}$-spectrum.
The degree $p$ and weight $q$ stable homotopy group $\pi_{p,q}\E$ is defined in $A$-sections by
\begin{equation*}
\xymatrix{
\pi_{p,q}\E(A)\equiv
{\colim}\,\,
\bigl(
[S^{p+n}\otimes C_{0}(\R^{q+n}),\E_{n}\vert A]\ar[r] & 
[S^{p+n+1}\otimes C_{0}(\R^{q+n+1}),\E_{n+1}\vert A]\ar[r] & \cdots\bigr). }
\end{equation*}
\end{definition}
In $A$-sections there are natural isomorphisms
\begin{equation}
\label{pq--p-q}
\pi_{p,q}\E(A)=\pi_{p-q}\Theta^{\infty}\RRR\E_{-q}(A).
\end{equation}
\begin{lemma}
Define $\Omega_{S^{1}}(-)\equiv\underline{\Hom}(S^{1},-)$ and 
$\Omega_{C_{0}(\R)}(-)\equiv\underline{\Hom}(C_{0}(\R),-)$.
For every cubical $\CC^{\ast}$-spectrum $\E$ there are isomorphisms
\begin{eqnarray*}
\pi_{p,q}\E= 
\begin{cases}
\pi_0\Omega^{p-q}_{S^{1}}\Theta^{\infty}(\RRR\E[-q])_0 & p\geq q \\
\pi_0\Omega^{q-p}_{C_{0}(\R)}\Theta^{\infty}(\RRR\E[-p])_0 & p\leq q.
\end{cases}
\end{eqnarray*}
\end{lemma}
\begin{proof}
There are isomorphisms
\begin{align*}
{\colim}\,\,[S^{p+n}\otimes C_{0}(\R^{q+n}),\E_{n}\vert A]
& =
{\colim}\,\,[S^{p-q+n}\otimes C_{0}(\R^{n}),\E[-q]_{n}\vert A] \\
& = 
{\colim}\,\,[S^{n}\otimes C_{0}(\R^{n}),\Omega^{p-q}_{S^{1}}\RRR\E[-q]_{n}\vert A]
\end{align*}
if $p\geq q$, 
and 
\begin{align*}
{\colim}\,\,[S^{p+n}\otimes C_{0}(\R^{q+n}),\E_{n}\vert A]
& = 
{\colim}\,\,[S^{n}\otimes C_{0}(\R^{q-p+n}),\E[-p]_{n}\vert A] \\
& = 
{\colim}\,\,[S^{n}\otimes C_{0}(\R^{n}),\Omega^{q-p}_{C_{0}(\R)}\RRR\E[-p]_{n}\vert A]
\end{align*}
if $q\geq p$.
\end{proof}

We are ready to give an algebraic characterization of stable weak equivalences.
\begin{lemma}
\label{lemma:stableweakequivalence-homotopy}
The map $\E\rightarrow\F$ is a stable weak equivalence if and only if there is an induced isomorphism of 
bigraded stable homotopy groups
\begin{equation*}
\pi_{p,q}\E=\pi_{p,q}\F.
\end{equation*}
\end{lemma}
\begin{proof}
A stable equivalence between $\E$ and $\F$ induces for every integer $m\in\ZZ$ a level weak equivalence
$\Theta^{\infty}\RRR\E[m]\rightarrow\Theta^{\infty}\RRR\F[m]$.
Hence, 
in all sections, 
the induced maps between the bigraded stable homotopy groups of $\E$ and $\F$ are isomorphisms 
(\ref{pq--p-q}).
Conversely, 
if $\pi_{p,q}\E\rightarrow\pi_{p,q}\F$ is an isomorphism of presheaves for $p\geq q\leq 0$, 
then again by (\ref{pq--p-q}) there is a level weak equivalence 
$\Theta^{\infty}\RRR\E\rightarrow\Theta^{\infty}\RRR\F$.
\end{proof}

The usual approach verifies that stable homotopy groups preserve colimits in the following sense. 
\begin{lemma}
Let $\{\E^{i}\}_{i\in I}$ be a filtered diagram of cubical $\CC^{\ast}$-spectra.
Then the natural map 
\begin{equation*}
\xymatrix{
\colim_{i\in I}\, \pi_{p,q}\E^{i}\ar[r] &
\pi_{p,q}(\hocolim_{i\in I}\, \E^{i}) }
\end{equation*}
is an isomorphism for all $p,q\in\ZZ$.
\end{lemma}
\begin{proof}
Without loss of generality we may assume the diagram consists of stably fibrant $\CC^{\ast}$-spectra.
Then the colimit $\E$ of the diagram - where $\E_{n}=\colim_{i\in I}\E^{i}_{n}$ - is stably fibrant:
Each $\E_{n}$ is a filtered colimit of $\CC^{\ast}$-projective fibrant pointed cubical $\CC^{\ast}$-spaces. 
Lemma \ref{lemma:pointedprojectiveinternalhomprojectiveCstarfibrant} shows that $\underline{\Hom}(C,\E^{i}_{n+1})$ 
is $\CC^{\ast}$-projective fibrant.
It follows that the map $\E^{i}_{n}\rightarrow\underline{\Hom}(C,\E^{i}_{n+1})$ is a pointwise weak equivalence.
Using that $\underline{\Hom}(C,-)$ commutes with filtered colimits we conclude that the adjoints of the structure maps 
\begin{equation*}
\xymatrix{
\E_{n}=\colim_{i\in I}\E^{i}_{n}\ar[r] &
\colim_{i\in I}\underline{\Hom}(C,\E^{i}_{n+1})=\Omega_{C}\E_{n+1} }
\end{equation*}
are pointwise weak equivalence, and thus $\CC^{\ast}$-weak equivalences.
\\ \indent
There is a natural isomorphism between  
\begin{equation*}
\colim_{i\in I}\, \pi_{0}{\hom}_{\Box\CC^{\ast}-\Spc_{0}}\bigl(S^{p}\otimes C_{0}(\R^{p-q}),\E^{i}\bigr)
\end{equation*}
and
\begin{equation*}
\pi_{0}{\hom}_{\Box\CC^{\ast}-\Spc_{0}}\bigl(S^{p}\otimes C_{0}(\R^{p-q}),\colim_{i\in I}\, \E^{i}\bigr) 
\end{equation*}
since the functor 
\begin{equation*}
{\hom}_{\Box\CC^{\ast}-\Spc_{0}}\bigl(S^{p}\otimes C_{0}(\R^{p-q}),-\bigr) 
\end{equation*}
commutes with filtered colimits.
Now suppose there is a stable weak equivalence $\E^{i}\rightarrow\F^{i}$ between stably fibrant $\CC^{\ast}$-spectra for every $i\in I$.
Then levelwise there are pointwise weak equivalences $\E^{i}_{n}\rightarrow\F^{i}_{n}$ so that the map 
\begin{equation*}
\xymatrix{
\colim_{i\in I}\, \E^{i}_{n}\ar[r] &
\colim_{i\in I}\, \F^{i}_{n} }
\end{equation*}
is a pointwise weak equivalence, 
which implies $\E\rightarrow\F$ is a stable weak equivalence.
Hence the homotopy colimit of $\{\E^{i}\}_{i\in I}$ maps by a natural stable weak equivalence to $\E$ and we are done.
\end{proof}

The next lemma dealing with the cyclic permutation condition on the circle $C$ is a key input in stable $\CC^{\ast}$-homotopy theory.
We shall refer to this lemma when comparing $\CC^{\ast}$-spectra with $\CC^{\ast}$-symmetric spectra.
It ensures that the stable $\CC^{\ast}$-homotopy category inherits a symmetric monoidal product.
\begin{lemma}
\label{lemma:cyclicpermutationcondition}
The circle $C$ satisfies the cyclic permutation condition. 
That is, 
there exists a homotopy $C\otimes C\otimes C\otimes C(I)\rightarrow C\otimes C\otimes C$
from the cyclic permutation to the identity on $C^{\otimes 3}$. 
\end{lemma}
\begin{proof}
In $\Box\CC^{\ast}-\Spc_{0}$ there is an isomorphism $C^{\otimes 3}=S^3\otimes C_{0}(\R^3)$.
Clearly the cyclic permutation condition holds for the cubical $3$-sphere since
$(321)\colon S^3\rightarrow S^3$, 
$x\otimes y\otimes z\mapsto y\otimes z\otimes x$,  
has degree one. 
Using the isomorphism $C_{0}(\R^3)\otimes C(I)=C_{0}(\R^3\times I)$ the claim for $C_{0}(\R^3)$ 
follows by defining a homotopy $C_{0}(\R^3\times I)\rightarrow C_{0}(\R^3)$ in terms of the matrix: 
\[ \left( \begin{array}{ccc}
t & 1-t & 0 \\
0 & t & 1-t \\
1-t & 0 & t \end{array} \right)\] 
These two observations imply that $C$ satisfies the cyclic permutation condition.
\end{proof}
\begin{lemma}
\label{lemma:compositionstableweakequivalence}
For every cubical $\CC^{\ast}$-space $\X$ there is a stable weak equivalence
\begin{equation*}
\xymatrix{
\Sigma^{\infty}_{C}\X\ar[r] & 
\Omega_{C}(\Sigma^{\infty}_{C}\X\wedge C)\ar[r] &
\Omega_{C}\RRR(\Sigma^{\infty}_{C}\X\wedge C). }
\end{equation*}
\end{lemma}
\begin{proof}
Using Lemma \ref{lemma:cyclicpermutationcondition} this follows as in the proof of
\cite[Lemma 3.14]{Jardine:MSS}.
\end{proof}
\begin{theorem}
\label{theorem:compositionstableweakequivalence}
For every cubical $\CC^{\ast}$-spectrum $\E$ there is a stable weak equivalence
\begin{equation*}
\xymatrix{
\E\ar[r] & 
\Omega_{C}(\E\wedge C)\ar[r] &
\Omega_{C}\RRR(\E\wedge C). }
\end{equation*}
\end{theorem}
\begin{proof}
The idea is to reduce the proof to Lemma \ref{lemma:compositionstableweakequivalence}
using the layer filtration.
Indeed, since the shift functor preserves stable weak equivalences,
Lemma \ref{lemma:compositionstableweakequivalence} implies there are stable weak equivalences
\begin{equation*}
\xymatrix{
\Sigma^{\infty}_{C}\E_{n}[-n]\ar[r] & 
\Omega_{C}\RRR(\Sigma^{\infty}_{C}\E_{n}[-n]\wedge C). }
\end{equation*}
Clearly the functors $-\wedge C$ and $\Omega_{C}\RRR(-\wedge C)$ preserve stable weak equivalences 
and there is a stable weak equivalence $\Sigma^{\infty}_{C}\E_{n}[-n]\rightarrow L_{n}\E$.
To conclude the map in question is a stable weak equivalence we refer to the next lemma,
which follows using the arguments in \cite[Lemma 3.12, pp.~498-499]{Jardine:MSS}.
\end{proof}
\begin{lemma}
\label{lemma:concludecompositionstableweakequivalence}
If $\E_{0}\rightarrow\E_{1}\rightarrow\cdots$ is a sequential diagram of cubical $\CC^{\ast}$-spectra
such that $\E_{i}\rightarrow\Omega_{C}\RRR(\E_{i}\wedge C)$ is a stable weak equivalence,
then so is 
${\colim}\,\,\E_{i}\rightarrow\Omega_{C}\RRR\bigl(({\colim}\,\,\E_{i})\wedge C\bigr)$.
\end{lemma}
We include the next results for completeness.
The proofs can be patched from the arguments in \cite[\S3.4]{Jardine:MSS}. 
\begin{corollary}
\label{corollary:patched}
Let $\E$ and $\F$ be cubical $\CC^{\ast}$-spectra.
\begin{itemize}
\item 
If $\E$ is level fibrant the evaluation map $\Omega_{C}\E\wedge C\rightarrow\E$
is a stable weak equivalence.
\item
A map $\E\wedge C\rightarrow\F$ is a stable weak equivalence if and only if
$\E\rightarrow\Omega_{C}\F\rightarrow\Omega_{C}\RRR\F$ is a stable weak equivalence.
\item
A map $\E\rightarrow\F$ is a stable weak equivalence if and only if
$\E\wedge C\rightarrow\F\wedge C$ is so.
\item
There are natural stable weak equivalences between $\Sigma_{C}\E$, $\E[1]$ and $\E\wedge C$.
\item
If $\E$ is level fibrant there are natural stable weak equivalences between 
$\Omega_{C}^{\ell}\E$, $\E[-1]$ and $\Omega_{C}\E$.
\end{itemize}
\end{corollary}
\vspace{0.1in}

As for the circle $C$,
we may define categories of $S^{1}$-spectra $\Spt_{S^{1}}$ and $C_{0}(\R)$-spectra 
$\Spt_{C_{0}(\R)}$ of pointed cubical $\CC^{\ast}$-spaces and construct level and 
stable model structures.  
In particular,
a map $\E\rightarrow\F$ of cubical $S^{1}$-spectra is a stable weak equivalence
if and only if $\Theta^{\infty}\RRR\E\rightarrow\Theta^{\infty}\RRR\F$ is a 
pointwise level weak equivalence, 
i.e.~all the induced maps 
$\pi_{p}\Theta^{\infty}\RRR\E(A)\rightarrow\pi_{p}\Theta^{\infty}\RRR\F(A)$ 
of homotopy groups are isomorphisms.
Here $\Theta^{\infty}$ is defined as above using $\underline{\Hom}(S^{1},-)$ and $\RRR$ 
is a fibrant replacement functor in the level model structure on $\Spt_{S^{1}}$.
The group $\pi_{p}\Theta^{\infty}\RRR\E(A)$ is isomorphic to the colimit of the sequential diagram
\begin{equation}
\label{stablehomotopygroupspresheavesinAsection}
\xymatrix{
\cdots\ar[r] &
[S^p,\E_{n}\vert A]\ar[r] & 
[S^{p+1},\E_{n+1}\vert A]\ar[r] & 
[S^{p+2},\E_{n+2}\vert A]\ar[r] &\cdots}
\end{equation}
in the homotopy category $\HH(\Box\CC^{\ast}-\Spc_{0}\downarrow A)$.
Define $\pi_{p}\E$ in $A$-sections to be the colimit of (\ref{stablehomotopygroupspresheavesinAsection}).
Thus $\E\rightarrow\F$ is a stable weak equivalence if and only if there are induced isomorphisms
$\pi_{p}\E=\pi_{p}\F$ for every integer $p$.
Every level fiber sequence $\F\rightarrow\E\rightarrow\E'$ can be functorially replaced up to level
weak equivalence by a fiber sequence of level fibrant cubical $S^{1}$-spectra, so that in $A$-sections 
we get a level fiber sequence $\RRR\F(A)\rightarrow\RRR\E(A)\rightarrow\RRR\E'(A)$.  
This implies there is a long exact sequence: 
\begin{equation}
\label{longexactsequence}
\xymatrix{
\cdots\ar[r] &
\pi_{p+1}\E'\ar[r] &
\pi_{p}\F\ar[r] &
\pi_{p}\E\ar[r] &
\pi_{p}\E'\ar[r] &
\cdots }
\end{equation}
Applying the proof of \cite[Corollary 3.6]{Jardine:MSS} to our setting we get the next result.
\begin{lemma}
\label{lemma:cofibersequenceinducelongexactsequence}
Every cofiber sequence $\E\rightarrow\E'\rightarrow\E'/\E$ induces a natural long exact sequence:
\begin{equation*}
\xymatrix{
\cdots\ar[r] &
\pi_{p+1}\E'/\E\ar[r] &
\pi_{p}\E\ar[r] &
\pi_{p}\E'\ar[r] &
\pi_{p}\E'/\E\ar[r] &
\cdots }
\end{equation*}
\end{lemma}
We end this section with a useful result which reduces questions about $\SH^{\ast}$ to properties 
of a set of well-behaved compact generators.
It is the stable analog of Theorem \ref{theorem:unstableweakgenerators}.
For completeness we indicate the proof following \cite[Section~7.4]{Hovey:Modelcategories}.
\begin{theorem}
\label{theorem:stableweakgenerators}
The cofibers of the generating projective cofibrations 
\begin{equation*}
\Fr_{m}\bigl(A\otimes(\partial\Box^{n}\subset\Box^{n})_{+}\bigr)
\end{equation*}
form a set of compact generators for the stable homotopy category $\SH^{\ast}$.
\end{theorem}
\begin{proof}
By \cite[Section~7.3]{Hovey:Modelcategories} the cofibers form a set of weak generators.
It remains to show that the cofibrant and finitely presentable cofibers are compact in $\SH^{\ast}$.
\\ \indent
Suppose $\E$ is cofibrant and finitely presentable in $\Spt_{C}$. 
Let $\lambda$ be an ordinal and denote by $\lambda^{\fin}$ its set of finite subsets.
Now for every $\lambda$-indexed collection of cubical $\CC^{\ast}$-spectra we need to show that the canonical map 
\begin{equation}
\label{equation:generators}
\xymatrix{
\colim_{\alpha^{\fin}} \SH^{\ast}(\E,\coprod_{\alpha\in\alpha^{\fin}}\F_{\alpha})\ar[r] &
\SH^{\ast}(\E,\coprod_{\alpha<\lambda}\F_{\alpha}) }
\end{equation}
is an isomorphism.
Injectivity of the map (\ref{equation:generators}) holds because the inclusion of every finite subcoproduct has a retraction.
To prove surjectivity we use transfinite induction.
The subcategory of finite subsets of $\lambda+1$ containing $\lambda$ is final in the category of finite subsets 
of $\lambda+1$. 
Thus the successor ordinal case holds and we may assume $\lambda$ is a limit ordinal such that (\ref{equation:generators})
is surjective for every $\beta<\lambda$.
Without loss of generality we may assume $\F_{\alpha}$ is bifibrant and hence that $\coprod_{\alpha<\lambda}\F_{\alpha}$
is cofibrant.
Since a filtered colimit of stably fibrant cubical $\CC^{\ast}$-spectra is stably fibrant applying a fibrant replacement 
functor $\RRR$ yields a weak equivalence
\begin{equation*}
\xymatrix{
\coprod_{\alpha<\lambda}\F_{\alpha}=
\colim_{\beta<\lambda} \coprod_{\alpha<\beta}\F_{\alpha}\ar[r] &
\colim_{\beta<\lambda} \RRR\coprod_{\alpha<\beta}\F_{\alpha}. } 
\end{equation*}
Thus by finite presentability of $\E$ every map to $\coprod_{\alpha<\lambda}\F_{\alpha}$ factors through 
$\RRR\coprod_{\alpha<\beta}\F_{\alpha}$ for some ordinal $\beta<\lambda$ for which surjectivity of 
(\ref{equation:generators}) holds by the transfinite induction assumption.
\end{proof}

On account of Corollary \ref{homotopyweaklyfinitelygenerated} it follows that the stable model structure on $\Spt_{C}$ is 
weakly finitely generated.
By applying Lemma \ref{lemma:afgmodelcategories} we deduce the following useful result.
\begin{lemma}
\label{lemma:spectraweaklyfinitelygenerated}
In the stable model structure on $\Spt_{C}$ the classes of
\begin{itemize}
\item 
acyclic fibrations
\item
fibrations with fibrant codomains
\item
fibrant objects
\item
stable weak equivalences
\end{itemize}
are closed under filtered colimits.
\end{lemma}
\newpage

\subsection{Bispectra}
\label{subsection:bispectra}
\begin{definition}
\label{definition:Cstarbispectra}
Let $m,n\geq 0$ be integers.
A cubical $\CC^{\ast}$-bispectrum $\E$ consists of pointed cubical $\CC^{\ast}$-spaces $\E_{m,n}$ 
together with structure maps 
\begin{equation*}
\xymatrix{
\sigma_{h}\colon S^{1}\otimes\E_{m,n}\ar[r] & \E_{m+1,n}}
\end{equation*} 
and
\begin{equation*}
\xymatrix{
\sigma_{v}\colon C_{0}(\R)\otimes\E_{m,n}\ar[r] & \E_{m,n+1}.} 
\end{equation*}
In addition, 
the structure maps are required to be compatible in the sense that the following diagrams commute:
\begin{equation*}
\xymatrix{
S^{1}\otimes C_{0}(\R)\otimes\E_{m,n} \ar[rr]^-{\tau\otimes\E_{m,n}}
\ar[d]_-{S^{1}\otimes\sigma_{v}} & & C_{0}(\R)\otimes S^{1}\otimes\E_{m,n}
\ar[d]^-{C_{0}(\R)\otimes\sigma_{h}} \\ 
S^{1}\otimes\E_{m,n+1}\ar[r]^-{\sigma_{h}} & \E_{m+1,n+1} & 
C_{0}(\R)\otimes\E_{m+1,n}\ar[l]_-{\sigma_{v}} }
\end{equation*} 
Here, $\tau$ flips the copies of $S^{1}$ and $C_{0}(\R)$. 
Let $\Spt_{S^{1},C_{0}(\R)}$ denote the category of cubical $\CC^{\ast}$-bispectra.
\end{definition}

As we hope to illustrate in this section,
the separation of suspension coordinates underlying the notion of a cubical $\CC^{\ast}$-bispectrum is a helpful tool. 

A cubical $\CC^{\ast}$-bispectrum can and will be interpreted as a cubical $C_{0}(\R)$-spectrum 
object in the category of cubical $S^{1}$-spectra;
that is, as a collection of cubical $S^{1}$-spectra $\E_{n}\equiv\E_{\ast,n}$ together with maps 
of cubical $S^{1}$-spectra induced by the structure maps $C_{0}(\R)\otimes\E_{n}\rightarrow\E_{n+1}$.
If $\X$ is a pointed cubical $\CC^{\ast}$-space, 
we let $\Sigma^{\infty}_{S^{1},C_{0}(\R)}\X$ denote the corresponding suspension cubical 
$\CC^{\ast}$-bispectrum.
\begin{proposition}
\label{proposition:homofsuspensionofCstaralgebras}
For $\X\in\CC^{\ast}-\Alg$ and $\E\in\Spt_{S^{1},C_{0}(\R)}$ there is an isomorphism
\begin{equation*}
\SH(\Sigma^{\infty}_{S^{1},C_{0}(\R)}\X,\E)=
\underset{n}{\colim}\,\,\SH^{\ast}_{S^{1}}\bigl(C_{0}(\R^{n})\wedge\Sigma^{\infty}_{S^{1}}\X,\E_{n}\bigr).
\end{equation*}
\end{proposition}
\begin{proof}
Let $\E_{n}\equiv(\E_{0,n},\E_{1,n},\cdots)$ by the cubical $S^{1}$-spectrum corresponding to $\E$.
There exists a fibrant replacement $\E^{f}$ of $\E$ in $\Spt_{S^{1},C_{0}(\R)}$ and isomorphisms
\begin{equation*}
\SH(\Sigma^{\infty}_{S^{1},C_{0}(\R)}\X,\E) 
=\Spt_{s,t}(\Sigma^{\infty}_{S^{1},C_{0}(\R)}\X,\E^{f})/\simeq \\ 
=\Spt_{S^{1}}(\Sigma^{\infty}_{S^{1}}\X,\E^{f}_{0})/\simeq.
\end{equation*}
The relation $\simeq$ is the homotopy relation on maps. 
In our setting, homotopies are parametrized by $C_{0}(\Box^{1}_{\ttop})$. 
We may choose a fibrant replacement $\E^{f}$ so that
\begin{equation*}
\E^{f}_{0}=\underset{n}{\colim}\,\,
\bigl((\E_{0})^{f} \xymatrix{\ar[r] &} \Omega_{C_{0}(\R)}\bigl((\E_{1})^{f}\bigr) 
\xymatrix{\ar[r] &} \Omega_{C_{0}(\R)}^{2}\bigl((\E_{2})^{f}\bigr) 
\xymatrix{\ar[r] &} \cdots\bigr).
\end{equation*}
Here, 
$(\E_{n})^{f}$ is a fibrant replacement of $\E_{n}$ in $\Spt_{S^{1}}$, 
and $\Omega_{C_{0}(\R)}$ is the right adjoint of the functor 
$C_{0}(\R)\otimes -\colon\Spt_{S^{1}}\rightarrow\Spt_{S^{1}}$.
Since the $S^{1}$-suspension spectra 
$\Sigma^{\infty}_{S^{1}}\X$ and $\Sigma^{\infty}_{S^{1}}\X\otimes C_{0}(\Box^{1}_{\ttop})$ 
are finitely presentable objects in $\Spt_{S^{1}}$, 
we get an isomorphism
\begin{equation*}
\Spt_{S^{1}}(\Sigma^{\infty}_{S^{1}}\X,\E^{f}_{0})/\simeq 
\,\, 
=
\,\,
\underset{n}{\colim}\,\,
\Spt_{S^{1}}\bigl(\Sigma^{\infty}_{S^{1}}\X,\Omega_{C_{0}(\R)}^{n}(\E_{n})^{f}\bigr)/\simeq. 
\end{equation*}
The latter combined with the isomorphism 
\begin{equation*}
\SH^{\ast}_{S^{1}}\bigl(\Sigma^{\infty}_{S^{1}}\X,\Omega_{C_{0}(\R)}^{n}(\E_{n})\bigr) 
= 
\SH^{\ast}_{S^{1}}\bigl(C_{0}(\R^{n})\otimes\Sigma^{\infty}_{S^{1}}\X,\E_{n}\bigr),
\end{equation*}
obtained from the suspension-loop adjunction imply the group isomorphism.
\end{proof}

Using the monoidal product description $C=S^{1}\otimes C_{0}(\R)$ it follows that every cubical 
$\CC^{\ast}$-spectrum $\E$ yields a cubical $\CC^{\ast}$-bispectrum $\E_{S^{1},C_{0}(\R)}$:
\begin{equation*}
\xymatrix{
\vdots & \vdots & \vdots\\
C_{0}(\R^{2})\otimes\E_{0} & C_{0}(\R)\otimes\E_{1} & \E_{2} & \cdots\\
C_{0}(\R)\otimes\E_{0} & \E_{1} & S^{1}\otimes\E_{2} & \cdots\\
\E_{0} & S^{1}\otimes\E_{1} & S^{2}\otimes\E_{2} & \cdots }
\end{equation*}
\vspace{0.1in}

The horizontal structure maps $\sigma_{h}\colon S^{1}\otimes\E_{m,n}\rightarrow\E_{m+1,n}$ 
are defined by the identity map when $m\geq n$, 
and for $m<n$ by the map obtained from switching monoidal factors as in the composite 
\begin{equation*}
\xymatrix{
S^{1}\otimes C_{0}(\R^{n})\otimes\E_{m} \ar[r]^-{\tau\otimes \E_{m}} & 
C_{0}(\R^{n-1})\otimes S^{1}\otimes C_{0}(\R)\otimes\E_{m}
\ar[r]^-{C_{0}(\R^{n-1})\otimes\sigma} & C_{0}(\R^{n-1})\otimes\E_{m+1}.}
\end{equation*}
The vertical structure maps $\sigma_{v}\colon C_{0}(\R)\otimes \E_{m,n}\rightarrow \E_{m,n+1}$
are defined by the identity map if $m\leq n$, 
and otherwise by 
\begin{equation*}
\xymatrix{
C_{0}(\R)\otimes S^m\otimes\E_{n} \ar[r]^-{\tau\otimes\E_{n}} & 
S^{m-1}\otimes C_{0}(\R)\otimes S^{1}\otimes\E_{n}
\ar[r]^-{S^{m-1}\otimes\sigma} & S^{m-1}\otimes\E_{n+1}.}
\end{equation*}
Associated with a cubical $\CC^{\ast}$-bispectrum $\E_{S^{1},C_{0}(\R)}$ there are presheaves of 
bigraded stable homotopy groups $\pi_{p,q}\E$.
In $A$-sections, 
it is defined as the colimit of the diagram:
\begin{equation*}
\xymatrix{
\vdots & \vdots\\
[S^{p+m}\otimes C_{0}(\R^{q+n+1}),\E_{m,n+1}\vert A]
\ar[r]^-{\widetilde{\sigma_{h}}}\ar[u]^-{\widetilde{\sigma_{v}}} &
[S^{p+m+1}\otimes C_{0}(\R^{q+n+1}),\E_{m+1,n+1}\vert A]
\ar[r]\ar[u]^-{\widetilde{\sigma_{v}}} &
\cdots \\
[S^{p+m}\otimes C_{0}(\R^{q+n}),\E_{m,n}\vert A]
\ar[r]^-{\widetilde{\sigma_{h}}}\ar[u]^-{\widetilde{\sigma_{v}}} &
[S^{p+m+1}\otimes C_{0}(\R^{q+n}),\E_{m+1,n}\vert A]
\ar[r]\ar[u]^-{\widetilde{\sigma_{v}}} &
\cdots}
\end{equation*}
Here we may assume $\E_{m,n}$ is projective $\CC^{\ast}$-fibrant for all integers $m,n\in\ZZ$.
A cofinality argument shows the colimit can be computed by taking the 
diagonal and using the transition maps $\widetilde{\sigma_{h}}$ and $\widetilde{\sigma_{v}}$ 
in either order.
In particular, 
the degree $p$ and weight $q$ stable homotopy presheaf of 
a cubical $\CC^{\ast}$-spectrum $\E$ 
is isomorphic to the bigraded presheaf 
$\pi_{p,q}\E_{S^{1},C_{0}(\R)}$ of its associated cubical $\CC^{\ast}$-bispectrum. 
Thus,
Lemma \ref{lemma:stableweakequivalence-homotopy} and the previous observation show that 
$\E\rightarrow\F$ is a stable equivalence if and only if there is an induced isomorphism 
of bigraded presheaves $\pi_{p,q}\E_{S^{1},C_{0}(\R)}=\pi_{p,q}\F_{S^{1},C_{0}(\R)}$.

A level weak equivalences (respectively level cofibrations and level fibrations) of cubical $\CC^{\ast}$-bispectra 
is map $\E\rightarrow\F$ such that $\E_{m,n}\rightarrow\F_{m,n}$ are $\CC^{\ast}$-weak equivalences
(respectively projective cofibrations and projective $\CC^{\ast}$-fibration) for all $m$ and $n$.
We observe that every level fiber sequence $\F\rightarrow\E\rightarrow\E'$ induces a long exact sequence:
\begin{equation}
\label{bilongexactsequence}
\xymatrix{
\cdots\ar[r] &
\pi_{p+1,q}\E'\ar[r] &
\pi_{p,q}\F\ar[r] &
\pi_{p,q}\E\ar[r] &
\pi_{p,q}\E'\ar[r] &
\cdots }
\end{equation}
In effect,
it is harmless to assume that $\E'$ is level fibrant so that we have fiber sequences of cubical $S^{1}$-spectra 
\begin{equation*}
\xymatrix{
\Omega_{C_{0}(\R)}^{q+n}\F_{\ast,n}\ar[r] &
\Omega_{C_{0}(\R)}^{q+n}\E_{\ast,n}\ar[r] &
\Omega_{C_{0}(\R)}^{q+n}\E'_{\ast,n}  }
\end{equation*}
for every $n$, and the corresponding long exact sequences:
\begin{equation*}
\xymatrix{
\cdots\ar[r] &
\pi_{p+1}\Omega_{C_{0}(\R)}^{q+n}\E'_{\ast,n}\ar[r] &
\pi_{p}\Omega_{C_{0}(\R)}^{q+n}\F_{\ast,n}\ar[r] &
\pi_{p}\Omega_{C_{0}(\R)}^{q+n}\E_{\ast,n}\ar[r] &
\pi_{p}\Omega_{C_{0}(\R)}^{q+n}\E'_{\ast,n}\ar[r] & \cdots  }
\end{equation*}
Taking the filtered colimit of these diagrams furnishes the long exact sequence (\ref{bilongexactsequence}). 
This setup is familiar by now and we are ready to sketch a proof of the next result.
\begin{lemma}
\label{lemma:bispectracofibersequenceinducelongexactsequence}
Every level cofiber sequence $\E\rightarrow\E'\rightarrow\E'/\E$ of cubical $\CC^{\ast}$-bispectra
induces a natural long exact sequence:
\begin{equation*}
\xymatrix{
\cdots\ar[r] &
\pi_{p+1,q}\E'/\E\ar[r] &
\pi_{p,q}\E\ar[r] &
\pi_{p,q}\E'\ar[r] &
\pi_{p,q}\E'/\E\ar[r] &
\cdots }
\end{equation*}
\end{lemma}
\begin{proof}
There is a commutative diagram
\begin{equation*}
\xymatrix{
\E\ar[d]\ar[r] &
\E'\ar[d]\ar[r] &
\E'/\E & \\
\F\ar[r] &
\F'\ar[ur] }
\end{equation*}
where $\E'\rightarrow\F'\rightarrow\E'/\E$ is the composite of a level acyclic projective 
cofibration and a level fibration.
Observe that $\E_{\ast,n}\rightarrow\F_{\ast,n}$ are stable weak equivalences of cubical 
$S^{1}$-spectra.
Hence there are isomorphisms $\pi_{p,q}\E=\pi_{p,q}\F$ for all $p$ and $q$.
Combining this with (\ref{bilongexactsequence}) yields the long exact sequence.
\end{proof} 
\begin{corollary}
For every cubical $\CC^{\ast}$-spectrum $\E$ there are natural isomorphisms
\begin{equation*}
\pi_{p,q}\E=\pi_{p+1,q}(\E\wedge C)
\end{equation*}
for all $p$ and $q$.
\end{corollary}
\vspace{0.1in}

One checks easily that the proofs of \cite[Lemma 3.9, Corollary 3.10]{Jardine:MSS} translate into the 
following results. 
\begin{corollary}
Suppose $\F\rightarrow\E\rightarrow\E'$ is a level fiber sequence and $\E\rightarrow\E'\rightarrow\E'/\E$
a level cofiber sequence of cubical $\CC^{\ast}$-spectra.
\begin{itemize}
\item
There is an induced stable weak equivalence $\E/\F\rightarrow\E'$.
\item
If $\E'\rightarrow\F'\rightarrow\E'/\E$ is a factoring into a level weak equivalence and a
level fibration,
there is an induced stable weak equivalence from $\E$ to the fiber of $\F'\rightarrow\E'/\E$.
\end{itemize}
\end{corollary}
\newpage

\subsection{Triangulated structure}
\label{subsection:triangulatesstructure}
Triangulated categories in the sense of homotopical algebra \cite{Hovey:Modelcategories} 
satisfy the axioms of a classical triangulated category as in \cite{GM:Homologicalalgebra}.
In what follows,
the term triangulated category is used in the sense of the former.
We shall explicate a triangulated structure on $\SH^{\ast}$ exploiting its structure as a 
closed $\Ho(\Delta\Set_{\ast})$-module obtained from \cite{Hovey:Modelcategories}.
In particular, 
it turns out every short exact sequence of $\CC^{\ast}$-algebras gives rise to long exact 
sequences of abelian groups; we include some examples to illustrate the utility of the
triangulated structure.

Now suppose $\E\rightarrow\F$ is a projective cofibration between cofibrant objects in $\Spt_{C}$.
Its cone is defined by the pushout diagram:
\begin{equation*}
\xymatrix{
\E\ar[r]\ar[d] & \E\wedge\Delta^{1}_{+}\ar[d]\\
\F\ar[r] & \cone(\E\rightarrow\F) }
\end{equation*}
\begin{example}
The cone of the projection $\E\rightarrow\ast$ is $\E\wedge S^{1}$.
\end{example}
\label{example:coneexample}
The next result follows by gluing \cite[II Lemma 8.8]{GJ:Modelcategories} since cones are examples of 
pushouts.
\begin{lemma}
\label{lemma:conelemma}
If the vertical maps in the following diagram of horizontal cofibrations between projective cofibrant 
objects in $\Spt_{C}$
\begin{equation*}
\xymatrix{
\E\ar[r]\ar[d] & \F\ar[d]\\
\E'\ar[r] & \F' }
\end{equation*}
is a stable weak equivalence, 
then so is the induced map $\cone(\E\rightarrow\F)\rightarrow\cone(\E'\rightarrow\F')$.
\end{lemma}
In the special case when $\E\rightarrow\E'$ is the identity map and $\F'=\ast$ in 
Lemma \ref{lemma:conelemma},
we get a map $\cone(\E\rightarrow\F)\rightarrow\cone(\E\rightarrow\ast)$,
and hence a diagram 
\begin{equation}
\label{equation:distinguishedtriangle}
\xymatrix{
\E\ar[r] & \F\ar[r] & \cone(\E\rightarrow\F)\ar[r] & \E\wedge S^{1} }
\end{equation}
natural in $\E\rightarrow\F$, 
cf.~Example \ref{example:coneexample}.
\begin{definition}
\label{definition:cofibersequences}
A cofiber sequence in $\SH^{\ast}$ is a diagram $\E\rightarrow\F\rightarrow\G$ together with 
a coaction of $\Sigma_{S^{1}}\E$ on $\G$ that is isomorphic in $\SH^{\ast}$ to a diagram of 
the form (\ref{equation:distinguishedtriangle}).
\end{definition}
\begin{theorem}
\label{theorem:stabletriangulatedstructure}
The stable homotopy category $\SH^{\ast}$ is a triangulated category.
\end{theorem}
\begin{proof}
The homotopy category of every pointed model category is a pre-triangulated category on merit
of its (co)fiber sequences \cite[\S6.5]{Hovey:Modelcategories}.
Since the shift functor a.k.a.~the $S^{1}$-suspension functor 
$\Sigma_{S^{1}}\colon\SH^{\ast}\rightarrow\SH^{\ast}$ 
is an equivalence of categories,
the assertion follows.
\end{proof}

\begin{remark}
The sum of two maps $f,g\colon\E\rightarrow\F$ is represented by the composite
\begin{equation*}
\xymatrix{
\E\ar[r]^-{\bigtriangleup} &
\E\times\E\ar[r]^-{f\times g} &
\F\times\F &
\F\vee\F \ar[l]_-{\simeq}\ar[r]^-{\bigtriangledown} & \F. }
\end{equation*} 
\end{remark}

Every distinguished triangle $\E\rightarrow\F\rightarrow\G$ in $\SH^{\ast}$ induces 
long exact sequences of abelian groups
\begin{equation*}
\xymatrix{
\cdots\ar[r] & 
[\Hspt,\E[n]]\ar[r] & 
[\Hspt,\F[n]]\ar[r] &
[\Hspt,\G[n]]\ar[r] &
[\Hspt,\Hspt[n+1]]\ar[r]&
\cdots, \\
\cdots\ar[r] & 
[\G[n],\Hspt]\ar[r] & 
[\F[n],\Hspt]\ar[r] &
[\E[n],\Hspt]\ar[r] &
[\G[n-1],\Hspt]\ar[r] &
\cdots.}
\end{equation*}
Here $[-,-]$ denotes maps in $\SH^{\ast}$.
We proceed with some other facts concerning the triangulated structure collected from \cite[Chapter 7]{Hovey:Modelcategories}.

\begin{proposition}
\label{proposition:patched}
The following holds in the triangulated structure on $\SH^{\ast}$.
\begin{itemize}
\item
The class of cofiber sequences is replete, 
i.e.~every diagram isomorphic to a cofiber sequence is a cofiber sequence.
\item 
For every commutative diagram of cofiber sequences
\begin{equation*}
\xymatrix{
\E\ar[r]\ar[d]_{e} & \F\ar[r]\ar[d] & \G\ar@{-->}[d] \\
\E'\ar[r] & \F'\ar[r] & \G' }
\end{equation*}
there exists a nonunique $\Sigma_{S^{1}}e$-equivariant filler $\G\rightarrow\G'$.
\item 
If in a commutative diagram of cofiber sequences 
\begin{equation*}
\xymatrix{
\E\ar[r]\ar[d]_{e} & \F\ar[r]\ar[d]_{f} & \G\ar[d]_{g} \\
\E'\ar[r] & \F'\ar[r] & \G' }
\end{equation*}
the maps $e$ and $f$ are isomorphisms and $g$ is $\Sigma_{S^{1}}e$-equivariant, 
then $g$ is an isomorphism.
\item
If $K$ is a pointed simplicial set and $\E\rightarrow\F\rightarrow\G$ a cofiber sequence,
then $\E\wedge K\rightarrow\F\wedge K\rightarrow\G\wedge K$ and 
$\RR\underline{\Hom}(K,\E)\rightarrow\RR\underline{\Hom}(K,\F)\rightarrow\RR\underline{\Hom}(K,\G)$
are cofiber sequences.
\end{itemize}
\end{proposition}

The next result gives a way of producing distinguished triangles in $\SH^{\ast}$.
\begin{lemma}
\label{lemma:homotopypushoutdistinguishedtriangles}
Every homotopy pushout square of simplicial $\CC^{\ast}$-spectra
\begin{equation}
\label{diagram:homotopypushout}
\xymatrix{
\E\ar[r]\ar[d] & \G\ar[d]\\
\F\ar[r] & \Hspt }
\end{equation}
gives rise to a distinguished triangle
\begin{equation*}
\xymatrix{
\cdots\ar[r] & 
\E\ar[r] & 
\F\oplus\G\ar[r] & 
\Hspt\ar[r] & 
\E[1]\ar[r] & \cdots.}
\end{equation*}
\end{lemma}
\begin{proof}
We may assume $\E\rightarrow\G$ is a projective cofibration and (\ref{diagram:homotopypushout}) 
is a pushout, 
and it suffices to show $\E\wedge S^{1}$ and $\cone(\F\oplus\G\rightarrow\Hspt)$ are isomorphic in 
$\SH^{\ast}$.
Since $\E\wedge S^{1}$ and $\cone(\E\rightarrow\E\wedge\Box^{1}_{+})$ are isomorphic in $\SH^{\ast}$
the assertion follows by noting there is a naturally induced stable weak equivalence between the 
lower right corners in the following pushouts diagrams:
\begin{equation*}
\xymatrix{
\E\ar[r]\ar[d] & \E\wedge\Delta^{1}_{+}\ar[d]\\
\E\wedge\Delta^{1}_{+}\ar[r] & \cone(\E\rightarrow\E\wedge\Delta^{1}_{+}) }
\;\;\;\;\;
\xymatrix{
\E\ar[r]\ar[d] & \G\wedge\Delta^{1}_{+}\ar[d]\\
\F\wedge\Delta^{1}_{+}\ar[r] & \cone(\F\oplus\G\rightarrow\Hspt) }
\end{equation*}
\end{proof}
\begin{corollary}
\label{corollary:sesdistinguishedtriangles}
Every short exact sequence of $\CC^{\ast}$-algebras gives rise to a distinguished triangle in $\SH^{\ast}$.
\end{corollary}
\begin{proof}
The suspension functor preserves homotopy pushout squares.
\end{proof}
Next we employ the triangulated structure to compute maps from sequential colimits into arbitrary objects 
in $\SH^{\ast}$.
\begin{lemma}
\label{lemma:lim1ses}
Suppose ${\bf X}\colon\mathbf{N}\rightarrow\Spc_{C}$ is a sequential diagram and $\E$ a cubical $\CC^{\ast}$-spectrum.
Then there is a short exact sequence
\begin{equation*}
\xymatrix{
0\ar[r] & 
\underset{\mathbf{N}}{\lim^{1}}\,\,\SH^{\ast}\bigl(S^{1}\wedge {\bf X}(n),\E\bigr)
\ar[r] & 
\SH^{\ast}(\underset{\mathbf{N}}{\colim}\,\,{\bf X},\E)
\ar[r] & 
\underset{\mathbf{N}}{\lim}\,\,\SH^{\ast}\bigl({\bf X}(n),\E\bigr)
\ar[r] & 
0. }
\end{equation*}
\end{lemma}
\begin{proof}
First we observe there is a naturally induced distinguished triangle
\begin{equation}
\label{equation:induceddistinguished triangle}
\xymatrix{
\bigvee_{\mathbf{N}}{\bf X}(n)
\ar[r] & 
\bigvee_{\mathbf{N}}{\bf X}(n)
\ar[r] & 
\underset{\mathbf{N}}{\colim}\,\,{\bf X}
\ar[r] & 
\bigvee_{\mathbf{N}} S^{1}\wedge {\bf X}(n). }
\end{equation}
To begin with, 
the two natural maps ${\bf X}(n)\hookrightarrow \Delta^{1}_{+}\wedge {\bf X}(n)$ induced by $0_{+}$ and $1_{+}$
induce the diagram
\begin{equation}
\label{equation:coequalizerdiagram}
\xymatrix{
\bigvee_{\mathbf{N}}{\bf X}(n)
\ar@<3pt>[r]\ar@<-3pt>[r] & 
\bigvee_{\mathbf{N}}\Delta^{1}_{+}\wedge {\bf X}(n). }
\end{equation}
The coequalizer of (\ref{equation:coequalizerdiagram}) maps by a weak equivalence to the colimit of ${\bf X}$.
By taking the difference of the two maps in (\ref{equation:coequalizerdiagram}) in the additive structure on 
$\SH^{\ast}$ we deduce that (\ref{equation:induceddistinguished triangle}) is a distinguished triangle.

Now applying $\SH^{\ast}(-,\E)$ to (\ref{equation:induceddistinguished triangle}) yields the long exact sequence
\begin{equation*}
\xymatrix{
\cdots\ar[r] & 
\prod_{\mathbf{N}}[S^{1}\wedge {\bf X}(n),\E]
\ar[r] &
[\underset{\mathbf{N}}{\colim}\,\,{\bf X},\E]
\ar[r] & 
\prod_{\mathbf{N}}[{\bf X}(n),\E]
\ar[r] & 
\cdots. }
\end{equation*}
By the definition of the $\lim^{1}$-term the long exact sequence breaks up into the claimed short exact sequences.
\end{proof}
\begin{remark}
In concrete examples it is often of interest to know whether the $\lim^{1}$-term in Lemma \ref{lemma:lim1ses} vanishes.
\end{remark}

By combining the distinguished triangle (\ref{equation:induceddistinguished triangle}) with the level filtrations of a cubical 
$\CC^{\ast}$-spectrum $\E$ we deduce the next result.
\begin{corollary}
For every cubical $\CC^{\ast}$-spectrum $\E$ there is a naturally induced distinguished triangle
\begin{equation*}
\xymatrix{
\bigvee_{\mathbf{N}}L_{n}\E
\ar[r] & 
\bigvee_{\mathbf{N}}L_{n}\E
\ar[r] & 
\E
\ar[r] & 
\bigvee_{\mathbf{N}} S^{1}\wedge L_{n}\E. }
\end{equation*}
\end{corollary}
\begin{corollary}
For cubical $\CC^{\ast}$-spectra $\E$ and $\F$ there is a canonical short exact sequence
\begin{equation*}
\xymatrix{
0\ar[r] & 
\underset{\mathbf{N}}{\lim^{1}}\,\,\E^{p+2i-1,q+i}(\F_{i})
\ar[r] & 
\E^{p,q}(\F)
\ar[r] & 
\underset{\mathbf{N}}{\lim}\,\,\E^{p+2i,q+i}(\F_{i})
\ar[r] & 
0. }
\end{equation*}
\end{corollary}

Due to the formal nature of the proofs in this section the results remain valid in the stable $\Group$-equivariant setting.

\begin{example}
\label{example:exactsequences}
We include some additional examples to illustrate the applicability of the results in this section.
\begin{itemize}
\item
The suspension extension of a $\Group-\CC^{\ast}$-algebra $A$ is 
\begin{equation*}
\xymatrix{
0\ar[r] & A\otimes C_{0}(0,1)\ar[r] & A\otimes C_{0}(0,1]\ar[r] & A\ar[r] & 0.}
\end{equation*}
The $\Group-\CC^{\ast}$-actions are determined by the pointwise action on $A$. 
\item
The Toeplitz extension
\begin{equation*}
\xymatrix{
0\ar[r] & A\otimes\K\ar[r] & A\otimes\T\ar[r] & \ar[r] A\otimes C(S^{1})& 0.}
\end{equation*}
In the equivariant setting, 
the compact operators $\K$ and the Toeplitz algebra $\T$ are equipped with the trivial actions. 
\item
If $\alpha$ is an automorphism of a unital $\CC^{\ast}$-algebra $A$, 
for the crossed product $A\rtimes_{\alpha}\ZZ$ of $A$ by the action of $\ZZ$ on $A$ given by $\alpha$
there is a short exact sequence of $\CC^{\ast}$-algebras,
the Pimsner-Voiculescu ``Toeplitz extension''
\begin{equation*}
\label{PVexactsequence}
\xymatrix{
0\ar[r] & A\otimes\K\ar[r] & \T_{\alpha}\ar[r] & A\rtimes_{\alpha}\ZZ\ar[r] & 0.}
\end{equation*}
Recall that $\T_{\alpha}$ is the $\CC^{\ast}$-subalgebra of $(A\rtimes_{\alpha}\ZZ)\otimes\T$ generated 
by $a\otimes 1$ and $u\otimes v$, where $u$ is the unitary element such that $uau^{\ast}=\alpha(a)$ for
all $a\in A$ and $v$ is the non-unitary isometry generating the ordinary Toeplitz algebra $\T$. 
\item
Let $s_{1},\cdots,s_{n}$,
$n\geq 2$, 
be isometries on a Hilbert space $\Hilbert$ whose range projections add up to the identity. 
The Cuntz algebra $\mathcal{O}_{n}$ is the unique up to isomorphism simple, purely infinite and unital 
$\CC^{\ast}$-subalgebra of $\K(\Hilbert)$ generated by $\{s_{1},\cdots,s_{n}\}$ subject to the relations 
$s_{i}^{\ast}s_{i}=I$ for all $1\leq i\leq n$ and $\sum_{i=1}^{n}s_{i}s_{i}^{\ast}=I$. 
There exists a short exact sequence
\begin{equation*}
\xymatrix{
0\ar[r] &\K\ar[r] &\E_{n}\ar[r] &\mathcal{O}_{n}\ar[r] & 0.}
\end{equation*}
\item
The $\CC^{\ast}$-algebra $C(S^{1})$ acts on the Hilbert space $L^{2}(S^{1})$ by multiplication 
$f(g)\equiv fg$.
For $\Theta\in[0,1]$, 
let $A_{\Theta}$ be the noncommutative torus generated by multiplication operators and the
unitary rotation $u_{\Theta}$ by $2\pi\Theta$ on $L^{2}(S^{1})$, 
i.e.~$u_{\Theta}(g)(s)\equiv g(s-2\pi\Theta)$.
Recall that $A_{\Theta}$ is simple when $\Theta$ is irrational.
There exists a short exact sequence
\begin{equation*}
\xymatrix{
0\ar[r] &\K\otimes C(S^{1})\ar[r] &\T_{\Theta}\ar[r] & A_{\Theta}\ar[r] & 0.}
\end{equation*}
\item
Let $M$ be a compact manifold with cotangent sphere bundle $T^{\ast}M$.
If $\Psi_{0}$ denotes the closure in the operator norm of the algebra of pseudodifferential operators
of negative order, 
then there exists a short exact sequence
\begin{equation*}
\xymatrix{
0\ar[r] &\K\bigl(L^{2}(M)\bigr)\ar[r] &\Psi_{0}\ar[r] & C(T^{\ast}M)\ar[r] & 0.}
\end{equation*}
\item
Let $T(A)$ be the tensor $\CC^{\ast}$-algebra of $A$,
i.e.~the completion of $A\oplus (A\otimes A)\oplus\cdots$ with respect to the $\CC^{\ast}$-norm
given by the supremum of its $\CC^{\ast}$-seminorms.
In \cite{Cuntz:bivariant} the Kasparov $KK$-theory functor $KK_{n}(A,-)$ is defined using the 
short exact sequence
\begin{equation*}
\xymatrix{
0\ar[r] & J(A)\ar[r] & T(A)\ar[r] & A\ar[r] & 0.}
\end{equation*}
If $A$ is a $\Group-\CC^{\ast}$-algebra, 
then so is $T(A)$.
For the important universal extension property of this short exact sequence we refer to \cite{Cuntz:bivariant}.
\end{itemize}
\end{example}
\newpage

\subsection{Brown representability}
\label{subsection:brownrepresentability}
In this section we note a Brown representability theorem in $\SH^{\ast}$.
A key result due to Rosicky \cite[Proposition 6.10]{Rosicky:brownrepresentability} shows the 
homotopy category of a combinatorial stable model category is well generated in the sense of Neeman 
\cite[Remark 8.1.7]{Neeman:Triangulatedcategories}.  
The precise definition of ``well generated'' is not easily stated and will not be repeated here since 
there is no real need for it.
By \cite[Proposition 8.4.2]{Neeman:Triangulatedcategories} the homotopy category satisfies the 
Brown representability theorem formulated in \cite[Definition 8.4.1]{Neeman:Triangulatedcategories}.
\vspace{0.1in}

Combining the general setup with our results for $\SH^{\ast}$ we deduce the following 
representability result referred to as Brown representability for cohomology:
\begin{theorem}
\label{theorem:stablebrown}
If $\F$ is a contravariant functor from $\SH^{\ast}$ to the category of abelian groups which 
is homological and sends coproducts to products,
there exists an object $\E$ of $\SH^{\ast}$ and a natural isomorphism
\begin{equation*}
\xymatrix{
\SH^{\ast}(-,\E)\ar[r] & \F(-). }
\end{equation*} 
\end{theorem}
\vspace{0.1in}

The functor $\F$ is homological if it sends every distinguished triangle in $\SH^{\ast}$ to an exact sequence of 
abelian groups \cite[Definition 1.1.7]{Neeman:Triangulatedcategories}.
Recall from Corollary \ref{corollary:sesdistinguishedtriangles} that every short exact sequence of $\CC^{\ast}$-algebras 
gives rise to a distinguished triangle in $\SH^{\ast}$.
Note also that $\F$ is matrix invariant and homotopy invariant by virtue of being a (contravariant) functor
on the stable $\CC^{\ast}$-homotopy category.
\newpage

\subsection{$\CC^{\ast}$-symmetric spectra}
\label{subsection:symmetricspectra}

The existence of model categories with strictly associative and commutative monoidal products 
which model the ordinary stable homotopy category is a relatively recent discovery.  
Symmetric spectra of pointed simplicial sets introduced in \cite{HSS:SS} furnishes one such model.
In this section we work out a theory of symmetric spectra for pointed cubical $\CC^{\ast}$-spaces.

Let $\Sigma=\coprod_{n\geq 0}\Sigma_{n}$ be the category with objects $\overline{n}=\{1,2,\cdots,n\}$ 
for $n\geq 0$,
where $\overline{0}=\emptyset $.  
The maps of $\Sigma$ from $\overline{m}$ to $\overline{n}$ are the bijections,
i.e.~the empty set when $m\neq n$ and the symmetric group $\Sigma_{n}$ when $m=n$. 
Note that $\Sigma$ is a skeleton for the category of finite sets and isomorphisms.

\begin{definition}
\label{definition:symmetric-sequence}
A $\CC^{\ast}$-cubical symmetric sequence $\E$ is a functor from $\Sigma$ to $\Box\CC^{\ast}-\Spc_{0}$ or 
equivalently a sequence of pointed cubical $\CC^{\ast}$-spaces $(\E_{n})_{n\geq 0}$ where $\Sigma_{n}$ acts on $\E_{n}$.
Let $\Box\CC^{\ast}-\Spc_{0}^{\Sigma}$ denote the functor category of $\CC^{\ast}$-cubical symmetric sequences.  
\end{definition}
\begin{example}
Every pointed cubical $\CC^{\ast}$-space $\X$ determines a $\CC^{\ast}$-cubical symmetric sequence 
$\Sym(\X)\equiv (\X^{\otimes n})_{n\geq 0}$ where $\Sigma_{n}$ acts on the product $\X^{\otimes n}$ 
by permutation.
\end{example}
\begin{example}
For $n\geq 0$, 
the free symmetric sequence $\Sigma[n]\equiv\Sigma(\overline{n},-)$ gives rise to the free functor 
$\Sigma[n]_{+}\otimes-\colon\Box\CC^{\ast}-\Spc_{0}\rightarrow\Box\CC^{\ast}-\Spc_{0}^{\Sigma}$.
\end{example}

The monoidal structure on $\CC^{\ast}$-cubical symmetric sequences 
\begin{equation}
\xymatrix{
-\otimes^{\Sigma}-\colon\Box\CC^{\ast}-\Spc_{0}^{\Sigma}\times\Box\CC^{\ast}-\Spc_{0}^{\Sigma}
\ar[r] & \Box\CC^{\ast}-\Spc_{0}^{\Sigma}}
\end{equation}
is defined by 
\begin{equation}
\label{monoidalproductsymmetricsequences}
(\E\otimes^{\Sigma}\F)_{n}=
\coprod_{p+q=n}\Sigma_{n}\times_{\Sigma_{p}\times\Sigma_{q}} (\E_{p}\otimes\F_{q}).
\end{equation}
Here, 
if $(\Z,z_{0})$ is a pointed set with an $\Sigma_{n}$-action, 
$\Sigma_{n}\times_{\Sigma_{p}\times\Sigma_{q}}\Z$ denotes the pointed by $(1,z_{0})$ quotient of the 
coproduct of $n!$ copies of $\Z$ by the equivalence relation generated by identifying the elements of 
$\Sigma_{n}\times\{z_{0}\}$ and elements $\bigl(\sigma_{n}(\sigma_{p},\sigma_{q}),z\bigr)$ with 
$\bigl(\sigma_{n},(\sigma_{p},\sigma_{q})z\bigr)$, where $\sigma_{i}\in\Sigma_{i}$ for $i=p,q,n$ and
the group homomorphism $\Sigma_{p}\times\Sigma_{q}\rightarrow\Sigma_{n}$ is defined by 
\begin{equation*}
(\sigma_{p},\sigma_{q})(k)\equiv
\begin{cases}
\sigma_{p}(k) & \text{ if } 1\leq k\leq p\\
\sigma_{q}(k-p)+p & \text{ if } p+1\leq k\leq p+q.\\
\end{cases}
\end{equation*}
To complete the definition of (\ref{monoidalproductsymmetricsequences}) one extends this construction 
in the natural way to all pointed cubical $\CC^{\ast}$-spaces.
The monoidal structure is rigged so that $(\C,0,0,\cdots)$ is the unit,
and $\Sym(\X)$ is freely generated by $(0,\X,0,0,\cdots)$. 
Note that $\E\otimes^{\Sigma}\F$ represents the functor which to $\G$ associates all 
$\Sigma_{p}\times\Sigma_{q}$-equivariant maps $\phi_{p,q}\colon\E_{p}\otimes\F_{q}\rightarrow\G_{p+q}$ 
in $\Box\CC^{\ast}-\Spc_{0}$.
To show that $\otimes^{\Sigma}$ is symmetric it suffices, 
by the Yoneda lemma, 
to define natural bijections 
$\Box\CC^{\ast}-\Spc_{0}^{\Sigma}(\E\otimes^{\Sigma}\F,\G)
\rightarrow
\Box\CC^{\ast}-\Spc_{0}^{\Sigma}(\F\otimes^{\Sigma}\E,\G)$.
In effect, 
if $(\phi_{p,q})\in\Box\CC^{\ast}-\Spc_{0}^{\Sigma}(\E\otimes^{\Sigma}\F,\G)$ define 
$(\phi'_{p,q})\in\Box\CC^{\ast}-\Spc_{0}^{\Sigma}(\F\otimes^{\Sigma}\E,\G)$ by the commutative diagrams:
\begin{displaymath}
\xymatrix{
\F_{p}\otimes\E_{q}\ar[r]^-{\phi'_{p,q}}\ar[d]_-{\tau} & \G_{p+q}\\
\E_{q}\otimes\F_{p}\ar[r]^-{\phi_{q,p}} & \G_{q+p}\ar[u]_-{c_{p,q}} }
\end{displaymath}
Here $\tau$ is the symmetry isomorphism in $\Box\CC^{\ast}-\Spc_{0}$ and $\sigma_{p,q}$ is the permutation 
defined by 
\begin{equation*}
\sigma_{p,q}(k)\equiv
\begin{cases}
k+q & \text{ if } 1\leq k\leq p\\
k-p & \text{ if } p+1\leq k\leq p+q.\\
\end{cases}
\end{equation*}
Then $\phi'_{p,q}$ is $\Sigma_{p}\times\Sigma_{q}$-equivariant, and hence there exists a 
commutativity isomorphism $\sigma\colon\E\otimes^{\Sigma}\F\rightarrow\F\otimes^{\Sigma}\E$.
According to (\ref{monoidalproductsymmetricsequences}) the associativity isomorphism for 
$\otimes^{\Sigma}$ follow using the associativity isomorphism for $\otimes$ in $\Box\CC^{\ast}-\Spc_{0}$.

Now if $\E_{n},\F_{n}\in\Box\CC^{\ast}-\Spc_{0}^{\Sigma_{n}}$ the internal hom 
$\underline{\Hom}^{\Sigma_{n}}(\E_{n},\F_{n})$ exists for formal reasons as the equalizer of the 
two maps $\underline{\Hom}(\E_{n},\F_{n})\rightarrow\underline{\Hom}(\Sigma_{n}\times\E_{n},\F_{n})$ 
induced by the $\Sigma_{n}$-actions on $\E_{n}$ and $\F_{n}$.  
Internal hom objects in $\Box\CC^{\ast}-\Spc_{0}^{\Sigma}$ are defined by
\begin{equation*}
\underline{\Hom}^{\Sigma}(\E,\F)_{k}\equiv
\prod _{n}\underline{\Hom}^{\Sigma_{n}}(X_{n},\F_{n+k}).
\end{equation*}
Note that $\underline{\Hom}^{\Sigma_{n}}(\E_{n},\F_{n+k})$ is the $\Sigma_{n}$-invariants of the internal 
hom $\underline{\Hom}(\E_{n},\F_{n+k})$ in $\Box\CC^{\ast}-\Spc_{0}$ for the action given by associating 
to $\sigma_{n}\in\Sigma_{n}$ the map 
\begin{equation*}
\xymatrix{
\underline{\Hom}(\E_{n},\F_{n+k}) 
\ar[r]^-{{\underline{\Hom}(\sigma_{n},\F_{n+k})}} &
\underline{\Hom}(\E_{n},\F_{n+k}) 
\ar[r]^-{{\underline{\Hom}\bigl(\E_{n},\phi_{n,k}(\sigma_{n},1)\bigr)}} &
\underline{\Hom}(\E_{n},\F_{n+k}).}
\end{equation*}
With these definitions there are natural bijections
\begin{equation*}
\Box\CC^{\ast}-\Spc_{0}^{\Sigma}(\E\otimes^{\Sigma}\F,\G)=
\Box\CC^{\ast}-\Spc_{0}^{\Sigma}\bigl(\E,\underline{\Hom}^{\Sigma}(\F,\G)\bigr).
\end{equation*}

Small limits and colimits in the functor category $\Box\CC^{\ast}-\Spc_{0}^{\Sigma}$ exist and are formed pointwise.
By reference to \cite{Day:closedfunctorsI} or by inspection of the above constructions we get the 
next result.
\begin{lemma}
\label{lemma:symmetric-sequencebcsmfc}
The triple $(\Box\CC^{\ast}-\Spc_{0}^{\Sigma},\otimes^{\Sigma},\underline{\Hom}^{\Sigma})$ forms a bicomplete and 
closed symmetric monoidal category.
\end{lemma}
\begin{example}
\label{example:symmetricexample}
For a pointed cubical $\CC^{\ast}$-space $\X$ the $\CC^{\ast}$-cubical symmetric sequence 
$\Sym(\X)$ is a commutative monoid in the closed symmetric structure on $\Box\CC^{\ast}-\Spc_{0}^{\Sigma}$.
Recall the monoid structure is induced by the canonical maps 
\begin{equation*}
\xymatrix{
\Sigma_{p+q}\times_{\Sigma_{p}\times\Sigma_{q}}(\X^{\otimes p}\otimes\X^{\otimes q})\ar[r] &
\X^{\otimes(p+q)}. }
\end{equation*}
\end{example}

Lemma \ref{lemma:symmetric-sequencebcsmfc} and Example \ref{example:symmetricexample} imply the category 
of modules over $\Sym(\X)$ in $\Box\CC^{\ast}-\Spc_{0}^{\Sigma}$ is bicomplete and closed symmetric monoidal.
The monoidal product 
\begin{equation*}
\xymatrix{
-\otimes_{\Sym(\X)}-\colon\Sym(\X)-\Mod\times\Sym(\X)-\Mod\ar[r] & \Sym(\X)-\Mod }
\end{equation*}
is defined by coequalizers in $\Box\CC^{\ast}-\Spc_{0}^{\Sigma}$ of the form
\begin{equation*}
\xymatrix{
\Sym(\X)\otimes^{\Sigma}\F\otimes^{\Sigma}\G\ar@<3pt>[r]\ar@<-3pt>[r] & 
\F\otimes^{\Sigma}\G\ar[r] &
\F\otimes_{\Sym(\X)}\G }
\end{equation*}
induced by 
$\Sym(\X)\otimes^{\Sigma}\F\rightarrow\F$ and $\Sym(\X)\otimes^{\Sigma}\F\otimes^{\Sigma}\G
\rightarrow\F\otimes^{\Sigma}\Sym(\X)\otimes^{\Sigma}\G\rightarrow\F\otimes^{\Sigma}\G$.
Moreover, the internal hom 
\begin{equation*}
\xymatrix{
\underline{\Hom}_{\Sym(\X)}(-,-)\colon
(\Sym(\X)-\Mod)^{\op}\times\Sym(\X)-\Mod\ar[r] & \Sym(\X)-\Mod}  
\end{equation*}
is defined by equalizers in $\Box\CC^{\ast}-\Spc_{0}^{\Sigma}$ of the form
\begin{equation*}
\xymatrix{
\underline{\Hom}_{\Sym(\X)}(\F,\G)
\ar[r] &
\underline{\Hom}^{\Sigma}(\F,\G)
\ar@<3pt>[r]\ar@<-3pt>[r] & 
\underline{\Hom}^{\Sigma}(\Sym(\X)\otimes\F,\G)}.
\end{equation*}
The first map in the equalizer is induced by the $\Sym(\X)$-action on $\F$ and the second map 
is the composition of $\Sym(\X)\otimes-$ and the $\Sym(\X)$-action on $\G$.
Note that $\Sym(\X)$ is the unit for the monoidal product. 

Next we specialize these constructions to the projective cofibrant pointed cubical $\CC^{\ast}$-space 
$C=S^{1}\otimes C_{0}(\R)$. 

\begin{definition}
\label{definition:Cstarcubicalsymmectricspectra}
The category of $\CC^{\ast}$-cubical symmetric spectra $\Spt_{C}^{\Sigma}$ 
is the category of modules in $\Box\CC^{\ast}-\Spc_{0}^{\Sigma}$ over the commutative monoid $\Sym(C)$.
In detail, 
a module over $\Sym(C)$ consists of a sequence of pointed cubical $\CC^{\ast}$-spaces 
$\E\equiv (\E_{n})_{n\geq 0}$ in $\Box\CC^{\ast}-\Spc_{0}^{\Sigma}$ together with $\Sigma_{n}$-equivariant 
structure maps $\sigma_{n}\colon C\otimes \E_{n}\rightarrow\E_{n+1}$ such that the composite 
\begin{equation*}
\xymatrix{
C^{\otimes p}\otimes\E_{n}\ar[r] & C^{\otimes p-1}\otimes\E_{n+1}\ar[r] & \dots\ar[r] & \E_{n+p}}
\end{equation*}
is $\Sigma_{n}\times\Sigma_{p}$-equivariant for all $n,p\geq 0$.  
A map of $\CC^{\ast}$-cubical symmetric spectra $f\colon\E\rightarrow\F$ is a collection of 
$\Sigma_{n}$-equivariant maps 
$f_{n}\colon\E_{n}\rightarrow\F_{n}$ compatible with the structure maps of $\E$ and $\F$ in the 
sense that there are commutative diagrams:
\begin{equation*}
\xymatrix{
C\otimes \E_{n}\ar[d]_-{C\otimes f_{n}}\ar[r]^-{\sigma_{n}} & \E_{n+1}\ar[d]^-{f_{n+1}}\\
C\otimes \F_{n}\ar[r]^-{\sigma_{n}} & \F_{n+1}}
\end{equation*}
\end{definition}
There is a forgetful functor $\Spt_{C}^{\Sigma}\rightarrow\Box\CC^{\ast}-\Spc_{0}^{\Sigma}$ and for all
$\E\in\Box\CC^{\ast}-\Spc_{0}^{\Sigma}$ and $\F\in\Spt_{C}^{\Sigma}$ a natural bijection 
$\Spt_{C}^{\Sigma}(\E\otimes C,\F)\rightarrow\Box\CC^{\ast}-\Spc_{0}^{\Sigma}(\E,\F)$.
We denote by $\wedge$ the monoidal product on $\Spt_{C}^{\Sigma}$. 
Next we give a series of examples.
\begin{example}
\label{example:Cstarsymmetriccubicalprolongation}
If $\RRR\colon\Box\CC^{\ast}-\Spc_{0}\rightarrow\Box\CC^{\ast}-\Spc_{0}$ is a cubical $\CC^{\ast}$-functor,
define the induced functor $\RRR\colon\Spt_{C}^{\Sigma}\rightarrow\Spt_{C}^{\Sigma}$ by 
$\RRR(\E)_{n}\equiv\RRR(\E_{n})$.
Here $\Sigma_{n}$ acts by applying $\RRR$ to the $\Sigma_{n}$-action on $\E_{n}$. 
The structure maps are given by the compositions 
$C\otimes\RRR(\E)_{n}\rightarrow\RRR(C\otimes\E_{n})\rightarrow\RRR(\E_{n+1})$.
For a map $\E\rightarrow\F$ between $\CC^{\ast}$-cubical symmetric spectra $\RRR(\E\rightarrow\F)$
is the sequence of maps  $\RRR(\E_{n}\rightarrow\F_{n})$ for $n\geq 0$.
In particular, 
using the tensor (\ref{cubicalCstarspacetensor}) and cotensor (\ref{cubicalCstarspacecotensor})
structures on pointed cubical $\CC^{\ast}$-spaces we obtain an adjoint functor pair:
\begin{equation*}
\xymatrix{
-\wedge K\colon\Spt_{C}^{\Sigma} \ar@<3pt>[r] & 
\Spt_{C}^{\Sigma}\colon(-)^{K}   \ar@<3pt>[l].}
\end{equation*}
The $\CC^{\ast}$-cubical symmetric spectrum $\E^{K}$ is defined in level $n$ by 
$\hom_{\Box\CC^{\ast}-\Spc_{0}}(K,\E_{n})$ and the structure map 
$C^{p}\otimes\hom_{\Box\CC^{\ast}-\Spc_{0}}(K,\E_{n})
\rightarrow
\hom_{\Box\CC^{\ast}-\Spc_{0}}(K,\E_{p+n})$ 
is the unique map of pointed cubical $\CC^{\ast}$-spaces making the diagram
\begin{equation*}
\xymatrix{
C^{p}\otimes\hom_{\Box\CC^{\ast}-\Spc_{0}}(K,\E_{n})\otimes K\ar[d]\ar[r] &
\hom_{\Box\CC^{\ast}-\Spc_{0}}(K,\E_{p+n})\otimes K\ar[d] \\
C^{p}\otimes\E_{n}\ar[r] & \E_{p+n} }
\end{equation*}
commute.
This construction is natural in $K$ and $\E$, 
and for all pointed cubical sets $K$ and $L$ there is a natural isomorphism
\begin{equation*}
\E^{K\otimes_{\Box\Set_{\ast}}L}=
(\E^L)^{K}.
\end{equation*}
\end{example}
\begin{example}
\label{example:Cstarsymmetricfunctioncomplexes}
The cubical function complex $\hom_{\Spt_{C}^{\Sigma}}(\E,\F)$ of $\CC^{\ast}$-cubical symmetric spectra 
$\E$ and $\F$ is defined by 
\begin{equation*}
\hom_{\Spt_{C}^{\Sigma}}(\E,\F)_{n}\equiv\Spt_{C}^{\Sigma}(\E\otimes\Box^{n}_+,\F).
\end{equation*}
By definition, a $0$-cell of $\hom_{\Spt_{C}^{\Sigma}}(\E,\F)$ is a map $\E\rightarrow\F$.
A $1$-cell is a cubical homotopy $H\colon\E\otimes\Box^{1}_+\rightarrow\F$ from $H\circ (\E\wedge i_{0})$
to $H\circ (\E\wedge i_{1})$ where $i_{0}$ and $i_{1}$ are the two inclusions $\Box^{0}\rightarrow\Box^{1}$.
The $1$-cells generate an equivalence relation on $\Spt_{C}^{\Sigma}(\E,\F)$ and the quotient is 
$\pi_{0}\hom_{\Spt_{C}^{\Sigma}}(\E,\F)$.
Note there is an adjoint functor pair:
\begin{equation*}
\xymatrix{
\E\wedge -\colon\Box\Set_{\ast} \ar@<3pt>[r] & 
\Spt_{C}^{\Sigma}\colon\hom_{\Spt_{C}^{\Sigma}}(\E,-) \ar@<3pt>[l].}
\end{equation*}
Moreover, there exist natural isomorphisms 
\begin{equation*}
\hom_{\Spt_{C}^{\Sigma}}(\E\wedge K,\F)=
\hom_{\Spt_{C}^{\Sigma}}(\E,\F^{K})=
\hom_{\Spt_{C}^{\Sigma}}(\E,\F)^{K}.
\end{equation*}
\end{example}
\begin{example}
\label{example:Cstarsymmetricinternal}
The internal hom of $\CC^{\ast}$-cubical symmetric spectra $\E$ and $\F$ are defined by 
$\underline{\Hom}_{\Spt_{C}^{\Sigma}}(\E,\F)\equiv\underline{\Hom}_{\Sym(C)}(\E,\F)$.
There are natural adjunction isomorphisms
\begin{equation*}
\Spt_{C}^{\Sigma}(\E\wedge\F,\G)=
\Spt_{C}^{\Sigma}\bigl(\E,\underline{\Hom}_{\Spt_{C}^{\Sigma}}(\F,\G)\bigr).
\end{equation*}
In addition, there are natural cubical and internal isomorphisms
\begin{equation*}
\hom_{\Spt_{C}^{\Sigma}}(\E\wedge\F,\G)=
\hom_{\Spt_{C}^{\Sigma}}\bigl(\E,\underline{\Hom}_{\Spt_{C}^{\Sigma}}(\F,\G)\bigr),
\end{equation*}
\begin{equation*}
\underline{\Hom}_{\Spt_{C}^{\Sigma}}(\E\wedge\F,\G)=
\underline{\Hom}_{\Spt_{C}^{\Sigma}}\bigl(\E,\underline{\Hom}_{\Spt_{C}^{\Sigma}}(\F,\G)\bigr).
\end{equation*}
If $\X$ is a pointed cubical $\CC^{\ast}$-space and $\E$ is a $\CC^{\ast}$-cubical symmetric spectrum, 
denote by $\underline{\Hom}_{\Spt_{C}^{\Sigma}}(\X,\E)$ the $\CC^{\ast}$-cubical symmetric spectrum with 
$n$th term the internal hom $\underline{\Hom}(\X,\E_{n})$ with $\Sigma_{n}$-action induced by the action 
on $\E_{n}$. 
Define the $n$th structure map 
$\sigma_{n}\colon C\otimes\underline{\Hom}(\X,\E_{n})\rightarrow\underline{\Hom}(\X,\E_{n+1})$ 
as the adjoint of the composite of the evaluation 
$C\otimes\Ev\colon C\otimes\underline{\Hom}(\X,\E_{n})\otimes\X\rightarrow C\otimes\E_{n}$ 
with the structure map $C\otimes\E_{n}\rightarrow\E_{n+1}$.
With these definitions it follows that $\underline{\Hom}_{\Spt_{C}^{\Sigma}}(\X,-)$ is right adjoint to 
$-\wedge\X$ as endofunctors of $\Spt_{C}^{\Sigma}$.
\end{example}
\begin{example}
\label{example:Cstarsymmetricloops}
The loop $\CC^{\ast}$-cubical symmetric spectrum of $\E$ is 
$\Omega_{C}\E\equiv\underline{\Hom}_{\Spt_{C}^{\Sigma}}(C,\E)$.  
Note that the functor $\Omega_{C}$ is finitely presentable.
\end{example}
\begin{example}
\label{example:Cstarsymmetricshift}
Let $k>0$.
The $\CC^{\ast}$-cubical symmetric spectrum $\E[1]$ has $n$th term $\E_{1+n}$ with $\Sigma_{n}$-action 
given by $1\oplus\sigma_{n}\in\Sigma_{k+n}$.
That is, $1\oplus\sigma_{n}(1)=1$ and $1\oplus\sigma_{n}(i)=1+\sigma_{n}(i-1)$ for $i>1$.
The structure map $C^{p}\otimes\E[1]_{n}\rightarrow\E[1]_{p+n}$ is defined to be the composite of 
$C^{p}\otimes\E_{1+n}\rightarrow\E_{p+1+n}$ with $\sigma_{p,1}\oplus 1$, 
where $\sigma_{p,1}$ is the cyclic permutation of order $p+1$.
Inductively, one defines $\E[k]\equiv\E[k-1][1]$.
\end{example}
\begin{definition}
\label{definition:Cstarsymmetricevaluation}
The $n$th evaluation functor $\Ev_{n}\colon\Spt_{C}^{\Sigma}\rightarrow\Box\CC^{\ast}-\Spc_{0}$ 
sends $\E$ to $\E_{n}$.
Its left adjoint, 
the shift desuspension functor $\Fr_{n}\colon\Box\CC^{\ast}-\Spc_{0}\rightarrow\Spt_{C}^{\Sigma}$ 
is defined by setting $\Fr_{n}\E\equiv\widetilde{\Fr}_{n}\E\otimes^{\Sigma}\Sym(K)$, 
where $\widetilde{\Fr}_{n}\E\equiv(0,\cdots,0,\Sigma[n]_{+}\otimes\E,0,0,\cdots)$.
\end{definition}
\begin{example}
\label{example:Cstarsymmetricanotherdescriptionofinternalhom}
For a $\CC^{\ast}$-cubical symmetric spectrum $\E$, 
the pointed cubical set of maps $\hom_{\Spt_{C}^{\Sigma}}(\Fr_{n}\C,\E)$ is naturally 
$\Sigma_{n}$-equivariant isomorphic to $\hom_{\Box\CC^{\ast}-\Spc_{0}}(\C,\Ev_{n}\E)=\E_{n}$.
In effect,  
$\Fr_{n}\C$ is the $\Sym(C)$-module $\Sym(C)\otimes^{\Sigma}\Sigma[n]_{+}$ and 
$\Sym(C)\otimes^{\Sigma}\Sigma[-]_{+}$ defines a functor $\Sigma^{\op}\rightarrow\Spt_{C}^{\Sigma}$ 
so that $\hom_{\Spt_{C}^{\Sigma}}\bigl(\Sym(C)\otimes^{\Sigma}\Sigma[-]_{+},\E\bigl)$ is the underlying 
$\CC^{\ast}$-cubical symmetric sequence of $\E$.  
In particular, 
$\hom_{\Spt_{C}^{\Sigma}}\bigl(\E\wedge (\Sym(C)\otimes^{\Sigma}\Sigma [-]_{+}),Y\bigr)$ is the underlying 
$\CC^{\ast}$-cubical symmetric sequence of the internal hom $\underline{\Hom}_{\Spt_{C}^{\Sigma}}(\E,Y)$.
\vspace{0.1in}

The point is now to derive an alternate description of the structure maps of $\E$. 
Let $\lambda\colon\Fr_{1}C\rightarrow\Fr_{0}\C$ be the adjoint of the identity map 
$C\rightarrow\Ev_{1}\Fr_{0}\C$ and consider the induced map  
$\underline{\Hom}_{\Spt_{C}^{\Sigma}}(\lambda,\E)\colon
\underline{\Hom}_{\Spt_{C}^{\Sigma}}(\Fr_{0}\C,\E)\rightarrow
\underline{\Hom}_{\Spt_{C}^{\Sigma}}(\Fr_{1}C,\E)$.
By evaluating in level $n$ we get a map 
$\E_{n}\rightarrow\underline{\Hom}_{\Box\CC^{\ast}-\Spc_{0}}(C,\E_{n+1})$ 
which is adjoint to the structure map $\sigma_{n}\colon C\otimes\E_{n}\rightarrow\E_{n+1}$. 
In particular, 
$\underline{\Hom}_{\Spt_{C}^{\Sigma}}(\Fr_{k}\C,\E)$ is the $k$-shift of $\E$; its
underlying symmetric sequence is the sequence of pointed cubical $\CC^{\ast}$-spaces
$\E_k,\E_{1+k},\cdots,\E_{n+k}\cdots$ with $\Sigma_{n}$ acting on $\E_{n+k}$ by restricting the action 
of $\Sigma_{n+k}$ to the copy of $\Sigma_{n}$ that permutes the first $n$ elements of $\overline{n+k}$.  
The structure maps of the $k$-shifted spectrum of $\E$ are the structure maps 
$\sigma_{n+k}\colon C\otimes\E_{n+k}\rightarrow \E_{n+k+1}$.
\end{example}
\begin{example}
\label{example:CstarsymmetricQ}
The adjoint of the $n$th structure map $\sigma_{n}\colon C\otimes\E_{n}\rightarrow\E_{n+1}$ of $\E$ 
yields a map $\tilde{\sigma}_{n}\colon\E_{n}\rightarrow\Omega_{C}\E_{n+1}=\Omega_{C}\E[1]_{n}$ and 
there is an induced map of $\CC^{\ast}$-cubical symmetric spectra $\E\rightarrow\Omega_{C}\E[1]$.

Denote by $(\Theta^{\infty})^{\Sigma}\E$ the colimit of the diagram:
\begin{equation*}
\xymatrix{
\E\ar[r] & \Omega_{C}\E[1]\ar[r] & \Omega_{C}^{2}\E[2]\ar[r] &\cdots }
\end{equation*}

For every $K$ in $\Box\Set_{\ast}$ there is a canonical map  
\begin{equation*}
\xymatrix{
\bigl((\Theta^{\infty})^{\Sigma}\E\bigr)\wedge K\ar[r] & 
(\Theta^{\infty})^{\Sigma}(K\otimes\E)  }
\end{equation*}
so that $(\Theta^{\infty})^{\Sigma}$ is a cubical $\CC^{\ast}$-functor. 
Hence there are induced maps of cubical function complexes 
\begin{equation*}
\xymatrix{
\hom_{\Spt_{C}^{\Sigma}}(\E,\F)\ar[r] & 
\hom_{\Spt_{C}^{\Sigma}}\bigl((\Theta^{\infty})^{\Sigma}\E,(\Theta^{\infty})^{\Sigma}\F\bigr). }
\end{equation*}
\end{example}
\begin{remark}
\label{remark:Cstarsymmetricrightadjoint}
The functor $\Ev_{n}$ has a right adjoint $R_{n}\colon\Box\CC^{\ast}-\Spc_{0}\rightarrow\Spt_{C}^{\Sigma}$.  
Indeed, 
$R_{n}\E\equiv\underline{\Hom}^{\Sigma}(\Sym(C),\widetilde{R}_{n}L)$, 
where $\widetilde{R}_{n}L$ is the $\CC^{\ast}$-cubical symmetric sequence whose $n$th term is the 
cofree $\Sigma_{n}$-object $\E^{\Sigma_{n}}$ and with the terminal object in all other degrees.  
\end{remark}

We are ready to define the level model structures on $\Spt_{C}^{\Sigma}$.

\begin{definition}
\label{definition:Cstarsymmetriclevel}
A map $f\colon\E\rightarrow\F$ between $\CC^{\ast}$-cubical symmetric spectra is a level equivalence 
if $\Ev_{n}f\colon\E_{n}\rightarrow\F_{n}$ is a $\CC^{\ast}$-weak equivalence in 
$\Box\CC^{\ast}-\Spc_{0}$ for every $n\geq 0$.  
And $f$ is a level fibration 
(respectively level cofibration, level acyclic fibration, level acyclic cofibration) 
if $\Ev_{n}f$ is a projective $\CC^{\ast}$-fibration 
(respectively projective cofibration, $\CC^{\ast}$-acyclic projective fibration, 
$\CC^{\ast}$-acyclic projective cofibration) in $\Box\CC^{\ast}-\Spc_{0}$ for every $n\geq 0$.  
A map is a projective cofibration if it has the left lifting property with respect 
to every level acyclic fibration, and an injective fibration if it has the right lifting property 
with respect to every level acyclic cofibration.
\end{definition}

Let $I$ and $J$ denote the generating cofibrations and generating acyclic cofibrations in the 
homotopy invariant model structure on $\Box\CC^{\ast}-\Spc_{0}$.
Set $I_{C}\equiv\bigcup_{n}\Fr_{n}I$ and $J_{C}\equiv\bigcup_{n}\Fr_{n}J$.
We get the next result following the usual script, cf.~\cite[Theorem 8.2]{Hovey:spectra}.

\begin{theorem}
\label{theorem:Cstarsymmetricprojective}
The projective cofibrations, the level fibrations and the level equivalences define a left proper 
combinatorial model structure on $\Spt_{C}^{\Sigma}$.
The cofibrations are generated by $I_{C}$ and the acyclic cofibrations are generated by $J_{C}$. 
\end{theorem}

Note that $\Ev_{n}$ takes level (acyclic) fibrations to (acyclic) fibrations, 
so $\Ev_{n}$ is a right Quillen functor and $\Fr_{n}$ is a left Quillen functor.  
From \cite[Theorem 8.3]{Hovey:spectra} we have:
\begin{theorem}
\label{theorem:Cstarsymmectricprojectivemonoidal}
The category $\Spt_{C}^{\Sigma}$ equipped with its level projective model structure is a 
$\Box\CC^{\ast}-\Spc_{0}$-model category.
\end{theorem}
With some additional work one arrives at the following model structure.
\begin{theorem}
\label{theorem:Cstarsymmetricinjective}
The level cofibrations, the injective fibrations and the level equivalences define a left proper 
combinatorial model structure on $\Spt_{C}^{\Sigma}$.
\end{theorem}
The $\CC^{\ast}$-stable model structures are now only one Bousfield localization away;
the next definition emphasizes the role of $\Omega_{C}$-symmetric spectra as the stably fibrant objects.
We leave the formulation of the injective version to the reader. 
\begin{definition}
\label{definition:stablefibrant}
A level fibrant $\CC^{\ast}$-cubical symmetric spectrum $\G$ is $\CC^{\ast}$-stably fibrant if 
the adjoints $\G_{n}\rightarrow\underline{\Hom}(C,\G_{n+1})$ of the structure maps of $\G$ are 
$\CC^{\ast}$-weak equivalences.
A map $f\colon\E\rightarrow\F$ is a $\CC^{\ast}$-stable weak equivalence if for every 
$\CC^{\ast}$-stably fibrant $\G$ there is an induced weak equivalence of pointed cubical sets
\begin{equation*}
\xymatrix{
\hom_{\Spt_{C}^{\Sigma}}(\QQ f,\G)\colon
\hom_{\Spt_{C}^{\Sigma}}(\QQ\F,\G)\ar[r] &
\hom_{\Spt_{C}^{\Sigma}}(\QQ\E,\G).}
\end{equation*}
\end{definition}
\begin{theorem}
\label{theorem:Cstarsymmetricprojectivestable}
The projective cofibrations and $\CC^{\ast}$-stable weak equivalences define a left proper 
combinatorial monoidal model structure on $\Spt_{C}^{\Sigma}$. 
In this model structure $C\wedge-$ is a Quillen equivalence.
\end{theorem}
\begin{proof}
The existence of the projective $\CC^{\ast}$-stable model structure with these properties
follows from Theorem \ref{theorem:Cstarsymmetricprojective} and results in \cite[\S8]{Hovey:spectra}.
In effect, 
\cite[Theorem 8.8]{Hovey:spectra} shows the $\CC^{\ast}$-stably fibrant objects 
a.k.a.~$\Omega_{C}$-symmetric spectra coincides with the fibrant objects in the stable model structure 
on $\Spt_{C}^{\Sigma}$ obtained from \cite[Theorem 8.10]{Hovey:spectra}.
The latter model structure is defined by localizing the projective level model structure at the 
adjoints $\Fr_{n+1}(C\otimes\X)\rightarrow\Fr_{n+1}(\X)$ of the maps 
$C\otimes\X\rightarrow\Ev_{n+1}\Fr_{n}C$ where $\X$ is either a domain or a codomain of the set of 
generating projective cofibrations $I_{\Box\CC^{\ast}-\Spc_{0}}$.
\end{proof}
\begin{remark}
\label{remark:Cstarsymmetricprojectivestableinvariance}
If $C'$ is a projective cofibrant pointed $\CC^{\ast}$-space weakly equivalent to $C$,
then \cite[Theorem 9.4]{Hovey:spectra} implies there is a Quillen equivalence between the 
projective $\CC^{\ast}$-stable model structures on $\Spt_{C'}^{\Sigma}$ and $\Spt_{C}^{\Sigma}$.
By replacing pointed cubical $\CC^{\ast}$-spaces with pointed simplicial $\CC^{\ast}$-spaces,
but otherwise forming the same constructions, 
we may define the Quillen equivalent to $\Spt_{C}^{\Sigma}$ category of $\CC^{\ast}$-simplicial symmetric spectra.
Our results for $\Spt_{C}^{\Sigma}$,  
in particular Theorem \ref{theorem:Cstarsymmetricprojectivestable}, 
hold verbatim in the simplicial context. 
We leave the formulation of Lemma \ref{lemma:spectraweaklyfinitelygenerated} for symmetric spectra to the reader.
\end{remark}
\vspace{0.1in}

Relying on \cite{Hovey:monoidsandalgebras} we may first refine Theorem \ref{theorem:Cstarsymmetricprojectivestable} 
to categories of modules in $\Spt_{C}^{\Sigma}$.
A monoid $\E$ has a multiplication $\E\wedge\E\rightarrow\E$ and a unit map ${\bf 1}\rightarrow\E$ from the sphere 
$\CC^{\ast}$-spectrum subject to the usual associativity and unit conditions.
It is commutative if the multiplication map is unchanged when composed with the twist isomorphism of $\E\wedge\E$.
The category $\Mod_{\E}$ of modules over a commutative monoid $\E$ is closed symmetric with unit $\E$. 
\begin{theorem}
\label{theorem:modulemodelstructure}
Suppose $\E$ is a cofibrant commutative monoid in $\Spt_{C}^{\Sigma}$.
\begin{itemize}
\item 
The category $\Mod_{\E}$ of $\E$-modules is a left proper combinatorial (and cellular) symmetric monoidal model category with 
the classes of weak equivalences and fibrations defined on the underlying category of $\CC^{\ast}$-cubical symmetric spectra 
(cofibrations defined by the left lifting property).
\item
If $\E\rightarrow\F$ is a $\CC^{\ast}$-stable weak equivalence between cofibrant commutative monoid in $\Spt_{C}^{\Sigma}$,
then the corresponding induction functor yields a Quillen equivalence between the module categories $\Mod_{\E}$ and $\Mod_{\F}$.
\item 
The monoidal Quillen equivalence between $\CC^{\ast}$-cubical and $\CC^{\ast}$-simplicial symmetric spectra yields a
Quillen equivalence between $\Mod_{\E}$ and modules over the image of $\E$ in $\CC^{\ast}$-simplicial symmetric spectra.
\end{itemize}
\end{theorem}

The result for algebras in $\Spt_{C}^{\Sigma}$ provided by \cite{Hovey:monoidsandalgebras} is less streamlined; 
refining the following result to the level of model structures is an open problem.
An $\E$-algebra is a monoid in $\Mod_{\E}$.
\begin{theorem}
\label{theorem:algebramodelstructure}
Suppose $\E$ is a cofibrant commutative monoid in $\Spt_{C}^{\Sigma}$.
\begin{itemize}
\item 
The category $\Alg_{\E}$ of $\E$-algebras comprised of monoids in $\Mod_{\E}$ equipped with the classes of 
weak equivalences and fibrations defined on the underlying category of $\CC^{\ast}$-cubical symmetric spectra 
(cofibrations defined by the left lifting property) is a semimodel category in the following sense:
(1) {\bf CM} 1-{\bf CM} 3 holds, 
(2) Acyclic cofibrations whose domain is cofibrant in $\Mod_{\E}$ have the left lifting property with respect to fibrations,
and
(3) Every map whose domain is cofibrant in $\Mod_{\E}$ factors functorially into a cofibration followed by an acyclic fibration
and as an acyclic cofibration followed by a fibration.
Moreover,
cofibrations whose domain is cofibrant in $\Mod_{\E}$ are cofibrations in $\Mod_{\E}$, 
and (acyclic) fibrations are closed under pullback.
\item
The homotopy category $\Ho(\Alg_{\E})$ obtained from $\Alg_{\E}$ by inverting the weak equivalences is equivalent to the full subcategory 
of cofibrant and fibrant $\E$-algebras modulo homotopy equivalence.
\item
If $\E\rightarrow\F$ is a $\CC^{\ast}$-stable weak equivalence between cofibrant commutative monoid in $\Spt_{C}^{\Sigma}$,
then the corresponding induction functor yields an equivalence between $\Ho(\Alg_{\E})$ and $\Ho(\Alg_{\F})$.
\item 
The monoidal Quillen equivalence between $\CC^{\ast}$-cubical and $\CC^{\ast}$-simplicial symmetric spectra yields an equivalence between 
the homotopy categories of $\Alg_{\E}$ and of algebras over the image of $\E$ in $\CC^{\ast}$-simplicial symmetric spectra.
\item
$\Ho(\Alg_{\E})$ acquires an action by the homotopy category of simplicial sets.
\end{itemize}
\end{theorem}
\begin{remark}
For semimodel structures at large see the work of Spitzweck \cite{Spitzweck:semimodelstructures}.
\end{remark}

It turns out there is a perfectly good homotopical comparison between $\Spt_{C}$ and $\CC^{\ast}$-cubical symmetric spectra.
The following result is special to our situation.
\begin{theorem}
\label{theorem:spectrasymmetricspectraprojectivestable}
The forgetful functor induces a Quillen equivalence between the projective 
$\CC^{\ast}$-stable model structures on $\Spt_{C}^{\Sigma}$ and $\Spt_{C}$.
\end{theorem}
\begin{proof}
By \cite[Theorem 10.1]{Hovey:spectra} there exists a zig-zag of Quillen equivalences between
$\Spt_{C}^{\Sigma}$ and $\Spt_{C}$ since the cyclic permutation condition holds for $C$ by 
Lemma \ref{lemma:cyclicpermutationcondition}.
The improved result stating that the forgetful and symmetrization adjoint functor pair between $\Spt_{C}^{\Sigma}$ and $\Spt_{C}$
defines a Quillen equivalence is rather long and involves bispectra and layer filtrations.
We have established the required ingredients needed for the arguments in \cite[\S4.4]{Jardine:MSS} to go through in our setting.
\end{proof}

The next result concerning monoidalness of the stable $\CC^{\ast}$-homotopy category is a consequence of
Theorem \ref{theorem:spectrasymmetricspectraprojectivestable}.
\begin{corollary}
\label{corollary:stablehomotopycategorymonoidalstructure}
The total left derived functor of the smash product $\wedge$ on $\Spt_{C}^{\Sigma}$ yields a symmetric monoidal product $\wedge^{{\bf L}}$ 
on the stable homotopy category $\SH^{\ast}$.
In addition, 
the suspension functor $\Fr_{0}$ induces a symmetric monoidal functor
\begin{equation*}
\xymatrix{
\bigl(\HH^{\ast},\otimes^{{\bf L}},\C\bigr)
\ar[r] &
\bigl(\SH^{\ast},\wedge^{{\bf L}},{\bf 1}\bigr).}
\end{equation*}
\end{corollary}

With the results for $\Spt_{C}^{\Sigma}$ in hand we are ready to move deeper into our treatment of the triangulated structure of $\SH^{\ast}$.
The notion we are interested in is that of a closed symmetric monoidal category with a compatible triangulation,
as introduced in \cite{May:additivityoftraces}.
The importance of this notion is evident from the next theorem which is a consequence of our results for $\Spt_{C}^{\Sigma}$ and specialization 
of the main result in \cite{May:additivityoftraces}.
\begin{theorem}
\label{additivityofeulercharacteristics}
The Euler characteristic is additive for distinguished triangles of dualizable objects in $\SH^{\ast}$.
\end{theorem}

Next we define Euler characteristics and discuss the content of Theorem \ref{additivityofeulercharacteristics}.
\vspace{0.1in}

The dual of $\E$ in $\SH^{\ast}$ is $\D\E\equiv\underline{\SH^{\ast}}(\E,{\bf 1})$ where $\underline{\SH^{\ast}}(\E,\F)$ is the derived version 
of the internal hom of $\E$ and $\F$ in $\Spt_{C}^{\Sigma}$,
i.e.~take a cofibrant replacement $\E^{c}$ of $\E$ a fibrant replacement $\F^{f}$ of $\F$ and form $\underline{\Hom}_{\Spt_{C}^{\Sigma}}(\E^{c},\F^{f})$.
Let $\epsilon_{\E}\colon\D\E\wedge^{{\bf L}}\E\rightarrow {\bf 1}$ denote the evident evaluation map.
There is a canonical map 
\begin{equation}
\label{dualizablemap}
\xymatrix{
\D\E\wedge^{{\bf L}}\E\ar[r] &
\underline{\SH^{\ast}}(\E,\E). }
\end{equation}
Recall that $\E$ is called dualizable if (\ref{dualizablemap}) is an isomorphism.
For $\E$ dualizable there is a coevaluation map $\eta_{\E}\colon {\bf 1}\rightarrow\E\wedge^{{\bf L}}\D\E$.
The Euler characteristic $\chi(\E)$ of $\E$ is defined as the composite map
\begin{equation}
\label{eulercharacteristics}
\xymatrix{
{\bf 1}\ar[r]^-{\eta_{\E}} &
\E\wedge^{{\bf L}}\D\E\ar[r]^-{\tau} &
\D\E\wedge^{{\bf L}}\E\ar[r]^-{\epsilon_{\E}} &
{\bf 1}. }
\end{equation}
Here,
$\tau$ denotes the twist map.
The categorical definition of Euler characteristics given above and the generalization reviewed below, 
putting trace maps in algebra and topology into a convenient framework, 
was introduced by Dold and Puppe \cite{DP}.
Theorem \ref{additivityofeulercharacteristics} states that for every distinguished triangle 
\begin{equation*}
\xymatrix{
\E\ar[r] & \F\ar[r] & \G\ar[r] & \Sigma_{S^{1}}\E }
\end{equation*}
of dualizable objects in $\SH^{\ast}$, 
the formula
\begin{equation}
\label{eulercharacteristicsformula}
\chi(\F)=\chi(\E)+\chi(\G)
\end{equation}
holds for the Euler characteristics (\ref{eulercharacteristics}) in the endomorphism ring of the sphere $\CC^{\ast}$-spectrum.
Note that if $\E$ and $\F$ are dualizable, then so is $\G$.
As emphasized in \cite{May:additivityoftraces}, 
the proof of the additivity theorem for Euler characteristics makes heavily use of the stable model categorical situation, 
so that a generalization of the formula (\ref{eulercharacteristicsformula}) to arbitrary triangulated categories seems a bit unlikely.
In order to explain this point in some details we shall briefly review the important notion of a closed symmetric monoidal category with 
a compatible triangulation in the sense specified by May \cite{May:additivityoftraces}. 
\begin{remark}
In our treatment of zeta functions of $\CC^{\ast}$-algebras in Section \ref{subsection:zetafunctions} we shall make use of Euler characteristics in
``rationalized'' stable homotopy categories.
\end{remark}

If $\E$ is dualizable as above and there is a  ``coaction'' map $\Delta_{\E}\colon \E\rightarrow\E\wedge^{{\bf L}} C_{\E}$ for some object $C_{\E}$ 
of $\SH^{\ast}$, 
typically arising form a comonoid structure on $C_{\E}$, 
define the trace $\tr(f)$ of a self map $f$ of $\E$ by the diagram:
\begin{equation*}
\xymatrix{ 
{\bf 1} \ar[r]^-{\eta_{\E}} \ar[d]_{\tr(f)} & 
\E\wedge^{{\bf L}}\D\E \ar[r]^-{\tau} & 
\D\E\wedge^{{\bf L}}\E \ar[d]^{\D(f)\wedge^{{\bf L}}\Delta_{\E}} \\
C_{\E} & 
{\bf 1}\wedge^{{\bf L}} C_{\E}\ar[l]_-{\cong} & 
\D\E\wedge^{{\bf L}}\E\wedge^{{\bf L}} C_{\E} \ar[l]_-{\epsilon_{\E}\wedge^{{\bf L}}\id} }
\end{equation*}
For completeness we include some of the basic properties of trace maps proven in \cite{May:additivityoftraces}.
Define a map
$$
(f,\alpha)\colon (\E,\Delta_{\E})\rightarrow(\F,\Delta_{\F})
$$
to consist of a pair of maps $f\colon\E\rightarrow\F$ and $\alpha\colon C_{\E}\rightarrow C_{\F}$ such that the following diagram commutes:
\begin{equation*}
\xymatrix{
\E\ar[r]^-{\Delta_{\E}}\ar[d]_-{f} & \E\wedge^{{\bf L}} C_{\E} \ar[d]^-{f\wedge^{{\bf L}}\alpha}\\
\F\ar[r]^-{\Delta_{\F}} & \F\wedge^{{\bf L}} C_{\F} }
\end{equation*}
\begin{lemma}
\label{tracelemma}
The trace satisfies the following properties, where $\E$ and $\F$ are dualizable
and $\Delta_{\E}$ and $\Delta_{\F}$ are given.
\begin{itemize}
\item 
If $f$ is a self map of the sphere $\CC^{\ast}$-spectrum, then $\chi(f)=f$.
\item
If $(f,\alpha)$ is a self map of $(\E,\Delta_{\E})$, then $\alpha\circ \tr(f)=\tr(f)$.
\item
If $\E\stackrel{i}{\rightarrow}\F\stackrel{r}{\rightarrow}\E$ is a retract, 
$f$ a self map of $\E$, 
and $(i,\alpha)$ a map $(\E,\Delta_{\E})\rightarrow (\F,\Delta_{\F})$, 
then $\alpha\circ\tr(f)=\tr(i\circ f\circ r)$.
\item
If $f$ and $g$ are self maps of $\E$ and $\F$ respectively, 
then $\tr(f\wedge^{{\bf L}} g)=\tr(f)\wedge^{{\bf L}} \tr(g)$,
where 
$\Delta_{\E\wedge^{{\bf L}}\F}= 
(\text{id}\wedge^{{\bf L}}\tau\wedge^{{\bf L}}\text{id})\circ (\Delta_{\E}\wedge^{{\bf L}} \Delta_{\F})$ 
with $\tau$ the transposition.
\item
If $h\colon\E\vee\F\rightarrow\E\vee\F$ induces $f\colon\E\rightarrow\E$ and $g\colon\F\rightarrow\F$ by inclusion and retraction, 
then $\tr(h)=\tr(f) + \tr(g)$,
where $C_{\E} = C_{\F} = C_{\E\vee \F}$ and $\Delta_{\E\vee \F} = \Delta{\E} \vee \Delta{\F}$.
\item
For every self map $f$, 
$\tr(\Sigma_{S^{1}} f)=-\tr(f)$,
where $\Delta{\Sigma_{S^{1}} \E} = \Sigma_{S^{1}} \Delta{\E}$.
\end{itemize}
\end{lemma}

The following additivity theorem was shown by May in \cite[Theorem 1.9]{May:additivityoftraces}.
\begin{theorem} 
\label{additivitytheorem}
Let $\E$ and $\F$ be dualizable in $\SH^{\ast}$, 
$\Delta{\E}$ and $\Delta{\F}$ be given, where $C=C_{\E}=C_{\F}$. 
Let $(f,\text{id})$ be a map $(\E,\Delta{\E})\rightarrow (\F,\Delta{\F})$ and extend $f$ to a distinguished triangle
\begin{equation*}
\xymatrix{
\E\ar[r]^-{f}  & \F\ar[r]^-{g} & \G \ar[r]^-{h} & \Sigma_{S^{1}}\E. }
\end{equation*}
Assume given maps $\phi$ and $\psi$ that make the left square commute in the
first of the following two diagrams:
\begin{equation*}
\xymatrix{
\E \ar[r]^-{f} \ar[d]_{\phi}  &  \F \ar[r]^-{g} \ar[d]^{\psi}
& \G \ar[r]^-{h} \ar[d]^{\omega} & \Sigma_{S^{1}} \E \ar[d]^{\Sigma_{S^{1}}\phi}\\
\E \ar[r]^-{f}  &  \F \ar[r]^-{g} & \G \ar[r]^-{h} & \Sigma_{S^{1}} \E }
\end{equation*}
\begin{equation*}
\xymatrix{
\E \ar[r]^-{f}\ar[d]_{\Delta{\E}}  &  \F \ar[r]^-{g} \ar[d]^{\Delta{\F}}
& \G \ar[r]^-{h} \ar[d]^{\Delta\G} & \Sigma_{S^{1}} \E \ar[d]^{\Sigma_{S^{1}}\Delta{\E}}\\
\E\wedge^{{\bf L}} C \ar[r]^-{f\wedge^{{\bf L}}\text{id}}  &  \F\wedge^{{\bf L}} C \ar[r]^-{g\wedge^{{\bf L}}\text{id}}
& \G\wedge^{{\bf L}} C \ar[r]^-{h\wedge^{{\bf L}}\text{id}} & \Sigma_{S^{1}} (\E\wedge^{{\bf L}} C) }
\end{equation*}
Then there are maps $\omega$ and $\Delta\G$ as indicated above rendering the diagrams commutative and
\begin{equation*}
\tr(\psi)=\tr(\phi)+\tr(\omega).
\end{equation*}
\end{theorem}

Additivity of Euler characteristics follows from this theorem by starting out with the data of a distinguished triangle.
\vspace{0.1in}

The proof of Theorem \ref{additivitytheorem} uses the fact that $\SH^{\ast}$ is the homotopy category of a closed symmetric monoidal 
stable model structure such that the smash product $\wedge^{{\bf L}}$ is compatible with the triangulated structure in the sense made
precise by the axioms (TC1)-(TC5) stated in \cite[\S4]{May:additivityoftraces}.

The axiom (TC1) asserts there exists a natural isomorphism $\alpha\colon\E\wedge^{{\bf L}} S^{1}\rightarrow\Sigma_{S^{1}}\E$ such that the 
composite map
\begin{equation*}
\xymatrix{
\Sigma_{S^{1}}S^{1}\ar[r]^-{\alpha^{-1}} &
S^{1}\wedge^{{\bf L}} S^{1}\ar[r]^-{\tau} &
S^{1}\wedge^{{\bf L}} S^{1}\ar[r]^-{\alpha} &
\Sigma_{S^{1}}S^{1} }
\end{equation*}
is multiplication by $-1$, 
while (TC2) basically asserts that smashing or taking internal hom objects with every object of $\SH^{\ast}$ preserves distinguished triangles.
These axioms are analogs of the elementary axioms (T1), (T2) for a triangulated category, and are easily verified.
\vspace{0.1in}

Next we formulate the braid axiom (TC3):
Suppose there exist distinguished triangles
\begin{equation*}
\xymatrix{
\E\ar[r]^-{f} &
\F\ar[r]^-{g} &
\G\ar[r]^-{h} &
\Sigma_{S^{1}}\E,}
\end{equation*}
and
\begin{equation*}
\xymatrix{
\E'\ar[r]^-{f'} &
\F'\ar[r]^-{g'} &
\G'\ar[r]^-{h'} &
\Sigma_{S^{1}}\E'.}
\end{equation*}
Then there exist distinguished triangles
\begin{equation*}
\xymatrix{
\F\wedge^{{\bf L}}\E' \ar[r]^-{p_{1}} &
\HHH\ar[r]^-{q_{1}} &
\E\wedge^{{\bf L}}\G'\ar[r]^-{f\wedge^{{\bf L}} h'} &
\Sigma_{S^{1}}(\F\wedge^{{\bf L}}\E'),}
\end{equation*}
\begin{equation*}
\xymatrix{
\Sigma_{S^{1}}^{-1}(\G\wedge^{{\bf L}}\G')\ar[r]^-{p_{2}} &
\HHH\ar[r]^-{q_{2}} &
\F\wedge^{{\bf L}}\F' \ar[r]^-{-g\wedge^{{\bf L}} g'} &
\G\wedge^{{\bf L}}\G', }
\end{equation*}
\begin{equation*}
\xymatrix{
\E\wedge^{{\bf L}}\F' \ar[r]^-{p_{3}} &
\HHH'\ar[r]^-{q_{3}} &
\G\wedge^{{\bf L}}\E'\ar[r]^-{h\wedge^{{\bf L}} f'} &
\Sigma_{S^{1}}(\E\wedge^{{\bf L}}\F'),}
\end{equation*}
such that the following diagrams commute:
\begin{equation}
\label{braiddiagram}
\xymatrix@C=17pt@R=10pt{
\Sigma_{S^{1}}^{-1}(\F\wedge^{{\bf L}}\G')
\ar@/_1pc/[ddd]_-{\Sigma_{S^{1}}^{-1}(\id\wedge^{{\bf L}} h')} 
\ar[rrrddd]^(0.7){\Sigma_{S^{1}}^{-1}(g\wedge^{{\bf L}}\id)} 
&&&
\E\wedge^{{\bf L}}\E' 
\ar[lllddd]_(0.3){f\wedge^{{\bf L}}\id} 
\ar[rrrddd]^(0.3){\id\wedge^{{\bf L}} f'} 
&&&
\Sigma_{S^{1}}^{-1}(\G\wedge^{{\bf L}}\F')
\ar[lllddd]_(0.7){\Sigma_{S^{1}}^{-1}(\id\wedge^{{\bf L}} g')} 
\ar@/^1pc/[ddd]^-{\Sigma_{S^{1}}^{-1}(h\wedge^{{\bf L}}\id)} 
\\ \\ \\
\F\wedge^{{\bf L}}\E' 
\ar@/_1pc/[dddddd]_-{g\wedge^{{\bf L}}\id} 
\ar[rrrdddddd]^(0.4){\id\wedge^{{\bf L}} f'}
\ar[rrrddd]^(0.4){p_{1}}
&&&
\Sigma_{S^{1}}^{-1}(\G\wedge^{{\bf L}}\G') 
\ar[llldddddd]_(0.7){\Sigma_{S^{1}}^{-1}(\id\wedge^{{\bf L}} h')}
\ar[ddd]^-{p_{2}}
\ar[rrrdddddd]^(0.7){\Sigma_{S^{1}}^{-1}(h\wedge^{{\bf L}}\id)}
&&&
\E\wedge^{{\bf L}}\F' 
\ar[lllddd]_(0.4){p_{3}}
\ar[llldddddd]_(0.4){f\wedge^{{\bf L}}\id}
\ar@/^1pc/[dddddd]^-{\id\wedge^{{\bf L}} g'} 
\\ \\ \\
&&& 
\HHH
\ar[lllddd]_(0.7){q_{3}} 
\ar[ddd]^-{q_{2}}
\ar[rrrddd]^(0.7){q_{1}}
\\ \\ \\
\G\wedge^{{\bf L}}\E'
\ar@/_1pc/[ddd]_-{\id\wedge^{{\bf L}} f'} 
\ar[rrrddd]^(0.7){h\wedge^{{\bf L}}\id} 
&&&
\F\wedge^{{\bf L}}\F' 
\ar[lllddd]_(0.3){g\wedge^{{\bf L}}\id} 
\ar[rrrddd]^(0.3){\id\wedge^{{\bf L}} g'} 
&&&
\E\wedge^{{\bf L}}\G' 
\ar[lllddd]_(0.7){-\id\wedge^{{\bf L}} h'} 
\ar@/^1pc/[ddd]^-{f\wedge^{{\bf L}}\id} 
\\ \\ \\
\G\wedge^{{\bf L}}\F'
&&&
\Sigma_{S^{1}}(\E\wedge^{{\bf L}}\E')
&&&
\F\wedge^{{\bf L}}\G' }
\end{equation}
The axiom (TC3) is more complicated than (TC1) and (TC2) in that it involves a simultaneous use of smash products and internal hom objects, 
a.k.a.~desuspensions.
If $\E$ and $\E'$ are subobjects of $\F$ and $\F'$ then $\HHH$ is typically the pushout of $\E\wedge^{{\bf L}}\F'$ and $\F\wedge^{{\bf L}}\E'$
along $\E\wedge^{{\bf L}}\E'$,
$q_{2}\colon\HHH\rightarrow\F\wedge^{{\bf L}}\F'$ the evident inclusion,
while $q_{1}\colon\HHH\rightarrow\E\wedge^{{\bf L}}\G'$ and $q_{2}\colon\HHH\rightarrow\G\wedge^{{\bf L}}\E'$ are obtained by quotienting out by 
$\F\wedge^{{\bf L}}\E'$ and $\E\wedge^{{\bf L}}\F'$ respectively.
For an interpretation of axiom (TC3) in terms of Verdier's axiom (T3) for a triangulated category we refer to \cite{May:additivityoftraces}.
There is an equivalent way of formulating axiom (TC3) which asserts the existence of distinguished triangles
\begin{equation*}
\xymatrix{
\E\wedge^{{\bf L}}\G' \ar[r]^-{r_{1}} &
\HHH'\ar[r]^-{s_{1}} &
\G\wedge^{{\bf L}}\G'\ar[r]^-{h\wedge^{{\bf L}} g'} &
\Sigma_{S^{1}}(\E\wedge^{{\bf L}}\G'),}
\end{equation*}
\begin{equation*}
\xymatrix{
\G\wedge^{{\bf L}}\G')\ar[r]^-{r_{2}} &
\HHH'\ar[r]^-{s_{2}} &
\Sigma_{S^{1}}(\E\wedge^{{\bf L}}\E') \ar[r]^-{-\Sigma_{S^{1}}(f\wedge^{{\bf L}} f')} &
\Sigma_{S^{1}}(\G\wedge^{{\bf L}}\G'), }
\end{equation*}
\begin{equation*}
\xymatrix{
\G\wedge^{{\bf L}}\E' \ar[r]^-{r_{3}} &
\HHH'\ar[r]^-{s_{3}} &
\F\wedge^{{\bf L}}\G'\ar[r]^-{g\wedge^{{\bf L}} h'} &
\Sigma_{S^{1}}(\G\wedge^{{\bf L}}\E'),}
\end{equation*}
such that the diagram (\ref{braiddiagram}) corresponding to the distinguished triangles $(-\Sigma_{S^{1}}^{-1}h,f,g)$ and 
$(-\Sigma_{S^{1}}^{-1}h',f',g')$ commutes \cite[Lemma 4.7]{May:additivityoftraces}. 
This axiom is called (TC3')
\vspace{0.1in}

The additivity axiom (TC4) concerns compatibility of the maps $q_{i}$ and $r_{i}$ in the sense that there is a weak pushout 
and weak pullback diagram:
\begin{equation*}
\xymatrix{
\HHH\ar[r]^-{q_{2}}\ar[d]_-{(q_{1},q_{3})} & \F\wedge^{{\bf L}}\F'\ar[d]^-{r_{2}} \\
(\E\wedge^{{\bf L}}\G')\vee (\G\wedge^{{\bf L}}\E')\ar[r]^-{(r_{1},r_{3})} & \HHH' }
\end{equation*}
In particular, 
$r_{2}\circ q_{2}=r_{1}\circ q_{1}+r_{3}\circ q_{3}$.
Recall that weak limits and weak colimits satisfy the existence but not necessarily the uniqueness part in the defining universal 
property of limits and colimits respectively.
We refer to \cite{May:additivityoftraces} for the precise definition of the subtle braid duality axiom (TC5) involving $\D\E$, $\D\F$ and $\D\G$, 
and the duals of the diagrams appearing in the axioms (TC3) and (TC3').
Assuming axioms (TC1)-(TC5), 
additivity of Euler characteristics is shown in \cite[\S4]{May:additivityoftraces}. 
\vspace{0.1in}

We shall leave the straightforward formulations of the corresponding base change and also the equivariant generalizations of the results in this 
section to the interested reader, 
and refer to \cite{KN} for further developments on the subject of May's axioms.
\newpage

\subsection{$\CC^{\ast}$-functors}
\label{subsection:Cstarfunctors}
The purpose of this section is to construct a convenient and highly structured enriched functor model for the stable $\CC^{\ast}$-homotopy category.
Although there are several recent works on this subject, 
none of the existing setups of enriched functor categories as models for homotopy types apply directly to the stable $\CC^{\ast}$-homotopy category. 
More precisely, 
this subject was initiated with \cite{Lydakis} and vastly generalized in \cite{DRO:general} with the purpose of including examples arising in 
algebraic geometry. 
A further development of the setup is given in \cite{Biedermann} and \cite{BCR}.
Dealing effectively with enriched functors in stable $\CC^{\ast}$-homotopy theory in its present state of the art requires some additional input, 
which in turn is likely to provide a broader range of applications in homotopy theory at large.
The algebro-geometric example of motivic functors in \cite{DRO:motivic} has been pivotal in the construction of a homotopy theoretic model for 
motives \cite{RO1}, \cite{RO2}.
\vspace{0.1in}

For background in enriched category theory we refer to \cite{Kelly}.
We shall be working with the closed symmetric monoidal category $\Box\CC^{\ast}-\Spc$ of cubical $\CC^{\ast}$-spaces relative to 
some (essentially small) symmetric monoidal $\Box\CC^{\ast}-\Spc$-subcategory $\f\Box\CC^{\ast}-\Spc$.
Denote by $[\f\Box\CC^{\ast}-\Spc,\Box\CC^{\ast}-\Spc]$ the $\Box\CC^{\ast}-\Spc$-category of $\Box\CC^{\ast}-\Spc$-functors from 
$\f\Box\CC^{\ast}-\Spc$ to $\Box\CC^{\ast}-\Spc$ equipped with the projective homotopy invariant model structure.
It acquires the structure of a closed symmetric monoidal category \cite{Day:closedfunctorsI}.
Every object $\X$ of $\f\Box\CC^{\ast}-\Spc$ represents a $\Box\CC^{\ast}-\Spc$-functor which we, 
by abuse of notation, 
denote by $\Box\CC^{\ast}-\Spc(\X,-)$.
\begin{theorem}
\label{theorem:enrichedprojectivemodel}
There exists a pointwise model structure on $[\f\Box\CC^{\ast}-\Spc,\Box\CC^{\ast}-\Spc]$ defined by declaring $\SSSS\rightarrow\TTTT$ 
is a pointwise fibration or weak equivalence if $\SSSS(\X)\rightarrow\TTTT(\X)$ is so in $\Box\CC^{\ast}-\Spc$ for every member $\X$ of
$\f\Box\CC^{\ast}-\Spc$.
The pointwise model structure is combinatorial and left proper.
The cofibrations are generated by the set consisting of the maps 
\begin{equation*}
f\otimes \Box\CC^{\ast}-\Spc(\X,-)
\end{equation*}
where $f$ runs through the generating cofibrations of $\Box\CC^{\ast}-\Spc$ and $\X$ through the objects of $\f\Box\CC^{\ast}-\Spc$.
Likewise, 
the acyclic cofibrations are generated by the set consisting of maps of the form 
\begin{equation*}
g\otimes \Box\CC^{\ast}-\Spc(\X,-)
\end{equation*}
where $g$ runs through the generating acyclic cofibrations of $\Box\CC^{\ast}-\Spc$.
\end{theorem}
\begin{proof}
Most parts of the proof is a standard application of Kan's recognition lemma \cite[Theorem 2.1.19]{Hovey:Modelcategories}.
The required smallness assumption in that result holds because $[\f\Box\CC^{\ast}-\Spc,\Box\CC^{\ast}-\Spc]$ is locally presentable.
It is clear that the pointwise weak equivalences satisfy the two-out-of-three axiom and are closed under retracts.
Let $\W_{\ppt}$ denote the class of pointwise weak equivalences.
An evident adjunction argument shows
\begin{equation*}
\{f\otimes \Box\CC^{\ast}-\Spc(\X,-)\}-\text{inj}=
\{g\otimes \Box\CC^{\ast}-\Spc(\X,-)\}-\text{inj}\cap
\W_{\ppt}.
\end{equation*}
It remains to show there is an inclusion
\begin{equation*}
\{g\otimes \Box\CC^{\ast}-\Spc(\X,-)\}-\text{cell}\subseteq
\{f\otimes \Box\CC^{\ast}-\Spc(\X,-)\}-\text{cof}\cap
\W_{\ppt}.
\end{equation*}
We note that it suffices to show maps of $\{f\otimes\underline{\Hom}(\X,\Y)\}$-cell are weak equivalences in $\Box\CC^{\ast}-\Spc$, 
where $f$ is a generating acyclic cofibration and $\X$, $\Y$ are objects of $\f\Box\CC^{\ast}-\Spc$:
Using the inclusion $\{g\otimes \Box\CC^{\ast}-\Spc(\X,-)\}-\text{cof}\subseteq\{f\otimes \Box\CC^{\ast}-\Spc(\X,-)\}-\text{cof}$ it suffices 
to show that maps of $\{g\otimes \Box\CC^{\ast}-\Spc(\X,-)\}-\text{cell}$ are pointwise weak equivalences.
Since colimits in $[\f\Box\CC^{\ast}-\Spc,\Box\CC^{\ast}-\Spc]$ are formed pointwise this follows immediately from the statement about maps of
$\{f\otimes\underline{\Hom}(\X,\Y)\}$-cell.
To prove the remaining claim we shall employ the injective homotopy invariant model structure.
Recall the weak equivalences in the injective model structure coincides with the weak equivalences in the projective model structure, 
but an advantage of the former is that $\underline{\Hom}(\X,\Y)$ is cofibrant.
Thus every map $f\otimes\underline{\Hom}(\X,\Y)$ as above is an acyclic cofibration in the injective homotopy invariant model structure, 
and hence the same holds on the level of cells; 
in particular, these maps are $\CC^{\ast}$-weak equivalences.

Left properness follows provided cofibrations in $[\f\Box\CC^{\ast}-\Spc,\Box\CC^{\ast}-\Spc]$ are pointwise cofibrations in the (left proper) 
injective homotopy invariant model structure.
To prove this we note that the generating cofibrations $f\otimes \Box\CC^{\ast}-\Spc(\X,-)$ are pointwise cofibrations, 
so that every cofibration is a pointwise cofibration.
\end{proof}

For every object $\X$ of $\f\Box\CC^{\ast}-\Spc$ the functor $-\otimes \Box\CC^{\ast}-\Spc(\X,-)$ is a left Quillen functor because evaluating 
at $\X$ clearly preserves fibrations and acyclic fibrations.
There is an evident pairing
\begin{equation}
\label{equation:pairing}
\xymatrix{
\Box\CC^{\ast}-\Spc
\times
[\f\Box\CC^{\ast}-\Spc,\Box\CC^{\ast}-\Spc]
\ar[r] &
[\f\Box\CC^{\ast}-\Spc,\Box\CC^{\ast}-\Spc]. }
\end{equation}
\begin{lemma}
The pairing (\ref{equation:pairing}) is a Quillen bifunctor with respect to the pointwise model structure on 
$[\f\Box\CC^{\ast}-\Spc,\Box\CC^{\ast}-\Spc]$ and the projective homotopy invariant model structure on $\Box\CC^{\ast}-\Spc$.
\end{lemma}
\begin{proof}
For the pushout product of $\SSSS\rightarrow\TTTT$ and $h\otimes \Box\CC^{\ast}-\Spc(\X,-)$ there is a natural isomorphism
\begin{equation*}
(\SSSS\rightarrow\TTTT)
\Box 
\bigl(h\otimes\Box\CC^{\ast}-\Spc(\X,-)\bigr)=
\bigl((\SSSS\rightarrow\TTTT)\Box h\bigr)\otimes\Box\CC^{\ast}-\Spc(\X,-). 
\end{equation*}
Since $\Box\CC^{\ast}-\Spc$ is monoidal by Proposition \ref{proposition:hiprojectivemonoidal} it follows that 
$\bigl((\SSSS\rightarrow\TTTT)\Box h\bigr)$ is a cofibration and an acyclic cofibration if either $\SSSS\rightarrow\TTTT$ or 
$h\otimes \Box\CC^{\ast}-\Spc(\X,-)$, 
and hence $h$, 
is so.
This finishes the proof because $-\otimes \Box\CC^{\ast}-\Spc(\X,-)$ is a left Quillen functor.
\end{proof}

As we have noted repeatedly the next type of result is imperative for a highly structured model structure.
\begin{lemma}
The pointwise model structure on $[\f\Box\CC^{\ast}-\Spc,\Box\CC^{\ast}-\Spc]$ is monoidal.
\end{lemma}
\begin{proof}
The inclusion of $\f\Box\CC^{\ast}-\Spc$ into $\Box\CC^{\ast}-\Spc$ is a $\Box\CC^{\ast}-\Spc$-functor and the unit of 
$[\f\Box\CC^{\ast}-\Spc,\Box\CC^{\ast}-\Spc]$.
It is cofibrant because the unit of $\Box\CC^{\ast}-\Spc$ is cofibrant.
The natural isomorphism
\begin{equation*}
\bigl(f\otimes\Box\CC^{\ast}-\Spc(\X,-)\bigr)
\Box 
\bigl(g\otimes\Box\CC^{\ast}-\Spc(\Y,-)\bigr)=
(f\Box g)\otimes\Box\CC^{\ast}-\Spc(\X\otimes\Y,-)
\end{equation*}
combined with the facts that $\Box\CC^{\ast}-\Spc$ is monoidal and $-\otimes \Box\CC^{\ast}-\Spc(\X\otimes\Y,-)$ is a left Quillen functor
finishes the proof.
\end{proof}

The following model structure on $\CC^{\ast}$-functors takes into account that $\f\Box\CC^{\ast}-\Spc$ has homotopical content in the form of 
weak equivalences (as a full subcategory of $\Box\CC^{\ast}-\Spc$).
A homotopy $\CC^{\ast}$-functor is an object of $[\f\Box\CC^{\ast}-\Spc,\Box\CC^{\ast}-\Spc]$ which preserves weak equivalences.
What we shall do next is localize the pointwise model structure in such a way that the fibrant objects in the localized model structure are 
precisely the pointwise fibrant homotopy $\CC^{\ast}$-functors.
It will be convenient to let $\X\rightarrow\X'$ denote a generic weak equivalence in $\f\Box\CC^{\ast}-\Spc$. 
\begin{theorem}
There is a homotopy functor model structure on $[\f\Box\CC^{\ast}-\Spc,\Box\CC^{\ast}-\Spc]$ with fibrant objects the pointwise fibrant 
homotopy $\CC^{\ast}$-functors and cofibrations the cofibrations in the pointwise model structure.
The homotopy functor model structure is combinatorial and left proper.
\end{theorem}
\begin{proof}
The existence and the properties of the homotopy functor model structure follow by observing that a pointwise fibrant $\CC^{\ast}$-functor
$\SSSS$ is a homotopy $\CC^{\ast}$-functor if and only if the naturally induced map of simplicial sets
\begin{equation*}
\xymatrix{
& \hom_{[\f\Box\CC^{\ast}-\Spc,\Box\CC^{\ast}-\Spc]}(\Y\otimes\Box\CC^{\ast}-\Spc(\X',-),\SSSS) \\
\ar[r] &
\hom_{[\f\Box\CC^{\ast}-\Spc,\Box\CC^{\ast}-\Spc]}(\Y\otimes\Box\CC^{\ast}-\Spc(\X,-),\SSSS) }
\end{equation*}
is a weak equivalence for every domain and codomain $\Y$ of the generating cofibrations of $\Box\CC^{\ast}-\Spc$. 
That is, 
the homotopy functor model structure is the localization of the pointwise model structure with respect to the set of map
\begin{equation*}
\xymatrix{
\Y\otimes\Box\CC^{\ast}-\Spc(\X',-)
\ar[r] &
\Y\otimes\Box\CC^{\ast}-\Spc(\X,-). }
\end{equation*}
\end{proof}
\begin{remark}
In the above there is no need to apply a cofibrant replacement functor $\QQ$ in the pointwise projective model structure on $\Box\CC^{\ast}-\Spc$
to $\Y$ since all the domains and codomains of the generating cofibrations of $\Box\CC^{\ast}-\Spc$ are cofibrant according to
Lemma \ref{lemma:projectivecofibrant}.
However,
using the same script for more general model categories requires taking a cofibrant replacement.
\end{remark}

We shall refer to the weak equivalences in the homotopy functor model structure as homotopy functor weak equivalences.
\begin{corollary}
If $\Y$ is projective cofibrant in $\Box\CC^{\ast}-\Spc$ then the naturally induced map
\begin{equation*}
\xymatrix{
\Y\otimes\Box\CC^{\ast}-\Spc(\X',-)
\ar[r] &
\Y\otimes\Box\CC^{\ast}-\Spc(\X,-) }
\end{equation*}
is a homotopy functor weak equivalence.
\end{corollary}
\vspace{0.1in}

We shall leave implicit the proofs of the following three results which the interested reader can verify following in outline the proofs
of the corresponding results for the pointwise model structure.

\begin{lemma}
\label{lemma:homotopyfunctorpairing}
The pairing (\ref{equation:pairing}) is a Quillen bifunctor with respect to the homotopy functor model structure on 
$[\f\Box\CC^{\ast}-\Spc,\Box\CC^{\ast}-\Spc]$ and the projective homotopy invariant model structure on $\Box\CC^{\ast}-\Spc$.
\end{lemma}
\vspace{0.1in}

In what follows, 
assume that every object of $\f\Box\CC^{\ast}-\Spc$ is cofibrant.
\begin{lemma}
The functor $-\otimes \Box\CC^{\ast}-\Spc(\X,-)$ is a left Quillen functor with respect to the homotopy functor model structure.
\end{lemma}

\begin{proposition}
The homotopy functor model structure is monoidal.
\end{proposition}
\vspace{0.1in}

Although the work in \cite{Biedermann} which makes a heavy use of \cite{BousfieldFriedlander} and \cite{DRO:general} does not apply 
directly to our setting,
it offers an approach which we believe is worthwhile to pursue when the model categories in question are not necessarily right proper.
We shall give such a generalization by using as input the recent paper \cite{Stanculescu}.
\vspace{0.1in}

Stanculescu \cite{Stanculescu} has shown the following result.

\begin{theorem} 
\label{theorem:stanculescu}
Let $\M$ be a combinatorial model category with localization functor $\gamma:\M\rightarrow\Ho(\M)$. 
Suppose there is an accessible functor $\FF:\M\rightarrow\M$ and a natural transformation 
$\alpha:\id\rightarrow\FF$ satisfying the following properties:
\begin{enumerate}[{\bf A} 1:]
\item The functor $\FF$ preserves weak equivalences.
\item For every $X\in\M$, 
the map $\FF(\alpha_{X})$ is a weak equivalence and $\gamma(\alpha_{\FF(X)})$ is a monomorphism.
\end{enumerate}
Then $\M$ acquires a left Bousfield localization with $\FF$-equivalences as weak equivalences.
\end{theorem}
\vspace{0.1in}

The assumption that $\FF$ be an accessible functor allows one to verify the hypothesis in Smith's main 
theorem on combinatorial model categories:
\begin{theorem}
\label{theorem:smith}
Suppose $\M$ is a locally presentable category, 
$\W$ a full accessible subcategory of the morphism category of $\M$,  
and $I$ a set of morphisms of $\M$ such that the following conditions hold:
\begin{enumerate}[{\bf C} 1:]
\item $\W$ has the three-out-of-two property.
\item $I-\text{inj}\subseteq \W$.
\item The class $I-\text{cof}\cap\W$ is closed under transfinite compositions and pushouts.
\end{enumerate}
Then $\M$ acquires a cofibrantly generated model structure with classes of weak equivalences $\W$,
cofibrations $I-\text{cof}$, 
and fibrations $(I-\text{cof}\cap \W)-\text{inj}$.
\end{theorem}
\begin{remark}
The ``only if'' implication follows since every accessible functor satisfies the solution-set condition,  
see \cite[Corollary 2.45]{AR:book}, 
and every class of weak equivalences in some combinatorial model category is an accessible subcategory of its morphism category.
\end{remark}

In order to prove Theorem \ref{theorem:stanculescu},
note that conditions $({\bf C} 1)$ and $({\bf C} 2)$ hold,  
so it remains to verify $({\bf C} 3)$.
This follows from the characterization of acyclic $\FF$-cofibrations by the left lifting property described in 
\cite[Lemma 2.4]{Stanculescu}.
\vspace{0.1in}

Suppose $(-)^{\ppt}$ is an accessible fibrant replacement functor in the pointwise model structure on $\CC^{\ast}$-functors.
To construct the homotopy functor model structure using Theorem \ref{theorem:stanculescu}, 
we set
\begin{equation*}
\FF^{\hht}(\SSSS)\equiv
\SSSS\circ (-)^{\ppt}.
\end{equation*} 
The verification of the axioms {\bf A} 1 and {\bf A} 2 for $\FF^{\hht}$ follows as in \cite[Proposition 3.3]{BCR}.
\begin{lemma}
The $\FF^{\hht}$-model structure coincides with the homotopy functor model structure. 
\end{lemma}
\begin{proof}
The model structures have the same cofibrations and fibrant objects.
\end{proof}

By using the same type of localization method we construct next the stable model structure on $[\f\Box\CC^{\ast}-\Spc,\Box\CC^{\ast}-\Spc]$.
We fix an accessible fibrant replacement functor $(-)^{\hht}$ in the homotopy functor model structure.
Let $C'$ denote the right adjoint functor of $-\otimes C$ - again denoted by $C$ in what follows - given by cotensoring with $C$.
Note that $C'$ commutes with filtered colimits and homotopy colimits because $C$ is small. 
\vspace{0.1in}

Define the endofunctor $\FF^{\sst}$ of $[\f\Box\CC^{\ast}-\Spc,\Box\CC^{\ast}-\Spc]$ by setting 
\begin{equation*}
\FF^{\sst}(\SSSS)\equiv
\hocolim_{n} \, \bigl(C'^{\otimes n}\circ (\SSSS)^{\hht}\circ C^{\otimes n}\bigr).
\end{equation*} 
This is an accessible functor and it satisfies the axioms {\bf A} 1 and {\bf A} 2 by \cite[Lemma 8.9]{Biedermann}.
\vspace{0.1in}

We are ready to formulate the main result in this section.
Most parts of this result should be clear by now, 
and more details will appear in a revised version of the general setup in \cite{Biedermann} dodging the right properness assumption.

\begin{theorem}
The following holds for the stable model structure on the enriched category of $\CC^{\ast}$-functors 
$$[\f\Box\CC^{\ast}-\Spc,\Box\CC^{\ast}-\Spc].$$
\begin{itemize}
\item
It is a combinatorial and left proper model category.
\item
It is a symmetric monoidal model category.
\item
When $\f\Box\CC^{\ast}-\Spc$ is the category of $C$-spheres then there exists a Quillen equivalence between the stable model structures on 
$\CC^{\ast}$-functors and on cubical $\CC^{\ast}$-spectra. 
\end{itemize}
\end{theorem}
\begin{remark}
Recall that the category of $C$-spheres is the full subcategory of $\Box\CC^{\ast}-\Spc$ comprising objects $\X$ for which there exists an 
acyclic cofibration $C^{\otimes n}\rightarrow\X$ in the projective homotopy invariant model structure on $\Box\CC^{\ast}-\Spc$.
This is the ``minimal'' choice of $\f\Box\CC^{\ast}-\Spc$.
It is not clear whether the full subcategory of finitely presentable cubical $\CC^{\ast}$-spaces ${\bf fp}\Box\CC^{\ast}-\Spc$ gives a Quillen 
equivalent model structure \cite[\S7.2]{DRO:general}.
This point is also emphasized in \cite[\S10]{Biedermann}.
\end{remark}
\newpage

\section{Invariants}
\label{section:invariants}
In what follows we employ $\CC^{\ast}$-homotopy theory to define invariants for $\CC^{\ast}$-algebras.
Section \ref{subsection:cohomologytheories} introduces briefly bigraded homology and cohomology theories at large.
The main examples are certain canonical extensions of $KK$-theory, 
see Section \ref{subsection:KK-theory}, 
and local cyclic homology theory, 
see Section \ref{subsection:localcyclichomology}, 
to the framework of pointed simplicial $\CC^{\ast}$-spaces.
We also observe that there is an enhanced Chern-Connes character between $KK$-theory and local cyclic theory on the level of 
simplicial $\CC^{\ast}$-spectra.
Section \ref{subsection:KtheoryofCstaralgebras} deals with a form of $K$-theory of $\CC^{\ast}$-algebras which is constructed
using the model structures introduced earlier in this paper.
This form of $K$-theory is wildly different from the traditional $2$-periodic $K$-theory of $\CC^{\ast}$-algebras 
\cite[II]{Connes:noncommutative} and relates to topics in geometric topology.
Finally, in the last section we discuss zeta functions of $\CC^{\ast}$-algebras.

\subsection{Cohomology and homology theories}
\label{subsection:cohomologytheories}
We record the notions of (co)homology and bigraded (co)homology theories. 
\begin{definition}
\begin{itemize}
\item 
A homology theory on $\SH^{\ast}$ is a homological functor $\SH^{\ast}\rightarrow\Ab$ which preserves sums.
Dually,
a cohomology theory on $\SH^{\ast}$ is a homological functor ${\SH^{\ast}}^{\op}\rightarrow\Ab$ which takes sums to products.
\item
A bigraded cohomology theory on $\SH^{\ast}$ is a homological functor $\Phi$ from ${\SH^{\ast}}^{\op}$ to Adams graded graded 
abelian groups which takes sums to products together with natural isomorphisms
\begin{equation*}
\Phi(\E)^{p,q}\cong \Phi(\Sigma_{S^{1}}\E)^{p+1,q}
\end{equation*}
and
\begin{equation*}
\Phi(\E)^{p,q}\cong \Phi(\Sigma_{C_{0}(\R)}\E)^{p,q+1}
\end{equation*}
such that the diagram
\begin{equation*}
\xymatrix{
\Phi(\E)^{p,q} \ar[r] \ar[d] & \Phi(\Sigma_{S^{1}}\E)^{p+1,q} \ar[d] \\
\Phi(\Sigma_{C_{0}(\R)}\E)^{p,q+1} \ar[r] &
\Phi(\Sigma_{C}\E)^{p+1,q+1} }
\end{equation*}
commutes for all integers $p,q\in\ZZ$.

Bigraded homology theories are defined likewise.
\end{itemize}
\end{definition}

The category of graded abelian groups refers to integer-graded objects subject to the Koszul sign rule 
$a\otimes b=(-1)^{\vert\, a\vert\vert\, b\vert} b\otimes a$.
In the case of bigraded cohomology theories there is a supplementary graded structure. 
The category of Adams graded graded abelian groups refers to integer-graded objects in graded abelian groups, 
but no sign rule for the tensor product is introduced as a consequence of the Adams grading. 
It is helpful to think of the Adams grading as being even.
\vspace{0.1in}

As alluded to in the introduction the bigraded cohomology and homology theories associated with a $\CC^{\ast}$-spectrum $\E$ are 
defined by the formulas 
\begin{equation}
\label{cohomology2}
\E^{p,q}(\F)
\equiv
\SH^{\ast}\bigl(\F,S^{p-q}\otimes C_0(\R^q)\otimes\E\bigr),
\end{equation} 
and 
\begin{equation}
\label{homology2}
\E_{p,q}(\F)
\equiv
\SH^{\ast}\bigl(\Sigma_C^{\infty} S^{p-q}\otimes C_0(\R^q),\F\wedge^{{\bf L}} \E\bigr).
\end{equation} 
When $\E$ is the sphere $\CC^{\ast}$-spectrum ${\bf 1}$ then (\ref{cohomology2}) defines the stable cohomotopy groups 
$\pi^{p,q}(\F)\equiv {\bf 1}^{p,q}(\F)$ and (\ref{homology2}) the stable homotopy groups $\pi_{p,q}(\F)\equiv {\bf 1}_{p,q}(\F)$ of $\F$.
Invoking the symmetric monoidal product $\wedge^{{\bf L}}$ on $\SH^{\ast}$ there is a pairing 
\begin{equation}
\label{homologypairing}
\xymatrix{
\pi_{p,q}(\F)\otimes\pi_{p',q'}(\F')
\ar[r] & 
\pi_{p+p',q+q'}(\F\wedge^{{\bf L}}\F'). }
\end{equation}
More generally, 
there exists formally defined products
\begin{equation*}
\xymatrix{
\wedge\colon\E_{p,q}(\F)\otimes\E'_{p',q'}(\F')
\ar[r] & 
(\E\wedge^{{\bf L}}\E')_{p+p',q+q'}(\F\wedge^{{\bf L}}\F'),} 
\end{equation*}
\begin{equation*}
\xymatrix{
\cup\colon\E^{p,q}(\F)\otimes\E'^{p',q'}(\F')
\ar[r] & 
(\E\wedge^{{\bf L}}\E')^{p+p',q+q'}(\F\wedge^{{\bf L}}\F'),} 
\end{equation*}
\begin{equation*}
\xymatrix{
/\colon\E^{p,q}(\F\wedge^{{\bf L}}\F')\otimes\E'_{p',q'}(\F')
\ar[r] & 
(\E\wedge^{{\bf L}}\E')_{p-p',q-q'}(\F),} 
\end{equation*}
\begin{equation*}
\xymatrix{
\backslash \colon\E^{p,q}(\F)\otimes\E'_{p',q'}(\F\wedge^{{\bf L}}\F')
\ar[r] & 
(\E\wedge^{{\bf L}}\E')_{p'-p,q'-q}(\F').}
\end{equation*}
When $\E=\E'$ is a monoid in $\SH^{\ast}$ composing the external products with $\E\wedge^{{\bf L}}\E\rightarrow\E$ yields internal products.
The internalization of the slant product $\backslash$ is a type of cap product.
We refer the interested reader to \cite{May:additivityoftraces} and the references therein for more details concerning the formal deduction of 
the above products using function spectra or derived internal hom objects depending only on the structure of $\SH^{\ast}$ as a symmetric monoidal 
category with a compatible triangulation,  
and the corresponding constructs in classical stable homotopy theory.
\newpage

\subsection{KK-theory and the Eilenberg-MacLane spectrum}
\label{subsection:KK-theory}
The construction we perform in this section is a special case of ``twisting'' a classical spectrum with $KK$-theory:
By combining the integral Eilenberg-MacLane spectrum representing classical singular cohomology and homology with $KK$-theory 
we deduce a $\CC^{\ast}$-symmetric spectrum which is designed to represent $K$-homology and $K$-theory of $\CC^{\ast}$-algebras.
Certain parts involved in this example depend heavily on a theory of noncommutative motives developed in \cite{Ostvar:noncommutativemotives}.
Throughout this section we work with simplicial objects rather than cubical objects, 
basically because we want to emphasize the (simplicial) Dold-Kan equivalence.
\vspace{0.1in}

In the companion paper \cite{Ostvar:noncommutativemotives} we construct an adjoint functor pair:
\begin{equation}
\label{equation:KKadjunction}
\xymatrix{
(-)^{\KK}\colon\Delta\CC^{\ast}-\Spc_{0} \ar@<3pt>[r] & 
\Delta\CC^{\ast}-\Spc^{\KK}_{0}\colon{\bf U} \ar@<3pt>[l] }
\end{equation}
Here $\Delta\CC^{\ast}-\Spc^{\KK}_{0}$ is the category of pointed simplicial $\CC^{\ast}$-spaces with $KK$-transfers,
i.e.~additive functors from Kasparov's category of $KK$-correspondences to simplicial abelian groups.
This is a closed symmetric monoidal category enriched in abelian groups and the symmetric monoidal functor $(-)^{\KK}$ 
is uniquely determined by 
\begin{equation}
\label{equation:KKextension}
(A\otimes\Delta^{n}_{+})^{\KK}\equiv
\KK(A,-)\otimes\widetilde{\ZZ}[\Delta^{n}_{+}].
\end{equation}
The right adjoint of the functor adding $KK$-transfers to pointed simplicial $\CC^{\ast}$-spaces is the lax symmetric monoidal 
forgetful functor ${\bf U}$.
\vspace{0.1in}

With these definitions there are isomorphisms
\begin{align*}
\Delta\CC^{\ast}-\Spc^{\KK}_{0}\bigl((A\otimes\Delta^{n}_{+})^{\KK},\Y\bigr) 
& = 
\Delta\CC^{\ast}-\Spc^{\KK}_{0}\bigl(\KK(A,-),\underline{\Hom}(\ZZ[\Delta^{n}_{+}],\Y)\bigr) \\
& = 
\Delta\CC^{\ast}-\Spc_{0}\bigl(A,{\bf U}\underline{\Hom}(\ZZ[\Delta^{n}_{+}],\Y)\bigr) \\
& = 
\Delta\CC^{\ast}-\Spc_{0}\bigl(A,\underline{\Hom}(\Delta^{n}_{+},{\bf U}\Y)\bigr) \\
& = 
\Delta\CC^{\ast}-\Spc_{0}(A\otimes\Delta^{n}_{+},{\bf U}\Y).
\end{align*}

The above definition clearly extends $KK$-theory to a functor on pointed simplicial $\CC^{\ast}$-spaces.
Moreover, 
for pointed simplicial $\CC^{\ast}$-spaces $\X$ and $\Y$ there exist canonically induced maps
\begin{equation*}
\xymatrix{
\X\otimes\Y^{\KK}\ar[r] & \X^{\KK}\otimes\Y^{\KK}\ar[r] & (\X\otimes\Y)^{\KK}.}
\end{equation*}
\vspace{0.1in}

In particular,
when $\X$ equals the preferred suspension coordinate $C=S^1\otimes C_0(\R)$ in $\CC^{\ast}$-homotopy theory and $\Y$ its $n$th fold
tensor product $C^{\otimes n}$, 
there is a map 
\begin{equation*}
\xymatrix{
C\otimes (C^{\otimes n})^{\KK}\ar[r] &  (C^{\otimes n+1})^{\KK}. }
\end{equation*}
The above defines the structure maps in the $\CC^{\ast}$-algebra analog of the stably fibrant Eilenberg-MacLane spectrum 
\begin{equation*}
\xymatrix{
\HH\ZZ=\{n \ar@{|->}[r] & \widetilde{\ZZ}[S^{n}]\} }
\end{equation*}
studied in stable homotopy theory.
This description clarifies the earlier remark about twisting the integral Eilenberg-MacLane spectrum with $KK$-theory.
A straightforward analysis based on Bott periodicity in $KK$-theory reveals there exist isomorphisms of simplicial 
$\CC^{\ast}$-spaces
\begin{equation*}
(C^{\otimes n})^{\KK}
=
\begin{cases}
K_{0}(-)\otimes \widetilde{\ZZ}[S^{n}] & n\geq 0\text{ even} \\
K_{1}(-)\otimes \widetilde{\ZZ}[S^{n}] & n\geq 1\text{ odd.}
\end{cases}
\end{equation*}

The main result in this section shows the spectrum we are dealing with is stably fibrant.
\begin{theorem}
\label{theorem:KKisfibrant}
The simplicial $\CC^{\ast}$-spectrum
\begin{equation*}
\xymatrix{
\KK\equiv\{n\ar@{|->}[r] & (C^{\otimes n})^{\KK}\}}
\end{equation*}
is stably fibrant.
\end{theorem}
\begin{proof}
We shall note the constituent spaces $(C^{\otimes n})^{\KK}$ of $\KK$ are fibrant in the 
projective homotopy invariant model structure on $\Delta\CC^{\ast}-\Spc_{0}$.
First, 
$\Ev_{A}(C^{\otimes n})^{\KK}$ is a simplicial abelian group and hence fibrant in the 
model structure on $\Delta\Set_{\ast}$,
so that $(C^{\otimes n})^{\KK}$ is projective fibrant.
For $KK$-theory of $\CC^{\ast}$-algebras, 
homotopy invariance holds trivially, 
while matrix invariance and split exactness hold by \cite[Propositions 2.11,2.12]{Higson:KK}.
The same properties hold for the $KK$-theory of the pointed simplicial $\CC^{\ast}$-spaces $C^{\otimes n}$ using 
(\ref{equation:KKextension}).
These observations imply that $\KK$ is level fibrant.
\vspace{0.1in}

It remains to show that for every $A\in\CC^{\ast}-\Alg$ and $m\geq 0$, 
there is an isomorphism
\begin{equation}
\label{equation:mapforgenerators}
\xymatrix{
\HH^{\ast}\bigl(A\otimes S^m,(C^{\otimes n})^{\KK}\bigr)
\ar[r] &
\HH^{\ast}\bigl(C\otimes A\otimes S^m,(C^{\otimes n+1})^{\KK}\bigr).}
\end{equation}
This reduction step follows directly from Theorem \ref{theorem:unstableweakgenerators}; the latter shows that the set of isomorphism 
classes of cubical $\CC^{\ast}$-algebras of the form $A\otimes S^m$ generates the homotopy category $\HH^{\ast}$.
In the next step of the proof we shall invoke the category of noncommutative motives $\MM(\KK)$ \cite{Ostvar:noncommutativemotives}.
Its underlying category $\Ch(\KK)$ consists of chain complexes of pointed $\CC^{\ast}$-spaces with $\KK$-transfers.
Every $\CC^{\ast}$-algebra $A$ has a corresponding motive $\MM(A)$ and likewise for $C\otimes A$.
The category of motives is in fact constructed analogously to $\HH^{\ast}$.
Via the adjunction (\ref{equation:KKadjunction}), 
the shift operator $[-]$ on chain complexes identifies the map (\ref{equation:mapforgenerators}) with 
\begin{equation*}
\xymatrix{
\MM(\KK)\bigl(\MM(A),\ZZ(n)[2n-m]\bigr)
\ar[r] &
\MM(\KK)\bigl(\MM(C\otimes A)[-2],\ZZ(n+1)[2n-m]\bigr),}
\end{equation*}
where 
\begin{equation*}
\ZZ(1)\equiv\MM\bigl(C_{0}(\R)\bigr)[-1]
\end{equation*}
is the so-called $\CC^{\ast}$-algebraic Bott object \cite{Ostvar:noncommutativemotives}. 
Thus, 
using the symmetric monoidal product in $\Ch(\KK)$, 
the map identifies with $-\otimes\ZZ(1)$ from  $\MM(\KK)\bigl(\MM(A),\ZZ(n)[2n-m]\bigr)$ to 
$\MM(\KK)\bigl(\MM(A)\otimes\ZZ(1),\ZZ(n)[2n-m]\otimes\ZZ(1)\bigr)$.
With these results in hand, 
it remains to note that $-\otimes\ZZ(1)$ is an isomorphism by Bott periodicity. 
\end{proof}
\begin{lemma}
\label{lemma:SHCh(KK)isomorphism}
There is an isomorphism
\begin{equation*}
\KK^{p,q}(A)\equiv
\SH^{\ast}\bigl(\Sigma_C^{\infty} A,\KK\otimes S^{p-q}\otimes C_{0}(\R^q)\bigr)=
\MM(\KK)\bigl(\MM(A),\ZZ(q)[p]\bigr).
\end{equation*} 
\end{lemma}
\begin{proof}
Fix some integer $m\geq p-q,q$.
Then $\SH^{\ast}\bigl(\Sigma_C^{\infty}A\otimes C^{\otimes -m}\otimes S^{q-p-m}\otimes C_{0}(\R^{m-q}),\KK\bigr)$ 
is isomorphic to $\HH^{\ast}\bigl(A\otimes S^{m+q-p}\otimes C_{0}(\R^{m-q}),(C^{\otimes m})^{\KK}\bigr)$ and hence, 
by Theorem \ref{theorem:KKisfibrant}, 
to 
\begin{equation*} 
\MM(\KK)\Bigl(\MM(A)\otimes\MM\bigl(C_{0}(\R^{m-q})\bigr)[q-m][m-q],\ZZ(m)[p+(m-q)]\Bigr),
\end{equation*} 
or equivalently
\begin{equation*} 
\MM(\KK)\bigl(\MM(A)\otimes\ZZ(m-q),\ZZ(m)[p]\bigr).
\end{equation*}
By Bott periodicity, tensoring with $\ZZ(m-q)$ implies the identification.
\end{proof}

The simplicial $\CC^{\ast}$-spectrum $\KK$ is intrinsically a simplicial $\CC^{\ast}$-symmetric spectrum via the 
natural action of the symmetric groups on the tensor products $(C^{\otimes n})^{\KK}$.
It is straightforward to show that $\KK$ is a ring spectrum in a highly structured sense.
\begin{lemma}
\label{lemma:KKcommutativemonoid}
$\KK$ is a commutative monoid in the category of $\CC^{\ast}$-symmetric spectra.
\end{lemma}

The construction of motives and $\KK$ alluded to above works more generally for $\Group-\CC^{\ast}$-algebras.
In particular,
there exists an equivariant $KK$-theory $(-)^{\KK^{\Group}}$ for pointed simplicial $\Group-\CC^{\ast}$-spaces. 
For completeness we state the corresponding equivariant result: 
\begin{theorem}
\label{theorem:equivariantKKisfibrant}
The simplicial $\Group-\CC^{\ast}$-spectrum
\begin{equation*}
\xymatrix{
\KK^{\Group}\equiv\{n\ar@{|->}[r] & (C^{\otimes n})^{\KK^{\Group}}\}}
\end{equation*}
is stably fibrant and a commutative monoid in the category of $\Group-\CC^{\ast}$-symmetric spectra. 
\end{theorem}
\vspace{0.1in}

\begin{remark}
It is possible to give an explicit model of $\KK$ as a $\CC^{\ast}$-functor.
\end{remark}
\newpage

\subsection{HL-theory and the Eilenberg-MacLane spectrum}
\label{subsection:localcyclichomology}

For the background material required in this section and the next we refer to \cite{Meyer}, \cite{Puschnigg} and \cite{Voigt}.
In particular,
when $\CC^{\ast}$-algebras are viewed as bornological algebras we always work with the precompact bornology in order 
to ensure that local cyclic homology satisfy split exactness, matrix invariance and homotopy invariance.
\vspace{0.1in}

Setting  
\begin{equation}
\label{equation:HLextension}
(A\otimes\Delta^{n}_{+})^{\HL}\equiv
\HL(A,-)\otimes\widetilde{\ZZ}[\Delta^{n}_{+}]
\end{equation}
extends local cyclic homology $\HL$ of $\CC^{\ast}$-algebras to pointed simplicial $\CC^{\ast}$-spaces. 
With respect to the usual composition product in local cyclic homology, 
this gives rise to a symmetric monoidal functor taking values in pointed simplicial $\CC^{\ast}$-spaces equipped 
with $\HL$-transfers,  
and an adjoint functor pair where ${\bf U}$ denotes the lax symmetric monoidal forgetful functor:
\begin{equation}
\label{equation:HLadjunction}
\xymatrix{
(-)^{\HL}\colon\Delta\CC^{\ast}-\Spc_{0} \ar@<3pt>[r] & 
\Delta\CC^{\ast}-\Spc^{\HL}_{0}\colon{\bf U} \ar@<3pt>[l] }
\end{equation}
The existence of (\ref{equation:HLadjunction}) is shown using (\ref{equation:HLextension}) by following the exact same steps 
as for the $\KK$-theory adjunction displayed in (\ref{equation:KKadjunction}).
In analogy with $\KK$ we may now define the structure maps 
\begin{equation*}
\xymatrix{
C\otimes (C^{\otimes n})^{\HL}\ar[r] &  (C^{\otimes n+1})^{\HL} }
\end{equation*}
in a simplicial $\CC^{\ast}$-spectrum we shall denote by $\HL$. 
Moreover, 
the natural $\Sigma_{n}$-action on $(C^{\otimes n})^{\HL}$ equips $\HL$ with the structure of a commutative monoid in 
${\bf Spt}^{\Sigma}_{C}$.
\vspace{0.1in}

Using the multiplication by $(2\pi i)^{-1}$ Bott periodicity isomorphism in local cyclic homology and an argument which runs
in parallel with the proof of Theorem \ref{theorem:KKisfibrant}, 
we deduce that also the local cyclic homology twisted Eilenberg-MacLane spectrum is stably fibrant:
\begin{theorem}
\label{theorem:HLisfibrant}
The simplicial $\CC^{\ast}$-spectrum
\begin{equation*}
\xymatrix{
\HL\equiv\{n\ar@{|->}[r] & (C^{\otimes n})^{\HL}\}}
\end{equation*}
is stably fibrant and a commutative monoid in ${\bf Spt}^{\Sigma}_{C}$.
\end{theorem}
\vspace{0.1in}

We shall leave the formulation of the equivariant version of Theorem \ref{theorem:HLisfibrant}, 
for $\Group$ totally disconnected, 
to the reader. 
\newpage

\subsection{The Chern-Connes character}
\label{subsection:theChernConnescharacter}

Local cyclic homology of $\CC^{\ast}$-algebras defines an exact, matrix invariant and homotopy invariant functor into
abelian groups.
Thus the universal property of $\KK$-theory implies that there exists a unique natural transformation between $\KK$-theory 
and local cyclic homology.
Moreover, it turns out this is a symmetric monoidal natural transformation.
By the definition of $\KK$-theory and local cyclic homology for pointed simplicial $\CC^{\ast}$-spaces in terms of left Kan 
extensions, 
if follows that there exists a unique symmetric monoidal natural transformation 
\begin{equation*}
\xymatrix{
(-)^{\KK}\ar[r] & (-)^{\HL}.}
\end{equation*}
We have established the existence of the Chern-Connes character. 
\begin{theorem}
\label{theorem:CCcharacter}
There exists a ring map of simplicial $\CC^{\ast}$-symmetric spectra
\begin{equation}
\label{equation:CCcharacter}
\xymatrix{
\KK\ar[r] & \HL.}
\end{equation}
\end{theorem}

Since the local cyclic homology of a $\CC^{\ast}$-algebra $A$ is a complex vector space, 
the Chern-Connes character for $A$ induces a $\C$-linear map 
\begin{equation}
\label{equation:classicalCCwithC}
\KK_{\ast}(\C,A)\otimes_{\ZZ}\C\rightarrow\HL_{\ast}(\C,A).
\end{equation}
By naturality there exists an induced map of simplicial $\CC^{\ast}$-symmetric spectra
\begin{equation}
\label{equation:CCwithC}
\xymatrix{
\KK\otimes_{\ZZ}\C\ar[r] & \HL.}
\end{equation}
The constituent spaces in the spectrum on left hand side in (\ref{equation:CCwithC}) are 
$n\mapsto (C^{\otimes n})^{\KK}\otimes_{\ZZ}\C$.
Recall that the map (\ref{equation:classicalCCwithC}) is an isomorphism provided $A$ is a member of the so-called bootstrap category
comprising the $\CC^{\ast}$-algebras with a $\KK$-equivalence to a member of the smallest class of nuclear $\CC^{\ast}$-algebras
that contains $\CC$ and is closed under countable colimits, extensions and $\KK$-equivalences.
Equivalently, 
$A$ is in the bootstrap category if and only if it is $\KK$-equivalent to a commutative $\CC^{\ast}$-algebra.
This implies (\ref{equation:CCwithC}) is a pointwise weak equivalence when restricted to the bootstrap category.  
\vspace{0.1in}

For second countable totally disconnected locally compact groups the work of Voigt \cite{Voigt} allows us to construct as in 
(\ref{equation:CCcharacter}) an equivariant Chern-Connes character.
\newpage

\subsection{$K$-theory of $\CC^{\ast}$-algebras}
\label{subsection:KtheoryofCstaralgebras}
The prerequisite to the $K$-theory of $\CC^{\ast}$-algebras proposed here is Waldhausen's work on $K$-theory of 
categories with cofibrations and weak equivalences \cite{Waldhausen:LNM1126}.
In what follows we tend to confine the general setup to cofibrantly generated pointed model categories $\M$ in 
order to streamline our presentation.
Throughout we consider the homotopy invariant projective model structure on $\Box\CC^{\ast}-\Spc_{0}$ and the 
stable model structures on $S^{1}$-spectra of pointed cubical $\CC^{\ast}$-algebras.
\vspace{0.1in}

A Waldhausen subcategory of $\M$ is a full subcategory $\NN\subset\M$ of cofibrant objects including a zero-object 
$\ast$ with the property that if $\X\rightarrow\Y$ is a map in $\NN$ and $\X\rightarrow\Z$ is a map in $\M$, 
then the pushout $\Y\coprod_{\X}\W$ belongs to $\NN$.
With notions of cofibrations and weak equivalences induced from the model structure on $\M$ it follows that $\NN$ 
is a category with cofibrations and weak equivalences.
For those not familiar with the axioms for cofibrations and weak equivalences in $K$-theory we state these in detail.
\begin{definition}
\label{definition:categorywithcofibrationsandweakequivalences}
A category with cofibrations and weak equivalences consists of a pointed category $\CCC$ equipped with two subcategories 
of cofibrations ${\bf cof}\CCC$ and weak equivalences ${\bf weq}\CCC$ such that the following axioms hold.
\begin{enumerate}[{\bf Cof} 1:]
\item Every isomorphism is a cofibration.
\item Every object is cofibrant.
That is,
$\ast\rightarrow\X$ is in ${\bf cof}\CCC$ for every object $\X$ of $\CCC$.
\item If $\X\rightarrow\Y$ is a cofibration, 
then the pushout of every diagram of the form 
\begin{equation*}
\xymatrix{
\Z &  
\X\ar[r]\ar[l] &
\Y }
\end{equation*}
in $\CCC$ exists, 
and the cobase change map $\Z\rightarrow\Z\cup_{\X}\Y$ is in ${\bf cof}\CCC$.
\end{enumerate}
\begin{enumerate}[{\bf Weq} 1:]
\item Every isomorphism is a weak equivalence.
\item The gluing lemma holds.
That is, 
for every commutative diagram
\begin{equation*}
\xymatrix{ 
\Z\ar[d] &
\X\ar[l]\ar[r]\ar[d] &
\Y\ar[d] \\
\Z' &
\X'\ar[l]\ar[r] &
\Y' }
\end{equation*}
in $\CCC$ where the vertical maps are weak equivalences and the right hand horizontal maps are cofibrations, 
$\Z\cup_{\X}\Y\rightarrow\Z'\cup_{\X'}\Y'$ is in ${\bf weq}\CCC$.
\end{enumerate}
\end{definition}

Applying the $S_{\bullet}$-construction we obtain a simplicial category $wS_{\bullet}(\NN)$.  
Taking the nerve produces a simplicial space $[n]\mapsto NwS_{n}(\NN)$, 
and $K(\NN)$ is defined by the loops $\Omega \vert\,NwS_{\bullet}(\NN)\,\vert$ on the realization of this 
simplicial space.
The algebraic $K$-groups of $\NN$ are defined as $K_{n}(\NN)\equiv\pi_{n}K(\NN)$.
With these definitions one finds that the abelian group $K_{0}(\NN)$ is generated by symbols $[\X]$ where 
$\X$ is an object of $\NN$, 
subject to the relations $[\X]=[\Y]$ if there is a weak equivalence $\X\rightarrow\Y$ and 
$[\Z]=[\X]+[\Y]$ if there is a cofibration sequence $\X\rightarrow\Z\rightarrow\Y$.
\vspace{0.1in}

The algebraic $K$-theory spectrum of $\NN$ is defined by iterating the $S_{\bullet}$-construction forming
$\vert\,NwS^{(n)}_{\bullet}(\NN)\,\vert$ for $n\geq 1$.
It is not difficult to verify that there is a symmetric spectrum structure on the $K$-theory spectrum of $\NN$.
Since we will not make use of this important extra structure here we refer the reader to \cite{Schwede:SS} for details.
\vspace{0.1in}

The identity map on the $K$-theory of the full subcategory of cofibrant objects $\M_{cof}$ is 
null-homotopic by a version of the Eilenberg swindle.
For this reason, 
$K$-theory deals with subcategories of $\M_{cof}$ defined by finiteness conditions which are 
typically not preserved under infinite coproducts.
The cube lemma for cofibrant objects as in \cite[Lemma 5.2.6]{Hovey:Modelcategories} implies 
the full subcategory ${\bf fp}\M_{cof}$ of finitely presentable objects is also a Waldhausen 
subcategory of $\M$.
With these choices of subcategories of pointed model categories we get, 
by combining \cite[Corollary 3.9]{DS:Ktheory} and \cite[Theorem 3.3]{Sagave:diplom}, 
the next result.
\begin{corollary}
\label{corollary:stableQuillenKisomorphism}
Every Quillen equivalence $\M\rightarrow\NN$ between pointed stable model categories induces a weak equivalence 
\begin{equation*}
\xymatrix{
K({\bf fp}\M_{cof})\ar[r]^-{\sim} &
K({\bf fp}\NN_{cof}). }
\end{equation*}
\end{corollary}
\vspace{0.1in}

A functor between categories with cofibrations and weak equivalences is called exact if it preserves the zero-object,
cofibrations, weak equivalences and cobase change maps along cofibrations.
An exact functor $F$ is a $K$-theory equivalence if the induced map $\Omega \vert\,NwS_{\bullet}(F)\,\vert$ is a 
homotopy equivalence.
Every left Quillen functor induces an exact functor between the corresponding full subcategories of finitely presentable 
and cofibrant objects.
\vspace{0.1in}

We give the following widely applicable characterization of cofibrant and finitely presentable objects in terms of cell complexes.
\begin{lemma}
\label{lemma:cofibrantfinitelypresentable}
Suppose the domains and codomains of the generating cofibrations $I_{\M}$ are finitely presentable.
Then an object $\X$ of $\M$ is cofibrant and finitely presentable if and only if it is a retract of 
a finite $I_{\M}$-cell complex $\Z$ of the form
\begin{equation*}
\xymatrix{
\ast=\Z_{0}\ar[r] & \Z_{1}\ar[r] & \cdots\ar[r] & \Z_{n}=\Z,}
\end{equation*}
where $\Z_{i}\rightarrow\Z_{i+1}$ is the pushout of a generating projective cofibration.
\end{lemma}
\begin{proof}
If $\X$ is a retract of a finite $I_{\M}$-cell complex, 
then $\X$ is cofibrant and finitely presentable since a finite colimit of finitely presentable 
objects is finitely presentable and a retract of a finitely presentable object is finitely presentable.
\vspace{0.1in}

Conversely, 
every cofibrant object $\X$ is a retract of the colimit $\X_{\infty}$ of an $I_{\M}$-cell complex 
\begin{equation*}
\xymatrix{
\ast=\X_{0}\ar[r] & \X_{1}\ar[r] & \X_{2}\ar[r] & \cdots\ar[r] & \X_{n}\ar[r] & \cdots} 
\end{equation*}
for pushout diagrams where the top map is a coproduct of generating cofibrations:
\begin{equation}
\label{diagram:cellpushout}
\xymatrix{
\coprod_{\lambda\in\lambda_{i}} s_{\lambda} \ar[d]\ar[r] & 
\coprod_{\lambda\in\lambda_{i}} t_{\lambda} \ar[d]\\
\X_{i}\ar[r] & \X_{i+1}}
\end{equation}
If $\lambda'_{i}\subset\lambda_{i}$ define $\X(\lambda'_{i})\subset\X_{i+1}$ by taking the pushout 
along the attaching maps $s_{\lambda}\rightarrow\X_{i}$ for $\lambda\in\lambda'_{i}$ as in 
(\ref{diagram:cellpushout}).
Note that, 
since $\X$ is finitely presentable, 
there exists a factoring $\X\rightarrow\X(\lambda'_{i})\rightarrow\X_{\infty}$ for 
$\lambda'_{i}\subset\lambda_{i}$ a finite subset.
Likewise, since the coproduct is finitely presentable, 
$\coprod_{\lambda\in\lambda'_{i}}s_{\lambda}\rightarrow\X_{i}$ factors through $\X(\lambda''_{i-1})$ 
for some finite subset $\lambda''_{i-1}$ of $\lambda_{i-1}$.
Clearly $\X(\lambda'_{i})$ is the filtered colimit of $\X(\lambda'_{i-1})(\lambda'_{i})$ for finite 
$\lambda'_{i-1}\subset\lambda_{i-1}$ containing $\lambda''_{i-1}$. 
Hence the map $\X\rightarrow\X(\lambda'_{i})$ factors through some $\X(\lambda'_{i-1})(\lambda'_{i})$.

Iterating this argument we find a factoring of the form
\begin{equation*}
\xymatrix{
\X\ar[r] & \X(\lambda'_{0})(\lambda'_{1})\cdots(\lambda'_{i-1})(\lambda'_{i})\ar[r] & \X_{\infty}, } 
\end{equation*}
as desired.
\end{proof}
\begin{remark}
The last part of the proof does not require that the codomains of $I_{\M}$ are finitely presentable.
\end{remark}

\begin{lemma}
\label{lemma:fpcofibrantfactorization}
Suppose $\M$ is weakly finitely generated and $\X$ and $\Y$ are finitely presentable cofibrant objects in $\M$.
Let $\id_{\M}\rightarrow\RRR$ denote the fibrant replacement functor on $\M$ obtained by applying the small object 
argument to the set $J'_{\M}$.
Then for every map $\X\rightarrow\RRR\Y$ there exists a finitely presentable and cofibrant object $\Y'$ and a 
commutative diagram with horizontal weak equivalences:
\begin{equation*}
\xymatrix{
& X\ar[d]\ar@{=}[r] & \X\ar[d]\\
\Y\ar@{>->}[r] & \Y'\ar@{>->}[r] & \RRR\Y }
\end{equation*}
\end{lemma}
\begin{proof}
Denote by $\RRR_{i}$ all diagrams of the form
\begin{equation*}
\xymatrix{
s'_{\lambda} \ar[d]\ar[r] & \Y_{i-1}\ar[d]\\
t'_{\lambda} \ar[r] & \ast}
\end{equation*}
where $s'_{\lambda}\rightarrow t'_{\lambda}$ is a map in $J'_{\M}$.
Then $\RRR\Y$ is the colimit of the diagram
\begin{equation*}
\xymatrix{
\Y=\Y_{0}\ar@{>->}[r]^-{\sim} & \Y_{1}\ar@{>->}[r]^-{\sim} &
\Y_{2}\ar@{>->}[r]^-{\sim} &\cdots \ar@{>->}[r]^-{\sim} & 
\Y_{n}\ar@{>->}[r]^-{\sim} &\cdots }
\end{equation*}
for pushout diagrams 
\begin{equation*}
\xymatrix{
\coprod s'_{\lambda} \ar[d]\ar[r] & \coprod t'_{\lambda} \ar[d]\\
\Y_{i-1} \ar[r] & \Y_{i}}
\end{equation*}
indexed by the set $\RRR_{i}$.
Since $\X$ is finitely presentable the map $\X\rightarrow\RRR\Y$ factors through some $\Y_{i}$.
Now the trick is to observe that $\Y_{i}$ is a filtered colimit indexed by the finite subsets of 
$\RRR_{i}$ ordered by inclusion.
Hence there is a factoring 
$\X\rightarrow\Y_{\lambda_{i}}\overset{\sim}{\cof}\Y_{i}\overset{\sim}{\cof}\RRR\Y$ 
for some finite subset $\lambda_{i}\subset\RRR_{i}$.
Iterating this argument we find finite subsets $\lambda_{k}\subset\RRR_{k}$ for $1\leq k\leq i$, 
and some factoring 
$\X\rightarrow\Y'\equiv\Y_{\lambda_{1}}\overset{\sim}{\cof}\cdots\overset{\sim}{\cof}
\Y_{\lambda_{i}}\overset{\sim}{\cof}\Y_{i}\overset{\sim}{\cof}\RRR\Y$.
By construction, 
there is an acyclic cofibration $\Y\overset{\sim}{\cof}\Y'$ and $\Y'$ is both finitely presentable and cofibrant.
\end{proof}

Next we recall a much weaker homotopical finiteness condition first introduced in special cases in 
\cite[\S2.1]{Waldhausen:LNM1126}.
An object of $\M$ is called homotopy finitely presentable if it is isomorphic in the homotopy category 
of $\M$ to a finitely presentable cofibrant object.

Let ${\bf hfp}\M_{cof}$ denote the full subcategory of $\M$ of homotopy finitely presentable cofibrant 
objects. 
Note that $\X$ is homotopy finitely presentable if and only if there exist finitely presentable cofibrant 
objects $\Y$ and $\Z$ where $\Z$ is fibrant and weak equivalences 
$\X\overset{\sim}{\rightarrow}\Z\overset{\sim}{\leftarrow}\Y$.
Equivalently, 
there exists a finitely presentable cofibrant object $\Y$ and a weak equivalence from $\X$ to a fibrant
replacement $\RRR\Y$ of $\Y$.
\vspace{0.1in}

With no additional assumptions on the model structure on $\M$ one cannot expect that ${\bf hfp}\M_{cof}$ 
is a category with cofibrations and weak equivalences in the same way as ${\bf fp}\M_{cof}$.
The only trouble is that a pushout of homotopy finitely presentable objects need not be 
homotopy finitely presentable.
However, 
the next result which is reminiscent of \cite[Proposition 3.2]{Sagave:diplom} covers all the cases 
we shall consider in this paper.

\begin{lemma}
\label{lemma:hfpiscatwithcofandweakeq}
Suppose $\M$ is cubical, 
weakly finitely generated and $-\otimes\Box_{+}^{1}$ preserves finitely presentable objects.
Then ${\bf hfp}\M_{cof}$ is a Waldhausen subcategory of $\M$.
\end{lemma}
\begin{proof}
Suppose $\X$, $\Y$ and $\Z$ are homotopy finitely presentable cofibrant objects and there are maps $\Z\leftarrow\X\cof\Y$.
We show the pushout is homotopy finitely presentable by constructing a commutative diagram with vertical weak equivalences:
\begin{equation*}
\xymatrix{
\Z  \ar[d] & \X  \ar[l] \ar@{>->}[r] \ar[d] & \Y  \ar[d] \\
\RRR\widetilde{\Z} & \RRR\X' \ar[l] \ar@{>->}[r] & \W  \\
\Z' \ar[u] & \X' \ar[l] \ar@{>->}[r] \ar[u] & \Y' \ar[u]}
\end{equation*}
\vspace{0.1in}

Applying the gluing lemma for cofibrant objects \cite[II Lemma 8.8]{GJ:Modelcategories}
or the cube lemma \cite[Lemma 5.2.6]{Hovey:Modelcategories} we deduce that the induced map of pushouts 
$\Y\coprod_{\X}\Z\rightarrow\Y'\coprod_{\X'}\Z'$ is a weak equivalence.
This shows that $\Y\coprod_{\X}\Z$ is homotopy finitely presentable provided $\X'$, $\Y'$ and $\Z'$ are 
finitely presentable cofibrant objects.
Next, 
existence of the middle column in the diagram where $\X'$ is finitely presentable and cofibrant follows 
because $\X$ is homotopy finitely presentable.
For the same reason there exists a weak equivalence $\Z\rightarrow\RRR\widetilde{\Z}$ for some finitely 
presentable and cofibrant object $\widetilde{\Z}$. 
Since $\X$ maps to the fibrant object $\RRR\widetilde{\Z}$ there exists a map 
$\RRR\X'\rightarrow\RRR\widetilde{\Z}$ by the lifting axiom in $\M$.

Now Lemma \ref{lemma:fpcofibrantfactorization} shows the composite map 
\begin{equation*}
\xymatrix{
\X'\ar[r] &
\RRR\X'\ar[r] & \RRR\widetilde{\Z} }  
\end{equation*}
factors through some finitely presentable and cofibrant object $\Z'$ which maps by an acyclic cofibration 
to $\RRR\widetilde{\Z}$. 
It remains to construct $\W$ and $\Y'$. 
\vspace{0.1in}

Since $\Y$ is homotopy finitely presentable there exists a finitely presentable cofibrant object 
$\widetilde{\Y}$ and a weak equivalence from $\Y$ to $\RRR\widetilde{\Y}$. 
The lifting axiom ${\bf CM}4$ in $\M$ yields a map $\RRR\X'\rightarrow\RRR\widetilde{\Y}$ and by 
Lemma \ref{lemma:fpcofibrantfactorization} the composite $\X'\rightarrow\RRR\X'\rightarrow\RRR\widetilde{\Y}$ 
factors through some finitely presentable cofibrant object $\widetilde{\Y}'$.
Using the cubical mapping cylinder we may factor the latter map as a cofibration $\X'\cof\Y'$ composed with 
a cubical homotopy.
Note that $\Y'$ is finitely presentable and cofibrant by the assumption that $-\otimes\Box_{+}^{1}$ preserves 
finitely presentable objects.
By the factorization axiom ${\bf CM}5$ in $\M$ we deduce there is a cofibration $\RRR\X'\cof\W$, 
so that $\W$ is cofibrant,  
and an acyclic fibration $\W\overset{\sim}{\fib}\RRR\widetilde{\Y}$. 
\vspace{0.1in}

Finally, 
the weak equivalences between $\Y$, $\W$ and $\Y'$ follow since there exist liftings in the following diagrams 
(the lower horizontal maps are weak equivalences and so are the right vertical fibrations):
\begin{equation*}
\xymatrix{
\ast\ar@{>->}[r]\ar[d] & \W \ar@{->>}[d]\\
\Y \ar[r] & \RRR\widetilde{\Y}}
\;\;\;\;\;
\xymatrix{
\ast\ar@{>->}[r]\ar[d] & \W \ar@{->>}[d]\\
\Y'\ar@{>->}[r] & \RRR\widetilde{\Y}}
\end{equation*}
\end{proof}
\vspace{0.1in}

Note that ${\bf hfp}\M_{cof}$ contains more fibrant objects than ${\bf fp}\M_{cof}$ since non-constant fibrant objects 
need not be finitely presentable.
The next result follows easily from a version of Waldhausen's approximation theorem \cite[Theorem 2.8]{Sagave:diplom} and 
Lemma \ref{lemma:hfpiscatwithcofandweakeq}.
\vspace{0.1in}

For the convenience of the reader we recall the setup.
A category with cofibrations and weak equivalences $\CCC$ is equipped with special objects if there is a full subcategory
$\CCC'\subseteq\CCC$ and a functor $\QQ\colon\CCC\rightarrow\CCC'$  together with a natural transformation $\id_{\CCC}\rightarrow\QQ$ 
such that $\X\rightarrow\QQ\X$ is a cofibration and a weak equivalence for every object $\X$ of $\CCC$. 
Cofibrant replacement functors in model categories furnish the prime examples of special objects.
\vspace{0.1in}

Axioms ${\bf App} 1$ and ${\bf App} 2$ formulated below are used in the original formulation of the approximation theorem 
\cite[Theorem 1.6.7]{Waldhausen:LNM1126}. 
In the recent slightly modified version \cite[Theorem 2.8]{Sagave:diplom} which we shall refer to as the special approximation theorem, 
axiom ${\bf SApp} 2$ replaces ${\bf App} 2$.
\begin{definition}
\label{definition:approximation}
Let $F\colon\CCC\rightarrow\DDD$ be an exact functor.
\begin{enumerate}[{\bf App} 1:]
\item $F$ reflects weak equivalences.
\item Every map $F(\X)\rightarrow\Z$ in $\DDD$ factors as $F(\X\rightarrow\Y)$ for a cofibration $\X\rightarrow\Y$ in $\CCC$ 
composed with a weak equivalence $F(\Y)\rightarrow\Z$ in $\DDD$.  
\end{enumerate}

\begin{enumerate}[{\bf SApp} 2:]
\item Suppose $\CCC$ is equipped with special objects.
Then ${\bf App}2$ holds if $\Z$ is a special object.
\end{enumerate}
\end{definition}
\vspace{0.1in}

\begin{proposition}
\label{proposition:Ktheoryequivalencefphfp}
Suppose $\M$ is cubical,
weakly finitely generated and $-\otimes\Box_{+}^{1}$ preserves finitely presentable objects.
Then $({\bf fp}\subset{\bf hfp})\M_{cof}$ induces an equivalence in $K$-theory.
\end{proposition}
\vspace{0.1in}

Next we turn to the examples arising in $\CC^{\ast}$-homotopy theory. 
Let $(\X,\Box\CC^{\ast}-\Spc,\X)$ denote the retract category of a cubical $\CC^{\ast}$-space $\X$. 
The homotopy theory of this category was worked out in Lemma \ref{lemma:modelstructuresonretractivecategories}.
We are interested in the $K$-theory of its full subcategory of finitely presentable cofibrant objects.
\begin{definition}
\label{KtheoryofCstaralgebras}
The $K$-theory of a cubical $\CC^{\ast}$-space $\X$ is 
\begin{equation*}
K(\X)\equiv K\bigl({\bf fp}(\X,\Box\CC^{\ast}-\Spc_{cof},\X)\bigr).
\end{equation*}
\end{definition}
\vspace{0.1in}

Lemma \ref{lemma:hfpiscatwithcofandweakeq} implies ${\bf hfp}(\Box\CC^{\ast}-\Spc)_{cof}$ is a Waldhausen subcategory 
of $\Box\CC^{\ast}-\Spc$.
By applying Proposition \ref{proposition:Ktheoryequivalencefphfp} we get the next result.
\begin{lemma}
\label{lemma:fpKtheory=hpfKtheory}
The special approximation theorem applies to the inclusion 
\begin{equation*}
({\bf fp}\subset{\bf hfp})(\X,\Box\CC^{\ast}-\Spc_{cof},\X). 
\end{equation*}
Thus for every cubical $\CC^{\ast}$-space $\X$ there is an induced equivalence 
\begin{equation*}
\xymatrix{
K(\X)\ar[r] & 
K\bigl({\bf hfp}(\X,\Box\CC^{\ast}-\Spc_{cof},\X)\bigr).}
\end{equation*}
\end{lemma}
\vspace{0.1in}

In the remaining of this section we consider the $K$-theory of the trivial $\CC^{\ast}$-algebra.
Adopting the work of R{\"o}ndigs \cite{Roendigs} to our setting we show a fundamental result:
\begin{theorem}
\label{theorem:K-theoryintermsofspectra}
Denote by $\Spt_{S^{1}}$ the category of $S^{1}$-spectra of pointed cubical $\CC^{\ast}$-spaces.
Then the $K$-theory of the trivial $\CC^{\ast}$-algebra is equivalent to the $K$-theory of ${\bf hfp}(\Spt_{S^{1}})_{cof}$. 
\end{theorem}
\vspace{0.1in}

To begin with, 
consider a sequential diagram of categories with cofibrations and weak equivalence
\begin{equation}
\label{squentialcofibrationsweakequivalence}
\xymatrix{
\M_{0}\ar[r] & 
\M_{1}\ar[r] &
\cdots\ar[r] &
\M_{n}\ar[r] &
\M_{n+1}\ar[r] &
\cdots.}
\end{equation}

It is straightforward to check that the colimit $\M_{\infty}$ of (\ref{squentialcofibrationsweakequivalence}) taken in the 
category of small categories is a category with cofibrations and weak equivalences:  
A map in $\M_{\infty}$ is a cofibration if some representative of it is a cofibration, and likewise for weak equivalences.
Moreover, 
with this definition the canonical functor $\M_{n}\rightarrow\M_{\infty}$ is exact and there is a naturally induced isomorphism
\begin{equation*}
\xymatrix{
\colim_{n}\,S_{\bullet}\M_{n}\ar[r] &
S_{\bullet}\M_{\infty} }
\end{equation*}
of simplicial categories with cofibrations and weak equivalences.
Now specialize to the constant sequential diagram with value ${\bf fp}(\Box\CC^{\ast}-\Spc_{0})_{cof}$ and transition map the
suspension functor $S^{1}\otimes -$.
Let $S^{1}{\bf fp}(\Box\CC^{\ast}-\Spc_{0})_{cof}$ denote the corresponding colimit with cofibrations and weak equivalences as 
described above.
\begin{lemma}
\label{lemma:step1}
The canonical functor 
\begin{equation*}
\xymatrix{
{\bf fp}(\Box\CC^{\ast}-\Spc_{0})_{cof}\ar[r] & 
S^{1}{\bf fp}(\Box\CC^{\ast}-\Spc_{0})_{cof} }
\end{equation*}
is a $K$-theory equivalence.
\end{lemma}
\begin{proof}
Since the category ${\bf fp}(\Box\CC^{\ast}-\Spc_{0})_{cof}$ has a good cylinder functor, 
the suspension functor $S^{1}\otimes -$ induces a $K$-theory equivalence \cite[Proposition 1.6.2]{Waldhausen:LNM1126}.
\end{proof}

Next we relate the target of the $K$-theory equivalence in Lemma \ref{lemma:step1} to $S^{1}$-spectra of pointed cubical $\CC^{\ast}$-spaces.
An object $\E$ of $\Spt_{S^{1}}$ is called strictly finitely presentable if $\E_{n}$ is finitely presentable in $\Box\CC^{\ast}-\Spc_{0}$ for 
every $n\geq 0$, and there exists an integer $n(\E)$ such that the structure maps of $\E$ are identity maps for $n\geq n(\E)$.
Every finitely presentable $S^{1}$-spectrum is isomorphic to a strictly finitely presentable one.
It implies that the inclusion functor ${\bf sfp}(\Spt_{S^{1}})_{cof}\hookrightarrow {\bf fp}(\Spt_{S^{1}})_{cof}$ is an equivalence of categories,
and therefore: 
\begin{lemma}
\label{lemma:step2}
The inclusion functor 
\begin{equation*}
\xymatrix{
{\bf sfp}(\Spt_{S^{1}})_{cof}\ar[r] &  
{\bf fp}(\Spt_{S^{1}})_{cof} }
\end{equation*}
is a $K$-theory equivalence.
\end{lemma}
\vspace{0.1in}

Define the functor
\begin{equation*}
\xymatrix{
\Phi\colon 
{\bf sfp}(\Spt_{S^{1}})_{cof} \ar[r] &
S^{1}{\bf fp}(\Box\CC^{\ast}-\Spc_{0})_{cof} }
\end{equation*}
by sending $\E$ to $(\E_{n},n)$ and $f\colon\E\rightarrow\F$ to $(f_{n},n)$ for $n\geq n(\E),n(\F)$.
Note that $\Phi$ does not extend to a functor from ${\bf fp}(\Spt_{S^{1}})_{cof}$ to $S^{1}{\bf fp}(\Box\CC^{\ast}-\Spc_{0})_{cof}$.
\begin{proposition}
\label{proposition:mainstep}
The functor $\Phi$ is exact and has the approximation property.
\end{proposition}
\begin{proof}
It is clear that $\Phi$ preserves the point and also (projective) cofibrations because if $\E\rightarrow\F$ is a cofibration in $\Spt_{S^{1}}$, 
then $\E_{n}\rightarrow\F_{n}$ is a cofibration of pointed cubical $\CC^{\ast}$-spaces for every $n\geq 0$.
The interesting part of the proof consists of showing that $\Phi$ preserves stable weak equivalences.
More precisely, 
if $\E\rightarrow\F$ is a stable weak equivalence of finitely presentable cofibrant $S^{1}$-spectra of pointed cubical $\CC^{\ast}$-spaces, 
then $\E_{n}\rightarrow\F_{n}$ is a $\CC^{\ast}$-weak equivalence for all $n>>0$.
This uses that the stable model structure on $\Spt_{S^{1}}$ is weakly finitely generated.
\vspace{0.1in}

Next we show that $\Phi$ has the approximation property:
It clearly detects weak equivalences.
For a map $\Phi(\E)\rightarrow (\X,m)$ in $S^{1}{\bf fp}(\Box\CC^{\ast}-\Spc_{0})_{cof}$ we may choose a representative $\E_{n}\rightarrow\Y$,
and there exists an integer $k$ such that $S^{k+m}\otimes\Y=S^{k+n}\otimes\X$.
The map $\E_{n}\rightarrow\Y$ factors through $\cyl(\E_{n}\rightarrow\Y)$ for the good cylinder functor on ${\bf fp}(\Spt_{S^{1}})_{cof}$.
Define the strictly finitely presentable cofibrant $S^{1}$-spectrum $\cyl(\E\rightarrow\Y)$ of pointed cubical $\CC^{\ast}$-spaces by 
\begin{equation*}
\cyl(\E\rightarrow\Y)_{m}\equiv
\begin{cases}
\E_{m} & m<n\\
\cyl(\E_{n}\rightarrow\Y) & m=n\\
S^{1}\otimes\cyl(\E\rightarrow\Y)_{m-1} & m>n.
\end{cases}
\end{equation*}
The structure maps of $\cyl(\E\rightarrow\Y)$ are given by the structure maps of $\E$ if $m<n-1$, 
by $S^{1}\otimes\E_{n-1}\rightarrow\E_{n}\rightarrow\cyl(\E_{n}\rightarrow\Y)$ if $m=n-1$,
and by the appropriate identity map if $m\geq n$.
Clearly, 
$\E\rightarrow\cyl(\E\rightarrow\Y)$ is a (projective) cofibration which provides the required factoring by applying $\Phi$.
\end{proof}
In order to conclude that $\Phi$ is a $K$-theory equivalence, 
note that ${\bf sfp}(\Spt_{S^{1}})_{cof}$ inherits a good cylinder functor from ${\bf fp}(\Spt_{S^{1}})_{cof}$ so that Lemma \ref{lemma:step2} 
allows us to apply the approximation theorem.
\begin{lemma}
\label{lemma:step3}
The functor 
\begin{equation*}
\xymatrix{
\Phi\colon{\bf sfp}(\Spt_{S^{1}})_{cof}\ar[r] &
 S^{1}{\bf fp}(\Box\CC^{\ast}-\Spc_{0})_{cof} } 
\end{equation*}
is a $K$-theory equivalence. 
\end{lemma}
\vspace{0.1in}

Proposition \ref{proposition:Ktheoryequivalencefphfp} takes care of the remaining $K$-theory equivalence needed to finish the proof of 
Theorem \ref{theorem:K-theoryintermsofspectra}. 
\begin{lemma}
\label{lemma:step4}
The inclusion functor 
\begin{equation*}
\xymatrix{
{\bf fp}(\Spt_{S^{1}})_{cof}\ar[r] & 
{\bf hfp}(\Spt_{S^{1}})_{cof} } 
\end{equation*}
is a $K$-theory equivalence. 
\end{lemma}
\vspace{0.1in}

Localization techniques imply our last result in this section.
\begin{theorem}
\label{theorem:Ktheoryretract}
The $K$-theory of ${\bf fp}(\Box\Set_{\ast})_{cof}$ is a retract of the $K$-theory of ${\bf fp}(\Box\CC^{\ast}-\Spc_{0})_{cof}$ up to homotopy.
\end{theorem}
\vspace{0.1in}

\begin{remark}
Theorem \ref{theorem:Ktheoryretract} connects $\CC^{\ast}$-homotopy theory to geometric topology since the $K$-theory of ${\bf fp}(\Box\Set)_{cof}$ 
is Waldhausen's $A(\ast)$ or the $K$-theory of the sphere spectrum \cite[\S2]{Waldhausen:LNM1126}.
The spectrum $A(\ast)$ is of finite type \cite{Dwyer:homologystability} and rationally equivalent to the algebraic $K$-theory of the integers.
We refer to \cite{Rosenberg:handbook} for a recent survey and further references.
Theorem \ref{theorem:Ktheoryretract} shows the $K$-theory of the trivial $\CC^{\ast}$-algebra as defined by $\CC^{\ast}$-homotopy theory carries 
highly nontrivial invariants.
\end{remark}
\newpage

\subsection{Zeta functions}
\label{subsection:zetafunctions}
Our definition of zeta functions of $\CC^{\ast}$-algebras is deeply rooted in algebraic geometry.
If $X$ is a quasi-projective variety over a finite field $\FF_{q}$ then its Hasse-Weil zeta function is traditionally defined in terms of the number 
of $\FF_{q^{n}}$-points of $X$ by the formula 
\begin{equation*}
\zeta_{X}(t)\equiv
\exp\bigl(\Sigma_{n=1}^{\infty} \# X(\FF_{q^{n}})\frac{t^{n}}{n}\bigr).
\end{equation*} 
The symmetric group $\Sigma_{n}$ on $n$ letters acts on the $n$-fold product $X\times\cdots\times X$ and the symmetric power $\Sym^{n}(X)$ of $X$ 
is a quotient quasi-projective variety over $\FF_{q}$.
For example,
the higher dimensional affine spaces $\AAA^{n}_{\FF_{q}}=\Sym^{n}(\AAA^{1}_{\FF_{q}})$ and projective spaces 
$\PPP^{n}_{\FF_{q}}=\Sym^{n}(\PPP^{1}_{\FF_{q}})$ arise in this way.
Using symmetric powers the Hasse-Weil zeta function can be rewritten as the formal power series
\begin{equation*}
\zeta_{X}(t)=
\Sigma_{n=0}^{\infty} \#\Sym^{n}(X)(\FF_{q})t^{n}.
\end{equation*} 
Kapranov \cite{Kapranov} has generalized the whole setup by incorporating multiplicative Euler characteristics with compact support in 
the definition of zeta functions.
That is, 
if $\mu$ is an invariant of quasi-projective variety over $\FF_{q}$ with values in a ring $R$ for which $\mu(X)=\mu(X\smallsetminus Y)+\mu(Y)$ 
for every $Y\subset X$ closed and $\mu(X\times Y)=\mu(X)\mu(Y)$,
then the zeta function of $X$ with respect to $\mu$ is the formal power series
\begin{equation*}
\zeta_{X,\mu}(t)\equiv
\Sigma_{n=0}^{\infty} \mu\Sym^{n}(X)t^{n} \in R[[t]].
\end{equation*}
A typical choice of the ground ring is Grothendieck's $K_{0}$-group of varieties over $\FF_{q}$ with ring structure induced by products of 
varieties. 
\vspace{0.1in}

In our first setup the ring $R$ will be a $K_{0}$-group of the thick symmetric monoidal triangulated subcategory $\SH^{\ast\ccc}_{\Q}$ of compact 
objects in the rationalized stable homotopy category of $\CC^{\ast}$-algebras.
The Eilenberg swindle explains why we restrict to compact objects in $\SH^{\ast}_{\Q}$:
If $[\E]$ is a class in $K_{0}(\SH^{\ast})$ and $\coprod\E$ an infinite coproduct of copies of $\E$, 
the identification $\E\oplus\coprod\E=\coprod\E$ implies the class of $\E$ is trivial. 
A crux step toward the definition of zeta functions of $\CC^{\ast}$-algebras is to note there exist symmetric powers of compact objects giving 
rise to a $\lambda$-structure on the ring $K_{0}(\SH^{\ast\ccc}_{\Q})$ with multiplication induced by the monoidal product on $\SH^{\ast}$.
It is the $\lambda$-structure that ultimately allows us to push through the definition of zeta functions by a formula reminiscent of the classical 
one in algebraic geometry.
\vspace{0.1in}

Throughout what follows we shall tactically replace the category $\SH^{\ast}_{\Q}$ with its idempotent completion, 
turning it into a pseudoabelian symmetric monoidal $\Q$-linear category. 
Working with the idempotent completion, 
so that every projector acquires an image,
allows us to construct symmetric powers and wedge powers by using Young's work on the classical representation theory of symmetric groups dealing 
with idempotents and partitions.
We review this next, 
cf.~\cite{JK} and \cite{Weyl} for details.
\vspace{0.1in}

Recall the set of irreducible representations of $\Sigma_{n}$ over $\Q$ is in bijection with the set of partitions $\nu$ of $n$.
Moreover,
there exists a set of orthogonal idempotents $e_{\nu}$ in the group ring $\Q[\Sigma_{n}]$ called the Young symmetrizer so that 
$\Sigma_{\nu}e_{\nu}=1_{\Q[\Sigma_{n}]}$ and $e_{\nu}$ induces the corresponding representation of $\Sigma_{n}$ (up to isomorphism).
For every element $\E$ of $\SH^{\ast}_{\Q}$ there is an algebra map from $\Q[\Sigma_{n}]$ to the endomorphisms of the $n$-fold product
$\E^{(n)}\equiv\E\wedge^{\bf L}\cdots\wedge^{\bf L}\E$ given by sending $\sigma\in\Sigma_{n}$ to the endomorphism $\E_{\sigma}$ that permutes 
the factors accordingly.
That is,
writing $\sigma\in\Sigma_{n}$ as a product of elementary transpositions $(i,i+1)$ for $1\leq i<n$ and letting the latter act on $\E^{(n)}$ 
by applying the commutativity constrain between the $i$th and $i+1$st factor yields a well-defined action.
The identity $\Sigma_{\nu}\E_{e_{\nu}}=\id_{\E^{(n)}}$ follows immediately, 
where $\E_{e_{\nu}}$ is the endomorphism of $\E^{(n)}$ obtained from $e_{\nu}$.
Since $e^{2}_{\nu}=e_{\nu}$ and we are dealing with a pseudo-abelian category, 
the $n$-fold product of $\E$ splits into a direct sum of the images of the idempotents $\E_{e_{\nu}}$ in the endomorphism ring 
$\SH^{\ast}_{\Q}(\E^{(n)},\E^{(n)})$.

\begin{definition}
The Schur functor $\SSS_{\nu}$ of a partition $\nu$ of $n$ is the endofunctor of $\SH^{\ast}_{\Q}$ defined by 
$\SSS_{\nu}(\E)\equiv\E_{e_{\nu}}(\E^{(n)})$.
We say that $\E$ is Schur finite if there exists an integer $n$ and a partition $\nu$ of $n$ such that $\SSS_{\nu}(\E)=0$.
\end{definition}

The notion of Schur finiteness was introduced by Deligne \cite{Deligne2}.
See \cite[A 2.5]{Eisenbud} for more background. 
We thank Mazza for discussions about Schur finiteness in the algebro-geometric setting of motives \cite{Mazza}.

Next we define, 
corresponding to the partition $(n)$ of $n$,  
the $n$th symmetric power of $\E$ by 
\begin{equation*}
\Sym^{n}(\E)\equiv
\SSS_{(n)}(\E)=
\E_{e_{(n)}}(\E^{(n)})=
\frac{1}{n!}
\underset{\sigma\in\Sigma_{n}}{\Sigma}\E_{\sigma}
(\E^{(n)}).
\end{equation*}
Similarly, 
corresponding to the partition $(1,\dots,1)$ of $n$, 
we define the $n$th wedge power of $\E$ by 
\begin{equation*}
\Alt^{n}(\E)\equiv
\SSS_{(1,\dots,1)}(\E)=
\E_{e_{(1,\dots,1)}}(\E^{(n)})=
\frac{1}{n!}
\underset{\sigma\in\Sigma_{n}}{\Sigma}\sgn(\sigma)\E_{\sigma}
(\E^{(n)}).
\end{equation*}
\begin{remark}
The corresponding notions for rationalized Chow groups introduced in \cite{Kimura1} is the subject of much current research on motives in 
algebraic geometry;
in particular, 
the following notions receive much attention:
$\E$ is negative or positive finite dimensional provided $\Sym^{n}(\E)=0$ respectively $\Alt^{n}(\E)=0$ for some $n$, 
and finite dimensional if there exists a direct sum decomposition $\E=\E_{+}\oplus\E_{-}$ where $\E_{+}$ is positive and $\E_{-}$ is negative 
finite dimensional.
The monoidal product of two finite dimensional objects in $\SH_{\Q}$ is finite dimensional, 
and the same holds for Schur finite objects.
The main result in \cite{Guletski1} shows that negative and positive finite dimensional objects satisfy the two-out-of-three property for 
distinguished triangles. 
A thorough study of Schur finite and finite dimensional objects in some rigid setting of $\CC^{\ast}$-algebras remains to be conducted.
\end{remark}

Recall that the Grothendieck group $K_{0}(\SH^{\ast\ccc}_{\Q})$ is the quotient of the free abelian group generated by the isomorphism classes 
$[\E]$ of objects of $\SH^{\ast\ccc}_{\Q}$ by the subgroup generated by the elements $[\F]-[\E]-[\G]$ for every distinguished triangle 
$\E\rightarrow\F\rightarrow\G$.
Due to the monoidal product $\wedge^{\bf L}$ on $\SH^{\ast}_{\Q}$ there is an induced multiplication on $K_{0}(\SH^{\ast\ccc}_{\Q})$ 
which turns the latter into a commutative unital ring.
\vspace{0.1in}

The main result in Guletski{\u\i}'s paper \cite{Guletski2} shows that the wedge and symmetric power constructions define opposite 
$\lambda$-structures on $K_{0}(\SH^{\ast\ccc}_{\Q})$.
Next we recall these notions.
The interested reader can consult the papers by Atiyah and Tall \cite{AT} 
and by Grothendieck \cite{Grothendieck} for further details on this subject (which is important in $K$-theory).
\vspace{0.1in}

Let $R$ be a commutative unital ring.
Then a $\lambda$-ring structure on $R$ consists of maps $\lambda^{n}\colon R\rightarrow R$ for every integer $n\geq 0$ such that the following 
conditions hold:
\begin{itemize}
\item 
$\lambda^{0}(r)=1$ for all $r\in R$
\item 
$\lambda^{1}=\id_{R}$
\item 
$\lambda^{n}(r+r')=\underset{i+j=n}{\Sigma}\lambda^{i}(r)\lambda^{j}(r')$
\end{itemize}

A $\lambda$-ring structure on $R$ induces a group homomorphism
\begin{equation}
\label{equation:universallambdamap}
\xymatrix{
\lambda_{t}\colon
R\ar[r] &
1+tR[[t]];
\,\,
r\ar@{|->}[r] &
1+\Sigma_{n\geq 1}\lambda^{n}(r)t^{n}, }
\end{equation}
from the underlying additive group of $R$ to the multiplicative group of formal power series in an indeterminate $t$ over $R$ with constant 
term $1$, 
i.e.~$\lambda_{t}(r+r')=\lambda_{t}(r)\lambda_{t}(r')$.

The ring $1+tR[[t]]$ acquires a $\lambda$-ring structure when the addition is defined by multiplication of formal power series, 
multiplication and $\lambda$-operations are given by universal polynomials, 
which in turn are uniquely determined by the identities
\begin{equation*}
\bigl(\prod_{i=1}^{k}(1+a_{i}t)\bigr)
\bigl(\prod_{j=1}^{l}(1+b_{j}t)\bigr)=
\prod_{i=1}^{k}\prod_{j=1}^{l}(1+a_{i}b_{j}t)
\end{equation*}
and
\begin{equation*}
\lambda^{n}
\bigl(\prod_{i=1}^{k}(1+a_{i}t)\bigr)=
\prod_{S\subset\{1,\dots,k\}, \#S=n}\bigl(1+t\prod_{j\in S}a_{j}\bigr).
\end{equation*}

Ring homomorphisms between $\lambda$-rings commuting with all the $\lambda$-operations are called $\lambda$-ring homomorphisms.
And a $\lambda$-ring $R$ is called special if the map (\ref{equation:universallambdamap}) is a $\lambda$-ring homomorphism.
In particular, 
the $\lambda$-ring $1+tR[[t]]$ is special,
as was noted by Grothendieck \cite{AT}.
\vspace{0.1in}

Every ring homomorphism $\phi\colon R\rightarrow 1+tR[[t]]$ for which $\phi(r)=1+rt+$higher degree terms defines a $\lambda$-ring structure on $R$.
The opposite $\lambda$-ring structure of a $\lambda$-ring $R$ is the $\lambda$-ring structure associated with the ring homomorphism 
$\phi(r)\equiv\lambda_{-t}(-r)=\lambda_{-t}(r)^{-1}$.
\vspace{0.1in}

In the example of $K_{0}(\SH^{\ast\ccc}_{\Q})$ we set 
\begin{equation}
\label{equation:lambdastructure}
\lambda^{n}([\E])\equiv 
[\Sym^{n}(\E)].
\end{equation}
The main result in \cite{Guletski2} shows that (\ref{equation:lambdastructure}) defines a $\lambda$-ring structure on $K_{0}(\SH^{\ast\ccc}_{\Q})$.
Its opposite $\lambda$-structure arises by replacing the class of $\Sym^{n}(\E)$ by the class of $\Alt^{n}(\E)$ in the definition
(\ref{equation:lambdastructure}). 
The class $[\Sym^{0}(\E)]$ is the unit in $K_{0}(\SH^{\ast\ccc}_{\Q})$.

\begin{definition}
Let 
\begin{equation*}
\xymatrix{
K_{0}(\SH^{\ast\ccc}_{\Q})\ar[r] &
1+tK_{0}(\SH^{\ast\ccc}_{\Q})[[t]];
\,\,
[\E]\ar@{|->}[r] &
1+\Sigma_{n\geq 1}[\Sym^{n}(\E)]t^{n} }
\end{equation*}
be the $\lambda$-ring homomorphism determined by the $\lambda$-ring structure on $K_{0}(\SH^{\ast\ccc}_{\Q})$ in (\ref{equation:lambdastructure}).

The zeta function of $\E$ in $\SH^{\ast\ccc}_{\Q}$ is the formal power series
\begin{equation}
\zeta_{\E}(t)\equiv
\Sigma_{n\geq 0}[\Sym^{n}(\E)]t^{n}.
\end{equation}
\end{definition}
\begin{remark}
We trust the first part of this section makes it plain that our definition of zeta functions of $\CC^{\ast}$-algebras is deeply rooted in 
algebraic geometry.
\end{remark}

The definition of zeta functions makes it clear that the following result holds.
\begin{lemma}
If $\E\rightarrow\F\rightarrow\G$ is a distinguished triangle in $\SH^{\ast\ccc}_{\Q}$ then 
\begin{equation*}
\zeta_{\F}=\zeta_{\E}\zeta_{\G}.
\end{equation*}
\end{lemma}

\begin{definition}
A power series $f(t)\in K_{0}(\SH^{\ast\ccc}_{\Q})[[t]]$ is (globally) rational if there exists polynomials 
$g(t),h(t)\in K_{0}(\SH^{\ast\ccc}_{\Q})[t]$ such that $f(t)$ is the unique solution of the equation $g(t)x=h(t)$.
\end{definition}
\begin{corollary}
If $\E\rightarrow\F\rightarrow\G$ is a distinguished triangle in $\SH^{\ast\ccc}_{\Q}$ and two of the three zeta functions $\zeta_{F}$, 
$\zeta_{E}$ and $\zeta_{G}$ are rational, 
then so is the third.
\end{corollary}

For classes $[\E]$ and $[\F]$ in $K_{0}(\SH^{\ast\ccc}_{\Q})$, 
if the zeta functions $\zeta_{[\E]}$, $\zeta_{[\F]}$ are rational then so is $\zeta_{[\E]\oplus [\F]}$.
Moreover, 
since the $\lambda$-structure on $K_{0}(\SH^{\ast\ccc}_{\Q})$ is special, 
it follows that $\zeta_{[\E\wedge^{\bf L}\F]}=\zeta_{\E}\ast\zeta_{\F}$ is also rational, 
where the product $\ast$ on the right hand side of the equation is given by the multiplication in the $\lambda$-ring 
$1+tK_{0}(\SH^{\ast\ccc}_{\Q})[[t]]$.
Thus rationality of zeta functions are closed under addition and multiplication in $K_{0}(\SH^{\ast\ccc}_{\Q})$.
Moreover,
the shift functor in the triangulated structure on $\SH^{\ast}_{\Q}$ preserves rationality.
The next result follows easily from the equality 
\begin{equation*}
\bigl(\Sigma_{n\geq 0}[\Alt^{n}(\E)](-t)^{n}\bigr)
\bigl(\Sigma_{n\geq 0}[\Sym^{n}(\E)]t^{n}\bigr)
=1
\end{equation*}
in $K_{0}(\SH^{\ast\ccc}_{\Q})[[t]]$.
\begin{lemma}
\begin{itemize}
\item 
If $\E_{-}$ is negative finite dimensional, 
then $\zeta_{\E_{-}}(t)$ is a polynomial.
\item 
If $\E_{+}$ is positive finite dimensional, 
then $\zeta_{\E_{+}}(t)^{-1}$ is a polynomial.
\item
If $\E$ is finite dimensional, 
then $\zeta_{\E}(t)$ is rational.
\end{itemize}
\end{lemma}

For the purpose of showing a functional equation for zeta functions of $\CC^{\ast}$-algebras we shall shift focus to the rationalized category 
$\MM(\KK)_{\Q}$ of $KK$-motives.  
The latter is the homotopy category of a stable monoidal model structure on non-connective or $\ZZ$-graded chain complexes $\Ch(\KK)$ of pointed 
$\CC^{\ast}$-spaces with $\KK$-transfers constructed similarly to the homotopy invariant model structure on $\CC^{\ast}$-spaces.  
We leave open the question of comparing zeta functions defined in terms of $\SH^{\ast\ccc}_{\Q}$ and $\MM(\KK)^{\ccc}_{\Q}$, 
but note that the properties shown so far in this section hold for $K_{0}(\MM(\KK)^{\ccc}_{\Q})$ and hence $\zeta_{\E}(t)$, 
where now $\E$ is a compact object in $\MM(\KK)_{\Q}$.

Next we outline the construction of the category of $KK$-motives.
More details will appear in \cite{Ostvar:noncommutativemotives}.
\vspace{0.1in}

There is a category $\KK$ of $\CC^{\ast}$-algebras with maps $A\rightarrow B$ the elements of $\KK(A,B)$ and composition provided by the 
intersection product in $KK$-theory.
We note that $\KK$ is symmetric monoidal and enriched in abelian groups, but it is not abelian.
Let $\Ch(\KK)$ denote the abelian category of additive functors from $\KK$ to non-connective chain complexes $\Ch$ of abelian groups equipped 
with the standard projective model structure \cite[Theorem 2.3.11]{Hovey:Modelcategories}.
By \cite[Theorem 4.4]{DRO:general} there exists a pointwise model structure on $\Ch(\KK)$.
Moreover, 
the pointwise model structure is stable because the projective model structure on $\Ch$ is stable.
\vspace{0.1in}

Next we may localize the pointwise model structure in order to construct the exact, matrix invariant and homotopy invariant symmetric monoidal 
stable model structures on $\Ch(\KK)$.
With due diligence these steps can be carried out as for $\CC^{\ast}$-spaces.
The discussion of Euler characteristics in stable $\CC^{\ast}$-homotopy theory carries over to furnish $\MM(\KK)$ with an additive invariant 
related to triangulated structure.
Moreover, 
we may replace abelian groups with modules over some commutative ring with unit.
In particular, 
there is a rationalized category of motives $\MM(\KK)_{\Q}$ corresponding to non-connective chain complexes of rational vector spaces.
In this category we may form symmetric powers $\Sym^{n}(\E)$ and wedge powers $\Alt^{n}(\E)$ for every object $\E$ and $n\geq 0$. 
Finally, 
we note that the endomorphism ring of the unit $\MM(\KK)_{\Q}({\bf 1},{\bf 1})$ is a copy of the rational numbers.
\vspace{0.1in}

Recall that $\E$ is negative or positive finite dimensional provided $\Sym^{n}(\E)=0$ respectively $\Alt^{n}(\E)=0$ for some $n$, 
and finite dimensional if there exists a direct sum decomposition $\E=\E_{+}\oplus\E_{-}$ where $\E_{+}$ is positive and $\E_{-}$ is negative 
finite dimensional.
We denote by $\MM(\KK)^{\fd}_{\Q}$ the thick subcategory of finite dimensional objects in $\MM(\KK)_{\Q}$.
\vspace{0.1in}

The following result shows in particular that the Euler characteristics of negative and positive finite dimensional rational motives are integers.   
\begin{proposition}
\begin{itemize}
\item 
If $\E$ is finite dimensional, 
then a direct sum decomposition 
$$\E=\E_{+}\oplus\E_{-}$$ 
is unique up to isomorphism.  
\item 
If $\E$ is negative finite dimensional, 
then $\chi(\E)$ is a nonpositive integer and the smallest $n$ such that $\Sym^{n}(\E)=0$ equals $1-\chi(M)$.
\item 
If $\E$ is positive finite dimensional, 
then $\chi(\E)$ is a nonnegative integer and the smallest $n$ such that $\Alt^{n}(\E)=0$ equals $1+\chi(M)$.
\end{itemize}
\end{proposition}
\begin{proof}
The first part follows as in the proof of \cite[Proposition 6.3]{Kimura1}, cf.~\cite[Proposition 9.1.10]{AKOS}, 
and the last part as in \cite[Theorems 7.2.4, 9.1.7]{AKOS}, cf.~\cite[\S7.2]{Deligne1}. 
\end{proof}

\begin{definition}
If $\E$ is finite dimensional, 
define
\begin{equation*}
\chi_{+}(\E)=
\chi(\E_{+})
\text{ and } 
\chi_{-}(\E)= 
\chi(\E_{-}).
\end{equation*}
\end{definition}

The first part of the next definition is standard while the second part can be found in \cite[Definition 8.2.4]{Kimura2}.
\begin{definition}
In $\MM(\KK)_{\Q}$ we make the following definitions.
\begin{itemize}
\item 
An object $\E$ is invertible if there exists an object $\F$ and an isomorphism 
$$\E\wedge^{\bf L}\F={\bf 1}.$$
\item 
An object $\E$ is $1$-dimensional if it is either 
(1) negative finite dimensional and $\chi(\E)=-1$, 
or 
(2) positive finite dimensional and $\chi(\E)=1$.
\end{itemize}
\end{definition}

Note that if $\E$ is invertible,
then the dual $\D\E$ of $\E$ is an inverse $\F$ which is unique up to unique isomorphism.
The unit object ${\bf 1}$ is clearly $1$-dimensional; 
for a proof we refer to \cite[Example 8.2.5]{Kimura2}.
In fact, 
$\Alt^{2}({\bf 1})=0$ since the twist map on ${\bf 1}\wedge^{\bf L} {\bf 1}$ is the identity map.
\begin{lemma}
The following hold in $\MM(\KK)_{\Q}$.
\begin{itemize}
\item
An object $\E$ is invertible if and only if it is $1$-dimensional.
\item
If $\E$ is negative finite dimensional, 
then $\Sym^{-\chi(\E)}(\E)$ is invertible.
\item
If $\E$ is positive finite dimensional, 
then $\Alt^{\chi(\E)}(\E)$ is invertible.
\end{itemize}
\end{lemma}
\begin{proof} 
The first part is clear from \cite[8.2.6, 8.2.9]{Kimura2}. 
It remains to show that $\Sym^{-\chi(\E)}(\E)$ and $\Alt^{\chi(\E)}(\E)$ are $1$-dimensional objects. 
This follows from the easily verified formulas
\begin{equation*}
\chi\bigl(\Sym^{n}(\E)\bigr)=
\binom{\chi(\E)+n-1}{n}
\end{equation*}
and
\begin{equation*}
\chi\bigl(\Alt^{n}(\E)\bigr)= 
\binom{\chi(\E)}{n}
\end{equation*}
found in \cite[\S3]{Andre} and \cite[\S7.2]{Deligne1}. 
\end{proof}

Next we define the determinant of a finite dimensional rational motive in analogy with determinants of algebro-geometric motives appearing in 
\cite[Definition 2]{Kahn} and \cite[Definition 8.4.3]{Kimura2}.
\begin{definition}
If $\E$ is a finite dimensional rational motive, 
define the determinant of $\E$ by 
\begin{equation*}
\ddet(\E)\equiv
\Alt^{\chi_{+}(\E)}(\E_{+})\wedge^{\bf L} 
\D\Sym^{-\chi_{-}(\E)}(\E_{-}).
\end{equation*}
The determinant of $\E$ is well-defined up to isomorphism.
\end{definition}

According to the combinatorics of the Littlewood-Richardson numbers \cite[I9]{Macdonald} the following identities hold for the determinant,
cf.~\cite[\S3]{Andre}, \cite[\S1]{Deligne1}.
\begin{proposition}
Suppose $\E$ and $\F$ are finite dimensional objects of $\MM(\KK)_{\Q}$.
Then
\begin{itemize}
\item 
$\ddet(\E\oplus \F)=
\ddet(\E)\ddet(\F)$.
\item 
$\ddet(\E\wedge^{\bf L} \F)= 
\ddet(\E)^{\chi(\F)}\wedge^{\bf L} \ddet(\F)^{\chi(\E)}$.
\item
$\ddet(\D\E)=
\D\ddet(\E)$.
\item 
$\ddet\bigl(\Alt^{n}(\E)\bigr)= 
\ddet(\E)^{r}$, where  $r=n\binom{\chi(\E)}{n}/\chi(\E)$.
\item 
$\ddet\bigl(\Sym^{n}(\E)\bigr)= 
\ddet(\E)^{s}$, where $s=n\binom{\chi(\E)+n-1}{n}/\chi(\E)$.
\end{itemize}
\end{proposition}

\begin{lemma} 
\label{lemma:symandaltofduals}
\begin{itemize}
\item
If $\E$ is negative finite dimensional, 
there is an isomorphism
\begin{equation*}
\Sym^{n}(\D\E)\simeq \Sym^{-\chi(\E)-n}(\E)\wedge^{\bf L}\D\ddet(\E)
\end{equation*}
for all $n\in [0,-\chi(\E)]$.
\item 
If $\E$ is positive finite dimensional, 
there is an isomorphism
\begin{equation*}
\Alt^{n}(\D\E)= 
\Alt^{\chi(\E)-n}(\E)\wedge^{\bf L}\D\ddet(\E)
\end{equation*}
for all $n\in [0,\chi(\E)]$.
\end{itemize}
\end{lemma}
\begin{proof} 
Note that $\Alt^{n}(\D\E)$ is isomorphic to $\D\Alt^{n}(\E)$.
Using the evident map
\begin{equation}
\label{evidentmap}
\xymatrix{
\Alt^{n}(\E)\wedge^{\bf L} \Alt^{\chi(\E)-n}(\E)\ar[r] & \ddet(\E) }
\end{equation}
we get 
\begin{equation}
\label{evidentmap1}
\xymatrix{
\Alt^{n}(\E)\wedge^{\bf L} \bigl(\Alt^{\chi(\E)-n}(\E)\wedge^{\bf L}\D\ddet(\E)\bigr)\ar[r] & 
{\bf 1}. }
\end{equation}
Likewise, 
replacing $\E$ by its dual in (\ref{evidentmap}) yields 
\begin{equation}
\label{evidentmap2}
\xymatrix{
{\bf 1}\ar[r] & 
\bigl(\Alt^{\chi(\E)-n}(\E)\wedge^{\bf L} \D\ddet(\E)\bigr)\wedge^{\bf L}\Alt^{n}(\E). }
\end{equation}
The maps in (\ref{evidentmap1}) and (\ref{evidentmap2}) satisfy the Dold-Puppe duality axioms in \cite{DP}, cf.~\cite{Kahn}.
The proof for the symmetric powers is entirely similar. 
\end{proof}

We are ready to formulate and prove a functional equation for zeta functions of finite dimensional rational motives.
The result follows using the same steps as in the proof of the main result in \cite{Kahn}.
\begin{theorem}
Suppose $\E$ is finite dimensional.
Then the zeta functions of $\E$ and its dual are related by the functional equation
\begin{equation*}
\zeta_{\D\E}(t^{-1})=
(-1)^{\chi_{+}(\E)}\ddet(\E) t^{\chi(\E)}\zeta_{\E}(t).
\end{equation*}
\end{theorem}
\begin{proof} 
We may assume $\E$ is negative or positive finite dimensional since the zeta function is an additive invariant of the triangulated structure 
on $\MM(\KK)_{\Q}$.
In what follows we use that symmetric powers and wedge powers define opposite special $\lambda$-structures on $K_{0}(\MM(\KK)^{\ccc}_{\Q})$.
\begin{itemize}
\item
If $\E$ is negative finite dimensional, then
\begin{equation*}
\zeta_{\D\E}(t^{-1})=
\sum_{n=0}^{-\chi(\E)}\Sym^{n}(\D\E)t^{-n}.
\end{equation*}
By Lemma \ref{lemma:symandaltofduals}, 
the sum equals
\begin{align*}
[\D\ddet(\E)]\sum_{n=0}^{-\chi(\E)} [\Sym^{-\chi(\E)-n}(\E)] t^{-n} 
& = [\D\ddet(\E)]\sum_{n=0}^{-\chi(\E)} [\Sym^{n}(\E)] t^{n+\chi(\E)}\\
& = [\D\ddet(\E)]t^{\chi(\E)}\sum_{n\geq 0} [\Sym^{n}(\E)] t^{n}\\
& = [\D\ddet(\E)]t^{\chi(\E)}\zeta_{\E}(t).
\end{align*}
\item
If $\E$ is positive finite dimensional, then
\begin{equation*}
\zeta_{\D\E}(t^{-1})^{-1}= 
\sum_{n=0}^{\chi(\E)} [\Alt^{n}(\D\E)] (-t)^{-n}.
\end{equation*}
By Lemma \ref{lemma:symandaltofduals}, 
the sum equals
\begin{align*}
[\D\ddet(\E)]\sum_{n=0}^{\chi(\E)} [\Alt^{\chi(\E)-n}(\E)] (-t)^{-n}
& = [\D\ddet(\E)]\sum_{n=0}^{\chi(\E)} [\Alt^{n}(\E)] (-t)^{n-\chi(\E)}\\
& = [\D\ddet(\E)](-t)^{-\chi(\E)}\sum_{n=0}^{\chi(\E)} [\Alt^{n}(\E)] (-t)^{n}\\
& = [\D\ddet(\E)](-t)^{-\chi(\E)}\zeta_{\E}(t)^{-1}.
\end{align*}
\end{itemize}
It remains to note that $\E$ and its dual $\D\E$ have the same sign.
\end{proof}
\newpage

\section{The slice filtration}
\label{section:theslicefiltration}
In this section we construct a sequence of full triangulated subcategories 
\begin{equation}
\label{slicefiltration}
\cdots\subseteq 
\Sigma_{C}^{1}\SH^{\ast,\eff}\subseteq
\SH^{\ast,\eff}\subseteq
\Sigma_{C}^{-1}\SH^{\ast,\eff}\subseteq
\cdots
\end{equation}
of the stable $\CC^{\ast}$-homotopy category $\SH^{\ast}$.
Here, 
placed in degree zero is the smallest triangulated subcategory $\SH^{\ast,\eff}$ of $\SH^{\ast}$ that is closed under direct sums and contains every  
suspension spectrum $\Sigma^{\infty}_{C}\X$,
but none of the corresponding desuspension spectra $\Sigma^{-n}_{C}\Sigma^{\infty}_{C}\X$ for any $n\geq 1$.
We shall refer to $\SH^{\ast,\eff}$ as the effective stable $\CC^{\ast}$-homotopy category.
If $q$ is an integer, 
we define the category $\Sigma_{C}^{q}\SH^{\ast,\eff}$ as the smallest triangulated full subcategory of that is closed under direct sums and contains 
for all $t-m\geq q$ the $\CC^{\ast}$-spectra of the form
\begin{equation}
\label{levelcompactgenerators}
\Fr_{m}\bigl(S^{s}\otimes C_{0}(\R^{t})\otimes \X\bigr).
\end{equation}
With these definitions we deduce the slice filtration (\ref{slicefiltration}) which can be viewed as a Postnikov tower. 
The analogous construction in motivic homotopy theory is due to Voevodsky \cite{VVopenproblems}.
Much of the current research in the motivic theory evolves around his tantalizing set of conjectures concerning the slice filtration.
\vspace{0.1in}

In order to make precise the meaning of ``filtration'' in the above we note that the smallest triangulated subcategory of $\SH^{\ast}$ that contains 
$\Sigma_{C}^{q}\SH^{\ast,\eff}$ for every integer $q$ coincides with $\SH^{\ast}$ since the latter is a compactly generated triangulated category. 
Likewise, 
at each level of the slice filtration we have the following result. 
\begin{lemma}
\label{lemma:levelcompactlygenerated}
The category $\Sigma_{C}^{q}\SH^{\ast,\eff}$ is a compactly generated triangulated category with the set of compact generators given by 
(\ref{levelcompactgenerators}).
Thus a map $f\colon \E\rightarrow\F$ in $\Sigma_{C}^{q}\SH^{\ast,\eff}$ is an isomorphism if and only if there is a naturally induced isomorphism
\begin{equation*}
\xymatrix{
\Sigma_{C}^{q}\SH^{\ast,\eff}\bigl(\Fr_{m}\bigl(S^{s}\otimes C_{0}(\R^{t})\otimes \X\bigr),\E\bigr)\ar[r] &
\Sigma_{C}^{q}\SH^{\ast,\eff}\bigl(\Fr_{m}\bigl(S^{s}\otimes C_{0}(\R^{t})\otimes \X\bigr),\F\bigr) }
\end{equation*}
for every compact generator $\Fr_{m}\bigl(S^{s}\otimes C_{0}(\R^{t})\otimes \X\bigr)$.
\end{lemma}

The ``effective'' $s$-stable $\CC^{\ast}$-homotopy category $\SH^{\ast,\eff}_{s}$ is defined similarly to $\SH^{\ast,\eff}$ by replacing $C$-suspension 
spectra with $S^{1}$-suspension spectra.
\vspace{0.1in}

We are ready to discuss certain functors relating $\Sigma_{C}^{q}\SH^{\ast,\eff}$ and $\SH^{\ast}$.
\begin{proposition}
For every integer $q$ the full inclusion functor 
\begin{equation*}
\xymatrix{
i_{q}\colon 
\Sigma_{C}^{q}\SH^{\ast,\eff}\ar[r] &
\SH^{\ast} }
\end{equation*}
acquires an exact right adjoint
 \begin{equation*}
\xymatrix{
r_{q}\colon 
\SH^{\ast}\ar[r] &
\Sigma_{C}^{q}\SH^{\ast,\eff} }
\end{equation*}
such that the following hold:
\begin{itemize}
\item
The unit of the adjunction $\id\rightarrow r_{q}\circ i_{q}$ is an isomorphism.
\item
By defining $f_{q}\equiv i_{q}\circ r_{q}$ there exists a natural transformation $f_{q+1}\rightarrow f_{q}$ and $f_{q+1}=f_{q+1}\circ f_{q}$.
\end{itemize}
\end{proposition}
\begin{proof}
Existence of the right adjoint $r_{q}$ follows by combining Lemma \ref{lemma:levelcompactlygenerated} with a general result due to Neeman 
\cite[Theorem 4.1]{Neeman} since the inclusion functor $i_{q}$ is clearly exact and preserves coproducts.
The unit of the adjunction is an isomorphism because $i_{q}$ is a full embedding, 
while the last claim follows by contemplating the diagram:
\begin{equation*}
\xymatrix{
\SH^{\ast}\ar[r]^-{r_{q+1}} \ar@{=}[d] &
\Sigma_{C}^{q+1}\SH^{\ast,\eff}\ar@{>->}[r]^-{i_{q+1}} \ar@{>->}[d] & 
\SH^{\ast}\ar@{=}[d] & \\
\SH^{\ast}\ar[r]^-{r_{q}} &
\Sigma_{C}^{q}\SH^{\ast,\eff}\ar@{>->}[r]^-{i_{q}} & 
\SH^{\ast} }
\end{equation*}
\end{proof}

Next we discuss some properties of $f_{q}$ and the counit of the adjunction.
\begin{lemma}
For every integer $q$ and map $f\colon \E\rightarrow\F$ in $\SH^{\ast}$ the induced map $f_{q}\colon f_{q}\E\rightarrow f_{q}\F$ is an isomorphism 
in $\SH^{\ast}$ if and only if there is a naturally induced isomorphism
\begin{equation*}
\xymatrix{
\Sigma_{C}^{q}\SH^{\ast,\eff}\bigl(\Fr_{m}\bigl(S^{s}\otimes C_{0}(\R^{t})\otimes \X\bigr),\E\bigr)\ar[r] &
\Sigma_{C}^{q}\SH^{\ast,\eff}\bigl(\Fr_{m}\bigl(S^{s}\otimes C_{0}(\R^{t})\otimes \X\bigr),\F\bigr) }
\end{equation*}
for every compact generator $\Fr_{m}\bigl(S^{s}\otimes C_{0}(\R^{t})\otimes \X\bigr)$.
\end{lemma}
\begin{proof}
This follows from Lemma \ref{lemma:levelcompactlygenerated}.
\end{proof}
\begin{lemma}
For every integer $q$ the counit of the adjunction $f_{q}\rightarrow\id$ evaluated at $\E$ yields an isomorphism 
\begin{equation*}
\xymatrix{
\SH^{\ast}\bigl(\Fr_{m}\bigl(S^{s}\otimes C_{0}(\R^{t})\otimes \X\bigr),f_{q}\E\bigr)\ar[r] &
\SH^{\ast}\bigl(\Fr_{m}\bigl(S^{s}\otimes C_{0}(\R^{t})\otimes \X\bigr),\E\bigr) }
\end{equation*}
for every compact generator $\Fr_{m}\bigl(S^{s}\otimes C_{0}(\R^{t})\otimes \X\bigr)$ of $\Sigma_{C}^{q}\SH^{\ast,\eff}$.
\end{lemma}
\begin{proof}
This follows by using the canonical isomorphism
\begin{equation*}
\xymatrix{
\SH^{\ast}\bigl(\Fr_{m}\bigl(S^{s}\otimes C_{0}(\R^{t})\otimes \X\bigr),\F\bigr)=
\SH^{\ast}\bigl(i_{q}\Fr_{m}\bigl(S^{s}\otimes C_{0}(\R^{t})\otimes \X\bigr),\F\bigr) }
\end{equation*}
and the adjunction between $i_{q}$ and $r_{q}$.
\end{proof}

\begin{theorem}
\label{theorem:slicefunctor}
For every integer $q$ there exists an exact functor 
\begin{equation*}
\xymatrix{
s_{n}\colon \SH^{\ast}\ar[r] & \SH^{\ast}. }
\end{equation*}
There exist natural transformations $f_{q}\rightarrow s_{q}$ and $s_{q}\rightarrow \Sigma_{S^{1}}f_{q+1}$ such that the following hold:
\begin{itemize}
\item
For every $\E$ there exists a distinguished triangle in $\SH^{\ast}$
\begin{equation*}
\xymatrix{
f_{q+1}\E\ar[r] &
f_{q}\E\ar[r] &
s_{q}\E\ar[r] &
\Sigma_{S^{1}}f_{q+1}\E. }
\end{equation*}
\item
The functor $s_{q}$ takes values in the full subcategory $\Sigma_{C}^{q}\SH^{\ast,\eff}$ of $\SH^{\ast}$.
\item
Every map in $\SH^{\ast}$ from an object of $\Sigma_{C}^{q+1}\SH^{\ast,\eff}$ to $s_{q}\E$ is trivial.
\item
The above properties characterizes the exact functor $s_{q}$ up to canonical isomorphism.
\end{itemize}
\end{theorem}
\begin{proof}
Compact generatedness of the triangulated categories $\Sigma_{C}^{q+1}\SH^{\ast,\eff}$ and $\Sigma_{C}^{q}\SH^{\ast,\eff}$ imply the above according to 
\cite[Propositions 9.1.8, 9.1.19]{Neeman:Triangulatedcategories} and standard arguments.
\end{proof}

\begin{definition}
\label{definition:slices}
The $n$th slice of $\E$ is $s_{n}\E$.
\end{definition}
\begin{remark}
We note that $s_{n}\E$ is unique up to unique isomorphism.
If $\E\in \Sigma_{C}^{n}\SH^{\ast,\eff}$ and $q\leq n$, 
then $f_{q}\E=\E$ and the $q$th slice $s_{q}\E$ of $\E$ is trivial for all $q<n$.
\end{remark}

The functor $s_{q}$ is compatible with the smash product in the sense that for $\E$ and $\F$ there is a natural map
\begin{equation*}
\xymatrix{
s_{q}(\E)\wedge^{\bf L} s_{q'}(\F) 
\ar[r] & 
s_{q+q'}(\E\wedge^{\bf L}\F).}
\end{equation*}
In particular, there is a map
\begin{equation*}
\xymatrix{
s_{0}({\bf 1})\wedge^{\bf L} s_{q}(E)\ar[r] & 
s_{q}(E). }
\end{equation*}
This shows the zero-slice of the sphere spectrum has an important property in this setup.

\begin{lemma}
For every integer $q$ and map $f\colon \E\rightarrow\F$ in $\SH^{\ast}$ the induced map $s_{q}\colon s_{q}\E\rightarrow s_{q}\F$ is an isomorphism 
in $\SH^{\ast}$ if and only if there is a naturally induced isomorphism
\begin{equation*}
\xymatrix{
\SH^{\ast}\bigl(\Fr_{m}\bigl(S^{s}\otimes C_{0}(\R^{t})\otimes \X\bigr),\E\bigr)\ar[r] &
\SH^{\ast}\bigl(\Fr_{m}\bigl(S^{s}\otimes C_{0}(\R^{t})\otimes \X\bigr),\F\bigr) }
\end{equation*}
for every compact generator $\Fr_{t-q}\bigl(S^{s}\otimes C_{0}(\R^{t})\otimes \X\bigr)$.
\end{lemma}
\vspace{0.1in}

The distinguished triangles in $\SH^{\ast}$ 
\begin{equation*}
\xymatrix{
f_{q+1}\E\ar[r] &
f_{q}\E\ar[r] &
s_{q}\E\ar[r] &
\Sigma_{S^{1}}f_{q+1}\E }
\end{equation*}
induce in a standard way an exact couple and a spectral sequence with input the groups $\pi_{p,n}(s_{q}\E)$ where the $r$th differential go from 
tridegree $(p,n,q)$ to $(p-1,n,q+r)$.
It would be interesting to work out concrete examples of such spectral sequences.

\vspace{0.5in}

\begin{center}
Department of Mathematics, University of Oslo, Norway.\\
e-mail: paularne@math.uio.no
\end{center}

\newpage


\begin{thebibliography}{10}

\bibitem{AR:book}
J.~Ad{\'a}mek and J.~Rosick{\'y}.
\newblock {\em Locally presentable and accessible categories}, volume 189 of
  {\em London Mathematical Society Lecture Note Series}.
\newblock Cambridge University Press, Cambridge, 1994.

\bibitem{Andre}
Y.~Andr{\'e}.
\newblock Motifs de dimension finie (d'apr\`es {S}.-{I}. {K}imura, {P}.
  {O}'{S}ullivan{$\dots$}).
\newblock {\em Ast\'erisque}, (299):Exp. No. 929, viii, 115--145, 2005.
\newblock S{\'e}minaire Bourbaki. Vol. 2003/2004.

\bibitem{AKOS}
Y.~Andr{\'e}, B.~Kahn, and P.~O'Sullivan.
\newblock Nilpotence, radicaux et structures mono\"\i dales.
\newblock {\em Rend. Sem. Mat. Univ. Padova}, 108:107--291, 2002.

\bibitem{AT}
M.~F. Atiyah and D.~O. Tall.
\newblock Group representations, {$\lambda $}-rings and the {$J$}-homomorphism.
\newblock {\em Topology}, 8:253--297, 1969.

\bibitem{Barwick}
C.~Barwick.
\newblock On (enriched) left {B}ousfield localizations of model categories.
\newblock Preprint, arXiv 0708.2067.

\bibitem{Biedermann}
G.~Biedermann.
\newblock {$L$}-stable functors.
\newblock Preprint, arXiv 0704.2576.

\bibitem{BCR}
G.~Biedermann, B.~Chorny, and O.~R{\"o}ndigs.
\newblock Calculus of functors and model categories.
\newblock {\em Adv. Math.}, 214(1):92--115, 2007.

\bibitem{BlackadarKbook}
B.~Blackadar.
\newblock {\em {$K$}-theory for operator algebras}, volume~5 of {\em
  Mathematical Sciences Research Institute Publications}.
\newblock Cambridge University Press, Cambridge, second edition, 1998.

\bibitem{Blackadar}
B.~Blackadar.
\newblock {\em Operator algebras}, volume 122 of {\em Encyclopaedia of
  Mathematical Sciences}.
\newblock Springer-Verlag, Berlin, 2006.
\newblock Theory of $C{\sp{*}}$-algebras and von Neumann algebras, Operator
  Algebras and Non-commutative Geometry, III.

\bibitem{Borceux:Handbook2}
F.~Borceux.
\newblock {\em Handbook of categorical algebra. 2}, volume~51 of {\em
  Encyclopedia of Mathematics and its Applications}.
\newblock Cambridge University Press, Cambridge, 1994.
\newblock Categories and structures.

\bibitem{BousfieldFriedlander}
A.~K. Bousfield and E.~M. Friedlander.
\newblock Homotopy theory of {$\Gamma $}-spaces, spectra, and bisimplicial
  sets.
\newblock In {\em Geometric applications of homotopy theory ({P}roc. {C}onf.,
  {E}vanston, {I}ll., 1977), {II}}, volume 658 of {\em Lecture Notes in Math.},
  pages 80--130. Springer, Berlin, 1978.

\bibitem{Cisinski:presheavesasmodelsforhomotopytypes}
D.~C. Cisinski.
\newblock Les pr\'efaisceaux comme mod\`eles des types d'homotopie.
\newblock {\em Ast\'erisque}, (308):xxiv+390, 2006.

\bibitem{Connes:noncommutative}
A.~Connes.
\newblock {\em Noncommutative geometry}.
\newblock Academic Press Inc., San Diego, CA, 1994.

\bibitem{CCM}
A.~Connes, C.~Consani, and M.~Marcolli.
\newblock Noncommutative geometry and motives: the thermodynamics of
  endomotives.
\newblock {\em Adv. Math.}, 214(2):761--831, 2007.

\bibitem{CH:Etheory}
A.~Connes and N.~Higson.
\newblock D\'eformations, morphismes asymptotiques et {$K$}-th\'eorie
  bivariante.
\newblock {\em C. R. Acad. Sci. Paris S\'er. I Math.}, 311(2):101--106, 1990.

\bibitem{Cuntz:bivariant}
J.~Cuntz.
\newblock A general construction of bivariant {$K$}-theories on the category of
  {$C\sp \ast$}-algebras.
\newblock In {\em Operator algebras and operator theory (Shanghai, 1997)},
  volume 228 of {\em Contemp. Math.}, pages 31--43. Amer. Math. Soc.,
  Providence, RI, 1998.

\bibitem{Day:closedfunctorsI}
B.~Day.
\newblock On closed categories of functors.
\newblock In {\em Reports of the Midwest Category Seminar, IV}, Lecture Notes
  in Mathematics, Vol. 137, pages 1--38. Springer, Berlin, 1970.

\bibitem{Deligne1}
P.~Deligne.
\newblock Cat\'egories tannakiennes.
\newblock In {\em The {G}rothendieck {F}estschrift, {V}ol.\ {II}}, volume~87 of
  {\em Progr. Math.}, pages 111--195. Birkh\"auser Boston, Boston, MA, 1990.

\bibitem{Deligne2}
P.~Deligne.
\newblock Cat\'egories tensorielles.
\newblock {\em Mosc. Math. J.}, 2(2):227--248, 2002.
\newblock Dedicated to Yuri I. Manin on the occasion of his 65th birthday.

\bibitem{DP}
A.~Dold and D.~Puppe.
\newblock Duality, trace, and transfer.
\newblock In {\em Proceedings of the {I}nternational {C}onference on
  {G}eometric {T}opology ({W}arsaw, 1978)}, pages 81--102, Warsaw, 1980. PWN.

\bibitem{DS:Ktheory}
D.~Dugger and B.~E. Shipley.
\newblock {$K$}-theory and derived equivalences.
\newblock {\em Duke Math. J.}, 124(3):587--617, 2004.

\bibitem{DRO:general}
B.~I. Dundas, O.~R{\"o}ndigs, and P.~A. {\O}stv{\ae}r.
\newblock Enriched functors and stable homotopy theory.
\newblock {\em Doc. Math.}, 8:409--488 (electronic), 2003.

\bibitem{DRO:motivic}
B.~I. Dundas, O.~R{\"o}ndigs, and P.~A. {\O}stv{\ae}r.
\newblock Motivic functors.
\newblock {\em Doc. Math.}, 8:489--525 (electronic), 2003.

\bibitem{Dwyer:homologystability}
W.~G. Dwyer.
\newblock Twisted homological stability for general linear groups.
\newblock {\em Ann. of Math. (2)}, 111(2):239--251, 1980.

\bibitem{DS:Modelcategories}
W.~G. Dwyer and J.~Spali{\'n}ski.
\newblock Homotopy theories and model categories.
\newblock In {\em Handbook of algebraic topology}, pages 73--126.
  North-Holland, Amsterdam, 1995.

\bibitem{Eisenbud}
D.~Eisenbud.
\newblock {\em Commutative algebra, with a view toward algebraic geometry},
  volume 150 of {\em Graduate Texts in Mathematics}.
\newblock Springer-Verlag, New York, 1995.

\bibitem{GM:Homologicalalgebra}
S.~I. Gelfand and Y.~I. Manin.
\newblock {\em Methods of homological algebra}.
\newblock Springer Monographs in Mathematics. Springer-Verlag, Berlin, second
  edition, 2003.

\bibitem{GJ:Modelcategories}
P.~G. Goerss and J.~F. Jardine.
\newblock {\em Simplicial homotopy theory}, volume 174 of {\em Progress in
  Mathematics}.
\newblock Birkh\"auser Verlag, Basel, 1999.

\bibitem{Grothendieck}
A.~Grothendieck.
\newblock La th\'eorie des classes de {C}hern.
\newblock {\em Bull. Soc. Math. France}, 86:137--154, 1958.

\bibitem{Guletski2}
V.~Guletski{\u\i}.
\newblock Zeta functions in triangulated categories.
\newblock Preprint, arXiv 0605040.

\bibitem{Guletski1}
V.~Guletski{\u\i}.
\newblock Finite-dimensional objects in distinguished triangles.
\newblock {\em J. Number Theory}, 119(1):99--127, 2006.

\bibitem{Higson:KK}
N.~Higson.
\newblock A characterization of {$KK$}-theory.
\newblock {\em Pacific J. Math.}, 126(2):253--276, 1987.

\bibitem{Hirschhorn:Modelcategories}
P.~S. Hirschhorn.
\newblock {\em Model categories and their localizations}, volume~99 of {\em
  Mathematical Surveys and Monographs}.
\newblock American Mathematical Society, Providence, RI, 2003.

\bibitem{Hovey:monoidsandalgebras}
M.~Hovey.
\newblock Monoidal model categories.
\newblock Preprint, arXiv 9803002.

\bibitem{Hovey:Modelcategories}
M.~Hovey.
\newblock {\em Model categories}, volume~63 of {\em Mathematical Surveys and
  Monographs}.
\newblock American Mathematical Society, Providence, RI, 1999.

\bibitem{Hovey:spectra}
M.~Hovey.
\newblock Spectra and symmetric spectra in general model categories.
\newblock {\em J. Pure Appl. Algebra}, 165(1):63--127, 2001.

\bibitem{HSS:SS}
M.~Hovey, B.~E. Shipley, and J.~Smith.
\newblock Symmetric spectra.
\newblock {\em J. Amer. Math. Soc.}, 13(1):149--208, 2000.

\bibitem{JK}
G.~James and A.~Kerber.
\newblock {\em The representation theory of the symmetric group}, volume~16 of
  {\em Encyclopedia of Mathematics and its Applications}.
\newblock Addison-Wesley Publishing Co., Reading, Mass., 1981.
\newblock With a foreword by P. M. Cohn, With an introduction by Gilbert de B.
  Robinson.

\bibitem{Jardine:representabilityformodelcategories}
J.~F. Jardine.
\newblock Representability for model categories.
\newblock Preprint.

\bibitem{Jardine:MSS}
J.~F. Jardine.
\newblock Motivic symmetric spectra.
\newblock {\em Doc. Math.}, 5:445--553 (electronic), 2000.

\bibitem{Jardine:Categoricalhomotopytheory}
J.~F. Jardine.
\newblock Categorical homotopy theory.
\newblock {\em Homology, Homotopy Appl.}, 8(1):71--144 (electronic), 2006.

\bibitem{JJ:modelstructure}
M.~Joachim and M.~W. Johnson.
\newblock Realizing {K}asparov's {$KK$}-theory groups as the homotopy classes
  of maps of a {Q}uillen model category.
\newblock In {\em An alpine anthology of homotopy theory}, volume 399 of {\em
  Contemp. Math.}, pages 163--197. Amer. Math. Soc., Providence, RI, 2006.

\bibitem{Kahn}
B.~Kahn.
\newblock Motivic zeta functions of motives.
\newblock Preprint, arXiv 0606424.

\bibitem{Kandelaki:KKRep}
T.~Kandelaki.
\newblock Algebraic {$K$}-theory of {F}redholm modules and {$KK$}-theory.
\newblock {\em J. Homotopy Relat. Struct.}, 1(1):195--218 (electronic), 2006.

\bibitem{Kapranov}
M.~Kapranov.
\newblock The elliptic curve in the $s$-duality theory and {E}isenstein series
  for {K}ac-{M}oody groups.
\newblock Preprint, arXiv 0001005.

\bibitem{Kasparov:Hilbertmodules}
G.~G. Kasparov.
\newblock Hilbert {$C\sp{\ast} $}-modules: theorems of {S}tinespring and
  {V}oiculescu.
\newblock {\em J. Operator Theory}, 4(1):133--150, 1980.

\bibitem{Kasparov:KK1}
G.~G. Kasparov.
\newblock The operator {$K$}-functor and extensions of {$C\sp{\ast}$}-algebras.
\newblock {\em Izv. Akad. Nauk SSSR Ser. Mat.}, 44(3):571--636, 719, 1980.

\bibitem{Kasparov:KK2}
G.~G. Kasparov.
\newblock Equivariant {$KK$}-theory and the {N}ovikov conjecture.
\newblock {\em Invent. Math.}, 91(1):147--201, 1988.

\bibitem{KN}
B.~Keller and A.~Neeman.
\newblock The connection between {M}ay's axioms for a triangulated tensor
  product and {H}appel's description of the derived category of the quiver
  {$D\sb 4$}.
\newblock {\em Doc. Math.}, 7:535--560 (electronic), 2002.

\bibitem{Kelly}
G.~M. Kelly.
\newblock Basic concepts of enriched category theory.
\newblock {\em Repr. Theory Appl. Categ.}, (10):vi+137 pp. (electronic), 2005.
\newblock Reprint of the 1982 original [Cambridge Univ. Press, Cambridge;
  MR0651714].

\bibitem{Kimura1}
S.-I. Kimura.
\newblock Chow groups are finite dimensional, in some sense.
\newblock {\em Math. Ann.}, 331(1):173--201, 2005.

\bibitem{Kimura2}
S.-I. Kimura.
\newblock A note on finite dimensional motives.
\newblock In {\em Algebraic cycles and motives. {V}ol. 2}, volume 344 of {\em
  London Math. Soc. Lecture Note Ser.}, pages 203--213. Cambridge Univ. Press,
  Cambridge, 2007.

\bibitem{Lydakis}
M.~Lydakis.
\newblock Simplicial functors and stable homotopy theory.
\newblock Preprint, Hopf Topology Archive.

\bibitem{Macdonald}
I.~G. Macdonald.
\newblock {\em Symmetric functions and {H}all polynomials}.
\newblock Oxford Mathematical Monographs. The Clarendon Press Oxford University
  Press, New York, second edition, 1995.
\newblock With contributions by A. Zelevinsky, Oxford Science Publications.

\bibitem{May:additivityoftraces}
J.~P. May.
\newblock The additivity of traces in triangulated categories.
\newblock {\em Adv. Math.}, 163(1):34--73, 2001.

\bibitem{Mazza}
C.~Mazza.
\newblock Schur functors and motives.
\newblock {\em $K$-Theory}, 33(2):89--106, 2004.

\bibitem{Meyer}
R.~Meyer.
\newblock {\em Local and analytic cyclic homology}, volume~3 of {\em EMS Tracts
  in Mathematics}.
\newblock European Mathematical Society (EMS), Z\"urich, 2007.

\bibitem{Milnor:realization}
J.~Milnor.
\newblock The geometric realization of a semi-simplicial complex.
\newblock {\em Ann. of Math. (2)}, 65:357--362, 1957.

\bibitem{MV:IHES}
F.~Morel and V.~Voevodsky.
\newblock {${\bf A}\sp 1$}-homotopy theory of schemes.
\newblock {\em Inst. Hautes \'Etudes Sci. Publ. Math.}, (90):45--143 (2001),
  1999.

\bibitem{Neeman}
A.~Neeman.
\newblock The {G}rothendieck duality theorem via {B}ousfield's techniques and
  {B}rown representability.
\newblock {\em J. Amer. Math. Soc.}, 9(1):205--236, 1996.

\bibitem{Neeman:Triangulatedcategories}
A.~Neeman.
\newblock {\em Triangulated categories}, volume 148 of {\em Annals of
  Mathematics Studies}.
\newblock Princeton University Press, Princeton, NJ, 2001.

\bibitem{Ostvar:noncommutativemotives}
P.~A. {\O}stv{\ae}r.
\newblock Noncommutative motives.
\newblock In preparation.

\bibitem{Pedersen:pullbackpushout}
G.~K. Pedersen.
\newblock Pullback and pushout constructions in {$C\sp *$}-algebra theory.
\newblock {\em J. Funct. Anal.}, 167(2):243--344, 1999.

\bibitem{Puschnigg}
M.~Puschnigg.
\newblock Excision in cyclic homology theories.
\newblock {\em Invent. Math.}, 143(2):249--323, 2001.

\bibitem{Quillen:Homotopicalalgebra}
D.~G. Quillen.
\newblock {\em Homotopical algebra}.
\newblock Lecture Notes in Mathematics, No. 43. Springer-Verlag, Berlin, 1967.

\bibitem{Roendigs}
O.~R{\"o}ndigs.
\newblock Algebraic ${K}$-theory of spaces in terms of spectra.
\newblock Diplomarbeit, University of Bielefeld.

\bibitem{RO1}
O.~R{\"o}ndigs and P.~A. {\O}stv{\ae}r.
\newblock Motives and modules over motivic cohomology.
\newblock {\em C. R. Math. Acad. Sci. Paris}, 342(10):751--754, 2006.

\bibitem{RO2}
O.~R{\"o}ndigs and P.~A. {\O}stv{\ae}r.
\newblock Modules over motivic cohomology.
\newblock {\em Adv. Math.}, 219(2):689--727, 2008.

\bibitem{Rosenberg:noncommutativetopology}
J.~Rosenberg.
\newblock The role of {$K$}-theory in noncommutative algebraic topology.
\newblock In {\em Operator algebras and $K$-theory (San Francisco, Calif.,
  1981)}, volume~10 of {\em Contemp. Math.}, pages 155--182. Amer. Math. Soc.,
  Providence, R.I., 1982.

\bibitem{Rosenberg:handbook}
J.~Rosenberg.
\newblock {$K$}-theory and geometric topology.
\newblock In {\em Handbook of {$K$}-theory. {V}ol. 1, 2}, pages 577--610.
  Springer, Berlin, 2005.

\bibitem{Rosicky:brownrepresentability}
J.~Rosick{\'y}.
\newblock Generalized {B}rown representability in homotopy categories.
\newblock {\em Theory Appl. Categ.}, 14:no. 19, 451--479 (electronic), 2005.

\bibitem{Sagave:diplom}
S.~Sagave.
\newblock On the algebraic {$K$}-theory of model categories.
\newblock {\em J. Pure Appl. Algebra}, 190(1-3):329--340, 2004.

\bibitem{Schochet:III}
C.~Schochet.
\newblock Topological methods for {$C\sp{\ast} $}-algebras. {III}. {A}xiomatic
  homology.
\newblock {\em Pacific J. Math.}, 114(2):399--445, 1984.

\bibitem{Schwede:SS}
S.~Schwede.
\newblock An untitled book project about symmetric spectra.
\newblock In preparation.

\bibitem{SS:monoidal}
S.~Schwede and B.~E. Shipley.
\newblock Algebras and modules in monoidal model categories.
\newblock {\em Proc. London Math. Soc. (3)}, 80(2):491--511, 2000.

\bibitem{Spitzweck:semimodelstructures}
M.~Spitzweck.
\newblock Operads, algebras and modules in general model categories.
\newblock Preprint, arXiv 0101102.

\bibitem{Stanculescu}
A.~E. Stanculescu.
\newblock Note on a theorem of {B}ousfield and {F}riedlander.
\newblock {\em Topology Appl.}, 155(13):1434--1438, 2008.

\bibitem{TT}
R.~W. Thomason and T.~Trobaugh.
\newblock Higher algebraic {$K$}-theory of schemes and of derived categories.
\newblock In {\em The Grothendieck Festschrift, Vol.\ III}, volume~88 of {\em
  Progr. Math.}, pages 247--435. Birkh\"auser Boston, Boston, MA, 1990.

\bibitem{VVopenproblems}
V.~Voevodsky.
\newblock Open problems in the motivic stable homotopy theory. {I}.
\newblock In {\em Motives, polylogarithms and {H}odge theory, {P}art {I}
  ({I}rvine, {CA}, 1998)}, volume~3 of {\em Int. Press Lect. Ser.}, pages
  3--34. Int. Press, Somerville, MA, 2002.

\bibitem{Voigt}
C.~Voigt.
\newblock Equivariant local cyclic homology and the equivariant
  {C}hern-{C}onnes character.
\newblock {\em Doc. Math.}, 12:313--359 (electronic), 2007.

\bibitem{Waldhausen:LNM1126}
F.~Waldhausen.
\newblock Algebraic {$K$}-theory of spaces.
\newblock In {\em Algebraic and geometric topology (New Brunswick, N.J.,
  1983)}, volume 1126 of {\em Lecture Notes in Math.}, pages 318--419.
  Springer, Berlin, 1985.

\bibitem{Weyl}
H.~Weyl.
\newblock {\em The classical groups}.
\newblock Princeton Landmarks in Mathematics. Princeton University Press,
  Princeton, NJ, 1997.
\newblock Their invariants and representations, Fifteenth printing, Princeton
  Paperbacks.

\end{thebibliography}
\end{document}